\documentclass[11pt, oneside]{amsart}
\usepackage[utf8]{inputenc}
\usepackage{mathrsfs}
\usepackage{bm}
\usepackage{amstext}
\usepackage{amsthm}
\usepackage{amssymb}
\usepackage[all]{xy}
\usepackage{multirow}
\usepackage{graphics}

\makeatletter
\numberwithin{equation}{section}
\numberwithin{figure}{section}
\theoremstyle{plain}
\newtheorem{theorem}{Theorem}[section]
  \newtheorem{lemma}[theorem]{Lemma}
   
  \newtheorem{proposition}[theorem]{Proposition}

  \newtheorem*{claim*}{Claim}
  \newtheorem{corollary}[theorem]{Corollary}
  \newtheorem{conj}[theorem]{Conjecture}
  \theoremstyle{definition}
  \newtheorem{definition}[theorem]{Definition}
  \theoremstyle{remark}
  \newtheorem{notation}[theorem]{Notation}
	\newtheorem*{acknowledgment*}{Acknowledgment}  
  \theoremstyle{remark}
  \newtheorem{remark}[theorem]{Remark}
  \newtheorem{example}[theorem]{Example}
 
\usepackage{a4wide}

\usepackage{xcolor}
\usepackage{version}

\excludeversion{NB}

\newcommand{\Aut}{\operatorname{Aut}}

\newcommand{\Mat}{\operatorname{Mat}}

\newcommand{\dbracket}[1]{[\![#1]\!]}

\newcommand{\Uv}{\mathcal{U}_v}

\newcommand{\wt}{\operatorname{wt}}

\newcommand{\up}{\mathrm{up}}

\newcommand{\Gup}{G^{\up}}
\newcommand{\dprod}{\mathop{\overrightarrow{\prod}}\limits}

\newcommand{\Image}{\operatorname{Im}}

\newcommand{\res}{\operatorname{res}}

\makeatother

\makeatletter
\ifcsname phantomsection\endcsname
    \newcommand*{\qrr@gobblenexttocentry}[5]{}
\else
    \newcommand*{\qrr@gobblenexttocentry}[4]{}
\fi
\newcommand*{\addsubsection}{%
    \addtocontents{toc}{\protect\qrr@gobblenexttocentry}%
    \subsection}
\makeatother
 \makeatletter
    
    \@addtoreset{equation}{section}
  \makeatother
\begin{document}
\title[Quantum Grothendieck ring isomorphisms]{Quantum Grothendieck ring isomorphisms, cluster algebras and Kazhdan-Lusztig algorithm}

\author{David HERNANDEZ}

\address[David HERNANDEZ]{Universit\'e Paris Diderot, Sorbonne Universit\'e, CNRS, Institut de Math\'ematiques de Jussieu - Paris Rive Gauche, IMJ-PRG, F-75013, Paris, FRANCE}

\email{david.hernandez@imj-prg.fr}
\author{Hironori OYA}
\address[Hironori OYA]{Department of Mathematical Sciences, Shibaura Institute of Technology, 307 Fukasaku, Minuma-ku, Saitama-shi, Saitama, 337-8570, JAPAN}

\email{hoya@shibaura-it.ac.jp}




\date{}
\begin{abstract}\label{abs}
We establish ring isomorphisms between quantum Grothendieck rings of certain remarkable monoidal categories $\mathcal{C}_{\mathscr{Q},\mathrm{B}_n}$ and $\mathcal{C}_{\mathcal{Q}, \mathrm{A}_{2n-1}}$ of finite-dimensional representations of quantum affine algebras of types $\mathrm{B}_n^{(1)}$ and $\mathrm{A}_{2n-1}^{(1)}$, respectively. Our proof relies in part on the corresponding quantum cluster algebra structures. Moreover, we prove that our isomorphisms specialize at $t = 1$ to the isomorphisms of (classical) Grothendieck 
rings obtained recently by Kashiwara, Kim and Oh by other methods. As a consequence, we prove a conjecture formulated by the first author in 2002 : the multiplicities of simple modules in standard modules in $\mathcal{C}_{\mathscr{Q},\mathrm{B}_n}$ are given by the specialization of certain analogues of Kazhdan-Lusztig polynomials and the coefficients of these polynomials are positive.
\end{abstract}
\maketitle
\tableofcontents 

\section{Introduction}\label{s:intro}

Cluster algebras can be realized in terms of various important monoidal categories, such as finite-dimensional representations of quantum affine algebras, perverse sheaves on quiver varieties, representations of quiver-Hecke algebras, equivariant perverse coherent sheaves on the affine Grassmannian... (see the very recent \cite{cw} and \cite{H:bou} for a review).

In this paper, inspired by this general framework of monoidal categorifications, we get new results on the representation theory of quantum affine algebras and on the positivity of certain corresponding Kazhdan-Lusztig polynomials.

Let $\mathcal{U}_q(\hat{\mathfrak{g}})$ be an untwisted quantum affine algebra of quantum parameter $q\in\mathbb{C}^{\times}$ which is not a root of unity. Consider $\mathcal{C}$ to be the monoidal category of finite-dimensional representations of the algebra $\mathcal{U}_q(\hat{\mathfrak{g}})$ (those representations are also regarded as representations of a quantum loop algebra $\mathcal{U}_q(\mathcal{L}\mathfrak{g})$). This category has been studied from several geometric, algebraic, combinatorial perspectives in connections to various fields, see \cite{mo, KKKO:mon, GT} for important recent developments and \cite{cwy} for a very recent point of view in the context of physics.

The category $\mathcal{C}$ is non-semisimple, and non-braided. It has a very intricate structure and many basic questions are still open. The simple objects in $\mathcal{C}$ have been
classified by Chari-Pressley in terms of Drinfeld polynomials \cite{CP:guide}, but the dimensions and characters of these simple objects are not known in general. 

The fundamental modules in $\mathcal{C}$ are distinguished simple modules whose classes generate its Grothendieck ring $K(\mathcal{C})$. (We do not distinguish the term ``module'' from ``representation'' in this paper.) For simply-laced types, Nakajima established in \cite{Nak:quiver} 
a remarkable Kazhdan-Lusztig algorithm to compute the multiplicity $P_{m,m'}$ of a simple module $L(m')$ in a 
standard module $M(m)$, that is, a tensor product of fundamental modules.
Here $m$, $m'$ belong to an ordered set of monomials $\mathcal{M}$ which parametrizes both simple and standard objects. 
This allows one to calculate the classes of simple modules $[L(m)]\in K(\mathcal{C})$ in terms of the classes of fundamental modules. As the dimensions and characters of the latter are known, this gives an answer to the problem of calculating the dimensions and characters of arbitrary simple objects : 
the multiplicities $P_{m,m'}$ are proved to be the evaluation at $t = 1$ of analogues $P_{m,m'}(t)$ of Kazhdan-Lusztig polynomials. These polynomials are constructed 
from the structure of the quantum Grothendieck ring $K_t(\mathcal{C})$, which is a $t$-deformation of $K(\mathcal{C})$ in a certain quantum torus \cite{Nak:quiver, VV:qGro}.
Besides it is proved that the coefficients of the polynomials $P_{m,m'}(t)$ are positive, as well as the structure constants of the quantum Grothendieck ring $K_t(\mathcal{C})$ on 
a canonical basis whose elements are called {\it $(q,t)$-characters of simple modules}. These results are based on the
geometric construction of standard modules in terms of quiver varieties known only for simply-laced types 
(note however that geometric character formulas for standard modules have been obtained in \cite{HL:JEMS2016} for all types).

The existence of an algorithm to compute the multiplicity of a simple module in a 
standard module is still not known in the case of non-simply-laced untwisted quantum affine algebras. 
A conjectural answer was proposed by the first author in \cite{H:qt} by 
giving another construction of the quantum Grothendieck ring $K_t(\mathcal{C})$ and
the corresponding polynomials $P_{m,m'}(t)$ which can be extended to general types. 
This leads to a general precise conjectural formula for the multiplicity of
simple modules in standard modules 
\begin{equation}\label{fconj}
[M(m)] = [L(m)] + \sum_{m'< m}P_{m,m'}(1)[L(m')],
\end{equation}
where $<$ is the partial ordering on monomials as above. 
Besides it is not clear from \cite{H:qt} that the polynomials $P_{m,m'}(t)$ have positive
coefficients (although it has been checked in numerous cases by computer).
As far as the authors know, these two independent conjectures, \eqref{fconj} and positivity, 
are completely open at the moment. 
In the present paper, we establish these conjectures for certain remarkable monoidal categories in the case of type $\mathrm{B}_n^{(1)}$.

For a simply-laced type $\mathrm{X}_n^{(1)}$, the category $\mathcal{C}$ has interesting monoidal subcategories $\mathcal{C}_{\mathcal{Q}, \mathrm{X}_n}$ introduced in \cite{HL:qGro} depending on a quiver $\mathcal{Q}$ whose underlying graph is the Dynkin diagram of type $\mathrm{X}_n$. This monoidal subcategory $\mathcal{C}_{\mathcal{Q}, \mathrm{X}_n}$ is sometimes called a Hernandez-Leclerc subcategory. It is proved in \cite{HL:qGro} that the quantum Grothendieck ring $K_t(\mathcal{C}_{\mathcal{Q}, \mathrm{X}_n})$ is isomorphic to the quantized coordinate algebra $\mathcal{A}_{t^{-1}}[N_-^{\mathrm{X}_n}]$ of type $\mathrm{X}_n$ (it was recently established \cite{Fuj} that certain completions of these categories have a structure of affine highest weight category). For type $\mathrm{B}_n^{(1)}$, analogues $\mathcal{C}_{\mathscr{Q},\mathrm{B}_n}$ of the categories $\mathcal{C}_{\mathcal{Q}, \mathrm{X}_n}$ were defined in \cite{OhS:foldedAR-I}.  In this paper, we establish an analogue of the isomorphism in \cite{HL:qGro} for $\mathcal{C}_{\mathscr{Q},\mathrm{B}_n}$ :
\begin{equation}\label{HLisom_B}
K_t(\mathcal{C}_{\mathscr{Q},\mathrm{B}_n})\simeq \mathcal{A}_{t^{-1}}[N_-^{\mathrm{A}_{2n-1}}] 
\end{equation}
(Theorem \ref{t:mainisom}, Corollary \ref{c:mainisom}). One notable point about this isomorphism is that the quantized coordinate algebra side is of type $\mathrm{A}_{2n-1}$, and not of type $\mathrm{B}_{n}$.  

In the proof of \eqref{HLisom_B}, we use the quantum cluster structure of the quantized coordinate algebras as follows. We first show the isomorphism between ambient quantum tori of $K_t(\mathcal{C}_{\mathscr{Q},\mathrm{B}_n})$ and $\mathcal{A}_{t^{-1}}[N_-^{\mathrm{A}_{2n-1}}]$. The former arises from the definition of (truncated) $(q, t)$-characters and the latter comes from the quantum cluster algebra structure of $\mathcal{A}_{t^{-1}}[N_-^{\mathrm{A}_{2n-1}}]$. Then we prove the correspondence between distinguished systems of relations : the quantum $T$-system of type $\mathrm{B}$ that we establish in $K_t(\mathcal{C}_{\mathscr{Q},\mathrm{B}_n})$ (Theorem \ref{t:T-sysaff}) and some specific mutation sequence in the quantum cluster algebra structure. This strategy is partially in line with the one in \cite{HL:qGro} aside from some technical
differences. For example, we do not know \emph{a priori} some positivity properties of (truncated) $(q, t)$-characters for type $\mathrm{B}_n^{(1)}$. 

As an immediate corollary of \eqref{HLisom_B}, we can say that $K_t(\mathcal{C}_{\mathscr{Q},\mathrm{B}_n})$ has a quantum cluster algebra structure coming from that of $\mathcal{A}_{t^{-1}}[N_-^{\mathrm{A}_{2n-1}}]$. Then the initial quiver for this quantum cluster algebra structure is nothing but a subquiver of the quiver introduced by the first author and Leclerc  \cite{HL:JEMS2016} which is defined for a subcategory of the category $\mathcal{C}$ of type $\mathrm{B}_n^{(1)}$. See Remarks \ref{r:HLquiver} and \ref{r:qclustr}. 

We also show that the isomorphism \eqref{HLisom_B} maps a basis $\{L_t(m)\}_m$ of $K_t(\mathcal{C}_{\mathscr{Q},\mathrm{B}_n})$ parametrized by simple modules $L(m)$ in $\mathcal{C}_{\mathscr{Q},\mathrm{B}_n}$ (called the $(q,t)$-characters of the simple modules $L(m)$ as above) to the dual canonical basis of  the quantized coordinate algebra in the sense of Lusztig and Kashiwara.
This implies the positivity property conjectured in \cite{H:qt} for the category $\mathcal{C}_{\mathscr{Q},\mathrm{B}_n}$ :
\begin{itemize}
\setlength{\itemsep}{4pt}
\item[(P1)] The positivity of the coefficients of the polynomials $P_{m,m'}(t)$.
\end{itemize}
\noindent As a by-product we also establish two additional positivity properties :
\begin{itemize}
\setlength{\itemsep}{4pt}
\item[(P2)] The positivity of structure constants of the quantum Grothendieck ring $K_t(\mathcal{C}_{\mathscr{Q},\mathrm{B}_n})$.

\item[(P3)] The positivity of the coefficients of the \emph{truncated} $(q,t)$-characters of $L(m)$.
\end{itemize}
As mentioned before, the following isomorphism was proved in \cite{HL:qGro} : 
 \begin{equation}\label{HLisom}
 K_t(\mathcal{C}_{\mathcal{Q},\mathrm{A}_{2n - 1}})\simeq \mathcal{A}_{t^{-1}}[N_-^{\mathrm{A}_{2n-1}}].
 \end{equation}
Combining \eqref{HLisom} with our isomorphism \eqref{HLisom_B}, we obtain an isomorphism between quantum Grothendieck rings
 \begin{equation}\label{isomt}
 K_t(\mathcal{C}_{\mathscr{Q},\mathrm{B}_n}) \simeq K_t(\mathcal{C}_{\mathcal{Q},\mathrm{A}_{2n - 1}}).
\end{equation}
Here we should note that $\mathcal{C}_{\mathscr{Q},\mathrm{B}_n}$ is determined by the Dynkin quiver $\mathcal{Q}'$ \emph{of type $\mathrm{A}_{2n-2}$}, which is irrelevant to the \emph{$\mathrm{A}_{2n-1}$-quiver} $\mathcal{Q}$, and some additional data $\flat\in \{>, <\}$. See section \ref{s:subcat} for details (the category $\mathcal{C}_{\mathscr{Q},\mathrm{B}_n}$ above is denoted by $\mathcal{C}_{(\mathcal{Q}')^{\flat}}$ there). Since the isomorphism \eqref{HLisom} also restricts to a bijection between the basis consisting of the $(q, t)$-characters of simple modules and the dual canonical basis, the isomorphism \eqref{isomt} induces a bijection between the $(q, t)$-characters of simple modules in $\mathcal{C}_{\mathscr{Q},\mathrm{B}_n}$ and $\mathcal{C}_{\mathcal{Q},\mathrm{A}_{2n - 1}}$. See also Remark \ref{r:stdPBW}. 

We note that similarities between finite-dimensional modules of type $\mathrm{A}$ and $\mathrm{B}$ quantum affine algebras had been observed in \cite{H:minim}. The Kirillov-Reshetikhin modules are certain simple objects in $\mathcal{C}$ generalizing fundamental modules. It is proved they are \emph{thin} in types $\mathrm{A}$ and $\mathrm{B}$, that is, they have multiplicity-free Frenkel-Reshetikhin $q$-characters \cite{FR:q-chara}, and that these $q$-characters have analogous explicit formulas in both types (certain $(q, t)$-characters of Kirillov-Reshetikhin modules are used in the	construction of $K_t(\mathcal{C}_{\mathscr{Q},\mathrm{B}_n})$).

To prove that the elements $\{L_t(m)\}_m$ specialize to the classes of simple modules as $t\rightarrow 1$, and to establish the formula (\ref{fconj}) in $\mathcal{C}_{\mathscr{Q},\mathrm{B}_n}$, we combine our results with
those of Kashiwara, Kim and Oh \cite{KO,KKO:AB}. These authors used different tools as in the present paper, 
namely generalized quantum affine Schur-Weyl duality \cite{KKK:1}, to construct an isomorphism $K(\mathcal{C}_{\mathscr{Q},\mathrm{B}_n}) \simeq K(\mathcal{C}_{\mathcal{Q},\mathrm{A}_{2n - 1}})$
 between {\it classical } Grothendieck rings. 
Moreover they proved their isomorphism preserves the basis of simple modules. We show that our isomorphism \eqref{isomt} of quantum Grothendieck rings specializes at $t = 1$ to their isomorphism (note that, as far as the authors know, the isomorphism \eqref{isomt} can not be deduced directly from the results of the papers \cite{KO,KKO:AB}).
This coincidence implies for the category $\mathcal{C}_{\mathscr{Q},\mathrm{B}_n}$ that the $(q,t)$-character $L_t(m)$  specializes to the $q$-character of $L(m)$ as conjectured by the first author in \cite{H:qt}. In particular the conjectural formula \eqref{fconj} holds for this category \footnote{Another approach is also developed by the first author in a paper in preparation.}.

We obtain also the explicit correspondence of $(q, t)$-characters of simple modules under \eqref{isomt} for special quivers. This is computed from the formulas for the change of the Lusztig parametrizations of the dual canonical basis associated with a change of reduced words. We observe that (\ref{isomt}) does not preserve any of the classes of modules one might naturally 
expect (e.g. fundamental or Kirillov-Reshetikhin modules). See Examples \ref{ex:ABcorresp2} and \ref{ex:ABcorresp3}. 
When the authors started to work on this correspondence, the paper \cite{KKO:AB} of Kashiwara, Kim and Oh appeared, and they also compute such correspondence in a different way. Our correspondence from type $\mathrm{A}_{2n-1}^{(1)}$ to type $\mathrm{B}_{n}^{(1)}$ coincides with the correspondence in \cite[section 3]{KKO:AB} up to shift of spectral parameters when we restrict it to $\mathcal{C}_{\mathcal{Q},\mathrm{A}_{2n-1}}$. We note that we also describe the correspondence from type $\mathrm{B}_{n}^{(1)}$ to type $\mathrm{A}_{2n-1}^{(1)}$ in this paper, which seems not to be written in \cite{KO} and \cite{KKO:AB}.

Some parts of the proof in this paper, in particular for the quantum $T$-systems, rely on the thinness of Kirillov-Reshetikhin modules discussed above. Although this property is not satisfied in general, it is satisfied by fundamental modules in types $\mathrm{C}$ and $\mathrm{G}_2$ \cite{H:minim}, which gives hope to get analogous results in these situations. We plan to come back to this in another paper.

We expect that we could also study the whole category $\mathcal{C}$ of type $\mathrm{B}_n^{(1)}$ by ``gluing'' (infinitely many) copies of $\mathcal{C}_{\mathscr{Q},\mathrm{B}_n}$ together (cf. \cite[subsection 1.5, sections 7, 8]{HL:qGro}). For example, we hope that we could describe a structure of $K_t(\mathcal{C})$ by revealing how the copies of $K_t(\mathcal{C}_{\mathscr{Q},\mathrm{B}_n})$ are ``patched'' in $K_t(\mathcal{C})$. In simply-laced cases, a ``patching'' of $K_t(\mathcal{C}_{\mathcal{Q}, \mathrm{X}_n})$ was calculated in a uniform manner in \cite{HL:qGro} by means of derived Hall algebras, whose analogues for non-simply-laced cases are not known now. We also plan to discuss the whole category $\mathcal{C}$ in a future paper. 

Moreover, the meaning of the quantum Grothendieck ring $K_t(\mathcal{C})$ in terms of the structure of the category $\mathcal{C}$ has not been clarified in non-simply-laced cases yet. Hence it is not known at the moment whether the isomorphism (\ref{isomt}) 
arises as an equivalence of categories, but it is certainly a question we would like to address.

The paper is organized as follows. In section \ref{s:prel}, we briefly recall the definitions of quantum loop algebras and their finite-dimensional representation theory.  
In section \ref{s:subcat}, we explain the definition of the remarkable monoidal categories considered in this paper following \cite{HL:qGro} and \cite{OhS:foldedAR-I}. 
Here we also show the important \emph{convexity} of twisted Auslander-Reiten quivers (Lemma \ref{l:tw-dia}). In section \ref{s:qtori}, we recall the definition of quantum tori which will be ambient spaces of $(q, t)$-characters. The important data are the inverses of quantum Cartan matrices. Their explicit form for type $\mathrm{B}$ is provided in Example \ref{e:invcarB_2} and \ref{e:invcarB-pic}, and its calculation is explained in Appendix \ref{a:inv}. In section \ref{s:QCA}, we explain another kind of quantum tori, that is, the quantum tori arising from the quantum cluster algebra structures of quantized coordinate algebras. In section \ref{s:qtoriisom}, we prove the isomorphism between these two kinds of quantum tori as explained just after \eqref{HLisom_B}. In section \ref{s:qGro}, we recall the definition of the quantum Grothendieck rings and the $(q, t)$-characters of standard modules and simple modules. In section \ref{s:dualcan}, we briefly recall the basics of quantized coordinate algebras including (normalized) unipotent quantum minors, quantum determinantal identities, (normalized) dual Poincar\'e-Birkhoff-Witt type bases and dual canonical bases. In section \ref{s:T-sys}, we prove the quantum $T$-system for type $\mathrm{B}_n^{(1)}$ relying on a feature of type $\mathrm{B}_n^{(1)}$ : the thinness of Kirillov-Reshetikhin modules. In section \ref{s:mainisom}, we prove the isomorphism \eqref{HLisom_B} together with the correspondence between $(q, t)$-characters of simple modules and dual canonical basis elements. In section \ref{s:cor}, we give the proof of the positivities (P1)--(P3) and explain the coincidence between our isomorphism and the isomorphism established by Kashiwara and Oh \cite{KO}, which implies \eqref{fconj} in $\mathcal{C}_{\mathscr{Q},\mathrm{B}_n}$. In section \ref{s:AB}, we provide the explicit correspondence of simple modules under \eqref{isomt} for some explicit quivers. 

\addsubsection*{Acknowledgments} 
The authors would like to thank Bernard Leclerc for helpful comments. They are also thankful to the anonymous referee whose suggestions improve the present paper. The authors were supported by the European Research Council under the European Union's Framework Programme H2020 with ERC Grant Agreement number 647353 Qaffine.

\subsection{}The following are general notation and convention in this paper.
\begin{itemize}
\item[(1)] For some proposition $\wp$, we set 
\[
\delta(\wp):=\begin{cases}
1&\text{if}\ \wp\ \text{holds}, \\
0&\text{otherwise}. 
\end{cases}
\]
\item[(2)] For a totally ordered set $J=\{\cdots <j_{-1}<j_0<j_1<j_2<\cdots\}$, write 
\begin{align*}
\dprod_{j\in J}A_{j}&:=\cdots A_{j_{-1}}A_{j_0}A_{j_1}A_{j_2}\cdots, &
\overrightarrow{\bigotimes_{j\in J}}A_{j}&:=\cdots A_{j_{-1}}\otimes A_{j_0}\otimes A_{j_1}\otimes A_{j_2}\cdots. 
\end{align*}
\item[(3)] For $k\in \mathbb{Z}_{>0}$, a subset $\sigma\subset \mathbb{Z}$ is called \emph{a $k$-segment} if $\sigma$ is of the form $\{\ell+sk\mid s=0, \dots, t\}$ for some $\ell\in \mathbb{Z}$ and $t\in\mathbb{Z}_{\geq 0}$. \end{itemize}

\section{Quantum loop algebras and their representation theory}\label{s:prel}
In this section, we give reminders on quantum loop algebras, their 
finite-dimensional modules and the corresponding $q$-character theory.
\subsection{Quantum loop algebras}\label{ss:QLA}
Let $\mathfrak{g}$ be a complex simple Lie algebra of type $\mathrm{X}_N$ ($\mathrm{X}_N=\mathrm{A}_N, \mathrm{B}_N, \mathrm{C}_N,\dots,\mathrm{G}_2 $), and $\mathfrak{h}$ be a Cartan subalgebra of $\mathfrak{g}$. Set $I:=\{1,\dots, N\}$ and let $\{ \alpha_i \mid i \in I\}$  (resp.~$\{ h_i\mid i \in I\}$) be the set of simple roots (resp.~coroots), $\Delta_+$ the set of positive roots.  Let $C=(c_{ij})_{i, j\in I}=(\langle h_i, \alpha_j\rangle)_{i, j\in I}$ be the Cartan matrix of $\mathfrak{g}$. We follow the labelling in \cite[\S 4.8]{Kac:Liealgbook}. Write $Q:=\sum_{i\in I}\mathbb{Z}\alpha_i\subset\mathfrak{h}^{\ast}$, called the root lattice of $\mathfrak{g}$, and set $P:=\{\lambda\in \mathfrak{h}^{\ast}\mid \langle h_i, \lambda\rangle\in \mathbb{Z}\ \text{for all}\ i\in I\}$, called the weight lattice of $\mathfrak{g}$. For $i\in I$, define $\varpi_i\in P$ as the element such that $\langle h_j, \varpi_i\rangle=\delta_{ij}$ for $j\in I$, which is called an fundamental weight. Set $P_+:=\sum_{i\in I}\mathbb{Z}_{\geq 0}\varpi_i$, $Q_+:=\sum_{i\in I}\mathbb{Z}_{\geq 0}\alpha_i$, and $\rho:=\sum_{i\in I}\varpi_i\in P_+$. We define a partial order $\leq$ on $P$ by $\beta\leq \beta'$ if and only if $\beta'-\beta\in Q_+$. 

Let $W$ be the Weyl group of $\mathfrak{g}$, that is, the group generated by $\{s_{i}\}_{i\in I}$ with the defining relations $s_i^2=e$ for $i\in I$ and $(s_is_j)^{m_{ij}}=e$ for $i, j\in I$, $i\neq j$, where $e$ is the unit of $W$ and $m_{ij}:=2$ (resp.~$3, 4, 6$) if $c_{ij}c_{ji}=0$ (resp.~$1, 2, 3$). We have the group homomorphisms $W\to\Aut\mathfrak{h}$ and $W\to\Aut\mathfrak{h}^{\ast}$ given by
\begin{align*}
s_{i}\left(h\right)&=h-\langle h, \alpha_{i}\rangle h_{i}&s_{i}\left(\mu\right)&=\mu-\langle h_{i},\mu\rangle\alpha_{i}
\end{align*}
for $h\in\mathfrak{h}$ and $\mu\in\mathfrak{h}^{\ast}$. For an element $w$ of $W$, $\ell(w)$ denotes the length of $w$, that is, the smallest integer $\ell$ such that there exist $i_{1},\dots,i_{\ell}\in I$ with $w=s_{i_{1}}\cdots s_{i_{\ell}}$. It is well-known that there exists a unique element $w_0$ such that $\ell(w)\leq \ell(w_0)$ for all $w\in W$, which is called the longest element of $W$. For $w\in W$, set 
\[
I(w):=\{\bm{i}=(i_{1},\dots,i_{\ell(w)})\in I^{\ell(w)}\mid w=s_{i_{1}}\cdots s_{i_{\ell(w)}}\}.
\]
An element of $I(w)$ is called a reduced word of $w$. 

Let $(-, -)$ the symmetric $\mathbb{C}$-bilinear form on $\mathfrak{h}^{\ast}$ such that  
\begin{enumerate}
\item[(a)] $\displaystyle\langle h_{i},\alpha_{j}\rangle=2\frac{(\alpha_i,\alpha_j)}{(\alpha_{i},\alpha_{i})}$ for $i, j\in I$, here left-hand side is the canonical paring of $h_i$ and $\alpha_j$, 
\item[(b)] $\min\{(\alpha_{i},\alpha_{i}) \mid i\in I\}=2$.
\end{enumerate}
Note that this bilinear form $(-,-)$ is $W$-invariant. Set 
\[
r_i:=\frac{(\alpha_i, \alpha_i)}{2} 
\]
for $i\in I$, and $D:=(\delta_{ij}r_i)_{i, j\in I}$. Let $B=(b_{ij})_{i, j\in I}:=DC$. Then we have 
\[
b_{ij}=(\alpha_i, \alpha_j)
\]
for $i, j\in I$, in particular, $B$ is symmetric. 
\begin{notation}\label{n:qbinom}
Let $x$ be a non-zero complex number which is not a root of unity or an indeterminate. Set
\begin{align*}
&x_i:= x^{r_i}, \displaystyle [n]_{x}:= \frac{x^n-x^{-n}}{x-x^{-1}}\ \text{for\ }n\in \mathbb{Z},\\
&\left[ \begin{array}{c} n\\k \end{array} \right]_{x}:=\begin{cases}\displaystyle \frac{[n]_{x}[n-1]_{x}\cdots[n-k+1]_{x}}{[k]_{x}[k-1]_{x}\cdots [1]_{x}}&\text{if\ }n\in \mathbb{Z}, k\in \mathbb{Z}_{>0},\\ 1 &\text{if\ }n\in \mathbb{Z}, k=0,\end{cases}
\end{align*}
Note that $[n], \left[ \begin{array}{c} n\\k \end{array} \right]\in \mathbb{Z}[x^{\pm 1}]$ if $x$ is an indeterminate. 
\end{notation}

Let $q$ be a non-zero complex number which is not a root of unity. 
\begin{definition}[{\cite{Dri:real,B:braid}}]\label{d:QEA-Drinfeld}
\emph{The quantum loop algebra $\mathcal{U}_q(\mathcal{L}\mathfrak{g})(=\mathcal{U}_q(\mathrm{X}_N^{(1)}))$ of type $\mathrm{X}_N^{(1)}$} is the $\mathbb{C}$-algebra defined by the generators
\[
k_i^{\pm 1} (i\in I), x_{i, r}^{\pm} ((i, r)\in I\times \mathbb{Z}), h_{i, r} ((i, r)\in I\times (\mathbb{Z}\setminus \{0\})),
\]
and the relations below :
\begin{itemize}
\item[(I)] $k_ik_i^{-1}=1=k_i^{-1}k_i$, $k_ik_j=k_jk_i$ for $i, j\in I$,
\item[(II)] $k_ix_{j, r}^{\pm}=q_i^{\pm c_{ij}}x_{j, r}^{\pm}k_i$ for $i, j\in I, r\in\mathbb{Z}$,
\item[(III)] $[k_i, h_{j, r}]=0$ for $i, j\in I, r\in \mathbb{Z}\setminus \{0\}$,
\item[(IV)] $\displaystyle [x_{i, r}^+, x_{j, s}^-]=\delta_{ij}\frac{\phi_{i, r+s}^+-\phi_{i, r+s}^-}{q_i-q_i^{-1}}$ for $(i, r), (j, s)\in I\times \mathbb{Z}$,
\item[(V)] $\displaystyle [h_{i, r}, x_{j, s}^{\pm}]=\pm\frac{[rc_{ij}]_{q_i}}{r}x_{j, r+s}^{\pm}$ for $(i, r)\in I\times (\mathbb{Z}\setminus \{0\}), (j, s)\in I\times \mathbb{Z}$,
\item[(VI)] $\displaystyle [h_{i, r}, h_{j, s}]=0$ for $(i, r), (j, s)\in I\times (\mathbb{Z}\setminus \{0\})$,
\item[(VII)] $x_{i, r+1}^{\pm}x_{j, s}^{\pm}-q^{\pm (\alpha_i, \alpha_j)}x_{i, r}^{\pm}x_{j, s+1}^{\pm}=q^{\pm (\alpha_i, \alpha_j)}x_{j, s}^{\pm}x_{i, r+1}^{\pm}-x_{j, s+1}^{\pm}x_{i, r}^{\pm}$ for $(i, r), (j, s)\in I\times \mathbb{Z}$,
\item[(VIII)] $\displaystyle \sum_{\sigma\in \mathfrak{S}_{1-c_{ij}}}\sum_{k=0}^{1-c_{ij}}(-1)^k\left[ \begin{array}{c} 1-c_{ij}\\k \end{array} \right]_{q_i}x_{i, r_{\sigma(1)}}^{\pm}\cdots x_{i, r_{\sigma(k)}}^{\pm}x_{j, s}^{\pm}x_{i, r_{\sigma(k+1)}}^{\pm}\cdots x_{i, r_{\sigma(1-c_{ij})}}^{\pm}=0$ for $i,j\in I$ with $i\neq j$, and $r_1,\dots, r_{1-c_{ij}}, s\in \mathbb{Z}$,
\end{itemize}
where 
\[
\phi_i^{\pm}(z):=\sum_{r=0}^{\infty}\phi_{i, \pm r}^{\pm}z^{\pm r}=k_i^{\pm 1}\exp\left(\pm (q_i-q_i^{-1})\sum_{r>0}h_{i, \pm r}z^{\pm r}\right).
\]
\end{definition}
The quantum loop algebra is a quotient of the corresponding quantum affine algebra and has a Hopf algebra structure which is defined via another realization of $\mathcal{U}_q(\mathcal{L}\mathfrak{g})$ (called the Drinfeld-Jimbo presentation). See, for example, \cite{CH} for more details. 

\subsection{Finite-dimensional representations}

A finite-dimensional $\mathcal{U}_q(\mathcal{L}\mathfrak{g})$-module is said to be of \emph{type 1} if the eigenvalues of $\{k_i\mid i\in I\}$ are of the form $q^m$, $m\in \mathbb{Z}$.
Let $\mathcal{C}$ be the category of type 1 finite-dimensional $\mathcal{U}_q(\mathcal{L}\mathfrak{g})$-modules. 

Since the elements $\{h_{i, r}, k_i\mid (i, r)\in I\times (\mathbb{Z}\setminus \{0\})\}$ mutually commute, we have a decomposition
\[
V=\bigoplus_{\bm{\gamma}=(\gamma_{i, \pm r}^{\pm})_{i\in I,  r\geq 0}}V_{\bm{\gamma}}
\]
such that
\[ 
V_{\bm{\gamma}}=\{v\in V \mid \text{for all}\ i, r,\ \text{there exists}\ p_{i, \pm r}>0\ \text{such that}\  (\phi_{i, \pm r}^{\pm}-\gamma_{i, \pm r}^{\pm})^{p_{i, \pm r}}.v=0\}.
\]
The series $\bm{\gamma}$ is called \emph{$l$-weights} and $V_{\bm{\gamma}}$ is called \emph{an $l$-weight space} if $V_{\bm{\gamma}}\neq 0$.

Let $L$ be a simple $\mathcal{U}_q(\mathcal{L}\mathfrak{g})$-module in $\mathcal{C}$. Then there uniquely exists an $l$-weight space $L_{\bm{\gamma}_0}$ such that $x_{i, r}^+.L_{\bm{\gamma}_0}=0$ for all $i\in I$ and $r\in \mathbb{Z}$ \cite[subsection 12.2]{CP:guide}. The isomorphism class of a simple $\mathcal{U}_q(\mathcal{L}\mathfrak{g})$-module in $\mathcal{C}$ is determined by such $\bm{\gamma}_0$ (called \emph{the $l$-highest weight}). Moreover, we have the following :
\begin{theorem}[{\cite[Theorem 3.3]{CP:QAAandRep}, \cite[Theorem 12.2.6]{CP:guide}}]\label{t:CP-class}
If $\bm{\gamma}=(\gamma_{i, \pm r}^{\pm})_{i\in I,  r\geq 0}$ is an $l$-highest weight of a simple $\mathcal{U}_q(\mathcal{L}\mathfrak{g})$-module in $\mathcal{C}$, then there uniquely exists $P_i(z)\in \mathbb{C}[z]$ with $P_i(0)=1$ for $i\in I$ such that 
\begin{align}
\sum_{r=0}^{\infty}\gamma_{i, \pm r}^{\pm}z^{\pm r}=q_i^{\deg (P_i)}\frac{P_i(zq_i^{-1})}{P_i(zq_i)}\label{eq:Dripoly}
\end{align}
as elements of $\mathbb{C}\dbracket{z}$ and $\mathbb{C}\dbracket{z^{-1}}$, respectively. Conversely, for any $(P_i(z))_{i\in I}\in \mathbb{C}[z]^I$ with $P_i(0)=1$ for $i\in I$, we have a simple $\mathcal{U}_q(\mathcal{L}\mathfrak{g})$-module in $\mathcal{C}$ with the $l$-highest $\bm{\gamma}$ given by \eqref{eq:Dripoly}. 
\end{theorem}
This $I$-tuple of polynomials $(P_i(z))_{i\in I}$ corresponding to a simple module $L$ is called \emph{the Drinfeld polynomials of $L$}. 

\subsection{$q$-characters}\label{ss:q-ch}

The $q$-character, introduced by  Frenkel and Reshetikhin \cite{FR:q-chara}, is an important tool for the study of finite-dimensional modules over quantum loop algebras. In this subsection, we briefly recall their definition and properties. See, for example, \cite{CH} for the survey of this topic.

 Frenkel and Reshetikhin \cite[Proposition 1]{FR:q-chara} proved that any $l$-weight $\bm{\gamma}$ of a $\mathcal{U}_q(\mathcal{L}\mathfrak{g})$-module in $\mathcal{C}$ is of the form 
\[
\sum_{r=0}^{\infty}\gamma_{i, \pm r}^{\pm}z^{\pm r}=q_i^{\deg (Q_i)-\deg (R_i)}\frac{Q_i(zq_i^{-1})R_i(zq_i)}{Q_i(zq_i)R_i(zq_i^{-1})}
\]
as elements of $\mathbb{C}\dbracket{z}$ and $\mathbb{C}\dbracket{z^{-1}}$, respectively, for some $Q_i(z), R_i(z)\in \mathbb{C}[z]$ with $Q_i(0)=R_i(0)=1$. For this $\bm{\gamma}$, define a monomial $m_{\bm{\gamma}}$ in the Laurent polynomial ring $\mathbb{Z}[Y_{i, a}^{\pm 1}\mid i\in I, a\in\mathbb{C}^{\times}]$ over $\mathbb{Z}$ with infinitely many variables by 
\[
m_{\bm{\gamma}}:=\prod_{i\in I, a\in \mathbb{C}^{\times}}Y_{i, a}^{q_{i, a}-r_{i, a}},
\]
where 
\begin{align*}
Q_i&=\prod_{a\in \mathbb{C}^{\times}}(1-za)^{q_{i, a}}
&
R_i&=\prod_{a\in \mathbb{C}^{\times}}(1-za)^{r_{i, a}}. 
\end{align*}

Let $K(\mathcal{C})$ be the Grothendieck ring of $\mathcal{C}$. 
\begin{theorem}[{\cite[Theorem 3]{FR:q-chara}}]\label{t:q-chara}
The assignment
\[
[V]\mapsto \sum_{\bm{\gamma}:l\text{-weights}}\dim (V_{\bm{\gamma}})m_{\bm{\gamma}}. 
\]
defines an injective algebra homomorphism $\chi_q\colon K(\mathcal{C})\to \mathbb{Z}[Y_{i, a}^{\pm 1}\mid i\in I, a\in\mathbb{C}^{\times}]$. 
\end{theorem}
The map $\chi_q$ is called \emph{the $q$-character homomorphism}. A monomial which is a product of only positive powers of $Y_{i, a}$'s is said to be \emph{dominant}. Write the simple $\mathcal{U}_q(\mathcal{L}\mathfrak{g})$-module in $\mathcal{C}$ with $l$-highest weight $\bm{\gamma}$ as $L(m_{\bm{\gamma}})$. Then, by Theorem \ref{t:CP-class}, $m_{\bm{\gamma}}$ is dominant, and the dominant monomials classify the simple $\mathcal{U}_q(\mathcal{L}\mathfrak{g})$-module in $\mathcal{C}$ up to isomorphism. 
\begin{proposition}[{\cite[Corollary 3.5]{CP:QAAandRep}}]\label{p:tensor}
Let $m$ and $m'$ be dominant monomials. Then $L(mm')$ is isomorphic to a quotient of the submodule of $L(m)\otimes L(m')$ generated by
the tensor product of vectors in their $l$-highest weight spaces.
\end{proposition}

For $i\in I$ and $a\in \mathbb{C}^{\times}$, set 
\[
A_{i, a}:=Y_{i, aq_i^{-1}}Y_{i, aq_i}\left(\prod_{j\colon c_{ji}=-1}Y_{j, a}^{-1}\right)\left(\prod_{j\colon c_{ji}=-2}Y_{j, aq^{-1}}^{-1}Y_{j, aq}^{-1}\right)\left(\prod_{j\colon c_{ji}=-3}Y_{j, aq^{-2}}^{-1}Y_{j, a}^{-1}Y_{j, aq^2}^{-1}\right). 
\]

Define a partial ordering on the set of monomials in $\mathbb{Z}[Y_{i, a}^{\pm 1}\mid i\in I, a\in\mathbb{C}^{\times}]$ as follows : 
\[
m\leq m'\ \text{if and only if}\ m(m')^{-1}\ \text{is a product of elements of}\ \{A_{i, a}^{-1}\mid i\in I, a\in \mathbb{C}^{\times}\}.
\]
This is called \emph{the Nakajima ordering}, which is considered as a refinement of the usual Chevalley ordering $\leq$ on $P$.
 
\begin{theorem}[{\cite[Theorem 4.1]{FM}}]\label{t:highest}
Let $m$ be a dominant monomial. Then all monomials appearing in $\chi_q(L(m))-m$ are strictly less than $m$ with respect to the Nakajima ordering. 
\end{theorem}

Let $\mathcal{C}_{\bullet}$ be the full subcategory of $\mathcal{C}$ consisting of the objects all of whose composition factors are isomorphic to simple modules whose Drinfeld polynomials have their roots in $\{q^{r}\mid r\in\mathbb{Z}\}$. In fact, the description of the simple modules in $\mathcal{C}$ essentially reduces to the description of the simple modules in $\mathcal{C}_{\bullet}$. See \cite[subsection 3.7]{HL:cluaff} for more details.
Note that $\mathcal{C}_{\bullet}$ is stable by taking subquotient, tensor product and extension.

Let $K(\mathcal{C}_{\bullet})$ be the Grothendieck ring of $\mathcal{C}_{\bullet}$. Then this is a subalgebra of $K(\mathcal{C})$ and the $q$-character homomorphism restricts to the $\mathbb{Z}$-algebra homomorphism 
\[
K(\mathcal{C}_{\bullet})\to \mathbb{Z}[Y_{i, q^r}^{\pm 1}\mid i\in I, r\in\mathbb{Z}],
\]
which will be also denoted by $\chi_q$ \cite[Lemma 6.1]{FM}. 

In this paper, we always work in the subcategory $\mathcal{C}_{\bullet}$. From now on, we write $Y_{i, r}:=Y_{i, q^r}$, $A_{i, r}:=A_{i, q^r}$ and $\mathcal{Y}:=\mathbb{Z}[Y_{i, r}^{\pm 1}\mid i\in I, r\in\mathbb{Z}]$. 
\begin{proposition}[{\cite[Corollary 2]{FR:q-chara}}]
The $\mathbb{Z}$-algebra $K(\mathcal{C}_{\bullet})$ is isomorphic to the polynomial ring $\mathbb{Z}[X_{i, r}\mid i\in I, r\in \mathbb{Z}]$ via $[L(Y_{i, r})]\mapsto X_{i, r}$. 
\end{proposition}

For a monomial $m$ in $\mathcal{Y}$, write 
\begin{align}
m=\prod_{i\in I, r\in \mathbb{Z}}Y_{i, r}^{u_{i, r}(m)}\label{eq:power}
\end{align}
with $u_{i, l}(m)\in \mathbb{Z}$. 

\begin{definition}
For $i\in I$, $r\in \mathbb{Z}$ and $k\in \mathbb{Z}_{\geq 0}$, set 
\[
W_{k, r}^{(i)}:=L(m_{k, r}^{(i)}),\ m_{k, r}^{(i)}:=\prod_{s=1}^{k}Y_{i, r+2r_i(s-1)}. 
\]
These simple modules are called \emph{Kirillov-Reshetikhin modules}. Note that $W_{0,r}^{(i)}$ is the trivial $\mathcal{U}_q(\mathcal{L}\mathfrak{g})$-module. The modules of the form $W_{1,r}^{(i)}$ are called \emph{fundamental modules}.

\end{definition}
\begin{theorem}[{\cite[Theorem 3.2]{Nak:KR}, \cite[Theorem 4.1, Lemma 4.4]{H:KR-T}}]\label{t:minuscule}
Let $i\in I$, $r\in \mathbb{Z}$ and $k\in \mathbb{Z}_{\geq 0}$. Suppose that $m$ is a monomial appearing in $\chi_q(W_{k, r}^{(i)})$. Then $m=m_{k, r}^{(i)}$ or $m\leq m_{k, r}^{(i)}A_{i, r+r_i(2k-1)}^{-1}$. 
\end{theorem} 
In particular, dominant monomials do not occur in $\chi_q(W_{k, r}^{(i)})-m_{k, r}^{(i)}$. This property is called \emph{the affine-minuscule property}. 

\section{The monoidal subcategories associated with quivers}\label{s:subcat}

In \cite{HL:qGro}, the first author and Leclerc introduced specific monoidal subcategories $\mathcal{C}_{\mathcal{Q}}$ of $\mathcal{C}_{\bullet}$ in the case that the quantum loop algebra is of type $\mathrm{X}_N^{(1)}$ with $\mathrm{X}=\mathrm{A}, \mathrm{D}$ or $\mathrm{E}$. These subcategories are constructed from a Dynkin quiver $\mathcal{Q}$ of type $\mathrm{X}_N$. 
In particular, the essential datum is \emph{the Auslander-Reiten quiver $\Gamma_{\mathcal{Q}}$} associated with $\mathcal{Q}$. It is known that the subcategory $\mathcal{C}_{\mathcal{Q}}$ has interesting properties. For example, its complexified Grothendieck ring is isomorphic to the coordinate algebra $\mathbb{C}[N_-]$ of the unipotent group $N_-$ of type $\mathrm{X}_N$, and the simple modules in $\mathcal{C}_{\mathcal{Q}}$ \emph{categorify} the dual canonical basis of $\mathbb{C}[N_-]$ (see Theorem \ref{t:isomtypeA}). Moreover it is a ``heart'' of the whole category (cf.~\cite[subsection 1.5]{HL:qGro}). See subsection \ref{ss:HLsubcat} for the definition of $\mathcal{C}_{\mathcal{Q}}$, and see subsection \ref{ss:HLisom} and \cite{HL:qGro,Fuj} for detailed properties. 

In this paper, we deal also with an analogue of these subcategories in the case of type $\mathrm{B}_{n}^{(1)}$, which is introduced by Oh and Suh \cite{OhS:foldedAR-I}. To define these subcategories, we need combinatorially generalized Auslander-Reiten quivers, called \emph{the twisted Auslander-Reiten quivers} \cite{OhS:cAR,OhS:foldedAR-I}. 
We emphasize that the combinatorial generalization of  Auslander-Reiten quivers \emph{of type $\mathrm{A}_{2n-1}$} are used in order to define the subcategories of  $\mathcal{C}_{\bullet}$ \emph{of type $\mathrm{B}_{n}^{(1)}$}.

\subsection{The adapted reduced words and Auslander-Reiten quivers}\label{ss:adapted}
In this subsection, we assume that $\mathfrak{g}$ is of type $\mathrm{X}_N$ with $\mathrm{X}=\mathrm{A}, \mathrm{D}, \mathrm{E}$. Let $\mathcal{Q}=(\mathcal{Q}^0=I, \mathcal{Q}^1)$ be a Dynkin quiver of type $\mathrm{X}_N$, where $\mathcal{Q}^0$ is the set of vertices (identified with the index set $I$ of simple roots) and $\mathcal{Q}^1$ is the set of arrows. If there is an arrow from $i$ to $j$ in $\mathcal{Q}$, then we write $i\to j$. A vertex $i\in I$ is called \emph{a sink} (resp.\emph{a source}) if there is no arrow exiting out of (resp. entering into) $i$. 

Let $i$ be a sink or a source of $\mathcal{Q}$. Then define $s_i(\mathcal{Q})$ as the quiver obtained from $\mathcal{Q}$ by reversing the orientation of each arrows connected with $i$ in $\mathcal{Q}$.  A reduced word $\bm{i}=(i_1,\dots, i_{\ell})$ of an element $w$ of $W$ is said to be \emph{adapted to $\mathcal{Q}$} if
\begin{center}
$i_k$ is a sink of $s_{i_{k-1}}\cdots s_{i_2}s_{i_1}(\mathcal{Q})$ 
\end{center}
for all $k=1,\dots, \ell$. There is a unique element $\phi_{\mathcal{Q}}\in W$ such that
\begin{align}
\phi_{\mathcal{Q}}=s_{i_1}s_{i_2}\cdots s_{i_{N}}\label{eq:cox}
\end{align}
where $\{i_1,i_2,\dots, i_{N}\}=I$, and $(i_1,i_2,\dots, i_{N})$ is adapted to  $\mathcal{Q}$. The element $\phi_{\mathcal{Q}}$ is called \emph{the Coxeter element associated with $\mathcal{Q}$}.

Now we shall recall the combinatorial construction \cite{Hap,HL:qGro} of the Auslander-Reiten quiver $\Gamma_{\mathcal{Q}}$ of the category of finite-dimensional modules over the path algebra $\mathbb{C}\mathcal{Q}$. By definition, the vertex set of $\Gamma_{\mathcal{Q}}$ is the set of isomorphism classes of indecomposable $\mathbb{C}\mathcal{Q}$-modules. Gabriel's theorem states that this set is identified with the set $\Delta_+$ of positive roots by taking their dimension vectors. Hence, we consider the vertex set of $\Gamma_{\mathcal{Q}}$ as $\Delta_+$ in this paper. We do not review (in fact, do not need) here the details of the representation theory of $\mathbb{C}\mathcal{Q}$. See, for instance, \cite{ASS} for more details.

Fix a map $\xi\colon I\to\mathbb{Z}$ such that 
\begin{center}
$\xi(j)=\xi(i)+1$ whenever there exists an arrow $i\to j$ in $\mathcal{Q}$. 
\end{center}
Such a map is called \emph{a height function associated with $\mathcal{Q}$}. Set 
\[
\widehat{I}:=\{(i, r)\in I\times \mathbb{Z}\mid r-\xi(i)\in 2\mathbb{Z}\}. 
\]
Let $\widehat{\mathcal{Q}}$ be the quiver whose vertex set is $\widehat{I}$ and whose arrow set $\widehat{\mathcal{Q}}^1$ is given by
\[
\widehat{\mathcal{Q}}^1=\{(j, r-1)\to (i, r)\mid \text{there is an arrow between }i\ \text{and}\ j\ \text{in}\ \mathcal{Q}\}.
\]

Recall that $\phi_{\mathcal{Q}}$ is the Coxeter element associated with $\mathcal{Q}$. Take $(i_1,\dots, i_N)\in I(\phi_{\mathcal{Q}})$. 
For $i\in I$, set 
\[
\gamma_i=s_{i_1}\cdots s_{i_{k-1}}(\alpha_{i_k})\ \text{with}\ i_k=i.
\]
Note that $\gamma_i$ does not depend on the choice of $(i_1,\dots, i_N)\in I(\phi_{\mathcal{Q}})$. Consider the map $\Omega_{\xi}\colon \Delta_+\to \widehat{I}$ such that 
\begin{align*}
\gamma_i\mapsto (i, \xi(i)),&&&\phi_{\mathcal{Q}}(\beta)\mapsto (i, p-2)\ \text{if}\ \Omega_{\xi}(\beta)=(i, p)\ \text{and}\ \phi_{\mathcal{Q}}(\beta), \beta\in \Delta_+.
\end{align*}
\begin{proposition}[{see, for example,\cite[subsection 2.2]{HL:qGro}}]\label{p:ARlabel}
The full subquiver of $\widehat{\mathcal{Q}}$ whose vertex set is $\Omega_{\xi}(\Delta_+)$ is isomorphic to the Auslander-Reiten quiver $\Gamma_{\mathcal{Q}}$ via $\Omega_{\xi}$. 
\end{proposition}
According to this proposition, the vertices of the Auslander-Reiten quiver $\Gamma_{\mathcal{Q}}$ are also labelled by the subset $\Omega_{\xi}(\Delta_+)$ of $\widehat{I}$. We write $\widehat{I}_{\xi}:=\Omega_{\xi}(\Delta_+)$. 
\begin{example}\label{e:AR}
Consider the case of type $\mathrm{A}_{4}$. 

\noindent(i) Let $\mathcal{Q}$ be the following quiver :  

\hfill
\begin{xy} 0;<1pt,0pt>:<0pt,-1pt>::
(0,0) *+{1} ="0",
(60,0) *+{2} ="1",
(120,0) *+{3} ="2",
(180,0) *+{4} ="3",
"1", {\ar"0"},
"2", {\ar"1"},
"3", {\ar"2"},
\end{xy}
\hfill
\hfill

Then $\phi_{\mathcal{Q}}=s_1s_2s_3s_4$. Take an associated height function $\xi$ such that $\xi(1)=0$. Then the  Auslander-Reiten quiver $\Gamma_{\mathcal{Q}}$ and its labelling $\widehat{I}_{\xi}$ are described as follows : 

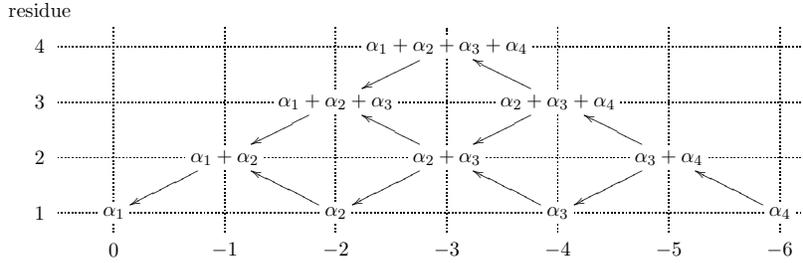
\hfill
\scalebox{0.7}[0.7]{
\begin{xy} 0;<1pt,0pt>:<0pt,-1pt>::
(-40,-20) *+{\text{residue}}, 
(-40,0) *+{4}, 
(-40,30) *+{3}, 
(-40,60) *+{2}, 
(-40,90) *+{1}, 
(0,110) *+{0}, 
(60,110) *+{-1}, 
(120,110) *+{-2}, 
(180,110) *+{-3}, 
(240,110) *+{-4}, 
(300,110) *+{-5}, 
(360,110) *+{-6}, 
(180,0) *+{\alpha_1+\alpha_2+\alpha_3+\alpha_4} ="0",
(120,30) *+{\alpha_1+\alpha_2+\alpha_3} ="1",
(240,30) *+{\alpha_2+\alpha_3+\alpha_4} ="2",
(60,60) *+{\alpha_1+\alpha_2} ="3",
(180,60) *+{\alpha_2+\alpha_3} ="4",
(300,60) *+{\alpha_3+\alpha_4} ="5",
(0,90) *+{\alpha_1} ="6",
(120,90) *+{\alpha_2} ="7",
(240,90) *+{\alpha_3} ="8",
(360,90) *+{\alpha_4} ="9",
{\ar@{.} (-30,0); "0"},
{\ar@{.} "0"; (375,0)},
{\ar@{.} (-30,30); "1"},
{\ar@{.} "1"; "2"},
{\ar@{.} "2"; (375,30)},
{\ar@{.} (-30,60); "3"},
{\ar@{.} "3"; "4"},
{\ar@{.} "4"; "5"},
{\ar@{.} "5"; (375,60)},
{\ar@{.} (-30,90); "6"},
{\ar@{.} "6"; "7"},
{\ar@{.} "7"; "8"},
{\ar@{.} "8"; "9"},
{\ar@{.} "9"; (375,90)},
{\ar@{.} (0,100); "6"},
{\ar@{.} "6"; (0,-10)},
{\ar@{.} (60,100); "3"},
{\ar@{.} "3"; (60,-10)},
{\ar@{.} (120,100); "7"},
{\ar@{.} "7"; "1"},
{\ar@{.} "1"; (120,-10)},
{\ar@{.} (180,100); "4"},
{\ar@{.} "4"; "0"},
{\ar@{.} "0"; (180,-10)},
{\ar@{.} (240,100); "8"},
{\ar@{.} "8"; "2"},
{\ar@{.} "2"; (240,-10)},
{\ar@{.} (300,100); "5"},
{\ar@{.} "5"; (300,-10)},
{\ar@{.} (360,100); "9"},
{\ar@{.} "9"; (360,-10)},
"0", {\ar"1"},
"2", {\ar"0"},
"1", {\ar"3"},
"4", {\ar"1"},
"2", {\ar"4"},
"5", {\ar"2"},
"3", {\ar"6"},
"7", {\ar"3"},
"4", {\ar"7"},
"8", {\ar"4"},
"5", {\ar"8"},
"9", {\ar"5"},
\end{xy}
}
\hfill
\hfill

\noindent(ii) Let $\mathcal{Q}$ be the following quiver :  

\hfill
\begin{xy} 0;<1pt,0pt>:<0pt,-1pt>::
(0,0) *+{1} ="0",
(60,0) *+{2} ="1",
(120,0) *+{3} ="2",
(180,0) *+{4} ="3",
"0", {\ar"1"},
"2", {\ar"1"},
"3", {\ar"2"},
\end{xy}
\hfill
\hfill

Then $\phi_{\mathcal{Q}}=s_2s_1s_3s_4$. Take an associated height function $\xi$ such that $\xi(1)=0$. Then the Auslander-Reiten quiver $\Gamma_{\mathcal{Q}}$ and its labelling $\widehat{I}_{\xi}$ are described as follows : 

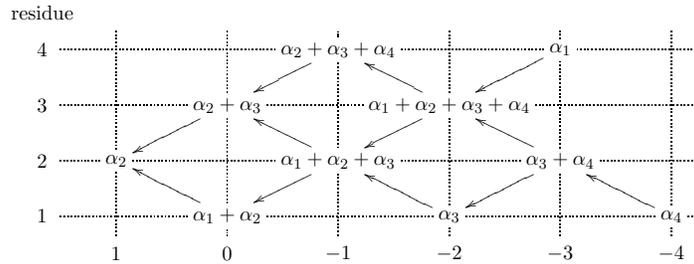
\hfill
\scalebox{0.7}[0.7]{
\begin{xy} 0;<1pt,0pt>:<0pt,-1pt>::
(-40,-20) *+{\text{residue}}, 
(-40,0) *+{4}, 
(-40,30) *+{3}, 
(-40,60) *+{2}, 
(-40,90) *+{1}, 
(0,110) *+{1}, 
(60,110) *+{0}, 
(120,110) *+{-1}, 
(180,110) *+{-2}, 
(240,110) *+{-3}, 
(300,110) *+{-4}, 
(120,0) *+{\alpha_2+\alpha_3+\alpha_4} ="0",
(240,0) *+{\alpha_1} ="6",
(60,30) *+{\alpha_2+\alpha_3} ="1",
(180,30) *+{\alpha_1+\alpha_2+\alpha_3+\alpha_4} ="2",
(0,60) *+{\alpha_2} ="3",
(120,60) *+{\alpha_1+\alpha_2+\alpha_3} ="4",
(240,60) *+{\alpha_3+\alpha_4} ="5",
(60,90) *+{\alpha_1+\alpha_2} ="7",
(180,90) *+{\alpha_3} ="8",
(300,90) *+{\alpha_4} ="9", 
{\ar@{.} (-30,0); "0"},
{\ar@{.} "0"; "6"},
{\ar@{.} "6"; (315,0)},
{\ar@{.} (-30,30); "1"},
{\ar@{.} "1"; "2"},
{\ar@{.} "2"; (315,30)},
{\ar@{.} (-30,60); "3"},
{\ar@{.} "3"; "4"},
{\ar@{.} "4"; "5"},
{\ar@{.} "5"; (315,60)},
{\ar@{.} (-30,90); "7"},
{\ar@{.} "7"; "8"},
{\ar@{.} "8"; "9"},
{\ar@{.} "9"; (315,90)},
{\ar@{.} (0,100); "3"},
{\ar@{.} "3"; (0,-10)},
{\ar@{.} (60,100); "7"},
{\ar@{.} "7"; "1"},
{\ar@{.} "1"; (60,-10)},
{\ar@{.} (120,100); "4"},
{\ar@{.} "4"; "0"},
{\ar@{.} "0"; (120,-10)},
{\ar@{.} (180,100); "8"},
{\ar@{.} "8"; "2"},
{\ar@{.} "2"; (180,-10)},
{\ar@{.} (240,100); "5"},
{\ar@{.} "5"; "6"},
{\ar@{.} "6"; (240,-10)},
{\ar@{.} (300,100); "9"},
{\ar@{.} "9"; (300,-10)},
"0", {\ar"1"},
"2", {\ar"0"},
"1", {\ar"3"},
"4", {\ar"1"},
"2", {\ar"4"},
"5", {\ar"2"},
"6", {\ar"2"},
"7", {\ar"3"},
"4", {\ar"7"},
"8", {\ar"4"},
"5", {\ar"8"},
"9", {\ar"5"},
\end{xy}
}
\hfill
\hfill

\end{example}

In the following, we frequently use the following \emph{convexity of $\Gamma_{\mathcal{Q}}$}. 
\begin{lemma}[{Convexity of $\Gamma_{\mathcal{Q}}$. cf.~\cite[Proof of Lemma 5.8]{HL:qGro}}]\label{l:dia}
For $i\in I$, the set $\{r\in \mathbb{Z}\mid (i, r)\in \widehat{I}_{\xi}\}$ is a 2-segment. Moreover, if $(i, r-1)$ and $(i, r+1)$ belong to $\widehat{I}_{\xi}$, then 
\[
\{(j, r)\mid j\in I, (\alpha_i, \alpha_j)<0\}\subset \widehat{I}_{\xi}.
\]
\end{lemma}
\begin{remark}\label{r:dia}
There are arrows from $(i, r-1)$ to the vertices $\{(j, r)\mid j\in I, (\alpha_i, \alpha_j)<0\}$ and arrows from them to $(i, r+1)$. Moreover, those arrows are all the arrows exiting out of $(i, r-1)$ and entering into $(i, r+1)$, respectively. 
\end{remark}

\subsection{The monoidal category $\mathcal{C}_{\mathcal{Q}}$}\label{ss:HLsubcat}
Let $\mathcal{Q}$ be the Dynkin quiver of type $\mathrm{X}_N$ with $\mathrm{X}=\mathrm{A}, \mathrm{D}, \mathrm{E}$, and $\xi$ an associated height function. As mentioned in the beginning of section \ref{s:subcat}, we provide the definition of the monoidal subcategory $\mathcal{C}_{\mathcal{Q}}$ of $\mathcal{C}_{\bullet}$ of type $\mathrm{X}_{N}^{(1)}$. In fact, we do not use $\mathcal{C}_{\mathcal{Q}}$ until section \ref{s:AB}, 
but the definition of the nice monoidal subcategories in the case of type $\mathrm{B}_n^{(1)}$ is a generalization of this construction (subsection \ref{ss:subcat}). See subsection \ref{ss:HLisom} and \cite{HL:qGro} for properties of  $\mathcal{C}_{\mathcal{Q}}$.

\begin{definition}[{\cite[subsection 5.11]{HL:qGro}}]\label{d:HLsubcat}
Recall the labelling $\widehat{I}_{\xi}$ of the positive roots $\Delta_+$ defined just after Proposition \ref{p:ARlabel}. Set $\mathbb{Y}^{\xi}:=\{Y_{i, r}\mid (i, r)\in \widehat{I}_{\xi}\}\subset\mathcal{Y}$, and $\mathbb{B}^{\xi}$ be the set of dominant monomials in the variables in $\mathbb{Y}^{\xi}$. 

Define $\mathcal{C}_{\mathcal{Q}}^{\xi}$ as the full subcategory of the category of $\mathcal{C}_{\bullet}$ of type $\mathrm{X}_{N}^{(1)}$ whose objects have all their composition factors isomorphic to $L(m)$ for some $m\in \mathbb{B}^{\xi}$. 
\end{definition}
\begin{remark}\label{r:shift}
In the following, we write $\mathcal{C}_{\mathcal{Q}}^{\xi}$ simply as
$\mathcal{C}_{\mathcal{Q}}$. Note that the height function is uniquely determined by the quiver $\mathcal{Q}$ \emph{up to shift}, that is, if $\xi'$ is another height function corresponding to $\mathcal{Q}$, then $\xi-\xi'$ is a constant function. Hence the categories $\mathcal{C}_{\mathcal{Q}}^{\xi}$ and $\mathcal{C}_{\mathcal{Q}}^{\xi'}$ are isomorphic via appropriate shift of the spectral parameter. 
\end{remark}
\begin{remark}\label{r:HLconv}
The monoidal subcategory $\mathcal{C}_{\mathcal{Q}}$ corresponds to the monoidal subcategory $\mathcal{C}_{\mathcal{Q}^{\mathrm{op}}}$ in \cite{HL:qGro}, here $\mathcal{Q}^{\mathrm{op}}$ is the Dynkin quiver obtained from $\mathcal{Q}$ by reversing the orientation of all arrows. This difference arises from conventional difference of the definitions of adaptedness and height functions.
\end{remark}
By using convexity of $\Gamma_{\mathcal{Q}}$, we can prove the following : 
\begin{lemma}[{\cite[Lemma 5.8]{HL:qGro}}]
The subcategory $\mathcal{C}_{\mathcal{Q}}$ is closed under tensor product. 
\end{lemma}
\begin{example}\label{e:HLsubcat}
Let $\mathcal{Q}$ and $\xi$ be the ones in Example \ref{e:AR} (i). Then $\mathcal{C}_{\mathcal{Q}}$ is the full subcategory of the category of $\mathcal{C}_{\bullet}$ of type $\mathrm{A}_{4}^{(1)}$ whose objects have all their composition factors isomorphic to 
\[
L(Y_{1,0}^{u_{1, 0}}Y_{1,-2}^{u_{1, -2}}Y_{1,-4}^{u_{1, -4}}Y_{1,-6}^{u_{1, -6}}Y_{2,-1}^{u_{2, -1}}Y_{2,-3}^{u_{2, -3}}Y_{2, -5}^{u_{2, -5}}Y_{3,-2}^{u_{3, -2}}Y_{3,-4}^{u_{3, -4}}Y_{4, -3}^{u_{4,-3}}) 
\]
for some $(u_{i, r})_{(i, r)\in \widehat{I}_{\xi}}\in \mathbb{Z}_{\geq 0}^{\widehat{I}_{\xi}}$. In fact, $\mathcal{C}_{\mathcal{Q}}$ is the smallest abelian full subcategory of $\mathcal{C}_{\bullet}$ such that 
\begin{itemize}
\item[(1)] it is stable by taking subquotient, tensor product and extension, 
\item[(2)] it contains the simple modules $L(Y_{i, r})$ with $(i, r)\in \widehat{I}_{\xi}$ and the trivial module. 
\end{itemize}
See the proof of Lemma \ref{l:subcat} below.
\end{example}
\subsection{Combinatorial Auslander-Reiten quivers}\label{ss:cAR}
The Auslander-Reiten quiver $\Gamma_{\mathcal{Q}}$, by definition,  describes the category of finite-dimensional modules over the path algebra $\mathbb{C}\mathcal{Q}$. Interestingly, it is known that $\Gamma_{\mathcal{Q}}$ also visualizes a certain partial ordering on the set $\Delta_+$ of positive roots determined from $\mathcal{Q}$ (see Proposition  \ref{p:adapted} below). This point of view leads to a combinatorial generalization of Auslander-Reiten quivers \cite{OhS:cAR}, which is an important tool in this paper. 

We consider the general setting in subsection \ref{ss:QLA} unless otherwise specified.  
\begin{definition}
Two sequences $\bm{i}$ and $\bm{i}'$ of elements of $I$ are said to be \emph{commutation equivalent} if $\bm{i}$ is obtained from $\bm{i}$ by a sequence of transformations of adjacent components from $(i, j)$ with $(\alpha_i, \alpha_j)=0$ into $(j, i)$. This is an equivalence relation, and the equivalence class of $\bm{i}$ is called \emph{the commutation class of $\bm{i}$}, denoted by $[\bm{i}]$. 
\end{definition}
Let $w_0$ be the longest element of the Weyl group $W$. For $\bm{i}=(i_1, i_2,\dots, i_{\ell})\in I(w_0)$, set 
\begin{align}
\beta_k^{\bm{i}}:=s_{i_1}\cdots s_{i_{k-1}}(\alpha_{i_k})\label{eq:label}
\end{align}
for $k=1,\dots, \ell$. Then it is well-known that $\beta_k^{\bm{i}}\neq \beta_{k'}^{\bm{i}}$ if $k\neq k'$ and the set $\{\beta_k^{\bm{i}}\mid k=1,\dots, \ell\}$ coincides with the set $\Delta_+$ of positive roots of $\mathfrak{g}$ \cite{Bour:Lie}. For $k=1,\dots, \ell$, $i_k$ is called \emph{the residue} of $\beta_k^{\bm{i}}$. Write 
\begin{align*}
\res^{[\bm{i}]}(\beta_{k}^{\bm{i}}):=i_k. 
\end{align*}
Note that the residues of the positive roots depend only on the commutation class of $\bm{i}$.

We consider the total ordering $<_{\bm{i}}$ on $\Delta_+$ such that 
\begin{center}
$\beta_k^{\bm{i}}<_{\bm{i}}\beta_{k'}^{\bm{i}}$ if and only if $k<k'$. 
\end{center}

There is also the partial ordering $\prec_{[\bm{i}]}$ on $\Delta_+$ associated with the commutation class $[\bm{i}]$ of $\bm{i}$ such that, for $\beta, \beta'\in \Delta_+$, 
\begin{center}
$\beta \prec_{[\bm{i}]} \beta'$ if and only if $\beta <_{\bm{i}'} \beta'$ for all $\bm{i}'\in [\bm{i}]$. 
\end{center}

\begin{proposition}[{\cite[section 4]{Lus:can1}, \cite[section 6]{Rin:PBW}, \cite[section 2]{Bed}}]\label{p:adapted}
Suppose that $\mathfrak{g}$ is of type $\mathrm{X}_N$ with $\mathrm{X}=\mathrm{A}, \mathrm{D}, \mathrm{E}$. Let $\mathcal{Q}$ be a Dynkin quiver of type $\mathrm{X}_N$.  
Then the following hold :

\item[(1)] A reduced word $\bm{i}\in I(w_0)$ is adapted to at most one Dynkin quiver $\mathcal{Q}$. 
\item[(2)] For each Dynkin quiver $\mathcal{Q}$, there exists a reduced word $\bm{i}_{\mathcal{Q}}$ of $w_0$ adapted to $\mathcal{Q}$. Moreover, all reduced words in its commutation class $[\bm{i}_{\mathcal{Q}}]$ are adapted to $\mathcal{Q}$, and all reduced words adapted to $\mathcal{Q}$ belong to $[\bm{i}_{\mathcal{Q}}]$. Hence this commutation class is denoted by $[\mathcal{Q}]$. 

\item[(3)] There exists $\bm{i}_c=(i_1,i_2\dots, i_{\ell})\in [\mathcal{Q}]$ such that $(i_1,i_2,\dots, i_{N})\in I(\phi_{\mathcal{Q}})$ (recall \eqref{eq:cox}).  

\item[(4)] For $\beta, \beta'\in \Delta_+$, we have $\beta\prec_{[\mathcal{Q}]}\beta'$ if and only if there is a path from $\beta'$ to $\beta$ in $\Gamma_{\mathcal{Q}}$. 

\item[(5)] For $\beta\in \Delta_+$ with $\Omega_{\xi}(\beta)=(i, r)$ (recall Proposition \ref{p:ARlabel}), we have $\res^{[\mathcal{Q}]}(\beta)=i$. 
\end{proposition}

Proposition \ref{p:adapted} (4) states that the Auslander-Reiten quiver $\Gamma_{\mathcal{Q}}$ is regarded as a visualization of the partial ordering $\prec_{[\mathcal{Q}]}$. A visualization of an arbitrary $\prec_{[\bm{i}]}$ (of an arbitrary type) is given by \emph{a combinatorial Auslander-Reiten quiver}, introduced by Oh and Suh \cite{OhS:cAR}. The following is their explicit constructive definition :  

\begin{definition}[{\cite[Algorithm 2.1]{OhS:cAR}}]
Let $\bm{i}\in I(w_0)$. The combinatorial Auslander-Reiten quiver $\Upsilon_{[\bm{i}]}$ associated with $[\bm{i}]$ is constructed as follows : 
\begin{itemize}
\item[(i)] the set of vertices consists of $\ell (w_0)$ elements labelled by $\Delta_+$. 
\item[(ii)] there is an arrow from $\beta_k^{\bm{i}}$ to $\beta_{k'}^{\bm{i}}$ (see \eqref{eq:label}) if and only if 
\begin{align*}
\mathrm{(a)}\ k>k'&&&\mathrm{(b)}\ (\alpha_{i_k}, \alpha_{i_{k'}})<0&&&\text{and}&&&\mathrm{(c)}\ \{t\mid k'<t<k, i_t=i_{k'}\ \text{or}\ i_k\}=\emptyset. 
\end{align*}
\end{itemize}
In (ii), we use the reduced word $\bm{i}$ but it is easy to show that, even if we choose another reduced word $\bm{i}'$ commutation equivalent to $\bm{i}$, we obtain the same quiver. Hence the quiver $\Upsilon_{[\bm{i}]}$ depends only on the commutation class $[\bm{i}]$ of $\bm{i}$. 
\end{definition}
\begin{remark}
Oh and Suh considered the combinatorial Auslander-Reiten quivers associated with (commutation classes of) reduced words of any element $w\in W$ in the same manner. Moreover, they considered colors of each arrow. However, we do not need them in this paper.   
\end{remark}
The combinatorial Auslander-Reiten quivers defined above satisfy the desired property  :
\begin{theorem}[{\cite[Theorem 2.22, Theorem 2.29]{OhS:cAR}}]\label{t:commequiv}
Two reduced words $\bm{i}, \bm{i}'\in I(w_0)$ are commutation equivalent if and only if $\Upsilon_{[\bm{i}]}=\Upsilon_{[\bm{i}']}$. Moreover, for $\beta, \beta'\in \Delta_+$, we have $\beta\prec_{[\bm{i}]}\beta'$ if and only if there is a path from $\beta'$ to $\beta$ in $\Upsilon_{[\bm{i}]}$. 
\end{theorem}
A sequence $(j_1, j_2,\dots, j_l)$ is called \emph{a compatible reading} of $\Upsilon_{[\bm{i}]}$ if 
\begin{align}
(j_1, j_2,\dots, j_l)=(\res^{[\bm{i}]}(\gamma_{1}), \res^{[\bm{i}]}(\gamma_{2}),\dots, \res^{[\bm{i}]}(\gamma_{\ell}))\label{eq:read}
\end{align}
for a total ordering of its vertex set $\Delta_+=\{\gamma_1, \gamma_2,\dots, \gamma_{\ell}\}$ such that $k'<k$ whenever there is an arrow from $\gamma_k$ to $\gamma_{k'}$.

\begin{theorem}[{\cite[Theorem 2.23]{OhS:cAR}}]\label{t:comread}
Every reduced expression of $\bm{i}'$ in $[\bm{i}]$ can be obtained by a compatible reading of $\Upsilon_{[\bm{i}]}$. 
\end{theorem}
The following theorem guarantees that the combinatorial Auslander-Reiten quiver is a combinatorial generalization of the usual Auslander-Reiten quiver. 
\begin{theorem}[{\cite[Theorem 2.28]{OhS:cAR}}]\label{t:AR}
Let $\mathcal{Q}$ be a Dynkin quiver of type $\mathrm{X}_N$ with $\mathrm{X}=\mathrm{A}, \mathrm{D}, \mathrm{E}$. Then the combinatorial Auslander-Reiten quiver $\Upsilon_{[\mathcal{Q}]}$ is isomorphic to the Auslander-Reiten quiver $\Gamma_{\mathcal{Q}}$, here the correspondence between their vertex sets is given by Gabriel's theorem.  
\end{theorem}
In the following, we identify $\Upsilon_{[\mathcal{Q}]}$ with $\Gamma_{\mathcal{Q}}$ by Theorem \ref{t:AR}. 

\subsection{Twisted adapted classes of type $\mathrm{A}_{2n-1}$}\label{ss:tadapted}
To define an analogue of $\mathcal{C}_{\mathcal{Q}}$ in the case of type $\mathrm{B}_n^{(1)}$, we need \emph{the twisted Auslander-Reiten quivers}  \cite{OhS:foldedAR-I}, which are specific combinatorial Auslander-Reiten quivers for type $\mathrm{A}_{2n-1}$. 

Let $\mathcal{Q}$ be a Dynkin quiver of type $\mathrm{A}_{2n-2}$, and $\xi\colon I=\{1, 2,\dots, 2n-2\}\to\mathbb{Z}$ be an associated height function. We identify the vertex set of $\Upsilon_{[\mathcal{Q}]}$ with $\widehat{I}_{\xi}$ (see the definition after Proposition \ref{p:ARlabel}). We construct the quiver $\Upsilon_{[\mathcal{Q}^{>}]}$ (resp.~$\Upsilon_{[\mathcal{Q}^{<}]}$), called \emph{a twisted Auslander-Reiten quiver}, from $\Upsilon_{[\mathcal{Q}]}$ as follows :
\begin{itemize}
\item[(T1)] Any vertex $(\imath, r)$ in $\widehat{I}_{\xi}$ with $\imath\geq n$ is renamed by $(\imath+1, r)$. 

\item[(T2)] Put a vertex $(n, r-\frac{1}{2})$ ($r\in \mathbb{Z}$) if there is an arrow $(n-1, r-1)\to (n+1, r)$ or $(n+1, r-1)\to (n-1, r)$. 

\item[(T3)] Each arrow of the form $(n-1, r-1)\to (n+1, r)$ is replaced by two arrows $(n-1, r-1)\to (n, r-\frac{1}{2})\to (n+1, r)$, and each arrow of the form $(n+1, r-1)\to (n-1, r)$ is replaced by two arrows $(n+1, r-1)\to (n, r-\frac{1}{2})\to (n-1, r)$. 

\item[(T4)] Let $r_{\mathrm{M}}:=\max\{r\in \mathbb{Z}\mid (n-1, r), (n, r)\in \widehat{I}_{\xi}\}$ (resp.~$r_{\mathrm{m}}:=\min\{r\in \mathbb{Z}\mid (n-1, r), (n, r)\in \widehat{I}_{\xi}\}$), and 
\begin{align*}
n_{\mathrm{M}}:=\begin{cases}
n-1&\text{if}\ (n-1, r_{\mathrm{M}})\in \widehat{I}_{\xi},\\
n+1&\text{if}\ (n, r_{\mathrm{M}})\in \widehat{I}_{\xi}. 
\end{cases}
&&&
\left(
n_{\mathrm{m}}:=\begin{cases}
n-1&\text{if}\ (n-1, r_{\mathrm{m}})\in \widehat{I}_{\xi},\\
n+1&\text{if}\ (n, r_{\mathrm{m}})\in \widehat{I}_{\xi}. 
\end{cases}
\right)
\end{align*}
Then we put a vertex $(n, r_{\mathrm{M}}+\frac{1}{2})$ (resp.~$(n, r_{\mathrm{m}}-\frac{1}{2})$) (on the quiver obtained in (T3)) and add an arrow $(n_{\mathrm{M}}, r_{\mathrm{M}})\to (n, r_{\mathrm{M}}+\frac{1}{2})$ (resp.~$(n, r_{\mathrm{m}}-\frac{1}{2})\to (n_{\mathrm{m}}, r_{\mathrm{m}})$). 
\end{itemize}

Let $\flat\in \{>, <\}$. For a vertex $(i, \frac{r}{2})$ of $\Upsilon_{[\mathcal{Q}^{\flat}]}$, set $\res^{[\mathcal{Q}^{\flat}]}((i, \frac{r}{2})):=i$. Then we can define a compatible reading of $\Upsilon_{[\mathcal{Q}^{\flat}]}$ in the same way as \eqref{eq:read}. The following theorem justifies these notations. 

\begin{theorem}[{\cite[Theorem 4.17, Remark 4.20, Algorithm 5.1]{OhS:foldedAR-I}}]\label{t:tw-const}
Let $\flat\in \{>, <\}$. Then a compatible reading of $\Upsilon_{[\mathcal{Q}^{\flat}]}$ gives a reduced word of the longest element of the Weyl group of type $\mathrm{A}_{2n-1}$. Moreover, the quiver $\Upsilon_{[\mathcal{Q}^{\flat}]}$ is isomorphic to its combinatorial Auslander-Reiten quiver in a residue preserving way. 
\end{theorem}
\begin{definition}\label{d:tadapted}
Let $\flat\in \{>, <\}$. The commutation class corresponding to the twisted Auslander-Reiten quiver $\Upsilon_{[\mathcal{Q}^{\flat}]}$ is called \emph{a twisted adapted class of type $\mathrm{A}_{2n-1}$}, and denoted by $[\mathcal{Q}^{\flat}]$. 
\end{definition}
\begin{remark}\label{r:tadapted}
Our definition of twisted adapted classes of type $\mathrm{A}_{2n-1}$ is equivalent to the one in \cite[Definition 4.2]{OhS:foldedAR-I} by the results of \cite[section 4]{OhS:foldedAR-I}. 
\end{remark}

\begin{remark}\label{r:adapted}
Every twisted adapted class of type $\mathrm{A}_{2n-1}$ is non-adapted \cite[section 4]{OhS:foldedAR-I}.
\end{remark}

Let $\flat\in \{>, <\}$.  In the following, we regard the vertex set of $\Upsilon_{[\mathcal{Q}^{\flat}]}$ as $\Delta_+$, as before. Its another labelling set obtained by (T1)--(T4) will be written as $\widehat{I}_{\xi}^{\mathrm{tw}, \flat}$. Moreover, this relabelling is written as $\Omega_{\xi}^{\mathrm{tw}, \flat}\colon \Delta_+ \to \widehat{I}_{\xi}^{\mathrm{tw}, \flat}$. In addition to $\Omega_{\xi}^{\mathrm{tw}, \flat}$, define $\overline{\Omega}_{\xi}^{\mathrm{tw}, \flat}\colon \Delta_+ \to \{1,2,\dots, n\}\times \mathbb{Z}$ as follows : for $\beta\in \Delta_+$ with $\Omega_{\xi}^{\mathrm{tw}, \flat}(\beta)=(\imath, \frac{r}{2})$, 
\begin{align*}
\overline{\Omega}_{\xi}^{\mathrm{tw}, \flat}(\beta)=\begin{cases}
(\imath, r)&\text{if}\ \imath\leq n,\\
(2n-\imath, r)&\text{if}\ \imath>n.
\end{cases}
\end{align*}
Then $\overline{\Omega}_{\xi}^{\mathrm{tw}, \flat}$ is also injective, and set 
\begin{align*}
\overline{I}_{\xi}^{\mathrm{tw}, \flat}:=\Image (\overline{\Omega}_{\xi}^{\mathrm{tw}, \flat})
\end{align*}

\begin{example}\label{e:tw-AR}
Consider the case of type $\mathrm{A}_{5}$ ($n=3$). 

\noindent(i) Let $\mathcal{Q}$ and $\xi$ be the ones in Example \ref{e:AR} (i). The twisted Auslander-Reiten quiver $\Upsilon_{[\mathcal{Q}^{>}]}$ and its labellings $\widehat{I}_{\xi}^{\mathrm{tw}, >}$, $\overline{I}_{\xi}^{\mathrm{tw}, >}$ are described as follows : 

\noindent
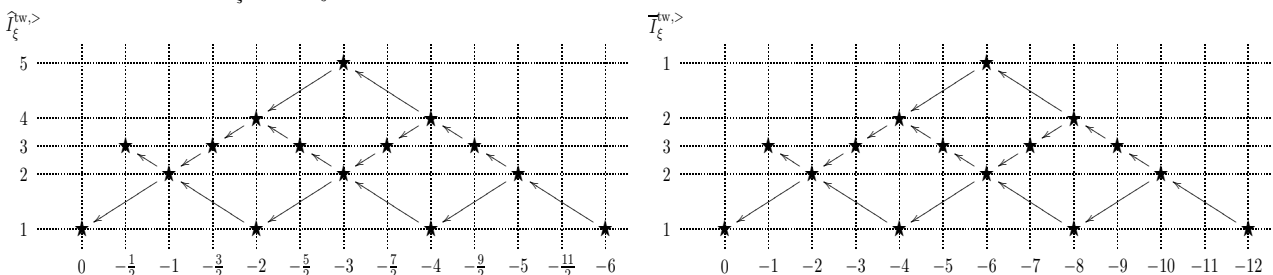
\scalebox{0.55}[0.7]{
\begin{xy} 0;<1pt,0pt>:<0pt,-1pt>::
(-40,-20) *+{\widehat{I}_{\xi}^{\mathrm{tw}, >}}, 
(-40,0) *+{5}, 
(-40,30) *+{4}, 
(-40,45) *+{3}, 
(-40,60) *+{2}, 
(-40,90) *+{1}, 
(0,110) *+{0}, 
(30,110) *+{-\frac{1}{2}}, 
(60,110) *+{-1}, 
(90,110) *+{-\frac{3}{2}}, 
(120,110) *+{-2}, 
(150,110) *+{-\frac{5}{2}}, 
(180,110) *+{-3}, 
(210,110) *+{-\frac{7}{2}}, 
(240,110) *+{-4}, 
(270,110) *+{-\frac{9}{2}}, 
(300,110) *+{-5}, 
(330,110) *+{-\frac{11}{2}}, 
(360,110) *+{-6}, 
(180,0) *+{\bigstar} ="0",
(120,30) *+{\bigstar} ="1",
(240,30) *+{\bigstar} ="2",
(60,60) *+{\bigstar} ="3",
(180,60) *+{\bigstar} ="4",
(300,60) *+{\bigstar} ="5",
(0,90) *+{\bigstar} ="6",
(120,90) *+{\bigstar} ="7",
(240,90) *+{\bigstar} ="8",
(360,90) *+{\bigstar} ="9",
(30,45) *+{\bigstar} ="10",
(90,45) *+{\bigstar} ="11",
(150,45) *+{\bigstar} ="12",
(210,45) *+{\bigstar} ="13",
(270,45) *+{\bigstar} ="14",
{\ar@{.} (-30,0); (375,0)},
{\ar@{.} (-30,30); (375,30)},
{\ar@{.} (-30,45); (375,45)},
{\ar@{.} (-30,60); (375,60)},
{\ar@{.} (-30,90); (375,90)},
{\ar@{.} (0,100); (0,-10)},
{\ar@{.} (30,100); (30,-10)},
{\ar@{.} (60,100); (60,-10)},
{\ar@{.} (90,100); (90,-10)},
{\ar@{.} (120,100); (120,-10)},
{\ar@{.} (150,100); (150,-10)},
{\ar@{.} (180,100); (180,-10)},
{\ar@{.} (210,100); (210,-10)},
{\ar@{.} (240,100); (240,-10)},
{\ar@{.} (270,100); (270,-10)},
{\ar@{.} (300,100); (300,-10)},
{\ar@{.} (330,100); (330,-10)},
{\ar@{.} (360,100); (360,-10)},
"0", {\ar"1"},
"2", {\ar"0"},
"3", {\ar"6"},
"7", {\ar"3"},
"4", {\ar"7"},
"8", {\ar"4"},
"5", {\ar"8"},
"9", {\ar"5"},
"3", {\ar"10"},
"1", {\ar"11"},
"11", {\ar"3"},
"12", {\ar"1"},
"4", {\ar"12"},
"2", {\ar"13"},
"13", {\ar"4"},
"14", {\ar"2"},
"5", {\ar"14"},
\end{xy}
}
\scalebox{0.55}[0.7]{
\begin{xy} 0;<1pt,0pt>:<0pt,-1pt>::
(-40,-20) *+{\overline{I}_{\xi}^{\mathrm{tw}, >}}, 
(-40,0) *+{1}, 
(-40,30) *+{2}, 
(-40,45) *+{3}, 
(-40,60) *+{2}, 
(-40,90) *+{1}, 
(0,110) *+{0}, 
(30,110) *+{-1}, 
(60,110) *+{-2}, 
(90,110) *+{-3}, 
(120,110) *+{-4}, 
(150,110) *+{-5}, 
(180,110) *+{-6}, 
(210,110) *+{-7}, 
(240,110) *+{-8}, 
(270,110) *+{-9}, 
(300,110) *+{-10}, 
(330,110) *+{-11}, 
(360,110) *+{-12}, 
(180,0) *+{\bigstar} ="0",
(120,30) *+{\bigstar} ="1",
(240,30) *+{\bigstar} ="2",
(60,60) *+{\bigstar} ="3",
(180,60) *+{\bigstar} ="4",
(300,60) *+{\bigstar} ="5",
(0,90) *+{\bigstar} ="6",
(120,90) *+{\bigstar} ="7",
(240,90) *+{\bigstar} ="8",
(360,90) *+{\bigstar} ="9",
(30,45) *+{\bigstar} ="10",
(90,45) *+{\bigstar} ="11",
(150,45) *+{\bigstar} ="12",
(210,45) *+{\bigstar} ="13",
(270,45) *+{\bigstar} ="14",
{\ar@{.} (-30,0); (375,0)},
{\ar@{.} (-30,30); (375,30)},
{\ar@{.} (-30,45); (375,45)},
{\ar@{.} (-30,60); (375,60)},
{\ar@{.} (-30,90); (375,90)},
{\ar@{.} (0,100); (0,-10)},
{\ar@{.} (30,100); (30,-10)},
{\ar@{.} (60,100); (60,-10)},
{\ar@{.} (90,100); (90,-10)},
{\ar@{.} (120,100); (120,-10)},
{\ar@{.} (150,100); (150,-10)},
{\ar@{.} (180,100); (180,-10)},
{\ar@{.} (210,100); (210,-10)},
{\ar@{.} (240,100); (240,-10)},
{\ar@{.} (270,100); (270,-10)},
{\ar@{.} (300,100); (300,-10)},
{\ar@{.} (330,100); (330,-10)},
{\ar@{.} (360,100); (360,-10)},
"0", {\ar"1"},
"2", {\ar"0"},
"3", {\ar"6"},
"7", {\ar"3"},
"4", {\ar"7"},
"8", {\ar"4"},
"5", {\ar"8"},
"9", {\ar"5"},
"3", {\ar"10"},
"1", {\ar"11"},
"11", {\ar"3"},
"12", {\ar"1"},
"4", {\ar"12"},
"2", {\ar"13"},
"13", {\ar"4"},
"14", {\ar"2"},
"5", {\ar"14"},
\end{xy}
}

For example, $(3, 1, 2, 3, 4, 1, 5, 3, 2, 3, 4, 1, 3, 2, 1), (1, 3, 2, 3, 4, 1, 3, 5, 2, 3, 4, 1, 3, 2, 1)\in [\mathcal{Q}^>]$. 

The twisted Auslander-Reiten quiver $\Upsilon_{[\mathcal{Q}^{<}]}$ and its labellings $\widehat{I}_{\xi}^{\mathrm{tw}, <}$, $\overline{I}_{\xi}^{\mathrm{tw}, <}$ are described as follows : 

\noindent
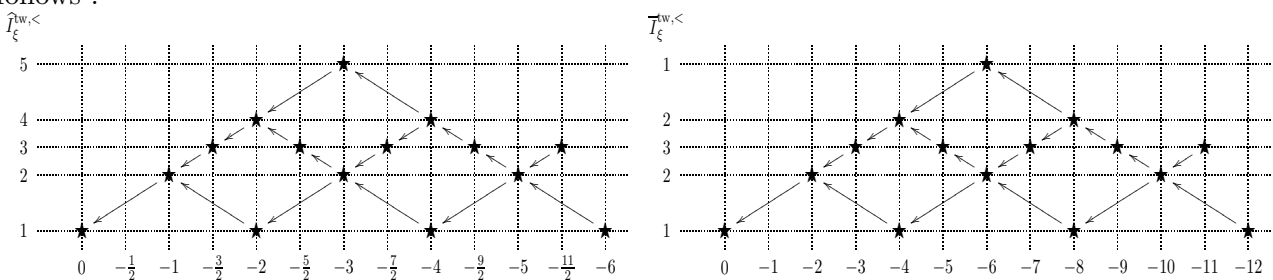
\scalebox{0.55}[0.7]{
\begin{xy} 0;<1pt,0pt>:<0pt,-1pt>::
(-40,-20) *+{\widehat{I}_{\xi}^{\mathrm{tw}, <}}, 
(-40,0) *+{5}, 
(-40,30) *+{4}, 
(-40,45) *+{3}, 
(-40,60) *+{2}, 
(-40,90) *+{1}, 
(0,110) *+{0}, 
(30,110) *+{-\frac{1}{2}}, 
(60,110) *+{-1}, 
(90,110) *+{-\frac{3}{2}}, 
(120,110) *+{-2}, 
(150,110) *+{-\frac{5}{2}}, 
(180,110) *+{-3}, 
(210,110) *+{-\frac{7}{2}}, 
(240,110) *+{-4}, 
(270,110) *+{-\frac{9}{2}}, 
(300,110) *+{-5}, 
(330,110) *+{-\frac{11}{2}}, 
(360,110) *+{-6}, 
(180,0) *+{\bigstar} ="0",
(120,30) *+{\bigstar} ="1",
(240,30) *+{\bigstar} ="2",
(60,60) *+{\bigstar} ="3",
(180,60) *+{\bigstar} ="4",
(300,60) *+{\bigstar} ="5",
(0,90) *+{\bigstar} ="6",
(120,90) *+{\bigstar} ="7",
(240,90) *+{\bigstar} ="8",
(360,90) *+{\bigstar} ="9",
(90,45) *+{\bigstar} ="11",
(150,45) *+{\bigstar} ="12",
(210,45) *+{\bigstar} ="13",
(270,45) *+{\bigstar} ="14",
(330,45) *+{\bigstar} ="10",
{\ar@{.} (-30,0); (375,0)},
{\ar@{.} (-30,30); (375,30)},
{\ar@{.} (-30,45); (375,45)},
{\ar@{.} (-30,60); (375,60)},
{\ar@{.} (-30,90); (375,90)},
{\ar@{.} (0,100); (0,-10)},
{\ar@{.} (30,100); (30,-10)},
{\ar@{.} (60,100); (60,-10)},
{\ar@{.} (90,100); (90,-10)},
{\ar@{.} (120,100); (120,-10)},
{\ar@{.} (150,100); (150,-10)},
{\ar@{.} (180,100); (180,-10)},
{\ar@{.} (210,100); (210,-10)},
{\ar@{.} (240,100); (240,-10)},
{\ar@{.} (270,100); (270,-10)},
{\ar@{.} (300,100); (300,-10)},
{\ar@{.} (330,100); (330,-10)},
{\ar@{.} (360,100); (360,-10)},
"0", {\ar"1"},
"2", {\ar"0"},
"3", {\ar"6"},
"7", {\ar"3"},
"4", {\ar"7"},
"8", {\ar"4"},
"5", {\ar"8"},
"9", {\ar"5"},
"10", {\ar"5"},
"1", {\ar"11"},
"11", {\ar"3"},
"12", {\ar"1"},
"4", {\ar"12"},
"2", {\ar"13"},
"13", {\ar"4"},
"14", {\ar"2"},
"5", {\ar"14"},
\end{xy}
}
\scalebox{0.55}[0.7]{
\begin{xy} 0;<1pt,0pt>:<0pt,-1pt>::
(-40,-20) *+{\overline{I}_{\xi}^{\mathrm{tw}, <}}, 
(-40,0) *+{1}, 
(-40,30) *+{2}, 
(-40,45) *+{3}, 
(-40,60) *+{2}, 
(-40,90) *+{1}, 
(0,110) *+{0}, 
(30,110) *+{-1}, 
(60,110) *+{-2}, 
(90,110) *+{-3}, 
(120,110) *+{-4}, 
(150,110) *+{-5}, 
(180,110) *+{-6}, 
(210,110) *+{-7}, 
(240,110) *+{-8}, 
(270,110) *+{-9}, 
(300,110) *+{-10}, 
(330,110) *+{-11}, 
(360,110) *+{-12}, 
(180,0) *+{\bigstar} ="0",
(120,30) *+{\bigstar} ="1",
(240,30) *+{\bigstar} ="2",
(60,60) *+{\bigstar} ="3",
(180,60) *+{\bigstar} ="4",
(300,60) *+{\bigstar} ="5",
(0,90) *+{\bigstar} ="6",
(120,90) *+{\bigstar} ="7",
(240,90) *+{\bigstar} ="8",
(360,90) *+{\bigstar} ="9",
(90,45) *+{\bigstar} ="11",
(150,45) *+{\bigstar} ="12",
(210,45) *+{\bigstar} ="13",
(270,45) *+{\bigstar} ="14",
(330,45) *+{\bigstar} ="10",
{\ar@{.} (-30,0); (375,0)},
{\ar@{.} (-30,30); (375,30)},
{\ar@{.} (-30,45); (375,45)},
{\ar@{.} (-30,60); (375,60)},
{\ar@{.} (-30,90); (375,90)},
{\ar@{.} (0,100); (0,-10)},
{\ar@{.} (30,100); (30,-10)},
{\ar@{.} (60,100); (60,-10)},
{\ar@{.} (90,100); (90,-10)},
{\ar@{.} (120,100); (120,-10)},
{\ar@{.} (150,100); (150,-10)},
{\ar@{.} (180,100); (180,-10)},
{\ar@{.} (210,100); (210,-10)},
{\ar@{.} (240,100); (240,-10)},
{\ar@{.} (270,100); (270,-10)},
{\ar@{.} (300,100); (300,-10)},
{\ar@{.} (330,100); (330,-10)},
{\ar@{.} (360,100); (360,-10)},
"0", {\ar"1"},
"2", {\ar"0"},
"3", {\ar"6"},
"7", {\ar"3"},
"4", {\ar"7"},
"8", {\ar"4"},
"5", {\ar"8"},
"9", {\ar"5"},
"10", {\ar"5"},
"1", {\ar"11"},
"11", {\ar"3"},
"12", {\ar"1"},
"4", {\ar"12"},
"2", {\ar"13"},
"13", {\ar"4"},
"14", {\ar"2"},
"5", {\ar"14"},
\end{xy}
}

For example, $(1, 2, 3, 4, 1, 5, 3, 2, 3, 4, 1, 3, 2, 1, 3), (1, 2, 3, 4, 1, 3, 5, 2, 3, 4, 1, 3, 2, 3, 1)\in [\mathcal{Q}^{<}]$. 

\noindent(ii) Let $\mathcal{Q}$ and $\xi$ be the ones in Example \ref{e:AR} (ii). The twisted Auslander-Reiten quiver $\Upsilon_{[\mathcal{Q}^{>}]}$ and its labellings $\widehat{I}_{\xi}^{\mathrm{tw}, >}$, $\overline{I}_{\xi}^{\mathrm{tw}, >}$ are described as follows : 

\noindent
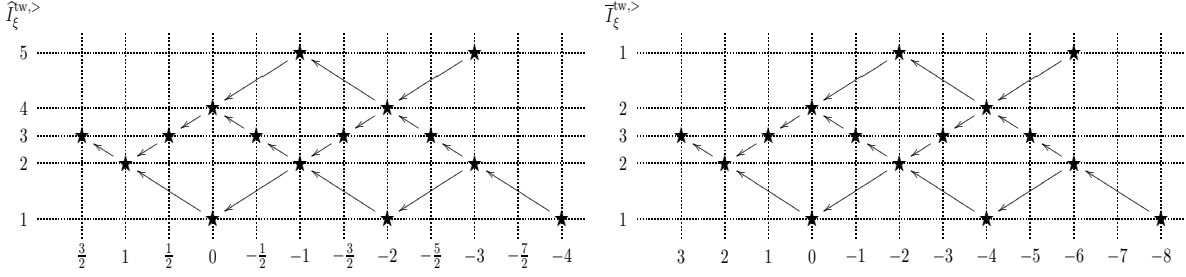
\scalebox{0.55}[0.7]{
\begin{xy} 0;<1pt,0pt>:<0pt,-1pt>::
(-40,-20) *+{\widehat{I}_{\xi}^{\mathrm{tw}, >}}, 
(-40,0) *+{5}, 
(-40,30) *+{4}, 
(-40,45) *+{3}, 
(-40,60) *+{2}, 
(-40,90) *+{1}, 
(0,110) *+{\frac{3}{2}}, 
(30,110) *+{1}, 
(60,110) *+{\frac{1}{2}}, 
(90,110) *+{0}, 
(120,110) *+{-\frac{1}{2}}, 
(150,110) *+{-1}, 
(180,110) *+{-\frac{3}{2}}, 
(210,110) *+{-2}, 
(240,110) *+{-\frac{5}{2}}, 
(270,110) *+{-3}, 
(300,110) *+{-\frac{7}{2}}, 
(330,110) *+{-4}, 
(150,0) *+{\bigstar} ="0",
(90,30) *+{\bigstar} ="1",
(210,30) *+{\bigstar} ="2",
(30,60) *+{\bigstar} ="3",
(150,60) *+{\bigstar} ="4",
(270,60) *+{\bigstar} ="5",
(270,0) *+{\bigstar} ="6",
(90,90) *+{\bigstar} ="7",
(210,90) *+{\bigstar} ="8",
(330,90) *+{\bigstar} ="9",
(0,45) *+{\bigstar} ="10",
(60,45) *+{\bigstar} ="11",
(120,45) *+{\bigstar} ="12",
(180,45) *+{\bigstar} ="13",
(240,45) *+{\bigstar} ="14",
{\ar@{.} (-30,0); (345,0)},
{\ar@{.} (-30,30); (345,30)},
{\ar@{.} (-30,45); (345,45)},
{\ar@{.} (-30,60); (345,60)},
{\ar@{.} (-30,90); (345,90)},
{\ar@{.} (0,100); (0,-10)},
{\ar@{.} (30,100); (30,-10)},
{\ar@{.} (60,100); (60,-10)},
{\ar@{.} (90,100); (90,-10)},
{\ar@{.} (120,100); (120,-10)},
{\ar@{.} (150,100); (150,-10)},
{\ar@{.} (180,100); (180,-10)},
{\ar@{.} (210,100); (210,-10)},
{\ar@{.} (240,100); (240,-10)},
{\ar@{.} (270,100); (270,-10)},
{\ar@{.} (300,100); (300,-10)},
{\ar@{.} (330,100); (330,-10)},
"0", {\ar"1"},
"2", {\ar"0"},
"7", {\ar"3"},
"4", {\ar"7"},
"8", {\ar"4"},
"5", {\ar"8"},
"9", {\ar"5"},
"3", {\ar"10"},
"1", {\ar"11"},
"11", {\ar"3"},
"12", {\ar"1"},
"4", {\ar"12"},
"2", {\ar"13"},
"13", {\ar"4"},
"14", {\ar"2"},
"5", {\ar"14"},
"6", {\ar"2"},
\end{xy}
}
\scalebox{0.55}[0.7]{
\begin{xy} 0;<1pt,0pt>:<0pt,-1pt>::
(-40,-20) *+{\overline{I}_{\xi}^{\mathrm{tw}, >}}, 
(-40,0) *+{1}, 
(-40,30) *+{2}, 
(-40,45) *+{3}, 
(-40,60) *+{2}, 
(-40,90) *+{1}, 
(0,110) *+{3}, 
(30,110) *+{2}, 
(60,110) *+{1}, 
(90,110) *+{0}, 
(120,110) *+{-1}, 
(150,110) *+{-2}, 
(180,110) *+{-3}, 
(210,110) *+{-4}, 
(240,110) *+{-5}, 
(270,110) *+{-6}, 
(300,110) *+{-7}, 
(330,110) *+{-8}, 
(150,0) *+{\bigstar} ="0",
(90,30) *+{\bigstar} ="1",
(210,30) *+{\bigstar} ="2",
(30,60) *+{\bigstar} ="3",
(150,60) *+{\bigstar} ="4",
(270,60) *+{\bigstar} ="5",
(270,0) *+{\bigstar} ="6",
(90,90) *+{\bigstar} ="7",
(210,90) *+{\bigstar} ="8",
(330,90) *+{\bigstar} ="9",
(0,45) *+{\bigstar} ="10",
(60,45) *+{\bigstar} ="11",
(120,45) *+{\bigstar} ="12",
(180,45) *+{\bigstar} ="13",
(240,45) *+{\bigstar} ="14",
{\ar@{.} (-30,0); (345,0)},
{\ar@{.} (-30,30); (345,30)},
{\ar@{.} (-30,45); (345,45)},
{\ar@{.} (-30,60); (345,60)},
{\ar@{.} (-30,90); (345,90)},
{\ar@{.} (0,100); (0,-10)},
{\ar@{.} (30,100); (30,-10)},
{\ar@{.} (60,100); (60,-10)},
{\ar@{.} (90,100); (90,-10)},
{\ar@{.} (120,100); (120,-10)},
{\ar@{.} (150,100); (150,-10)},
{\ar@{.} (180,100); (180,-10)},
{\ar@{.} (210,100); (210,-10)},
{\ar@{.} (240,100); (240,-10)},
{\ar@{.} (270,100); (270,-10)},
{\ar@{.} (300,100); (300,-10)},
{\ar@{.} (330,100); (330,-10)},
"0", {\ar"1"},
"2", {\ar"0"},
"7", {\ar"3"},
"4", {\ar"7"},
"8", {\ar"4"},
"5", {\ar"8"},
"9", {\ar"5"},
"3", {\ar"10"},
"1", {\ar"11"},
"11", {\ar"3"},
"12", {\ar"1"},
"4", {\ar"12"},
"2", {\ar"13"},
"13", {\ar"4"},
"14", {\ar"2"},
"5", {\ar"14"},
"6", {\ar"2"},
\end{xy}
}

For example, $(3, 2, 3, 4, 1, 5, 3, 2, 1, 3, 4, 5, 3, 2, 1), (3, 2, 3, 4, 1, 3, 5, 2, 3, 4, 1, 3, 5, 2, 1)\in [\mathcal{Q}^>]$. 

The twisted Auslander-Reiten quiver $\Upsilon_{[\mathcal{Q}^{<}]}$ and its labellings $\widehat{I}_{\xi}^{\mathrm{tw}, <}$, $\overline{I}_{\xi}^{\mathrm{tw}, <}$ are described as follows : 

\noindent
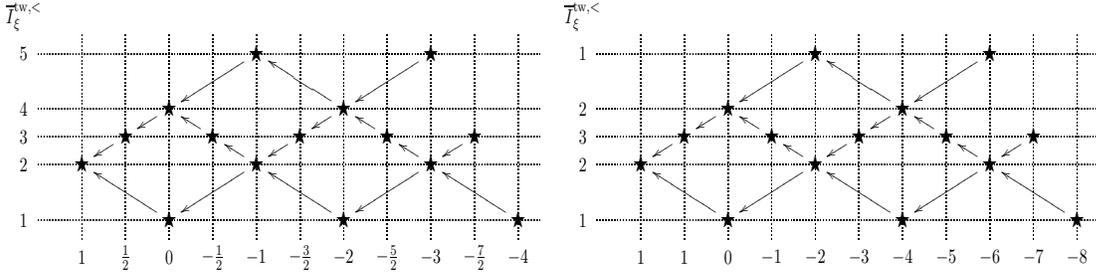
\scalebox{0.55}[0.7]{
\begin{xy} 0;<1pt,0pt>:<0pt,-1pt>::
(-10,-20) *+{\overline{I}_{\xi}^{\mathrm{tw}, <}}, 
(-10,0) *+{5}, 
(-10,30) *+{4}, 
(-10,45) *+{3}, 
(-10,60) *+{2}, 
(-10,90) *+{1}, 
(30,110) *+{1}, 
(60,110) *+{\frac{1}{2}}, 
(90,110) *+{0}, 
(120,110) *+{-\frac{1}{2}}, 
(150,110) *+{-1}, 
(180,110) *+{-\frac{3}{2}}, 
(210,110) *+{-2}, 
(240,110) *+{-\frac{5}{2}}, 
(270,110) *+{-3}, 
(300,110) *+{-\frac{7}{2}}, 
(330,110) *+{-4}, 
(150,0) *+{\bigstar} ="0",
(90,30) *+{\bigstar} ="1",
(210,30) *+{\bigstar} ="2",
(30,60) *+{\bigstar} ="3",
(150,60) *+{\bigstar} ="4",
(270,60) *+{\bigstar} ="5",
(270,0) *+{\bigstar} ="6",
(90,90) *+{\bigstar} ="7",
(210,90) *+{\bigstar} ="8",
(330,90) *+{\bigstar} ="9",
(60,45) *+{\bigstar} ="11",
(120,45) *+{\bigstar} ="12",
(180,45) *+{\bigstar} ="13",
(240,45) *+{\bigstar} ="14",
(300,45) *+{\bigstar} ="10",
{\ar@{.} (0,0); (345,0)},
{\ar@{.} (0,30); (345,30)},
{\ar@{.} (0,45); (345,45)},
{\ar@{.} (0,60); (345,60)},
{\ar@{.} (0,90); (345,90)},
{\ar@{.} (30,100); (30,-10)},
{\ar@{.} (60,100); (60,-10)},
{\ar@{.} (90,100); (90,-10)},
{\ar@{.} (120,100); (120,-10)},
{\ar@{.} (150,100); (150,-10)},
{\ar@{.} (180,100); (180,-10)},
{\ar@{.} (210,100); (210,-10)},
{\ar@{.} (240,100); (240,-10)},
{\ar@{.} (270,100); (270,-10)},
{\ar@{.} (300,100); (300,-10)},
{\ar@{.} (330,100); (330,-10)},
"0", {\ar"1"},
"2", {\ar"0"},
"7", {\ar"3"},
"4", {\ar"7"},
"8", {\ar"4"},
"5", {\ar"8"},
"9", {\ar"5"},
"10", {\ar"5"},
"1", {\ar"11"},
"11", {\ar"3"},
"12", {\ar"1"},
"4", {\ar"12"},
"2", {\ar"13"},
"13", {\ar"4"},
"14", {\ar"2"},
"5", {\ar"14"},
"6", {\ar"2"},
\end{xy}
}
\scalebox{0.55}[0.7]{
\begin{xy} 0;<1pt,0pt>:<0pt,-1pt>::
(-10,-20) *+{\overline{I}_{\xi}^{\mathrm{tw}, <}}, 
(-10,0) *+{1}, 
(-10,30) *+{2}, 
(-10,45) *+{3}, 
(-10,60) *+{2}, 
(-10,90) *+{1}, 
(30,110) *+{1}, 
(60,110) *+{1}, 
(90,110) *+{0}, 
(120,110) *+{-1}, 
(150,110) *+{-2}, 
(180,110) *+{-3}, 
(210,110) *+{-4}, 
(240,110) *+{-5}, 
(270,110) *+{-6}, 
(300,110) *+{-7}, 
(330,110) *+{-8}, 
(150,0) *+{\bigstar} ="0",
(90,30) *+{\bigstar} ="1",
(210,30) *+{\bigstar} ="2",
(30,60) *+{\bigstar} ="3",
(150,60) *+{\bigstar} ="4",
(270,60) *+{\bigstar} ="5",
(270,0) *+{\bigstar} ="6",
(90,90) *+{\bigstar} ="7",
(210,90) *+{\bigstar} ="8",
(330,90) *+{\bigstar} ="9",
(60,45) *+{\bigstar} ="11",
(120,45) *+{\bigstar} ="12",
(180,45) *+{\bigstar} ="13",
(240,45) *+{\bigstar} ="14",
(300,45) *+{\bigstar} ="10",
{\ar@{.} (0,0); (345,0)},
{\ar@{.} (0,30); (345,30)},
{\ar@{.} (0,45); (345,45)},
{\ar@{.} (0,60); (345,60)},
{\ar@{.} (0,90); (345,90)},
{\ar@{.} (30,100); (30,-10)},
{\ar@{.} (60,100); (60,-10)},
{\ar@{.} (90,100); (90,-10)},
{\ar@{.} (120,100); (120,-10)},
{\ar@{.} (150,100); (150,-10)},
{\ar@{.} (180,100); (180,-10)},
{\ar@{.} (210,100); (210,-10)},
{\ar@{.} (240,100); (240,-10)},
{\ar@{.} (270,100); (270,-10)},
{\ar@{.} (300,100); (300,-10)},
{\ar@{.} (330,100); (330,-10)},
"0", {\ar"1"},
"2", {\ar"0"},
"7", {\ar"3"},
"4", {\ar"7"},
"8", {\ar"4"},
"5", {\ar"8"},
"9", {\ar"5"},
"10", {\ar"5"},
"1", {\ar"11"},
"11", {\ar"3"},
"12", {\ar"1"},
"4", {\ar"12"},
"2", {\ar"13"},
"13", {\ar"4"},
"14", {\ar"2"},
"5", {\ar"14"},
"6", {\ar"2"},
\end{xy}
}

For example, $(2, 3, 4, 1, 5, 3, 2, 1, 3, 4, 5, 3, 2, 1, 3), (2, 3, 4, 1, 3, 5, 2, 3, 4, 1, 3, 5, 2, 3, 1)\in [\mathcal{Q}^<]$. 
\end{example}

The following is an analogue of convexity (Lemma \ref{l:dia}) for twisted Auslander-Reiten quivers. We call it \emph{twisted convexity}. Remark the difference between $\widehat{I}_{\xi}^{\mathrm{tw}, \flat}$ and $\overline{I}_{\xi}^{\mathrm{tw}, \flat}$ in the statement. 

\begin{lemma}[{Twisted convexity of $\Upsilon_{[\mathcal{Q}^{\flat}]}$}]\label{l:tw-dia} 
Let $\flat\in \{>, <\}$. For $\imath\in \{1,2, \dots, 2n-1\}$, the set $\{r\in \mathbb{Z}\mid (\imath, \frac{r}{2})\in \widehat{I}_{\xi}^{\mathrm{tw}, \flat}\}$ is a $4$-segment if $\imath\neq n$, and is a $2$-segment if $\imath=n$. Moreover, we have the following : 
\begin{itemize}
\item[(1)] If $(i, r-2)$ and $(i, r+2)$ with $i\leq n-2$ belong to $\overline{I}_{\xi}^{\mathrm{tw}, \flat}$, then $(i-1, r), (i+1, r)\in \overline{I}_{\xi}^{\mathrm{tw}, \flat}$ and there are arrows from $(i, r-2)$ to them and from them to $(i, r+2)$. Moreover, they are all the arrows exiting out of $(i, r-2)$ and entering into $(i, r+2)$, respectively. (When $i=1$, we ignore $(0, r)$ in the statement. We adopt a similar convention below.)

\item[(2)] If $(n-1, r-2)$ and $(n-1, r+2)$ belong to $\overline{I}_{\xi}^{\mathrm{tw}, \flat}$, then $(n-2, r), (n, r-1), (n, r+1)\in \overline{I}_{\xi}^{\mathrm{tw}, \flat}$ and there are arrows from $(n-1, r-2)$ to $(n-2, r)$, $(n, r-1)$ and from $(n-2, r)$, $(n, r+1)$ to $(i, r+2)$. Moreover, they are all the arrows exiting out of $(n-1, r-2)$ and entering into $(n-1, r+2)$, respectively. 

\item[(3)] If $(n, r-1)$ and $(n, r+1)$ belong to $\overline{I}_{\xi}^{\mathrm{tw}, \flat}$, then $(n-1, r)\in \overline{I}_{\xi}^{\mathrm{tw}, \flat}$  and there are arrows from $(n, r-1)$ to $(n-1, r)$ and from $(n-1, r)$ to $(n, r+1)$. Moreover, they are all the arrows exiting out of $(n, r-1)$ and entering into $(n, r+1)$, respectively. 
\end{itemize}
\end{lemma}
\begin{proof}
We only prove the statements for $\overline{I}_{\xi}^{\mathrm{tw}, >}$. The proof for $\overline{I}_{\xi}^{\mathrm{tw}, <}$ is exactly the same.

The construction of $\overline{\Omega}_{\xi}^{\mathrm{tw}, >}$ implies that  the full subquiver of $\Upsilon_{[\mathcal{Q}^{>}]}$ consisting of the vertices $\{\beta\mid \overline{\Omega}_{\xi}^{\mathrm{tw}, >}(\beta)=(i, r)\ \text{with}\ i\leq n-1\}$ is isomorphic to the quiver obtained from $\Upsilon_{[\mathcal{Q}]}$ by deleting the arrows between vertices of residue $n-1$ and $n$. Hence it follows from convexity of $\Upsilon_{[\mathcal{Q}]}$ and Remark \ref{r:dia} that (1) holds and the set $\{r\in \mathbb{Z}\mid (\imath, \frac{r}{2})\in \widehat{I}_{\xi}^{\mathrm{tw}, >}\}$ is a $4$-segment for $\imath\neq n$. 

Next, we consider the setting of (2). It also follows from the argument above that $(n-2, r)\in \overline{I}_{\xi}^{\mathrm{tw}, >}$ and there are arrows $(n-1, r-2)\to (n-2, r)$ and $(n-2, r)\to (n-1, r+2)$. Following (T1)--(T4) in reverse order, we obtain that the vertices $\{(n-1, r-2), (n-1, r+2)\}$ in $\overline{I}_{\xi}^{\mathrm{tw}, \flat}$ correspond to the vertices $\{(n-1, \frac{r}{2}-1), (n-1, \frac{r}{2}+1)\}$ or $\{(n+1, \frac{r}{2}-1), (n+1, \frac{r}{2}+1)\}$ in $\Upsilon_{[\mathcal{Q}]}$ (with the vertex set labelled by $\widehat{I}_{\xi}$). We consider the former case. The latter case is exactly the same. By convexity of $\Upsilon_{[\mathcal{Q}]}$ and Remark \ref{r:dia}, there are arrows $(n-1, \frac{r}{2}-1)\to (n, \frac{r}{2})$ and $(n, \frac{r}{2})\to (n-1, \frac{r}{2}+1)$. Hence it follows from (T1)--(T4) and the definition of $\overline{\Omega}_{\xi}^{\mathrm{tw}, >}$ that the vertices $(n, r-1)$ and $(n, r+1)$ belong to $\overline{I}_{\xi}^{\mathrm{tw}, >}$ and there are the arrows $(n-1, r-2)\to (n, r-1)$ and $(n, r+1)\to (n-1, r+2)$, which prove (2). Since any vertices of residue $n$ in $\Upsilon_{[\mathcal{Q}^{>}]}$ are obtained in this way except for $(\overline{\Omega}_{\xi}^{\mathrm{tw}, >})^{-1}((n, 2r_{\mathrm{M}}+1))$,  and since the sets $\{r\in \mathbb{Z}\mid (n-1, \frac{r}{2})\in \widehat{I}_{\xi}^{\mathrm{tw}, >}\}$ and $\{r\in \mathbb{Z}\mid (n+1, \frac{r}{2})\in \widehat{I}_{\xi}^{\mathrm{tw}, >}\}$ are 4-segments, we obtain that $\{r\in \mathbb{Z}\mid (n, \frac{r}{2})\in \widehat{I}_{\xi}^{\mathrm{tw}, >}\}$ is a 2-segment. 

The assertion (3) immediately follows from  (T1)--(T4) and the definition of $\overline{\Omega}_{\xi}^{\mathrm{tw}, >}$. 
\end{proof}

The following technical lemma will be used in the proof of Theorem \ref{t:torusisom}. The proof is immediate from the definition of $\widehat{I}_{\xi}^{\mathrm{tw}, \flat}$. 

\begin{lemma}\label{l:max}
Let $\flat\in \{>, <\}$. For $\imath\in \{1,2,\dots, 2n-1\}$, set $\Xi_{\imath}:=\max\{r\in \frac{1}{2}\mathbb{Z}\mid (\imath, r)\in \widehat{I}_{\xi}^{\mathrm{tw}, \flat}\}$. Then 
\[
\Xi_{\imath}=\begin{cases}
\xi(\imath)&\text{if}\ 1\leq \imath\leq n-1, \\
\frac{1}{2}(\xi(n-1)+\xi(n))+1&\text{if}\ \imath=n\ \text{and}\ \flat=>, \\
\frac{1}{2}(\xi(n-1)+\xi(n))&\text{if}\ \imath=n\ \text{and}\ \flat=<, \\
\xi(\imath -1)&\text{if}\ n+1\leq \imath\leq 2n-1.
\end{cases}
\] 
\end{lemma}

\subsection{The monoidal subcategories $\mathcal{C}_{\mathcal{Q}^{>}}$ and $\mathcal{C}_{\mathcal{Q}^{>}}$}\label{ss:subcat}
Let $\mathcal{Q}$ be a Dynkin quiver of type $\mathrm{A}_{2n-2}$, $\xi\colon I=\{1, 2,\dots, 2n-2\}\to\mathbb{Z}$ an associated height function, and $\flat\in \{>, <\}$. Since the first components of elements of $\overline{I}_{\xi}^{\mathrm{tw}, \flat}$ are elements of $\{1, 2,\dots, n\}$, we regard them as elements of the index set of the simple roots \emph{of type $\mathrm{B}_n$}. 
\begin{definition}\label{d:subcat}
Consider $\mathcal{Y}$ for type $\mathrm{B}_n^{(1)}$ (see subsection \ref{ss:q-ch} for the definition of $\mathcal{Y}$). Let $\mathbb{Y}^{\xi, \flat}:=\{Y_{i, r}\mid (i, r)\in \overline{I}_{\xi}^{\mathrm{tw}, \flat}\}\subset \mathcal{Y}$, and $\mathbb{B}^{\xi, \flat}$ be the set of dominant monomials in the variables in $\mathbb{Y}^{\xi, \flat}$. 

Define $\mathcal{C}_{\mathcal{Q}^{\flat}}^{\xi}$ as the full subcategory of $\mathcal{C}_{\bullet}$ of type $\mathrm{B}_n^{(1)}$ consisting of the objects all of whose composition factors are isomorphic to $L(m)$ for some $m\in \mathbb{B}^{\xi, \flat}$. 

Note that, by definition, a simple module in $\mathcal{C}_{\mathcal{Q}^{\flat}}$ is isomorphic to $L(m)$ for some $m\in \mathbb{B}^{\xi, \flat}$.  
\end{definition}
\begin{remark}
In the following, we write $\mathcal{C}_{\mathcal{Q}^{\flat}}^{\xi}$ simply as
$\mathcal{C}_{\mathcal{Q}^{\flat}}$ for the same reason as in Remark \ref{r:shift}.  
\end{remark}
Our definition of $\mathcal{C}_{\mathcal{Q}^{\flat}}$ is apparently different from the original definition in \cite[Definition 9.2]{OhS:foldedAR-I}. The following lemma guarantees that this subcategory coincides with the one in \cite{OhS:foldedAR-I} : 
\begin{lemma}\label{l:subcat}
The category $\mathcal{C}_{\mathcal{Q}^{\flat}}$ is the smallest abelian full subcategory of $\mathcal{C}_{\bullet}$ of type $\mathrm{B}_n^{(1)}$ such that 
\begin{itemize}
\item[(1)] it is stable by taking subquotient, tensor product and extension, 
\item[(2)] it contains the simple modules $L(Y_{i, r})$ with $(i, r)\in \overline{I}_{\xi}^{\mathrm{tw}, \flat}$ and the trivial module. 
\end{itemize}

\end{lemma}
\begin{proof}
Let $\mathcal{C}'$ be the full subcategory of $\mathcal{C}_{\bullet}$ defined as in the statement of the lemma. We shall prove that $\mathcal{C}'=\mathcal{C}_{\mathcal{Q}^{\flat}}$. By Proposition \ref{p:tensor}, any module $L(m)$ with $m\in \mathbb{B}^{\xi, \flat}$ belongs to $\mathcal{C}'$. Since $\mathcal{C}'$ is closed under the extension, all objects in $\mathcal{C}_{\mathcal{Q}^{\flat}}$ are contained in $\mathcal{C}'$. Therefore, we only have to check that $\mathcal{C}_{\mathcal{Q}^{\flat}}$ satisfies (1) and (2) in the definition of $\mathcal{C}'$.

The category $\mathcal{C}_{\mathcal{Q}^{\flat}}$ is obviously stable by taking subquotient and extension, and satisfies (2). Therefore it suffices to show that $L(m)\otimes L(m')$ is an object of $\mathcal{C}_{\mathcal{Q}^{\flat}}$ for all $m, m'\in \mathbb{B}^{\xi, \flat}$. 

Suppose that $L(m'')$ is a composition factor of $L(m)\otimes L(m')$. Then $m''$ is a product of monomials of $\chi_q(L(m))$ and $\chi_q(L(m'))$. Hence, by Theorem \ref{t:highest}, we have $m''=mm'M$ where $M$ is a monomial in the variables $A_{i, r}^{-1}$. If $M$ contains the term $A_{i, r}^{-1}$, then, by the explicit form of $A_{i, r}^{-1}$ (cf.~\emph{the left and right negativity of $A_{i, r}^{-1}$} \cite{FM}) and twisted convexity of $\Upsilon_{[\mathcal{Q}^{\flat}]}$ (the segment property), we have $(i, r+r_i), (i, r-r_i)\in \overline{I}_{\xi}^{\mathrm{tw}, \flat}$. Hence it follows from twisted convexity of $\Upsilon_{[\mathcal{Q}^{\flat}]}$ again that $m''$ contains only the variables in $\mathbb{Y}^{\xi, \flat}$, hence, $L(m)\otimes L(m')$ is an object of $\mathcal{C}_{\mathcal{Q}^{\flat}}$. 
\end{proof}
\begin{example}\label{e:subcat}
Let $\mathcal{Q}$ and $\xi$ be the ones in Example \ref{e:AR} (i). See also Example \ref{e:tw-AR}. Then $\mathcal{C}_{\mathcal{Q}^{>}}$ is the full subcategory of the category of $\mathcal{C}_{\bullet}$ of type $\mathrm{B}_{3}^{(1)}$ whose objects have all their composition factors isomorphic to 
\[
L\left(\left(\prod_{s=0}^3Y_{1,-4s}^{u_{1, -4s}}\right)\cdot Y_{1,-6}^{u_{1, -6}}\cdot \left(\prod_{s=0}^2Y_{2,-2-4s}^{u_{2, -2-4s}}\right)\cdot \left(\prod_{s=0}^1Y_{2,-4-4s}^{u_{2, -4-4s}}\right)\cdot \left(\prod_{s=0}^4Y_{3,-1-2s}^{u_{3, -1-2s}}\right)\right) 
\]
for some $(u_{i, r})_{(i, r)\in \overline{I}_{\xi}^{\mathrm{tw}, >}}\in \mathbb{Z}_{\geq 0}^{\overline{I}_{\xi}^{\mathrm{tw}, >}}$. By Lemma \ref{l:subcat}, $\mathcal{C}_{\mathcal{Q}^{>}}$ is the smallest abelian full subcategory of $\mathcal{C}_{\bullet}$ such that 
\begin{itemize}
\item[(1)] it is stable by taking subquotient, tensor product and extension, 
\item[(2)] it contains the simple modules $L(Y_{i, r})$ with $(i, r)\in \overline{I}_{\xi}^{\mathrm{tw}, >}$ and the trivial module. 
\end{itemize}
\end{example}

\begin{remark}\label{r:convention}
The notation for fundamental modules in \cite{OhS:foldedAR-I,KO} is different from ours. The correspondence is given as follows :
\begin{itemize}
\item $q_s(=q^{1/2})$ in \cite{OhS:foldedAR-I,KO} for type $\mathrm{B}_N^{(1)}$ corresponds to our $q$. More precisely, they treat $q_s$ as an indeterminate, however this difference does not cause essential difference concerning the topics of this paper.
\item $V(\varpi_i)_{(-1)^sq^r}$ in \cite{OhS:foldedAR-I,KO} corresponds to our $L(Y_{i, o(i)(-1)^{-h+1+s}q^{-dh^{\vee}+dr}})$, where
\begin{enumerate}
\item $o\colon I\to \{\pm 1\}$ is a map such that $o(i)o(j)=-1$ whenever $a_{ij}<0$, 
\item $h=\begin{cases}
N+1&\text{if}\ \mathrm{X}_N^{(1)}=\mathrm{A}_N^{(1)},\\ 
2N&\text{if}\ \mathrm{X}_N^{(1)}=\mathrm{B}_N^{(1)}, 
\end{cases}
$, $h^{\vee}=\begin{cases}
N+1&\text{if}\ \mathrm{X}_N^{(1)}=\mathrm{A}_N^{(1)},\\
2N-1&\text{if}\ \mathrm{X}_N^{(1)}=\mathrm{B}_N^{(1)}, 
\end{cases}
$, $d=\begin{cases}
1&\text{if}\ \mathrm{X}_N^{(1)}=\mathrm{A}_N^{(1)},\\ 
2&\text{if}\ \mathrm{X}_N^{(1)}=\mathrm{B}_N^{(1)}.
\end{cases}
$. 
\end{enumerate} 
\end{itemize}
The number $h$ (resp.~$h^{\vee}$) is called \emph{the Coxeter number} (resp.~\emph{the dual  Coxeter number}). A map $o$ of the above kind is needed to construct an isomorphism between the quantum loop algebra defined in Definition \ref{d:QEA-Drinfeld} and the one defined by the Drinfeld-Jimbo presentation. See \cite{B:braid}, \cite[Proposition 3.1, Remark 3.3]{Nak:extr} and \cite[subsection 1.3]{KKK:2} for more details. 

In particular, our $\mathcal{C}_{\mathcal{Q}^{\flat}}^{\xi-2h^{\vee}}$ corresponds to the subcategory in \cite{OhS:foldedAR-I} if we choose $o(i)$ as $o(1)=(-1)^{-h+2}$.
\end{remark}
\section{Quantum tori}\label{s:qtori}
In the following, we deal with a $t$-deformed $q$-character homomorphism, called \emph{a $(q, t)$-character}. First, we define a $t$-deformed version of $\mathcal{Y}=\mathbb{Z}[Y_{i, r}^{\pm 1}\mid i\in I, r\in\mathbb{Z}]$ following \cite{H:qt}. The definition in \cite{H:qt} comes from the original construction of $q$-characters in \cite{FR:q-chara} : the variables $Y_{i,p}$ are defined as formal power series that pairwise commute. In \cite{H:qt}, these formal power series are replaced by certain infinite sums which can be seen as vertex operators and give rise to the $t$-commutative variables (the parameter $t$ itself appears as a formal power series). In \cite{Nak:quiver,VV:qGro}, a slightly different quantum torus is defined for the simply-laced case from a convolution operation for certain perverse sheaves on quiver varieties.

\subsection{Quantum Cartan matrices}\label{ss:qCartan}
We return to the general setting in subsection \ref{ss:QLA} again. Recall that $C$ is a Cartan matrix of finite type. Define $C(z):=(C(z)_{ij})_{i,j\in I}, D(z):=(D(z)_{ij})_{i,j\in I}, B(z):=(B(z)_{ij})_{i,j\in I}\in \Mat_{I\times I}(\mathbb{Z}[z^{\pm 1}])$ by 
\begin{align*}
C(z)_{ij}&=\begin{cases}z_i+z_i^{-1}&\text{if}\ i=j \\ [c_{ij}]_z&\text{if}\ i\neq j \end{cases}&D(z)_{ij}&=\delta_{ij}[r_i]_z& B(z)&=D(z)C(z).
\end{align*}
(See Notation \ref{n:qbinom}.) Note that 
\[
B(z)=([(\alpha_i, \alpha_j)]_z)_{i, j\in I}, 
\]
in particular, $B(z)$ is symmetric. The inverse of $C(z)$ is denoted by $\widetilde{C}(z)=(\widetilde{C}(z)_{ij})_{i,j\in I}$. 

\begin{example}\label{e:qCartan}
If $C=\begin{pmatrix}2&-1\\ -1&2\end{pmatrix}$ (type $\mathrm{A}_2$), then 
\begin{align*}
C^{-1}&=\frac{1}{3}\begin{pmatrix}2&1\\ 1&2\end{pmatrix}& C(z)&=\begin{pmatrix}z+z^{-1}&-1\\ -1&z+z^{-1}\end{pmatrix}& \widetilde{C}(z)&=\frac{1}{z^{2}+1+z^{-2}}\begin{pmatrix}z+z^{-1}&1\\ 1&z+z^{-1}\end{pmatrix}.
\end{align*}

If $C=\begin{pmatrix}2&-1&0\\ -1&2&-1\\ 0&-1&2\end{pmatrix}$ (type $\mathrm{A}_3$), then 
\begin{align*}
C^{-1}&=\frac{1}{4}\begin{pmatrix}3&2&1\\ 2&4&2\\ 1&2&3\end{pmatrix}& C(z)&=\begin{pmatrix}[2]_z&-1&0\\ -1&[2]_z&-1\\ 0&-1&[2]_z\end{pmatrix}& \widetilde{C}(z)&=\frac{1}{[4]_z}\begin{pmatrix}[3]_z&[2]_z&1\\ [2]_z&([2]_z)^2&[2]_z\\ 1&[2]_z&[3]_z\end{pmatrix}.
\end{align*}

If $C=\begin{pmatrix}2&-1\\ -2&2\end{pmatrix}$ (type $\mathrm{B}_2$), then 
\begin{align*}
C^{-1}&=\frac{1}{2}\begin{pmatrix}2&1\\ 2&2\end{pmatrix}& C(z)&=\begin{pmatrix}z^2+z^{-2}&-1\\ -z-z^{-1}&z+z^{-1}\end{pmatrix}& \widetilde{C}(z)&=\frac{1}{z^{3}+z^{-3}}\begin{pmatrix}z+z^{-1}&1\\ z+z^{-1}&z^{2}+z^{-2}\end{pmatrix}.
\end{align*}
\end{example}

Let $\mathbb{Z}((z^{-1}))$ be the ring of series of the form $P=\sum_{r\leq R_P}P_rz^{r}$ such that $R_P, P_r\in \mathbb{Z}$. We will regard $\widetilde{C}(z)_{ij}$ as an element of $\mathbb{Z}((z^{-1}))$, and write 
\[
\widetilde{C}(z)_{ji}=\sum_{r\in \mathbb{Z}}\widetilde{c}_{ji}(r)z^r\in \mathbb{Z}((z^{-1})). 
\]
It follows from \cite[Lemma 1.1]{FM} that 
\[
\widetilde{c}_{ji}(r)=0
\]
for $r\in\mathbb{Z}_{\geq 0}$. By definition, we have 
\begin{align}
\sum_{k\in I, r\leq -1}C(z)_{jk}\widetilde{c}_{ki}(r)z^{r}=\delta_{ji}.\label{eq:invcar}
\end{align}
\begin{example}\label{e:invcarB_2}
Consider the case of type $\mathrm{B}_2$. By Example \ref{e:qCartan}, we have 
\begin{align*}
\widetilde{C}(z)_{11}&=\sum_{k\geq 0}(-1)^k(z^{-6k-2}+z^{-6k-4}),
&
\widetilde{C}(z)_{12}&=\sum_{k\geq 0}(-1)^kz^{-6k-3},\\
\widetilde{C}(z)_{21}&=\sum_{k\geq 0}(-1)^k(z^{-6k-2}+z^{-6k-4}),
&
\widetilde{C}(z)_{22}&=\sum_{k\geq 0}(-1)^k(z^{-6k-1}+z^{-6k-5}). 
\end{align*}
Hence $\widetilde{c}_{ji}(r)$'s are summarized as follows (blanks stand for $0$) : 

\hfill
\scalebox{1.0}{
\begin{xy} 0;<1pt,0pt>:<0pt,-1pt>::
(-30,-20) *+{\widetilde{c}_{j1}(r)}, 
(-20,0) *+{1}, 
(-20,20) *+{2}, 
(-35,10) *+{j}, 
(0,-20) *+{-2}, 
(30,-20) *+{-4}, 
(60,-20) *+{-6}, 
(90,-20) *+{-8}, 
(120,-20) *+{-10}, 
(150,-20) *+{-12}, 
(180,-20) *+{-14}, 
(210,-20) *+{-16}, 
(240,-20) *+{\cdots}, 
(240,10) *+{\cdots}, 
(105,-30) *+{r}, 
(0,0) *+{1}, 
(30,0) *+{1}, 
(60,0) *+{ }, 
(90,0) *+{-1}, 
(120,0) *+{-1}, 
(150,0) *+{ }, 
(180,0) *+{1}, 
(210,0) *+{1}, 
(0,20) *+{1}, 
(30,20) *+{1}, 
(60,20) *+{ }, 
(90,20) *+{-1}, 
(120,20) *+{-1}, 
(150,20) *+{ }, 
(180,20) *+{1}, 
(210,20) *+{1}, 
{\ar@{-} (-50,-10); (270,-10)},
{\ar@{-} (-10,-35); (-10,30)},
\end{xy}
}
\hfill
\hfill

\hfill
\scalebox{1.0}{
\begin{xy} 0;<1pt,0pt>:<0pt,-1pt>::
(-30,-20) *+{\widetilde{c}_{j2}(r)}, 
(-20,0) *+{1}, 
(-20,20) *+{2}, 
(-35,10) *+{j}, 
(0,-20) *+{-1}, 
(30,-20) *+{-3}, 
(60,-20) *+{-5}, 
(90,-20) *+{-7}, 
(120,-20) *+{-9}, 
(150,-20) *+{-11}, 
(180,-20) *+{-13}, 
(210,-20) *+{-15}, 
(240,-20) *+{\cdots}, 
(240,10) *+{\cdots}, 
(105,-30) *+{r}, 
(0,0) *+{ }, 
(30,0) *+{1}, 
(60,0) *+{ }, 
(90,0) *+{ }, 
(120,0) *+{-1}, 
(150,0) *+{ }, 
(180,0) *+{ }, 
(210,0) *+{1}, 
(0,20) *+{1}, 
(30,20) *+{ }, 
(60,20) *+{1}, 
(90,20) *+{-1}, 
(120,20) *+{ }, 
(150,20) *+{-1}, 
(180,20) *+{1}, 
(210,20) *+{ }, 
{\ar@{-} (-50,-10); (270,-10)},
{\ar@{-} (-10,-35); (-10,30)},
\end{xy}
}
\hfill
\hfill

We can also calculate $\widetilde{c}_{ji}(r)$ inductively by using \eqref{eq:invcar}.  Focusing on the coefficient of $z^{r}$ with $r\leq -1$ in \eqref{eq:invcar}, we have
\begin{align*}
&\widetilde{c}_{1i}(r-2)+\widetilde{c}_{1i}(r+2)-\widetilde{c}_{2i}(r)=0\ \text{for}\ i=1, 2,\\
&\widetilde{c}_{2i}(r-1)+\widetilde{c}_{2i}(r+1)-\widetilde{c}_{1i}(r-1)-\widetilde{c}_{1i}(r+1)=0 \ \text{for}\ i=1, 2.
\end{align*}
These equalities determine all $c_{ji}(r)$'s from the condition that $\widetilde{c}_{ji}(r)=0$ for $r>-r_j$ and $\widetilde{c}_{ji}(-r_j)=\delta_{ji}$. For example, 
\begin{align*}
&\widetilde{c}_{21}(-2)=-\widetilde{c}_{21}(0)+\widetilde{c}_{11}(-2)+\widetilde{c}_{11}(0)=1, \\
&\widetilde{c}_{11}(-4)=-\widetilde{c}_{11}(0)+\widetilde{c}_{21}(-2)=1, \\
&\widetilde{c}_{21}(-4)=-\widetilde{c}_{21}(-2)+\widetilde{c}_{11}(-4)+\widetilde{c}_{11}(-2)=1, \\
&\widetilde{c}_{11}(-6)=-\widetilde{c}_{11}(-2)+\widetilde{c}_{21}(-4)=0, \\
&\widetilde{c}_{21}(-6)=-\widetilde{c}_{21}(-4)+\widetilde{c}_{11}(-6)+\widetilde{c}_{11}(-4)=0, \\
&\cdots.
\end{align*}
In fact, this method is valid for all $\mathrm{B}_n$ with $n\geq 2$. See Lemma \ref{l:explicitinvcarB} and Example \ref{e:invcarB-pic} below. Moreover, the periodicity of $c_{ji}(r)$'s as in the tables above is also a general phenomenon. See Remark \ref{r:explicitinvcarB}. 
\end{example}

\begin{example}\label{e:invcarB-pic}
Consider the case of type $\mathrm{B}_n$ with $n\geq 2$. Then, by the method explained in Example \ref{e:invcarB_2}, we can calculate the tables of $\widetilde{c}_{ji}(r)$'s as follows. See Lemma \ref{l:explicitinvcarB} for more details (the following is a visualization of the statements of Lemma \ref{l:explicitinvcarB}) : 

Case 1 : $i\neq n$.

\hfill 
\scalebox{0.6}{
\begin{xy}0;<1pt,0pt>:<0pt,-1pt>::
(-65,127.5) *+{\scalebox{1.3}{$j$}},
(247.5,-35) *+{\scalebox{1.3}{$r$}},
(-25,0) *+{1},
(-25,100) *+{i},
(-35,85) *+{i-1},
(-35,140) *+{n-i},
(-35,240) *+{n-1},
(-25,255) *+{n},
(535,127.5) *+{\cdots},
(0,-15) *+{-2},
(15,-15) *+{-4},
(100,-15) *+{-2i},
(140,-15) *+{-2(n-i)},
(240,-15) *+{-2(n-1)},
(0,100) *+{1},
(15,85) *+{1},
(15,115) *+{1},
(30,70) *+{1},
(30,100) *+{1},
(30,130) *+{1},
(45,55) *+{1},
(45,85) *+{1},
(45,115) *+{1},
(45,145) *+{1},
,
(100,0) *+{1},
(85,15) *+{1},
(115,15) *+{1},
(70,30) *+{1},
(100,30) *+{1},
(130,30) *+{1},
,
(140,240) *+{1},
(125,225) *+{1},
(155,225) *+{1},
(110,210) *+{1},
(140,210) *+{1},
(170,210) *+{1},
,
(240,140) *+{1},
(225,125) *+{1},
(225,155) *+{1},
(210,110) *+{1},
(210,140) *+{1},
(210,170) *+{1},
,
(140,255) *+{1},
(155,240) *+{1},
(155,255) *+{1},
(170,225) *+{1},
(170,240) *+{1},
(170,255) *+{1},
,
(240,155) *+{1},
(240,170) *+{1},
(240,185) *+{1},
(240,200) *+{1},
(225,170) *+{1},
(225,185) *+{1},
(225,200) *+{1},
(210,185) *+{1},
(210,200) *+{1},
,
(240,240) *+{1},
(240,255) *+{1},
(225,240) *+{1},
(225,255) *+{1},
(210,240) *+{1},
(210,255) *+{1},
,
(510,-15) *+{-2(2n-1)},
(495,-25) *+{-2(2n-2)},
(395,-25) *+{-2(2n-i-1)},
(355,-15) *+{-2(n+i-1)},
(255,-5) *+{-2n},
(495,100) *+{1},
(480,85) *+{1},
(480,115) *+{1},
(465,70) *+{1},
(465,100) *+{1},
(465,130) *+{1},
(450,55) *+{1},
(450,85) *+{1},
(450,115) *+{1},
(450,145) *+{1},
,
(395,0) *+{1},
(410,15) *+{1},
(380,15) *+{1},
(425,30) *+{1},
(395,30) *+{1},
(365,30) *+{1},
,
(355,240) *+{1},
(370,225) *+{1},
(340,225) *+{1},
(385,210) *+{1},
(355,210) *+{1},
(325,210) *+{1},
,
(255,140) *+{1},
(270,125) *+{1},
(270,155) *+{1},
(285,110) *+{1},
(285,140) *+{1},
(285,170) *+{1},
,
(355,255) *+{1},
(340,240) *+{1},
(340,255) *+{1},
(325,225) *+{1},
(325,240) *+{1},
(325,255) *+{1},
,
(255,155) *+{1},
(255,170) *+{1},
(255,185) *+{1},
(255,200) *+{1},
(270,170) *+{1},
(270,185) *+{1},
(270,200) *+{1},
(285,185) *+{1},
(285,200) *+{1},
,
(255,240) *+{1},
(255,255) *+{1},
(270,240) *+{1},
(270,255) *+{1},
(285,240) *+{1},
(285,255) *+{1},
(85,70) *+{\rotatebox{45}{$\cdots$}},
(180,175) *+{\rotatebox{45}{$\cdots$}},
(190,225) *+{\cdots},
(225,220) *+{\rotatebox{90}{$\cdots$}},
(95,165) *+{\rotatebox{135}{$\cdots$}},
(180,100) *+{\rotatebox{135}{$\cdots$}},
(410,70) *+{\rotatebox{135}{$\cdots$}},
(315,175) *+{\rotatebox{135}{$\cdots$}},
(305,225) *+{\cdots},
(270,220) *+{\rotatebox{90}{$\cdots$}},
(400,165) *+{\rotatebox{45}{$\cdots$}},
(315,100) *+{\rotatebox{45}{$\cdots$}},
{\ar@{--} (95,0); (-20,0)},
{\ar@{--} (10,85); (-20,85)},
{\ar@{--} (-5,100); (-20,100)},
{\ar@{--} (235,140); (-20,140)},
{\ar@{--} (135,240); (-20,240)},
{\ar@{--} (135,255); (-20,255)},
{\ar@{--} (0,95); (0,-10)},
{\ar@{--} (15,80); (15,-10)},
{\ar@{--} (100,-5); (100,-10)},
{\ar@{--} (140,235); (140,-10)},
{\ar@{--} (240,135); (240,-10)},
{\ar@{--} (510,260); (510,-10)},
{\ar@{--} (495,95); (495,-20)},
{\ar@{--} (395,-5); (395,-20)},
{\ar@{--} (355,235); (355,-10)},
{\ar@{--} (255,135); (255,0)},
{\ar@{.} (-10,110); (110,-10)},
{\ar@{.} (90,-10); (250,150)},
{\ar@{.} (-10,90); (150,250)},
{\ar@{.} (130,250); (250,130)},
{\ar@{.} (505,110); (385,-10)},
{\ar@{.} (405,-10); (245,150)},
{\ar@{.} (505,90); (345,250)},
{\ar@{.} (365,250); (245,130)},
{\ar@{.} (247.5,260); (247.5,-25)},
\end{xy}
}
\hfill
\hfill

Case 2 : $i=n$.
 
\hfill
\scalebox{0.6}{
\begin{xy}0;<1pt,0pt>:<0pt,-1pt>::
(-65,127.5) *+{\scalebox{1.3}{$j$}},
(255,-35) *+{\scalebox{1.3}{$r$}},
(-25,0) *+{1},
(-35,240) *+{n-1},
(-25,255) *+{n},
(0,-15) *+{-1},
(15,-15) *+{-3},
(255,-15) *+{-2n+1},
(510,-15) *+{-2(2n-1)+1},
(535,127.5) *+{\cdots},
(0,255) *+{1},
(15,240) *+{1},
(30,225) *+{1},
(30,255) *+{1},
(45,210) *+{1},
(45,240) *+{1},
(60,195) *+{1},
(60,225) *+{1},
(60,255) *+{1},
,
(510,255) *+{1},
(495,240) *+{1},
(480,225) *+{1},
(480,255) *+{1},
(465,210) *+{1},
(465,240) *+{1},
(450,195) *+{1},
(450,225) *+{1},
(450,255) *+{1},
,
(255,0) *+{1},
(240,15) *+{1},
(270,15) *+{1},
(225,30) *+{1},
(255,30) *+{1},
(285,30) *+{1},
(210,45) *+{1},
(240,45) *+{1},
(270,45) *+{1},
(300,45) *+{1},
(195,60) *+{1},
(225,60) *+{1},
(255,60) *+{1},
(285,60) *+{1},
(315,60) *+{1},
,
(210,45) *+{1},
(240,45) *+{1},
(270,45) *+{1},
(300,45) *+{1},
(195,60) *+{1},
(225,60) *+{1},
(255,60) *+{1},
(285,60) *+{1},
(315,60) *+{1},
(170,150) *+{\rotatebox{45}{$\cdots$}},
(150,225) *+{\cdots},
(225,180) *+{\rotatebox{90}{$\cdots$}},
(340,150) *+{\rotatebox{135}{$\cdots$}},
(360,225) *+{\cdots},
(285,180) *+{\rotatebox{90}{$\cdots$}},
{\ar@{--} (250,0); (-20,0)},
{\ar@{--} (-5,255); (-20,255)},
{\ar@{--} (0,250); (0,-10)},
{\ar@{--} (15,235); (15,-10)},
{\ar@{--} (10,240); (-10,240)},
{\ar@{--} (255,-5); (255,-10)},
{\ar@{--} (510,250); (510,-10)},
{\ar@{.} (0,255); (265,-10)},
{\ar@{.} (245,-10); (510,255)},
{\ar@{.} (255,260); (255,-10)},
\end{xy}
}
\hfill
\hfill

For example, when $n=5$, we have 

\scalebox{0.6}{
\begin{xy} 0;<1pt,0pt>:<0pt,-1pt>::
(-30,-20) *+{\widetilde{c}_{j1}(r)}, 
(-20,0) *+{1}, 
(-20,20) *+{2}, 
(-20,40) *+{3}, 
(-20,60) *+{4}, 
(-20,80) *+{5}, 
(-35,40) *+{j}, 
(0,-20) *+{-2}, 
(30,-20) *+{-4}, 
(60,-20) *+{-6}, 
(90,-20) *+{-8}, 
(120,-20) *+{-10}, 
(150,-20) *+{-12}, 
(180,-20) *+{-14}, 
(210,-20) *+{-16}, 
(240,-20) *+{-18}, 
(270,-20) *+{\cdots}, 
(270,40) *+{\cdots}, 
(120,-30) *+{r}, 
(0,0) *+{1}, 
(210,0) *+{1}, 
(30,20) *+{1}, 
(180,20) *+{1}, 
(60,40) *+{1}, 
(150,40) *+{1}, 
(90,60) *+{1}, 
(120,60) *+{1}, 
(90,80) *+{1}, 
(120,80) *+{1}, 
{\ar@{-} (-50,-10); (300,-10)},
{\ar@{-} (-10,-35); (-10,90)},
\end{xy}
}
\scalebox{0.6}{
\begin{xy} 0;<1pt,0pt>:<0pt,-1pt>::
(-30,-20) *+{\widetilde{c}_{j2}(r)}, 
(-20,0) *+{1}, 
(-20,20) *+{2}, 
(-20,40) *+{3}, 
(-20,60) *+{4}, 
(-20,80) *+{5}, 
(-35,40) *+{j}, 
(0,-20) *+{-2}, 
(30,-20) *+{-4}, 
(60,-20) *+{-6}, 
(90,-20) *+{-8}, 
(120,-20) *+{-10}, 
(150,-20) *+{-12}, 
(180,-20) *+{-14}, 
(210,-20) *+{-16}, 
(240,-20) *+{-18}, 
(270,-20) *+{\cdots}, 
(270,40) *+{\cdots}, 
(120,-30) *+{r}, 
(30,0) *+{1}, 
(180,0) *+{1}, 
(0,20) *+{1}, 
(60,20) *+{1}, 
(150,20) *+{1}, 
(210,20) *+{1}, 
(30,40) *+{1}, 
(90,40) *+{1}, 
(120,40) *+{1}, 
(180,40) *+{1}, 
(60,60) *+{1}, 
(90,60) *+{1}, 
(120,60) *+{1}, 
(150,60) *+{1}, 
(60,80) *+{1}, 
(90,80) *+{1}, 
(120,80) *+{1}, 
(150,80) *+{1}, 
{\ar@{-} (-50,-10); (300,-10)},
{\ar@{-} (-10,-35); (-10,90)},
\end{xy}
}

\scalebox{0.6}{
\begin{xy} 0;<1pt,0pt>:<0pt,-1pt>::
(-30,-20) *+{\widetilde{c}_{j3}(r)}, 
(-20,0) *+{1}, 
(-20,20) *+{2}, 
(-20,40) *+{3}, 
(-20,60) *+{4}, 
(-20,80) *+{5}, 
(-35,40) *+{j}, 
(0,-20) *+{-2}, 
(30,-20) *+{-4}, 
(60,-20) *+{-6}, 
(90,-20) *+{-8}, 
(120,-20) *+{-10}, 
(150,-20) *+{-12}, 
(180,-20) *+{-14}, 
(210,-20) *+{-16}, 
(240,-20) *+{-18}, 
(270,-20) *+{\cdots}, 
(270,40) *+{\cdots}, 
(120,-30) *+{r}, 
(60,0) *+{1}, 
(150,0) *+{1}, 
(30,20) *+{1}, 
(90,20) *+{1}, 
(120,20) *+{1}, 
(180,20) *+{1}, 
(0,40) *+{1}, 
(60,40) *+{1}, 
(90,40) *+{1}, 
(120,40) *+{1}, 
(150,40) *+{1}, 
(210,40) *+{1}, 
(30,60) *+{1}, 
(60,60) *+{1}, 
(90,60) *+{1}, 
(120,60) *+{1}, 
(150,60) *+{1}, 
(180,60) *+{1}, 
(30,80) *+{1}, 
(60,80) *+{1}, 
(90,80) *+{1}, 
(120,80) *+{1}, 
(150,80) *+{1}, 
(180,80) *+{1}, 
{\ar@{-} (-50,-10); (300,-10)},
{\ar@{-} (-10,-35); (-10,90)},
\end{xy}
}
\scalebox{0.6}{
\begin{xy} 0;<1pt,0pt>:<0pt,-1pt>::
(-30,-20) *+{\widetilde{c}_{j4}(r)}, 
(-20,0) *+{1}, 
(-20,20) *+{2}, 
(-20,40) *+{3}, 
(-20,60) *+{4}, 
(-20,80) *+{5}, 
(-35,40) *+{j}, 
(0,-20) *+{-2}, 
(30,-20) *+{-4}, 
(60,-20) *+{-6}, 
(90,-20) *+{-8}, 
(120,-20) *+{-10}, 
(150,-20) *+{-12}, 
(180,-20) *+{-14}, 
(210,-20) *+{-16}, 
(240,-20) *+{-18}, 
(270,-20) *+{\cdots}, 
(270,40) *+{\cdots}, 
(120,-30) *+{r}, 
(90,0) *+{1}, 
(120,0) *+{1}, 
(60,20) *+{1}, 
(90,20) *+{1}, 
(120,20) *+{1}, 
(150,20) *+{1}, 
(30,40) *+{1}, 
(60,40) *+{1}, 
(90,40) *+{1}, 
(120,40) *+{1}, 
(150,40) *+{1}, 
(180,40) *+{1}, 
(0,60) *+{1}, 
(30,60) *+{1}, 
(60,60) *+{1}, 
(90,60) *+{1}, 
(120,60) *+{1}, 
(150,60) *+{1}, 
(180,60) *+{1}, 
(210,60) *+{1}, 
(0,80) *+{1}, 
(30,80) *+{1}, 
(60,80) *+{1}, 
(90,80) *+{1}, 
(120,80) *+{1}, 
(150,80) *+{1}, 
(180,80) *+{1}, 
(210,80) *+{1}, 
{\ar@{-} (-50,-10); (300,-10)},
{\ar@{-} (-10,-35); (-10,90)},
\end{xy}
}

\scalebox{0.6}{
\begin{xy} 0;<1pt,0pt>:<0pt,-1pt>::
(-30,-20) *+{\widetilde{c}_{j5}(r)}, 
(-20,0) *+{1}, 
(-20,20) *+{2}, 
(-20,40) *+{3}, 
(-20,60) *+{4}, 
(-20,80) *+{5}, 
(-35,40) *+{j}, 
(0,-20) *+{-1}, 
(30,-20) *+{-3}, 
(60,-20) *+{-5}, 
(90,-20) *+{-7}, 
(120,-20) *+{-9}, 
(150,-20) *+{-11}, 
(180,-20) *+{-13}, 
(210,-20) *+{-15}, 
(240,-20) *+{-17}, 
(270,-20) *+{\cdots}, 
(270,40) *+{\cdots}, 
(120,-30) *+{r}, 
(120,0) *+{1}, 
(90,20) *+{1}, 
(150,20) *+{1}, 
(60,40) *+{1}, 
(120,40) *+{1}, 
(180,40) *+{1}, 
(30,60) *+{1}, 
(90,60) *+{1}, 
(150,60) *+{1}, 
(210,60) *+{1}, 
(0,80) *+{1}, 
(60,80) *+{1}, 
(120,80) *+{1}, 
(180,80) *+{1}, 
(240,80) *+{1}, 
{\ar@{-} (-50,-10); (300,-10)},
{\ar@{-} (-10,-35); (-10,90)},
\end{xy}
}
\end{example}

\subsection{Quantum tori}\label{ss:qtori}
Following \cite{H:qt}, we define a $t$-deformed version $\mathcal{Y}_t$ of $\mathcal{Y}$ (\emph{quantum torus}) associated with a quantum Cartan matrix $C(z)$.

\begin{definition}
The quantum torus $\mathcal{Y}_t$ associated with a quantum Cartan matrix $C(z)$ is defined as the $\mathbb{Z}$-algebra given by the generators $\widetilde{Y}_{i, r}^{\pm 1}$ $(i\in I, r\in \mathbb{Z})$, $t^{\pm 1/2}$ and the following relations\footnote{There are typos in \cite[Lemma 3.5 (2), Theorem 3.11]{H:qt} : $t_{\tilde{C}_{j,i}(z)(z^{r_j}-z^{-r_j})(-z^{(l-k)}+z^{(k-l)})}$ in Lemma 3.5 (2) should be replaced by $t_{\tilde{C}_{j,i}(z)(z^{r_j}-z^{-r_j})(z^{(l-k)}-z^{(k-l)})}$. Hence, in the equality of $\gamma(i,l,j,k)$ in Theorem 3.11, we should multiply the right-hand side by $-1$.} : 
\begin{itemize}
\item[(1)] $t^{\pm 1/2}$ are central,
\item[(2)] $t^{1/2}t^{-1/2}=1$ and $\widetilde{Y}_{i, r}\widetilde{Y}_{i, r}^{-1}=1=\widetilde{Y}_{i, r}^{-1}\widetilde{Y}_{i, r}$,
\item[(3)] for $i, j\in I$ and $r, s\in \mathbb{Z}$, 
\[
\widetilde{Y}_{i, r}\widetilde{Y}_{j, s}=t^{\gamma(i, r; j, s)}\widetilde{Y}_{j, s}\widetilde{Y}_{i, r},
\]
where $\gamma\colon (I\times \mathbb{Z})^2\to\mathbb{Z}$ is given by
\[
\gamma(i, r; j, s)=\widetilde{c}_{ji}(-r_j-r+s)+\widetilde{c}_{ji}(r_j+r-s)-\widetilde{c}_{ji}(r_j-r+s)-\widetilde{c}_{ji}(-r_j+r-s). 
\]
\end{itemize}
\end{definition}
\begin{remark}\label{r:commval}
If $r>s$, then 
\begin{align}
\gamma(i, r; j, s)=\widetilde{c}_{ji}(-r_j-r+s)-\widetilde{c}_{ji}(r_j-r+s), 
\end{align}
since $\widetilde{c}_{ij}(r')=0$ for $r'>-r_i$ (see the proof of Lemma \ref{l:explicitinvcarB}).

Moreover, $\gamma(i, r; j, r)=0$, that is, 
\begin{align}
\widetilde{Y}_{i, r}\widetilde{Y}_{j, r}=\widetilde{Y}_{j, r}\widetilde{Y}_{i, r} \label{eq:sameindex}
\end{align}
for $i, j\in I$ and $r\in \mathbb{Z}$. 
\end{remark}

There exists a $\mathbb{Z}$-algebra homomorphism $\mathrm{ev}_{t=1}\colon \mathcal{Y}_t\to \mathcal{Y}$ given by 
\begin{align*}
t^{1/2}\mapsto 1&&& \widetilde{Y}_{i, r}\mapsto Y_{i, r}.
\end{align*}
This map is called \emph{the specialization at $t=1$}. 

An element $\widetilde{m}\in \mathcal{Y}_t$ is called \emph{a monomial} if it is a product of the generators $\widetilde{Y}_{i, r}^{\pm 1}$ $(i\in I, r\in \mathbb{Z})$ and $t^{\pm 1/2}$. For a monomial $\widetilde{m}$ in $\mathcal{Y}_t$, we set $u_{i, r}(\widetilde{m}):=u_{i, r}(\mathrm{ev}_{t=1}(\widetilde{m}))$ (recall the notation \eqref{eq:power}). A monomial $\widetilde{m}$ in $\mathcal{Y}_t$ is said to be \emph{dominant} if $\mathrm{ev}_{t=1}(\widetilde{m})$ is dominant, that is, $u_{i, r}(\widetilde{m})\geq 0$ for all $i\in I$ and $r\in \mathbb{Z}$. Moreover, for monomials $\widetilde{m}$, $\widetilde{m}'$ in $\mathcal{Y}_t$, set
\[
\widetilde{m}\leq \widetilde{m}'\ \text{if and only if}\ \mathrm{ev}_{t=1}(\widetilde{m})\leq \mathrm{ev}_{t=1}(\widetilde{m}').
\]

It is shown in \cite[section 6.3]{H:qt} that there is a $\mathbb{Z}$-algebra anti-involution $\overline{(\cdot )}$ on $\mathcal{Y}_{t}$ given by 
\begin{align*}
t^{1/2}\mapsto t^{-1/2}&&& \widetilde{Y}_{i, r}\mapsto t^{-1}\widetilde{Y}_{i, r}.
\end{align*}

It is easy to show that, for any monomial $\widetilde{m}$ in $\mathcal{Y}_t$, there uniquely exists $r\in \mathbb{Z}$ such that $t^{r/2}\widetilde{m}$ is $\overline{(\cdot )}$-invariant, and this element is denoted by $\underline{\widetilde{m}}$. Note that $\underline{(t^{k/2}\widetilde{m})}=\underline{\widetilde{m}}$ for any $k\in \mathbb{Z}$. Hence $\underline{\widetilde{m}}$ depends only on $\mathrm{ev}_{t=1}(\widetilde{m})$. Therefore, for every monomial $m$ in $\mathcal{Y}$, the element $\underline{m}$ is well-defined. The elements of this form are called \emph{commutative monomials}. For example, $\underline{Y_{i, r}}=t^{-1/2}\widetilde{Y}_{i, r}$. Note that, for every monomial $m$ in $\mathcal{Y}$, we have $\underline{(m^{-1})}=(\underline{m})^{-1}(=:\underline{m}^{-1})$.  

\begin{remark}
In \cite{H:qt}, the quantum torus $\mathcal{Y}_t$ is a $\mathbb{Z}[t^{\pm 1}]$-algebra. However, to guarantee the existence of commutative monomials, we need the square root of $t^{\pm 1}$. This is the reason why we add $t^{\pm 1/2}$.
\end{remark}

We have 
\[
\mathcal{Y}_t=\bigoplus_{m: \text{monomial in}\ \mathcal{Y}}\mathbb{Z}[t^{\pm 1/2}]\underline{m}
\]
as $\mathbb{Z}[t^{\pm 1/2}]$-modules. 

For $i\in I, r\in \mathbb{Z}$, set  
\[
\widetilde{A}_{i, r}:=\underline{A_{i, r}}\in \mathcal{Y}_t. 
\]

\begin{proposition}[{\cite[Proposition 3.12]{H:qt}}]\label{p:A-comm}
The $\mathbb{Z}[t^{\pm 1}]$-subalgebra of $\mathcal{Y}_t$ generated by $\{\widetilde{A}_{i, r}^{-1}\mid i\in I, r\in \mathbb{Z}\}$ is, in an obvious way, isomorphic to the $\mathbb{Z}$-algebra defined by the generators $\widetilde{A}_{i, r}^{\pm 1}$ $(i\in I, r\in \mathbb{Z})$, $t^{\pm 1/2}$ and the following relations :
\begin{itemize}
\item[(1)] $t^{\pm 1/2}$ are central,
\item[(2)] $t^{1/2}t^{-1/2}=1$,
\item[(3)] for $i, j\in I$ and $r,s\in \mathbb{Z}$, 
\begin{align}
\widetilde{A}_{i, r}^{-1}\widetilde{A}_{j, s}^{-1}=t^{\alpha(i, r; j, s)}\widetilde{A}_{j, s}^{-1}\widetilde{A}_{i, r}^{-1},\label{eq:A-comm}
\end{align}
where $\alpha\colon (I\times \mathbb{Z})^2\to\mathbb{Z}$ is given by
\begin{align*}
\alpha(i, r; j, s)&=2(-\delta_{r-s, (\alpha_i, \alpha_j)}+\delta_{r-s, -(\alpha_i, \alpha_j)})\\
&=\begin{cases}2(-\delta_{r-s, 2r_i}+\delta_{r-s, -2r_i})&\text{if}\ i=j,\\
2\sum_{l=0}^{-c_{ij}-1}(-\delta_{r-s, -r_i+c_{ij}+1+2l}+\delta_{r-s, r_i+c_{ij}+1+2l})&\text{if}\ i\neq j.\end{cases} 
\end{align*}
\end{itemize}
Moreover, we have the following relations in $\mathcal{Y}_t$ : 
\begin{align}
\widetilde{Y}_{i, r}\widetilde{A}_{j, s}^{-1}=t^{\beta(i, r; j, s)}\widetilde{A}_{j, s}^{-1}\widetilde{Y}_{i, r},\label{eq:AY-comm}
\end{align}
where $\beta\colon (I\times \mathbb{Z})^2\to\mathbb{Z}$ is given by
\[
\beta(i, r; j, s)=2\delta_{ij}(-\delta_{r-s, -r_i}+\delta_{r-s, r_i}).
\]
\end{proposition}

\section{Quantum tori for quantized coordinate algebras}\label{s:QCA}
We prepare the other quantum tori which arise from quantum cluster algebra structures of quantized coordinate algebras \cite{BZ:qcluster}.  
Quantum cluster algebras are specific subalgebras of skew field of fractions (and, \emph{a posteriori}, subalgebras of quantum tori) with infinitely many generators in general. It is known that many quantized coordinate algebras in Lie theory have quantum cluster algebra structures. See \cite{GLS:qcluster,GY:BZconj,GY:Memo} and references therein. 

In this section, we only recall the definitions concerning these quantum tori, and review the quantum tori associated with quantum cluster algebra structures of some quantized coordinate algebras. See Appendix \ref{a:qclus} for the complete definition of quantum cluster algebras. 

\subsection{Quantum tori in the theory of quantum cluster algebras}\label{ss:qclus}
Let $v$ be an indeterminate. Fix a finite set $J$ and take a subset $J_f\subset J$. Set $J_e:=J\setminus J_f$. 

Let $\Lambda=(\lambda_{ij})_{i,j\in J}$ be a skew-symmetric integer matrix and $\widetilde{B}=(b_{ij})_{i\in J,j\in J_e}$ an integer matrix.  The pair $(\Lambda,\widetilde{B})$ is said to be \emph{compatible} if, there exists $\bm{d}:=(d_{i})_{i\in J_e}\in (\mathbb{Z}_{>0})^{J_e}$ such that 
\[
\sum_{k\in J}b_{ki}\lambda_{kj}=d_{i}\delta_{ij}
\]
for all $i\in J_e$ and $j\in J$. Note that, if $(\Lambda,\widetilde{B})$ is compatible, then the rank of $\widetilde{B}$ is $\# J_e$ and 
\begin{align}
d_{i}b_{ij}=-d_jb_{ji} \label{eq:symble}
\end{align}
for all $i, j\in J_e$ \cite[Proposition 3.3]{BZ:qcluster}. The skew-symmetric integer matrix $\Lambda$ determines a skew-symmetric $\mathbb{Z}$-bilinear form $\mathbb{Z}^{J}\times\mathbb{Z}^{J}\to\mathbb{Z}$ by $\Lambda(\bm{e}_{i},\bm{e}_{j})=\lambda_{ij}$, $i,j\in J$, and vice versa. Here, $\{\bm{e}_{i}\mid i\in J\}$ is a standard basis of $\mathbb{Z}^{J}$. Hence we also write this form as $\Lambda$. Then \emph{the quantum torus $\mathcal{T}(=\mathcal{T}(\Lambda))$ associated with the bilinear form $\Lambda$} is the free $\mathbb{Z}[v^{\pm1/2}]$-module with a basis $\{X^{\bm{a}}\mid\bm{a}\in\mathbb{Z}^{J}\}$ equipped with the $\mathbb{Z}[v^{\pm1/2}]$-algebra structure given by 
\[
X^{\bm{a}}X^{\bm{b}}=v^{\Lambda(a,b)/2}X^{\bm{a}+\bm{b}}
\]
for $\bm{a},\bm{b}\in\mathbb{Z}^{J}$. Then 
\begin{itemize}
\item $X^{\bm{a}}X^{\bm{b}}=v^{\Lambda(\bm{a},\bm{b})}X^{\bm{b}}X^{\bm{a}}$, 
\item $X^{0}=1$ and $(X^{\bm{a}})^{-1}=X^{-\bm{a}}$ for $\bm{a}\in\mathbb{Z}^{J}$. 
\end{itemize}
Since the quantum torus $\mathcal{T}$ is an Ore domain \cite[Appendix A]{BZ:qcluster}, it is regarded as a subalgebra of the skew-field $\mathcal{F}(=\mathcal{F}(\Lambda))$ of fractions. Note that $\mathcal{F}$ is a $\mathbb{Q}(v^{1/2})$-algebra. 

For $i\in J$, set $X_i:=X^{\bm{e}_i}$. \emph{The quantum cluster algebra} $\mathcal{A}(\Lambda, \widetilde{B})$ is the $\mathbb{Z}[v^{\pm 1/2}]$-subalgebra of $\mathcal{F}$ generated by the union of the elements (called \emph{the quantum cluster variables}) obtained from $\{X_i\mid i\in J\}$ by arbitrary sequence of \emph{mutations in direction $k$} with $k\in J_e$. See Appendix \ref{a:qclus} for details. A fundamental theorem in the theory of quantum cluster algebras, called \emph{quantum Laurent phenomena} \cite[Corollary 5.2]{BZ:qcluster}, states that 
\[
\mathcal{A}(\Lambda, \widetilde{B})\subset\mathcal{T}(\Lambda). 
\]

For $j\in J_e$, write $\hat{X}_{j}:=X^{\widetilde{\bm{b}}^{j}}$ where $\widetilde{\bm{b}}^{j}:=(b_{ij})_{i\in J}\in \mathbb{Z}^J$. By the definition of compatibility and \eqref{eq:symble}, we have 
\begin{align*}
\Lambda(\bm{e}_i, \widetilde{\bm{b}}^{j})&=-d_i\delta_{ij}\ (i\in J, j\in J_e)
&
\Lambda(\widetilde{\bm{b}}^{i}, \widetilde{\bm{b}}^{j})&=d_ib_{ij}\ (i, j\in J_e).
\end{align*}
Therefore, we have 
\begin{align}
X_{i}\hat{X}_{j}&=v^{-d_i\delta_{ij}}\hat{X}_{j}\hat{X}_{i}\ (i\in J, j\in J_e)&
\hat{X}_{i}\hat{X}_{j}&=v^{d_ib_{ij}}\hat{X}_{j}\hat{X}_{i}\ (i, j\in J_e).\label{eq:Xhat-general}
\end{align}
In particular, these $v$-commutation relations depend only on the vector $\bm{d}$ and the submatrix $B=(b_{ij})_{i,j\in J_e}$ (called \emph{the principal part of $\widetilde{B}$}). 

There exists a $\mathbb{Z}$-linear automorphism $\overline{(\cdot)}$ on $\mathcal{T}$ given by 
\[
v^{r/2}X^{\bm{c}}\mapsto v^{-r/2}X^{\bm{c}}
\]
for $\bm{c}\in \mathbb{Z}^J$ and $r\in \mathbb{Z}$. By the defining relation of $\mathcal{T}$, $\overline{(\cdot)}$ is an algebra anti-involution. By \cite[Proposition 6.2]{BZ:qcluster}, it restricts to a $\mathbb{Z}$-algebra anti-involution on the quantum cluster algebra $\mathcal{A}(\Lambda, \widetilde{B})$. This is called \emph{the bar involution}. 

\subsection{Quantum tori for quantized coordinate algebras}\label{ss:qclus-coord}
Recall the setting in subsection \ref{ss:QLA}. 
\begin{notation}\label{n:indexplus}
When we fix $w\in W$ and $\bm{i}=(i_{1},\dots,i_{\ell})\in I(w)$, we write 
\begin{align*}
w_{\leq k}&:=s_{i_1}\cdots s_{i_k},\ w_{\leq 0}:=e,\\
k^{+} & :=\min(\{\ell+1\}\cup\{k+1\leq j\leq\ell\mid i_{j}=i_{k}\}),\\
k^{-} & :=\max(\{0\}\cup\{1\leq j\leq k-1\mid i_{j}=i_{k}\}),\\
k^-(i)&:=\max(\{0\}\cup \{1\leq j\leq k-1\mid i_{j}=i, c_{ii_k}\neq 0\}).
\end{align*}
for $k=1,\dots,\ell$ and $i\in I$.
\end{notation}
\begin{remark}\label{r:indexplus}
In \cite[subsection 5.4]{GLS:qcluster}, $k^-(i)$ is defined as $\max(\{0\}\cup \{1\leq j\leq k-1\mid i_{j}=i\})$. See also Notation \ref{n:compatible}. 
\end{remark}
The following compatible pairs give quantum tori for quantized coordinate algebras. See Theorem \ref{t:qcluster} for the precise relation to quantized coordinate algebras. 
\begin{proposition}[{\cite[Proposition 10.1, Lemma 11.3]{GLS:qcluster},\cite[Proposition 10.4]{GY:Memo}}]\label{p:compatible}
Let $w_0$ be the longest element of $W$. Fix $\bm{i}=(i_1,\dots i_{\ell(w_0)})\in I(w_0)$. 
Set 
\[
J:=\{1,\dots,\ell(w_0)\},\ J_f=\{j\in J\mid j^+=\ell(w_0)+1\}\ \text{and}\ J_e:=J\setminus J_f. 
\]
Define the $J\times J_e$-integer matrix $\widetilde{B}^{\bm{i}}=(b_{st})_{s\in J, t\in J_e}$ as  
\[
b_{st}=
\begin{cases}
1&\text{if}\ t=s^+, \\
c_{i_si_t}&\text{if}\ s<t<s^+<t^+,\\
-1&\text{if}\ t^+=s, \\
-c_{i_si_t}&\text{if}\ t<s<t^+<s^+, \\
0&\text{otherwise}.
\end{cases}
\]
Define $J\times J$-skew-symmetric integer matrix $\Lambda^{\bm{i}}=(\lambda_{st})_{s, t\in J}$ as 
\[
\lambda_{st}=(\varpi_{i_s}-w_{\leq s}\varpi_{i_s}, \varpi_{i_t}+w_{\leq t}\varpi_{i_t})\ \text{for}\ s<t.
\]
Then $(\Lambda^{\bm{i}}, \widetilde{B}^{\bm{i}})$ is a compatible pair, and 
\[
\sum_{k\in J}b_{ks}\lambda_{kt}=2r_{i_s}\delta_{st}
\] 
for all $s\in J_e$ and $t\in J$, (that is, $d_s=2r_{i_s}$ for all $s\in J_e$). 
\end{proposition}
\begin{notation}\label{n:compatible}
It is easy to show that, if we take another reduced word $\bm{i}'$ of $w_0$ commutation equivalent to $\bm{i}$ in Proposition \ref{p:compatible}, we can obtain the same compatible pair up to an obvious simultaneous renaming of indices in $J$. In this sense, we may say that $(\Lambda^{\bm{i}}, \widetilde{B}^{\bm{i}})$ depends only on the commutation class of $\bm{i}$, and we will write it as $(\Lambda^{[\bm{i}]}, \widetilde{B}^{[\bm{i}]})$. Moreover, the index set $J$ can be regarded as the vertex set $\Delta_+$ of the combinatorial Auslander-Reiten quiver of $[\bm{i}]$. See Theorem \ref{t:commequiv}. After this identification, the operations $k^+$ and $k^-$ and $k^-(i)$ on $\Delta_+$ in Notation \ref{n:indexplus} depends only on the commutation class of $\bm{i}$ (see, for example, \cite[Proposition 2.9]{OhS:cAR}). Note that, if we exclude ``$c_{ii_k}\neq 0$'' from the definition of $k^-(i)$, the operation $k^-(i)$ on $\Delta_+$ depends on a choice of reduced words in $[\bm{i}]$. 
\end{notation}
When $\mathfrak{g}$ is of type $\mathrm{A}_N$, $\mathrm{D}_N$ or $\mathrm{E}_N$, the entries of $\widetilde{B}^{[\bm{i}]}$ are $0$ or $\pm 1$, and the principal part of $\widetilde{B}^{[\bm{i}]}$ is skew-symmetric. Hence, $\widetilde{B}^{[\bm{i}]}$ is described as the quiver whose vertex set is the same as the vertex set of $\Upsilon_{[\bm{i}]}$ and whose arrow set is given by 
\[
\{\beta\to \beta'\mid b_{\beta, \beta'}=1\ \text{or}\  b_{\beta', \beta}=-1\}. 
\]
Note that there are no arrows between the vertices in $J_f$.

\begin{example}\label{e:initial}
\noindent(i) Let $\mathcal{Q}$ be the one in Example \ref{e:AR} (i). Then the quiver of $\widetilde{B}^{[\mathcal{Q}]}$ is described as follows : 

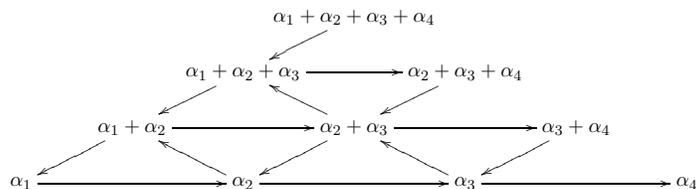
\hfill
\scalebox{0.7}[0.7]{
\begin{xy} 0;<1pt,0pt>:<0pt,-1pt>::
(180,0) *+{\alpha_1+\alpha_2+\alpha_3+\alpha_4} ="0",
(120,30) *+{\alpha_1+\alpha_2+\alpha_3} ="1",
(240,30) *+{\alpha_2+\alpha_3+\alpha_4} ="2",
(60,60) *+{\alpha_1+\alpha_2} ="3",
(180,60) *+{\alpha_2+\alpha_3} ="4",
(300,60) *+{\alpha_3+\alpha_4} ="5",
(0,90) *+{\alpha_1} ="6",
(120,90) *+{\alpha_2} ="7",
(240,90) *+{\alpha_3} ="8",
(360,90) *+{\alpha_4} ="9",
"0", {\ar"1"},
"1", {\ar"3"},
"4", {\ar"1"},
"2", {\ar"4"},
"3", {\ar"6"},
"7", {\ar"3"},
"4", {\ar"7"},
"8", {\ar"4"},
"5", {\ar"8"},
"1", {\ar"2"},
"3", {\ar"4"},
"4", {\ar"5"},
"6", {\ar"7"},
"7", {\ar"8"},
"8", {\ar"9"},
\end{xy}
}
\hfill
\hfill

\noindent(ii) Let $\mathcal{Q}$ and $\xi$ be the ones in Example \ref{e:tw-AR} (ii). Then the quiver of $\widetilde{B}^{[\mathcal{Q}^>]}$ is described as follows (here we consider the labelling $\overline{I}_{\xi}^{\mathrm{tw}, >}$) : 

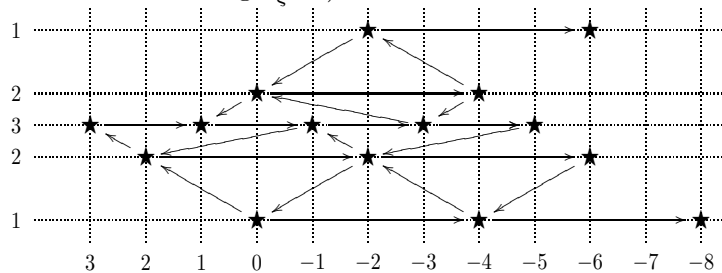
\hfill
\scalebox{0.7}[0.8]{
\begin{xy} 0;<1pt,0pt>:<0pt,-1pt>::
(-40,0) *+{1}, 
(-40,30) *+{2}, 
(-40,45) *+{3}, 
(-40,60) *+{2}, 
(-40,90) *+{1}, 
(0,110) *+{3}, 
(30,110) *+{2}, 
(60,110) *+{1}, 
(90,110) *+{0}, 
(120,110) *+{-1}, 
(150,110) *+{-2}, 
(180,110) *+{-3}, 
(210,110) *+{-4}, 
(240,110) *+{-5}, 
(270,110) *+{-6}, 
(300,110) *+{-7}, 
(330,110) *+{-8}, 
(150,0) *+{\bigstar} ="0",
(90,30) *+{\bigstar} ="1",
(210,30) *+{\bigstar} ="2",
(30,60) *+{\bigstar} ="3",
(150,60) *+{\bigstar} ="4",
(270,60) *+{\bigstar} ="5",
(270,0) *+{\bigstar} ="6",
(90,90) *+{\bigstar} ="7",
(210,90) *+{\bigstar} ="8",
(330,90) *+{\bigstar} ="9",
(0,45) *+{\bigstar} ="10",
(60,45) *+{\bigstar} ="11",
(120,45) *+{\bigstar} ="12",
(180,45) *+{\bigstar} ="13",
(240,45) *+{\bigstar} ="14",
{\ar@{.} (-30,0); (345,0)},
{\ar@{.} (-30,30); (345,30)},
{\ar@{.} (-30,45); (345,45)},
{\ar@{.} (-30,60); (345,60)},
{\ar@{.} (-30,90); (345,90)},
{\ar@{.} (0,100); (0,-10)},
{\ar@{.} (30,100); (30,-10)},
{\ar@{.} (60,100); (60,-10)},
{\ar@{.} (90,100); (90,-10)},
{\ar@{.} (120,100); (120,-10)},
{\ar@{.} (150,100); (150,-10)},
{\ar@{.} (180,100); (180,-10)},
{\ar@{.} (210,100); (210,-10)},
{\ar@{.} (240,100); (240,-10)},
{\ar@{.} (270,100); (270,-10)},
{\ar@{.} (300,100); (300,-10)},
{\ar@{.} (330,100); (330,-10)},
"0", {\ar"1"},
"2", {\ar"0"},
"7", {\ar"3"},
"4", {\ar"7"},
"8", {\ar"4"},
"5", {\ar"8"},
"3", {\ar"10"},
"13", {\ar"1"},
"12", {\ar"3"},
"4", {\ar"12"},
"2", {\ar"13"},
"14", {\ar"4"},
"10", {\ar"11"},
"11", {\ar"12"},
"12", {\ar"13"},
"13", {\ar"14"},
"13", {\ar"14"},
"0", {\ar"6"},
"1", {\ar"2"},
"3", {\ar"4"},
"4", {\ar"5"},
"7", {\ar"8"},
"8", {\ar"9"},
"1", {\ar"11"},
\end{xy}
}
\hfill
\hfill

\end{example}

\section{The isomorphism between quantum tori}\label{s:qtoriisom}
In this section, we prove an isomorphism between two kinds of quantum tori (Theorem \ref{t:torusisom}) : the appropriate subalgebra $\mathcal{Y}_{t, \mathcal{Q}^{\flat}}$ of $\mathcal{Y}_t$ of type $\mathrm{B}_n$, which will be defined before Theorem \ref{t:torusisom}, and the quantum torus corresponding to the compatible pair $(\Lambda^{[\mathcal{Q}^{\flat}]}, \widetilde{B}^{[\mathcal{Q}^{\flat}]})$ (Proposition \ref{p:compatible}, Notation \ref{n:compatible}). 

Recall the setting in subsection \ref{ss:subcat}, that is, let $\mathcal{Q}$ be a Dynkin quiver of type $\mathrm{A}_{2n-2}$, $\xi\colon I=\{1, 2,\dots, 2n-2\}\to\mathbb{Z}$ be an associated height function, and $\flat\in \{>, <\}$. Denote by $W^{\mathrm{A}_{2n-1}}$ the Weyl group of type $\mathrm{A}_{2n-1}$, and set $I_{\mathrm{A}}:=\{1,\dots, 2n-1\}$,  $I_{\mathrm{B}}:=\{1,\dots, n\}$. Recall that the first component of an element of $\overline{I}_{\xi}^{\mathrm{tw}, \flat}$ is in $I_{\mathrm{B}}$ and the set $I_{\mathrm{B}}$ is regarded as an index set of the simple roots of type $\mathrm{B}_n$. (However, note that the residue of an element of $\overline{I}_{\xi}^{\mathrm{tw}, \flat}$ is an element of $I_{\mathrm{A}}$.) In the following, we usually use the symbols $i, j,\dots$ for the elements of $I_{\mathrm{B}}$ and use the symbols $\imath, \jmath$ for the elements of $I_{\mathrm{A}}$.

By Proposition \ref{p:compatible} and Notation \ref{n:compatible}, we have a compatible pair $(\Lambda^{[\mathcal{Q}^{\flat}]}, \widetilde{B}^{[\mathcal{Q}^{\flat}]})$. We write the corresponding quantum torus $\mathcal{T}(\Lambda^{[\mathcal{Q}^{\flat}]})$ as $\mathcal{T}_{\mathcal{Q}^{\flat}}$ for short. 
The twisted Auslander-Reiten quiver $\Upsilon_{[\mathcal{Q}^{\flat}]}$ and the quiver of $\widetilde{B}^{[\mathcal{Q}^{\flat}]}$ are closely related as follows. Compare the quiver in Example \ref{e:tw-AR} (ii) with \ref{e:initial} (ii) :
\begin{proposition}\label{p:initial}
We regard the vertex set of $\Upsilon_{[\mathcal{Q}^{\flat}]}$ as $\overline{I}_{\xi}^{\mathrm{tw}, \flat}$ via $\overline{\Omega}_{\xi}^{\mathrm{tw}, \flat}$. Then the quiver of $\widetilde{B}^{[\mathcal{Q}^{\flat}]}$ is obtained from $\Upsilon_{[\mathcal{Q}^{\flat}]}$ by 
\begin{itemize}
\item[(B1)] attaching the arrows of the form $(i, r)\to (i, r-2r_i)$ whenever $(i, r), (i, r-2r_i)\in \overline{I}_{\xi}^{\mathrm{tw}, \flat}$, 
\item[(B2)] removing all the arrows of the form $(n, r)\to (n-1, r+1)$, 
\item[(B3)] attaching the arrows of the form $(n, r)\to (n-1, r+3)$ whenever $(n, r)\in \overline{I}_{\xi}^{\mathrm{tw}, \flat}$ and $(n-1, r+3)\in (\overline{I}_{\xi}^{\mathrm{tw}, \flat})_e$, 
\item[(B4)] removing the arrows between the vertices in  $(\overline{I}_{\xi}^{\mathrm{tw}, \flat})_f$. 
\end{itemize}
More explicitly, the set of arrows of the quiver of $\widetilde{B}^{[\mathcal{Q}^{\flat}]}$ is equal to the union of the following sets :
\begin{enumerate}
\item[(1)] $\{(i, r)\to (i, r-2r_i)\mid (i, r), (i, r-2r_i)\in \overline{I}_{\xi}^{\mathrm{tw}, \flat}\}$
\item[(2)] $\{(i, r)\to (j, r+2)\mid i, j\leq n-1, |i-j|=1, (i, r)\in \overline{I}_{\xi}^{\mathrm{tw}, \flat}, (j, r+2)\in (\overline{I}_{\xi}^{\mathrm{tw}, \flat})_e\}$
\item[(3)] $\{(n-1, r)\to (n, r+1)\mid (n-1, r)\in \overline{I}_{\xi}^{\mathrm{tw}, \flat}, (n, r+1)\in (\overline{I}_{\xi}^{\mathrm{tw}, \flat})_e\}$
\item[(4)] $\{(n, r)\to (n-1, r+3)\mid  (n, r)\in \overline{I}_{\xi}^{\mathrm{tw}, \flat}, (n-1, r+3)\in (\overline{I}_{\xi}^{\mathrm{tw}, \flat})_e\}$.
\end{enumerate}
\end{proposition}
\begin{proof}
By the definition of compatible reading and twisted convexity of $\Upsilon_{[\mathcal{Q}^{\flat}]}$, we have 
\[
(i, r)^+=(i, r-2r_i)
\]
for $(i, r)\in (\overline{I}_{\xi}^{\mathrm{tw}, \flat})_e$ (recall Notation \ref{n:indexplus}, \ref{n:compatible}). Hence the arrows (1), which are not the arrows of $\Upsilon_{[\mathcal{Q}^{\flat}]}$,  are in the quiver of $\widetilde{B}^{[\mathcal{Q}^{\flat}]}$. The remaining arrows in the quiver of $\widetilde{B}^{[\mathcal{Q}^{\flat}]}$ are of the form $(i, r)\to (j, s)$ such that 
\begin{itemize}
\item $(i, r)\in \overline{I}_{\xi}^{\mathrm{tw}, \flat}$ and $(j, s)\in (\overline{I}_{\xi}^{\mathrm{tw}, \flat})_e$, 
\item $|\res^{[\mathcal{Q}^{\flat}]}((i, r))-\res^{[\mathcal{Q}^{\flat}]}((j, s))|=1$, 
\item $(j, s)\prec_{[\mathcal{Q}^{\flat}]}(i, r)\prec_{[\mathcal{Q}^{\flat}]}(j, s-2r_j)\prec_{[\mathcal{Q}^{\flat}]}(i, r)^+$. 
\end{itemize}
Here we regard $(i, r)^+$ as a maximum element with respect to $\prec_{[\mathcal{Q}^{\flat}]}$ if $(i, r)^+\notin \overline{I}_{\xi}^{\mathrm{tw}, \flat}$. From now, we assume the first condition. Note that all arrows in $\Upsilon_{[\mathcal{Q}^{\flat}]}$ whose target is in $(\overline{I}_{\xi}^{\mathrm{tw}, \flat})_f$ are removed by (B2) and (B4) because of the twisted convexity of $\Upsilon_{[\mathcal{Q}^{\flat}]}$. 

Assume that $i, j\leq n-1$. Then we have 
\begin{center}
$|\res^{[\mathcal{Q}^{\flat}]}((i, r))-\res^{[\mathcal{Q}^{\flat}]}((j, s))|=1$ and $(j, s)\prec_{[\mathcal{Q}^{\flat}]}(i, r)\prec_{[\mathcal{Q}^{\flat}]}(j, s-4)$ 
\end{center}
if and only if
\begin{center}
$|i-j|=1$ and $r=s-2$
\end{center}
by twisted convexity of $\Upsilon_{[\mathcal{Q}^{\flat}]}$. Moreover, in this case, $(j, s-4)<(i, r)^+$ is automatically satisfied. Hence the arrows (2) are  all the arrows $(i, r)\to (j, s)$ with $i, j\leq n-1$ in the quiver of $\widetilde{B}^{[\mathcal{Q}^{\flat}]}$, and they are also the ones in the quiver $\Upsilon_{[\mathcal{Q}^{\flat}]}$ except for the arrows between the vertices in $(\overline{I}_{\xi}^{\mathrm{tw}, \flat})_f$. 

Next, we consider the case that $j=n$. Then,  
\begin{center}
$|\res^{[\mathcal{Q}^{\flat}]}((i, r))-\res^{[\mathcal{Q}^{\flat}]}((n, s))|=1$ and $(n, s)\prec_{[\mathcal{Q}^{\flat}]}(i, r)\prec_{[\mathcal{Q}^{\flat}]}(n, s-2)$ 
\end{center}
if and only if
\begin{center}
$i=n-1$ and $r=s-1$
\end{center}
by twisted convexity of $\Upsilon_{[\mathcal{Q}^{\flat}]}$. Moreover, in this case, $(n, s-2)<(i, r)^+$ is automatically satisfied. Hence the arrows (3) are  all the arrows of the form $(i, r)\to (n, s)$ in the quiver of $\widetilde{B}^{[\mathcal{Q}^{\flat}]}$, and they are also the ones in the quiver $\Upsilon_{[\mathcal{Q}^{\flat}]}$ except for the arrows between the vertices in $(\overline{I}_{\xi}^{\mathrm{tw}, \flat})_f$. 

Finally, we consider the case that $i=n$. Then, 
\begin{center}
$|\res^{[\mathcal{Q}^{\flat}]}((n, r))-\res^{[\mathcal{Q}^{\flat}]}((j, s))|=1$ and $(j, s)\prec_{[\mathcal{Q}^{\flat}]}(n, r)\prec_{[\mathcal{Q}^{\flat}]}(j, s-2r_j)$ 
\end{center}
if and only if
\begin{center}
$j=n-1$ and [$r=s-1$ or $s-3$]
\end{center}
by twisted convexity of $\Upsilon_{[\mathcal{Q}^{\flat}]}$. If $r=s-1$, then $(n-1, s)\prec_{[\mathcal{Q}^{\flat}]}(n,s-1)\prec_{[\mathcal{Q}^{\flat}]}(n, s-1)^+=(n, s-3)\prec_{[\mathcal{Q}^{\flat}]}(n-1, s-4)$. Hence there are no arrows from $(n, s-1)$ to $(n-1, s)$ in the quiver of $\widetilde{B}^{[\mathcal{Q}^{\flat}]}$. Note that such an arrow exists in $\Upsilon_{[\mathcal{Q}^{\flat}]}$. If $r=s-3$, then $(n-1, s-4)<(n, r)^+$ is satisfied. Hence the arrows (4) are all the arrows of the form $(n, r)\to (n, s)$ in the quiver of $\widetilde{B}^{[\mathcal{Q}^{\flat}]}$, and they are not arrows in the quiver $\Upsilon_{[\mathcal{Q}^{\flat}]}$, which completes the proof.
\end{proof}
\begin{remark}\label{r:HLquiver}
It follows from Proposition \ref{p:initial} that the opposite quiver of $\widetilde{B}^{[\mathcal{Q}^{\flat}]}$ is isomorphic to a full subquiver of the quiver $G^-$ for type $\mathrm{B}_n^{(1)}$ in \cite[section 2]{HL:JEMS2016}, introduced by the first author and Leclerc. Here the overall difference of directions of arrows is not essential. Note that, if we identify the vertex set of the opposite quiver of $\widetilde{B}^{[\mathcal{Q}^{\flat}]}$ with $\overline{I}_{\xi}^{\mathrm{tw}, \flat}$, then the labellings of vertices of these two quivers are the same up to shift of second components. 
\end{remark}
Define $\mathcal{Y}_{t, \mathcal{Q}^{\flat}}$ as the $\mathbb{Z}[t^{\pm 1/2}]$-subalgebra of the quantum torus $\mathcal{Y}_t$ \emph{of type $\mathrm{B}_n$} generated by $\{\widetilde{Y}_{i, r}^{\pm 1}\mid (i, r)\in \overline{I}_{\xi}^{\mathrm{tw}, \flat}\}$. For $(i, r)\in \overline{I}_{\xi}^{\mathrm{tw}, \flat}$, we set 
\begin{align}
k(i, r):=\# \{r'\in \mathbb{Z}\mid (i, r')\in \overline{I}_{\xi}^{\mathrm{tw}, \flat}, r'-r\in 2r_i\mathbb{Z}_{\geq 0}\}.\label{eq:kir} 
\end{align}

The following is our first main theorem. 

\begin{theorem}\label{t:torusisom}
We regard the vertex set of $\Upsilon_{[\mathcal{Q}^{\flat}]}$ as $\overline{I}_{\xi}^{\mathrm{tw}, \flat}$ via $\overline{\Omega}_{\xi}^{\mathrm{tw}, \flat}$. There exists a $\mathbb{Z}$-algebra isomorphism $\widetilde{\Phi}^T\colon\mathcal{T}_{\mathcal{Q}^{\flat}} \to \mathcal{Y}_{t, \mathcal{Q}^{\flat}}$ given by 
\begin{align*}
v^{\pm 1/2}\mapsto t^{\mp 1/2}&&
&X_{(i, r)}\mapsto \underline{m_{k(i, r), r}^{(i)}}=\underline{\prod_{s=1}^{k(i, r)}Y_{i, r+2r_i(s-1)}},  
\end{align*}
for $(i, r)\in \overline{I}_{\xi}^{\mathrm{tw}, \flat}$. Moreover, $\widetilde{\Phi}^T\circ \overline{(\cdot)}=\overline{(\cdot)}\circ\widetilde{\Phi}^T$. 
\end{theorem}

\begin{proof}
The equality $\widetilde{\Phi}^T\circ \overline{(\cdot)}=\overline{(\cdot)}\circ\widetilde{\Phi}^T$ follows from direct checks on generators. Indeed, $\widetilde{\Phi}^T\circ \overline{(\cdot)}$ and $\overline{(\cdot)}\circ\widetilde{\Phi}^T$ are algebra anti-homomorphisms. Moreover, if there exists a $\mathbb{Z}$-algebra homomorphism $\mathcal{T}_{\mathcal{Q}^{\flat}} \to \mathcal{Y}_{t, \mathcal{Q}^{\flat}}$ given by the correspondence in the theorem, then it is easy to show that it is an isomorphism. 

We will prove the existence of the $\mathbb{Z}$-algebra homomorphism $\widetilde{\Phi}^T$. It follows from the explicit form of $\widetilde{B}^{[\mathcal{Q}^{\flat}]}$ given by twisted convexity of $\Upsilon_{[\mathcal{Q}^{\flat}]}$ and Proposition \ref{p:initial} that $\hat{X}_{(i, r)}$ with $(i, r)\in (\overline{I}_{\xi}^{\mathrm{tw}, \flat})_e$ (recall the notation in subsection \ref{ss:qclus}) coincides, up to powers of $v^{\pm \frac{1}{2}}$, with 
\begin{align}
\begin{cases}
X_{(i, r)^-}X_{(i, r-4)}^{-1}\dprod_{j; |i-j|=1} \left(X_{(j, r-2)}X_{(j, r-2)^-}^{-1}\right)&\text{if}\  i\leq n-2\\
X_{(n-1, r)^-}X_{(n-1, r-4)}^{-1}X_{(n-2, r-2)}X_{(n-2, r-2)^-}^{-1}X_{(n, r-3)}X_{(n, r-1)^{-}}^{-1}&\text{if}\ i=n-1\\
X_{(n, r)^-}X_{(n, r-2)}^{-1}X_{(n-1, r-1)}X_{(n-1, r-1)^-}^{-1}&\text{if}\ i=n.
\end{cases}\label{eq:Xhat}
\end{align}
Here we set $X_0:=1$ ($0$ stands for $0$ in the sense of Notation \ref{n:indexplus}) and, when $n=2$, $X_{(0, r-2)}X_{(0, r-2)^-}^{-1}:=1$. For $\imath\in I_{\mathrm{A}}$, define $\Xi_{\imath}$ as in Lemma \ref{l:max} and set 
\begin{align*}
\Xi'_{\imath}&:=\begin{cases}
\xi(n-1)+\xi(n)&\text{if}\ \flat=> \text{and}\ \imath=n, \\
2\Xi_{\imath}&\text{otherwide},
\end{cases}
&
\Pi_{\imath}&:=\begin{cases}
(\imath, \Xi'_{\imath})&\text{if}\ 1\leq \imath\leq n, \\
(2n-\imath, \Xi'_{\imath})&\text{if}\ n+1\leq \imath \leq 2n-1.
\end{cases}
\end{align*}
Note that $\Xi'_{\imath}$ and $\Pi_{\imath}$ depend only on $\xi$, and they are independent of the choice of $\flat$. Let 
\[
X'_{\Pi_{\imath}}:=
\begin{cases}
v^{\frac{1}{2}\lambda_{(n, 2\Xi_n), \Pi_n}}X_{(n, 2\Xi_n)}^{-1}X_{\Pi_{n}}&\text{if}\ \flat=> \text{and}\ \imath=n,\\ 
X_{\Pi_{\imath}}&\text{otherwise} 
\end{cases}
\]
(recall the notation in Proposition \ref{p:compatible}). Then the explicit forms \eqref{eq:Xhat} imply that the quantum torus $\mathcal{T}_{\mathcal{Q}^{\flat}}$ is generated by 
\begin{align}
\{\hat{X}_{(i, r)}^{\pm 1}\mid (i, r)\in (\overline{I}_{\xi}^{\mathrm{tw}, \flat})_e\}\cup \{(X'_{\Pi_{\imath}})^{\pm 1}\mid \imath\in I_{\mathrm{A}}\}.\label{eq:gen}
\end{align} 
Moreover, because of the explicit forms \eqref{eq:Xhat} again, it suffices to show that there exists a $\mathbb{Z}$-algebra homomorphism $\widetilde{\Phi}^T\colon\mathcal{T}_{\mathcal{Q}^{\flat}} \to \mathcal{Y}_{t, \mathcal{Q}^{\flat}}$ given by 
\begin{align}
\begin{cases}
v^{1/2}\mapsto t^{-1/2},&\\
\hat{X}_{(i, r)}\mapsto \widetilde{A}_{i, r-r_i}^{-1}&\text{for}\ (i, r)\in (\overline{I}_{\xi}^{\mathrm{tw}, \flat})_e,\\
X'_{\Pi_{\imath}}\mapsto \underline{\widetilde{Y}_{\Pi_{\imath}}}&\text{for}\ \imath\in I_{\mathrm{A}}.
\end{cases}\label{eq:candidate}
\end{align}

For $\imath, \jmath\in I$, there exists $\lambda_{\imath, \jmath}\in \mathbb{Z}$ such that the generators of $\mathcal{T}_{\mathcal{Q}^{\flat}}$ satisfy the following relations (recall \eqref{eq:Xhat-general} and the notation in Proposition \ref{p:compatible}) :
\begin{align}
&\hat{X}_{(i, r)}\hat{X}_{(j, s)}=v^{2b_{(i, r), (j, s)}}\hat{X}_{(j, s)}\hat{X}_{(i, r)}\ \text{for}\ (i, r), (j, s)\in (\overline{I}_{\xi}^{\mathrm{tw}, \flat})_e, \label{eq:rel1}\\ 
&X_{\Pi_{\imath}}\hat{X}_{(j, s)}=v^{-2\delta(\Pi_{\imath}=(j, s))}\hat{X}_{(j, s)}X_{\Pi_{\imath}} \ \text{for}\ \imath \in I_{\mathrm{A}}\ \text{and}\ (j, s)\in (\overline{I}_{\xi}^{\mathrm{tw}, \flat})_e,
\label{eq:rel2}\\ 
&X'_{\Pi_{n}}\hat{X}_{(j, s)}=v^{2\delta((n, 2\Xi_n)=(j, s))-2\delta(\Pi_{n}=(j, s))}\hat{X}_{(j, s)}X'_{\Pi_{n}} \ \text{when}\ \flat=> \text{and}\ (j, s)\in (\overline{I}_{\xi}^{\mathrm{tw}, \flat})_e,
\label{eq:rel2'}\\
&X'_{\Pi_{\imath}}X'_{\Pi_{\jmath}}=v^{\lambda_{\imath, \jmath}}X'_{\Pi_{\jmath}}X'_{\Pi_{\imath}}\ \text{for}\ \imath, \jmath\in I_{\mathrm{A}}. \label{eq:rel3}
\end{align} 
Note that the relations of the generators \eqref{eq:gen} of $\mathcal{T}_{\mathcal{Q}^{\flat}}$ are exhausted by \eqref{eq:rel1}, \eqref{eq:rel2}, \eqref{eq:rel2'} and \eqref{eq:rel3}.

Here we compute $\lambda_{\imath, \jmath}$ explicitly. Define $\phi_{\mathcal{Q}}^{\mathrm{tw}}=s_{\imath_1}\cdots s_{\imath_{2n-1}}\in W^{\mathrm{A}_{2n-1}}$ by 
\begin{center}
$\{\imath_1,\dots, \imath_{2n-1}\}=I_{\mathrm{A}}$ with $k<k'$ whenever $\Pi_{\imath_k}\prec_{[\mathcal{Q}^{\flat}]}\Pi_{\imath_{k'}}$. 
\end{center}

By Proposition \ref{p:adapted} (3), the processes (T1)--(T4) in subsection \ref{ss:tadapted} and Proposition \ref{p:compatible}, we have the following : 

If $\flat=<$, then   
\begin{align*}
\lambda_{\imath, \jmath}&=(\varpi_{\imath}, \phi_{\mathcal{Q}}^{\mathrm{tw}}\varpi_{\jmath})-
(\phi_{\mathcal{Q}}^{\mathrm{tw}}\varpi_{\imath}, \varpi_{\jmath})\\
&=(\varpi_{\imath}, \phi_{\mathcal{Q}}^{\mathrm{tw}}\varpi_{\jmath}-\varpi_{\jmath})-(\varpi_{\imath}, (\phi_{\mathcal{Q}}^{\mathrm{tw}})^{-1}\varpi_{\jmath}-\varpi_{\jmath}).
\end{align*}

If $\flat=>$, $\imath\neq n$ and $\jmath\neq n$, then
\begin{align*}
\lambda_{\imath, \jmath}&=(\varpi_{\imath}, s_n\phi_{\mathcal{Q}}^{\mathrm{tw}}\varpi_{\jmath})-
(s_n\phi_{\mathcal{Q}}^{\mathrm{tw}}\varpi_{\imath}, \varpi_{\jmath})\\
&=(\varpi_{\imath}, \phi_{\mathcal{Q}}^{\mathrm{tw}}\varpi_{\jmath})-
(\phi_{\mathcal{Q}}^{\mathrm{tw}}\varpi_{\imath}, \varpi_{\jmath})\\
&=(\varpi_{\imath}, \phi_{\mathcal{Q}}^{\mathrm{tw}}\varpi_{\jmath}-\varpi_{\jmath})-(\varpi_{\imath}, (\phi_{\mathcal{Q}}^{\mathrm{tw}})^{-1}\varpi_{\jmath}-\varpi_{\jmath}).
\end{align*}

If $\flat=>$, $\imath=n$ and $\jmath\neq n$, then
\begin{align*}
\lambda_{n, \jmath}&=-(\alpha_n, \varpi_{\jmath}+s_n\phi_{\mathcal{Q}}^{\mathrm{tw}}\varpi_{\jmath})+(\varpi_{n}, s_n\phi_{\mathcal{Q}}^{\mathrm{tw}}\varpi_{\jmath})-
(s_n\phi_{\mathcal{Q}}^{\mathrm{tw}}\varpi_{n}, \varpi_{\jmath})\\
&=(\varpi_{n}, \phi_{\mathcal{Q}}^{\mathrm{tw}}\varpi_{\jmath})-
(\phi_{\mathcal{Q}}^{\mathrm{tw}}\varpi_{n}, \varpi_{\jmath})\\
&=(\varpi_{\imath}, \phi_{\mathcal{Q}}^{\mathrm{tw}}\varpi_{\jmath}-\varpi_{\jmath})-(\varpi_{\imath}, (\phi_{\mathcal{Q}}^{\mathrm{tw}})^{-1}\varpi_{\jmath}-\varpi_{\jmath}).
\end{align*}

Therefore, for all cases and all $\imath, \jmath\in I_{\mathrm{A}}$, we have
\begin{align*}
\lambda_{\imath, \jmath}=(\varpi_{\imath}, \phi_{\mathcal{Q}}^{\mathrm{tw}}\varpi_{\jmath}-\varpi_{\jmath})-(\varpi_{\imath}, (\phi_{\mathcal{Q}}^{\mathrm{tw}})^{-1}\varpi_{\jmath}-\varpi_{\jmath}).
\end{align*}
Let $\Upsilon_{\Xi}$ be the quiver whose vertex set is $I_{\mathrm{A}}$ and whose arrow set is $\{\imath\to \jmath\mid |\imath -\jmath|=1, \Xi'_{\jmath}>\Xi'_{\imath}\}$. (A reduced word of $\phi_{\mathcal{Q}}^{\mathrm{tw}}$ corresponds to a compatible reading of $\Upsilon_{\Xi}$.) For $\jmath\in I_{\mathrm{A}}$, set 
\begin{align*}
{}^{\jmath}B:=\{\jmath'\mid \text{there is a path from}\ \jmath\ \text{to}\ \jmath'\ \text{in}\ \Upsilon_{\Xi}\},\\
B^{\jmath}:=\{\jmath'\mid \text{there is a path from}\ \jmath'\ \text{to}\ \jmath\ \text{in}\ \Upsilon_{\Xi}\}.
\end{align*}
Then 
\begin{align*}
\phi_{\mathcal{Q}}^{\mathrm{tw}}\varpi_{\jmath}-\varpi_{\jmath}&=-\sum_{\jmath'\in {{}^{\jmath}B}}\alpha_{\jmath'},
&
(\phi_{\mathcal{Q}}^{\mathrm{tw}})^{-1}\varpi_{\jmath}-\varpi_{\jmath}&=-\sum_{\jmath'\in B^{\jmath}}\alpha_{\jmath'}.
\end{align*}
Therefore, for $\imath, \jmath\in I_{\mathrm{A}}$ with $\Xi'_{\imath}\geq \Xi'_{\jmath}$, we have
\begin{align}
X'_{\Pi_{\imath}}X'_{\Pi_{\jmath}}&=\begin{cases}
v^{-\delta(\imath\in {{}^{\jmath}B})}X'_{\Pi_{\jmath}}X'_{\Pi_{\imath}}&\text{if}\ \Xi'_{\imath}> \Xi'_{\jmath}\\
X'_{\Pi_{\jmath}}X'_{\Pi_{\imath}}&\text{if}\ \Xi'_{\imath}=\Xi'_{\jmath}. 
\end{cases}\label{eq:rel4}
\end{align}

It remains to show that the correspondence \eqref{eq:candidate} is compatible with the relations \eqref{eq:rel1}, \eqref{eq:rel2}, \eqref{eq:rel2'} and \eqref{eq:rel4}. The compatibility with \eqref{eq:rel1} directly follows from Proposition \ref{p:initial} and the relation \eqref{eq:A-comm} :
\begin{align*}
\widetilde{A}_{i, r}^{-1}\widetilde{A}_{j, s}^{-1}&=t^{\alpha(i, r; j, s)}\widetilde{A}_{j, s}^{-1}\widetilde{A}_{i, r}^{-1},\\
\text{where}\ \alpha(i, r; j, s)&=
\begin{cases}
2(-\delta_{r-s, 2r_i}+\delta_{r-s, -2r_i})&\text{if}\ i=j,\\
2(-\delta_{r-s, -2}+\delta_{r-s, 2})&\text{if}\ |i-j|=1,\\
0&\text{otherwise}.
\end{cases}
\end{align*}
The compatibility with \eqref{eq:rel2} and \eqref{eq:rel2'} is a straightforward consequence of the relation \eqref{eq:AY-comm}. 

Finally, we check the compatibility with \eqref{eq:rel4}. Recall that, for $i, j\in I_{\mathrm{B}}$ and $r,s\in \mathbb{Z}$ with $r\geq s$, we have
\[
\widetilde{Y}_{i, r}\widetilde{Y}_{j, s}=t^{\gamma(i, r; j, s)}\widetilde{Y}_{j, s}\widetilde{Y}_{i, r},
\]
where
\begin{align*}
\gamma(i, r; j, s)=\begin{cases}
\widetilde{c}_{ji}(-r_j-r+s)-\widetilde{c}_{ji}(r_j-r+s)&\text{if}\ r>s,\\
0&\text{if}\ r=s. 
\end{cases}
\end{align*}
(See Remark \ref{r:commval}). Hence it suffices to show that $\gamma(\Pi_{\imath}; \Pi_{\jmath})=\delta(\imath\in {{}^{\jmath}B})$ for $\imath, \jmath\in I_{\mathrm{A}}$ with $\Xi'_{\imath}> \Xi'_{\jmath}$. Consider the following cases : 
\begin{align*}
\text{(a)}\ 1\leq \imath \leq n-1,&&& 
\text{(b)}\ \imath =n,&&&
\text{(c)}\ n+1\leq \imath\leq 2n-1.
\end{align*}

We deal with only the case (a). The cases (b) and (c) can be treated exactly in the similar way. The details are left to the reader. 

For $\jmath\in I_{\mathrm{A}}$, set 
\[
r^{(\jmath)}:=
\begin{cases}
-2|\jmath-\imath|&\text{if}\ 1\leq \jmath\leq n-1,\\
-2(n-\imath)+1&\text{if}\ \jmath=n,\\
-2(\jmath-\imath)+2&\text{if}\ n+1\leq \jmath\leq 2n-1.\\
\end{cases}
\]
Then, for $\imath\in {{}^{\jmath}B}$, we have 
\[
-\Xi'_{\imath}+\Xi'_{\jmath}\begin{cases}
=r^{(\jmath)}&\text{if}\ \imath\in {{}^{\jmath}B},\\
\in r^{(\jmath)}+4\mathbb{Z}_{>0}&\text{if}\ \imath\notin {{}^{\jmath}B}\ \text{and}\ \jmath\neq n,\\
\in r^{(n)}+2\mathbb{Z}_{>0}&\text{if}\ \imath\notin {{}^{\jmath}B}\ \text{and}\ \jmath=n.
\end{cases} 
\]
Moreover, for $k\in \mathbb{Z}_{\geq 0}$, we have
\[
\begin{cases}
\widetilde{c}_{\jmath\imath}(-2+r^{(\jmath)}+4k)=\delta_{0, k}&\text{if}\ 1\leq \jmath\leq n-1,\\
\widetilde{c}_{2n-\jmath, \imath}(-2+r^{(\jmath)}+4k)=\delta_{0, k}&\text{if}\ n+1\leq \jmath\leq 2n-1,\\
\widetilde{c}_{n, \imath}(-1+r^{(n)}+2k)=\delta_{0, k}&\text{if}\ \jmath=n,\\
\end{cases} 
\]
by Example \ref{e:invcarB-pic} (or Lemma \ref{l:explicitinvcarB}). Therefore, we have 
\[
\gamma(\Pi_{\imath}, \Pi_{\jmath})=\delta(\imath\in {{}^{\jmath}B}). 
\]
\end{proof}
\section{Quantum Grothendieck rings}\label{s:qGro}
In this section, we recall the definition of $t$-deformed $q$-characters, in other words, a $t$-deformed Grothendieck ring $K(\mathcal{C}_{\bullet})$, following the approach in \cite{H:qt}. We should remark that in the case of type $\mathrm{A}_N^{(1)}$, $\mathrm{D}_N^{(1)}$ and $\mathrm{E}_N^{(1)}$, such a $t$-deformed Grothendieck ring was also constructed in \cite{Nak:quiver,VV:qGro} by another method using quiver varieties (and a slightly different quantum torus). 

\subsection{Quantum Grothendieck rings}\label{ss:qGro}
We return to the general setting in section \ref{s:qtori}. For $i\in I$, let $\mathcal{K}_{i, t}$ be a $\mathbb{Z}[t^{\pm 1/2}]$-algebra of $\mathcal{Y}_t$ generated by $\{\widetilde{Y}_{i, r}(1+t\widetilde{A}_{i,r+r_i}^{-1}), \widetilde{Y}_{j, r}^{\pm 1}\mid r\in \mathbb{Z}, j\in I\setminus \{i\}\}$.

Set
\[
\mathcal{K}_{t}(=K_t(\mathcal{C}_{\bullet})):=\bigcap_{i\in I}\mathcal{K}_{i, t}.
\]
See \cite[section 4, 5]{H:qt} for the background of these algebras. The ring $\mathcal{K}_{t}$ is called \emph{the quantum Grothendieck ring of $\mathcal{C}_{\bullet}$}. A priori it is not clear that $\mathcal{K}_t$ contains non-zero elements. However we have the following.
\begin{theorem}[{\cite[Theorem 5.11]{H:qt}}]\label{t:F-elem}
For every dominant monomial $\widetilde{m}$ in $\mathcal{Y}_t$, there uniquely exists an element $F_t(\widetilde{m})$ of $\mathcal{K}_{t}$ such that $\widetilde{m}$ is the unique dominant monomial occurring in $F_t(\widetilde{m})$. Moreover, any monomial $\widetilde{m}'$ occurring in $F_t(\widetilde{m})-\widetilde{m}$ satisfies $\widetilde{m}'<\widetilde{m}$. In particular, the set $\{F_t(\underline{m})\mid m\ \text{is a dominant monomial in }\mathcal{Y}\}$ forms a $\mathbb{Z}[t^{\pm 1/2}]$-basis of $\mathcal{K}_{t}$. 
\end{theorem}
Note that $F_t(\underline{m})$ is $\overline{(\cdot )}$-invariant for any dominant monomial $m$. 
\begin{remark}
In \cite{H:qt}, the first author considered a completion of $\mathcal{K}_{t}$ to construct the elements $F_t(\widetilde{m})$. However, a posteriori, we do not need this completion. See \cite[subsection 7.1]{H:qt} and \cite[subsection 7.3]{H:monom}. 
\end{remark}
For a dominant monomial $\widetilde{m}$ in $\mathcal{Y}_t$, set
\begin{align*}
E_t(\widetilde{m})=\widetilde{m}
\left(\dprod_{r\in \mathbb{Z}}\left(\prod_{i\in I}\widetilde{Y}_{i, r}^{u_{i, r}(\widetilde{m})}\right)\right)^{-1}
\dprod_{r\in \mathbb{Z}}\left(\prod_{i\in I}F_{t}(\widetilde{Y}_{i, r})^{u_{i, r}(\widetilde{m})}\right),
\end{align*}
here $\prod_{i\in I}\widetilde{Y}_{i, r}^{u_{i, r}(\widetilde{m})}$ is well-defined by \eqref{eq:sameindex}, and $\prod_{i\in I}F_{t}(\widetilde{Y}_{i, r})^{u_{i, r}(\widetilde{m})}$ is well-defined by \cite[Lemma 5.12]{H:qt}. 
Note that $\displaystyle\widetilde{m}
\left(\dprod_{r\in \mathbb{Z}}\left(\prod_{i\in I}\widetilde{Y}_{i, r}^{u_{i, r}(\widetilde{m})}\right)\right)^{-1}=t^{\beta}$ for some $\beta\in \mathbb{Z}$.

There is a maximal monomial in $E_t(\widetilde{m})$ and it is equal to $\widetilde{m}$. In particular, by Theorem \ref{t:F-elem}, 
\begin{align}
E_{t}(\underline{m})=F_t(\underline{m})+\sum_{
m'<m}C_{m, m'}F_t(\underline{m}')\label{eq:EtoF}
\end{align}
with $C_{m, m'}\in \mathbb{Z}[t^{\pm 1/2}]$ (in fact, $\mathbb{Z}[t^{\pm 1}]$). Note that the set $\{E_t(\underline{m})\mid m\ \text{is a dominant monomial in }\mathcal{Y}\}$ also forms a $\mathbb{Z}[t^{\pm 1/2}]$-basis of $\mathcal{K}_{t}$ since 
\begin{align}
\#\{m'\mid m'<m, m'\ \text{is a dominant monomial in }\mathcal{Y}\}<\infty\label{eq:finite}
\end{align}
for any dominant monomial $m$ \cite[subsection 3.4]{H:qt}. 
For a monomial $\widetilde{m}$ in $\mathcal{Y}_t$, we have 
\begin{align*}
\mathrm{ev}_{t=1}(E_{t}(\widetilde{m}))&=\chi_q\left(M(m)\right)
&
\mathrm{ev}_{t=1}(\mathcal{K}_{t})&=\chi_q(K(\mathcal{C}_{\bullet}))
\end{align*}
\cite[Theorem 6.2]{H:qt}, here we set 
\[
M(m):=\overrightarrow{\bigotimes_{r\in \mathbb{Z}}}\left(\bigotimes_{i\in I}L(Y_{i, r})^{\otimes u_{i, r}(\widetilde{m})}\right),
\]
called \emph{a standard module}. Note that, for any fixed $r\in \mathbb{Z}$, the isomorphism class of the tensor product $\bigotimes_{i\in I}L(Y_{i, r})^{\otimes u_{i, r}(\widetilde{m})}$ does not depend on the ordering of the factors, and it is in fact simple \cite[Proposition 6.15]{FM}. The element $E_t(\underline{m})$ is called \emph{the $(q, t)$-character of the standard module $M(m)$}. 

We consider another kind of elements $L_t(\underline{m})$ which is conjectually a $t$-deformed version of the $q$-character of simple modules. See Conjecture \ref{co:qt} below. 
\begin{theorem}[{\cite[Theorem 8.1]{Nak:quiver}, \cite[Theorem 6.9]{H:qt}}]
\label{t:normqt}  
For a dominant monomial $m$ in $\mathcal{Y}$, there uniquely exists an element $L_t(\underline{m})$ in $\mathcal{K}_{t}$ such that 
\begin{enumerate}
\item[(S1)] $\overline{L_t(\underline{m})}=L_t(\underline{m})$, and
\item[(S2)] $L_t(\underline{m})=E_t(\underline{m})+\sum_{m'<m}Q_{m, m'}E_t(\underline{m}')$
with $Q_{m, m'}\in t^{-1}\mathbb{Z}[t^{-1}]$.
\end{enumerate}
Moreover, the set $\{L_t(\underline{m})\mid m\ \text{is a dominant monomial}\}$ forms a $\mathbb{Z}[t^{\pm 1/2}]$-basis of $\mathcal{K}_t$ .
\end{theorem}
The element $L_t(\underline{m})$ is called \emph{the $(q, t)$-character} of the simple module $L(m)$. In fact, in the case when $\mathcal{U}_q(\mathcal{L}\mathfrak{g})$ is of simply-laced type, Nakajima \cite[Theorem 8.1]{Nak:quiver} also proved that 
\[
\mathrm{ev}_{t=1}(L_{t}(\underline{m}))=\chi_q(L(m)). 
\]
Its proof is based on a geometric method using quiver varieties, and it is valid only in the simply-laced case. Hence the following is still a conjecture in general :
\begin{conj}[{\cite[Conjecture 7.3]{H:qt}}]\label{co:qt}
For all dominant monomial $m$ in $\mathcal{Y}$, we have 
\[
\mathrm{ev}_{t=1}(L_{t}(\underline{m}))=\chi_q(L(m)).
\]
\end{conj}
In this paper, we will prove Conjecture \ref{co:qt} in the case that  $L(m)$ is an object of $\mathcal{C}_{\mathcal{Q}^{\flat}}$, $\flat \in \{>, <\}$  (Corollary \ref{c:KO}).  
\begin{remark}\label{r:KLalg}
Let $m$ be a dominant monomial in $\mathcal{Y}$. By \eqref{eq:finite} and (S2), we have 
\begin{align*}
E_t(\underline{m})=\sum_{m'}P_{m, m'}(t)L_t(\underline{m'})
\end{align*}
for some $P_{m, m'}(t)\in \mathbb{Z}[t^{-1}]$. More precisely, we have 
\[
P_{m, m'}(t)\begin{cases}
=1&\text{if }m'=m,\\
\in t^{-1}\mathbb{Z}[t^{-1}]&\text{if }m'<m,\\
=0&\text{otherwise.}
\end{cases}
\]
as \eqref{fconj}. Hence, if Conjecture \ref{co:qt} holds, we obtain 
\begin{align}
[M(m)]=\sum_{m'}P_{m, m'}(1)[L(m')].\label{eq:fconj}
\end{align}
in $K(\mathcal{C}_{\bullet})$. Therefore $P_{m, m'}(t)$ can be considered as an analogue of Kazhdan-Lusztig polynomials. Note that there is an inductive algorithm to compute $P_{m, m'}(t)$ (more precisely, $Q_{m, m'}$ in (S2)), sometimes called \emph{the Kazhdan-Lusztig algorithm} (see, for example,   \cite[Lemma 24.2.1]{Lus:intro} and Remark \ref{r:unitri-qt} below). 
\end{remark}
\begin{notation}\label{n:normalization}
For a monomial $m$ in $\mathcal{Y}$, we will write 
\begin{align*}
F_t(m)&:=F_t(\underline{m}) & E_t(m)&:=E_t(\underline{m})&L_t(m)&:=L_t(\underline{m})
\end{align*}
for simplicity. 
\end{notation}

\begin{remark}\label{r:unitri-qt}
The \emph{unitriangular property} of the $(q, t)$-characters of standard modules with respect to $\overline{(\cdot )}$ is also included in Theorem \ref{t:normqt}. Indeed, by \eqref{eq:finite} and (S2), we have 
\[
E_t(m)\in L_t(m)+\sum_{m'<m}t^{-1}\mathbb{Z}[t^{-1}]L_t(m').
\]
Hence, by (S1), 
\[
\overline{E_t(m)}\in L_t(m)+\sum_{m'<m}t\mathbb{Z}[t]L_t(m').
\] 
By (S2) again, we have
\[
\overline{E_t(m)}\in E_t(m)+\sum_{m'<m}\mathbb{Z}[t^{\pm 1}]E_t(m').
\] 
\end{remark}
\begin{remark}
\label{r:characterize-qt} 
The characterization property (S2) can be replaced by the following :
\begin{enumerate}
\item[(S2)$'$] $\displaystyle L_t(m)-E_t(m)\in\sum_{m'}t^{-1}\mathbb{Z}[t^{-1}]E_t(m')$.
\end{enumerate}
Indeed, if an element $x\in \mathcal{K}_t$ satisfies (S2)$'$, that is, $x-E_t(m)\in\sum_{m'}t^{-1}\mathbb{Z}[t^{-1}]E_t(m')$, we have 
\[
x\in L_t(m)+\sum_{m'}t^{-1}\mathbb{Z}[t^{-1}]L_t(m'),
\]
by \eqref{eq:finite} and (S2) (cf.~Remark \ref{r:unitri-qt}). Therefore $x$ is equal to $L_t(m)$ if $x$ also satisfies (S1) because it is the unique $\overline{(\cdot)}$-invariant element in $L_t(m)+\sum_{m'}t^{-1}\mathbb{Z}[t^{-1}]L_t(m')$. 
\end{remark}

\subsection{The quantum Grothendieck ring of $\mathcal{C}_{\mathcal{Q}^{\flat}}$}\label{ss:qGroB}
Recall the setting and notation in subsection \ref{ss:subcat}. Set
\[
\mathcal{K}_{t, \mathcal{Q}^{\flat}}(=K_t(\mathcal{C}_{\mathcal{Q}^{\flat}})):=\bigoplus_{m\in \mathbb{B}^{\xi, \flat}}\mathbb{Z}[t^{\pm 1/2}]F_t(m)\subset \mathcal{K}_{t}.
\]
Then, we have
\[
\mathcal{K}_{t, \mathcal{Q}^{\flat}}=\bigoplus_{m\in \mathbb{B}^{\xi, \flat}}\mathbb{Z}[t^{\pm 1/2}]E_t(m)=\bigoplus_{m\in \mathbb{B}^{\xi, \flat}}\mathbb{Z}[t^{\pm 1/2}]L_t(m)
\]
by \eqref{eq:EtoF}, \eqref{eq:finite} and (S2) in Theorem \ref{t:normqt}. Note that it follows from the argument in the proof of Lemma \ref{l:subcat} that, for dominant monomials $m, m'$ with $m'<m$, the condition $m\in \mathbb{B}^{\xi, \flat}$ implies $m'\in \mathbb{B}^{\xi, \flat}$. 
 
\begin{lemma}\label{l:monoidal}
The $\mathbb{Z}[t^{\pm 1/2}]$-module $\mathcal{K}_{t, \mathcal{Q}^{\flat}}$ is a $\mathbb{Z}[t^{\pm 1/2}]$-subalgebra of $\mathcal{Y}_t$ generated by $\{F_t(Y_{i, r})\mid (i, r)\in \overline{I}_{\xi}^{\mathrm{tw}, \flat}\}$. 
\end{lemma}
\begin{proof}
It suffices to show that, for $m, m'\in \mathbb{B}^{\xi, \flat}$, the dominant monomials appearing in $F_t(m)F_t(m')$ are contained in $\mathbb{B}^{\xi, \flat}$. We can prove it in the same manner as the proof of Lemma \ref{l:subcat}. 
\end{proof}
The algebra $\mathcal{K}_{t, \mathcal{Q}^{\flat}}$ is called \emph{the quantum Grothendieck ring of $\mathcal{C}_{\mathcal{Q}^{\flat}}$}. Recall $\mathcal{Y}_{t, \mathcal{Q}^{\flat}}$ defined before Theorem \ref{t:torusisom}. For $m\in \mathbb{B}^{\xi, \flat}$, define $E_t(m)^T$ as an element of $\mathcal{Y}_{t, \mathcal{Q}^{\flat}}$ which is obtained from $E_t(m)$ by discarding the monomials containing the terms $\widetilde{Y}_{i, r}^{\pm 1}$, $(i, r)\notin \overline{I}_{\xi}^{\mathrm{tw}, \flat}$. This is called \emph{the truncation of the $(q, t)$-character $E_t(m)$}. By the same argument as in \cite[Proposition 6.1]{HL:cluaff}, we obtain the following proposition. 
\begin{proposition}\label{p:trunc}
The assignment $E_t(m)\mapsto E_t(m)^T$ is extended to an injective algebra homomorphism $(\cdot)^T\colon \mathcal{K}_{t, \mathcal{Q}^{\flat}}\to \mathcal{Y}_{t, \mathcal{Q}^{\flat}}$. 
\end{proposition}
The image of $\mathcal{K}_{t, \mathcal{Q}^{\flat}}$ under the truncation $(\cdot)^T$ will be denoted by $\mathcal{K}_{t, \mathcal{Q}^{\flat}}^{T}$. Note that the algebra anti-involution $\overline{(\cdot)}$ on $\mathcal{Y}_t$ restricts to that on $\mathcal{Y}_{t, \mathcal{Q}^{\flat}}$, and $(\cdot)^T$ commutes with $\overline{(\cdot)}$. 

\begin{remark}\label{r:tr-qt}  
By the injectivity of $(\cdot)^T$, its compatibility with $\overline{(\cdot)}$ and Remark \ref{r:characterize-qt}, the truncation of a $(q, t)$-character $L_t(m)^{T}$ is characterized in $\mathcal{K}_{t, \mathcal{Q}^{\flat}}^{T}$ by the following properties :  
\begin{enumerate}
\item[(tr-S1)] $\overline{L_t(m)^T}=L_t(m)^T$, and
\item[(tr-S2$'$)] $L_t(m)^T-E_t(m)^T\in\sum_{m'\in \mathbb{B}^{\xi, \flat}}t^{-1}\mathbb{Z}[t^{-1}]E_t(m')^T$.
\end{enumerate}
\end{remark}

\section{Quantized coordinate algebras}\label{s:dualcan}
Through our isomorphism $\widetilde{\Phi}^T$, the two algebras arising from different contexts are related. The counterpart of $\mathcal{K}_{t, \mathcal{Q}^{\flat}}$ is actually a quantized coordinate algebra, which is our second main theorem (Theorem \ref{t:mainisom}). In this section, we briefly review results about quantized coordinate algebras. See also Appendix  \ref{a:QEA}.

We return to the general setting in subsection \ref{ss:QLA}. \emph{The negative half $\Uv^-(=\Uv^-(\mathrm{X}_N))$ of the quantized enveloping algebra of type $\mathrm{X}_{N}$} over $\mathbb{Q}(v^{1/2})$ is the unital associative $\mathbb{Q}(v^{1/2})$-algebra defined by the generators $\{f_{i}\mid i\in I\}$ and the relations
\[
\sum_{k=0}^{1-c_{ij}}(-1)^{k}\left[\begin{array}{c}
1-c_{ij}\\
k
\end{array}\right]_{v_i}f_{i}^{k}f_{j}f_{i}^{1-c_{ij}-k}=0
\]
for all $i,j\in I$ with $i\neq j$.

This is a $-Q_+$-graded algebra by $\deg f_{i}=-\alpha_i$. For a homogeneous element $x$ of $\Uv^-$ with respect to this grading, we write $\wt x:=\deg(x)\in -Q_+$, and it is called \emph{the weight of $x$}. 

\emph{The quantized coordinate algebra} $\mathcal{A}_v[N_-](=\mathcal{A}_v[N_-^{\mathrm{X}_N}])$ is a specific $\mathbb{Z}[v^{\pm 1/2}]$-subalgebra of $\Uv^-$, which satisfies the properties : 
\begin{align*}
\mathbb{Q}(v^{\pm 1/2})\otimes_{\mathbb{Z}[v^{\pm 1/2}]}\mathcal{A}_v[N_-]&\simeq \Uv^-
&
\mathbb{C}\otimes_{\mathbb{Z}[v^{\pm 1/2}]}\mathcal{A}_v[N_-]&\simeq \mathbb{C}[N_-]. 
\end{align*}
Here $\mathbb{C}$ is regarded as a $\mathbb{Z}[v^{\pm 1/2}]$-module via $v^{\pm 1/2}\mapsto 1$, and $\mathbb{C}[N_-]$ is the coordinate algebra of the complex unipotent algebraic group $N_-$ whose Lie algebra is the negative half of $\mathfrak{g}$ with respect to its triangular decomposition. See Appendix \ref{a:QEA} for the precise definition of $\mathcal{A}_v[N_-]$, and, for example, \cite[subsection 5.1]{GLS:Kac-Moody}, \cite[subsection 4.9]{Kimura:qunip} for more details. In particular, $\mathbb{C}[N_-]$  is commutative. We have an obvious $\mathbb{Z}$-algebra homomorphism
\[
\mathrm{ev}_{v=1}\colon \mathcal{A}_v[N_-]\to \mathbb{C}\otimes_{\mathbb{Z}[v^{\pm 1/2}]}\mathcal{A}_v[N_-]\simeq \mathbb{C}[N_-], 
\]
called \emph{the specialization at $v=1$}. 

For $\lambda\in P_+$ and $w, w'\in W$, we can define a homogeneous element $\widetilde{D}_{w\lambda, w'\lambda}$ of $\mathcal{A}_v[N_-]$, called \emph{a normalized unipotent quantum minor}, as in \cite[section 6]{Kimura:qunip}, \cite[section 5]{GLS:qcluster}. Its precise definition is written in Appendix \ref{a:QEA}. We recall here their properties which are important to this paper. 

\begin{remark}\label{r:minorconv}
There are several conventions of the definitions of quantized coordinate algebras and (normalized) unipotent quantum minors. For example, in \cite{HL:qGro}, quantized coordinate algebras are defined through the \emph{positive} half of the quantized enveloping algebras. The reference \cite[Remark 7.16]{KiO:BZtwist} may be convenient to consider the relations between several conventions. Note that our convention coincides with the one in \cite{KiO:BZtwist} by just adding $v^{\pm 1/2}$ in the base ring (see Appendix \ref{ss:unipminor} for the normalization of unipotent quantum minors).
\end{remark}

For $\lambda\in P_+$ and $w, w'\in W$, we have  
\begin{align*}
\widetilde{D}_{w\lambda, w'\lambda}&=\begin{cases}
1&\text{if}\ w\lambda=w'\lambda,\\
0&\text{if}\ w\lambda\not\leq w'\lambda,
\end{cases}
&
\wt \widetilde{D}_{w\lambda, w'\lambda}&=w\lambda-w'\lambda\ \text{if}\ w\lambda\leq w'\lambda. 
\end{align*}

The quantized coordinate algebra $\mathcal{A}_v[N_-]$ has a quantum cluster algebra structure whose initial quantum cluster variables are described by the normalized unipotent quantum minors : 
\begin{theorem}[{\cite[Theorem 12.3]{GLS:qcluster},\cite[Theorem 10.1]{GY:Memo}}]\label{t:qcluster}
Let $w_0$ be the longest element of $W$. Fix $\bm{i}=(i_1,\dots i_{\ell})\in I(w_0)$. 
Recall the notations in Proposition \ref{p:compatible}. Then there is a $\mathbb{Q}(v^{1/2})$-algebra isomorphism 
\[
\mathrm{CL}\colon\mathbb{Q}(v^{1/2})\otimes_{\mathbb{Z}[v^{\pm 1/2}]}\mathcal{A}(\Lambda^{\bm{i}}, \widetilde{B}^{\bm{i}})\to \mathbb{Q}(v^{1/2})\otimes_{\mathbb{Z}[v^{\pm 1/2}]}\mathcal{A}_v[N_{-}]
\]
given by  
\[
X_k\mapsto \widetilde{D}_{w_{\leq k}\varpi_{i_{k}}, \varpi_{i_{k}}}
\]
for all $k\in J=\{1,\dots, \ell\}$. Moreover, $\mathrm{CL}(\mathcal{A}(\Lambda^{\bm{i}}, \widetilde{B}^{\bm{i}}))\supset \mathcal{A}_v[N_{-}]$. 
\end{theorem}
\begin{remark}
We can easily check that the isomorphism $\mathrm{CL}$ is still well-defined even if we identify the quantum cluster algebra $\mathcal{A}(\Lambda^{\bm{i}}, \widetilde{B}^{\bm{i}})$ with the one arising from a different reduced word commutation equivalent to $\bm{i}$ as in Notation \ref{n:compatible}. Hence we will regard $\mathrm{CL}$ as the map from $\mathbb{Q}(v^{1/2})\otimes_{\mathbb{Z}[v^{\pm 1/2}]}\mathcal{A}(\Lambda^{[\bm{i}]}, \widetilde{B}^{[\bm{i}]})$. 
\end{remark}
\begin{remark}\label{r:qcluster}
In the following, through $\mathrm{CL}^{-1}$, we will regard $\mathcal{A}_v[N_{-}]$ as a subalgebra of $\mathcal{T}_{[\bm{i}]}$. 
\end{remark}

Fix $\bm{i}=(i_1,\dots, i_{\ell})\in I(w_0)$. For $\beta, \gamma\in \Delta_+$ with $\res^{[\bm{i}]}(\beta)=\res^{[\bm{i}]}(\gamma)(=i)$ and $\gamma\preceq_{[\bm{i}]}\beta$, set
\[
\widetilde{D}^{[\bm{i}]}(\beta, \gamma):=\widetilde{D}_{w_{\leq d}\varpi_{i}, w_{\leq b}\varpi_i},
\]
here $\beta=\beta_d^{\bm{i}}$ and $\gamma=\beta_b^{\bm{i}}$ ($d\geq b$). It can be checked easily that the right-hand side depends only on the commutation class of $\bm{i}$ and $\beta, \gamma$, hence this element is well-defined. Moreover, we set $\widetilde{D}^{[\bm{i}]}(\beta, 0)=\widetilde{D}_{w_{\leq d}\varpi_{i_d}, \varpi_{i_d}}$, where $\beta=\beta_d^{\bm{i}}$, for $\beta\in \Delta_+$ ($0$ stands for $0$ in the sense of Notation \ref{n:indexplus}, and these elements are also well-defined) and $\widetilde{D}^{[\bm{i}]}(\beta, \beta):=1$ for $\beta\in \Delta_+$.  
These elements satisfy a specific system of equalities as follows (recall Notation \ref{n:indexplus} and \ref{n:compatible}) : 
\begin{proposition}[{A system of quantum determinantal identities \cite[Proposition 5.5]{GLS:qcluster}}]\label{p:T-sys}
Let $\bm{i}\in I(w_0)$, $\beta, \gamma\in \Delta_+$ with $\res^{[\bm{i}]}(\beta)=\res^{[\bm{i}]}(\gamma)(=i)$ and $\gamma\prec_{[\bm{i}]}\beta$. Then there exist $a, b\in \mathbb{Z}$ such that 
\begin{align}
\widetilde{D}^{[\bm{i}]}(\beta, \gamma)\widetilde{D}^{[\bm{i}]}(\beta^-, \gamma^-)&=v^{a/2}\widetilde{D}^{[\bm{i}]}(\beta, \gamma^-)\widetilde{D}^{[\bm{i}]}(\beta^-, \gamma)+v^{b/2}\dprod_{j\in I\setminus\{i\}}\widetilde{D}^{[\bm{i}]}(\beta^-(j), \gamma^-(j))^{-c_{ji}}\label{eq:T-sys'}
\end{align}
in $\mathcal{A}_v[N_-]$.  
\end{proposition}
They are regarded as a quantum analogue of determinantal identities. It is shown in \cite[section 13]{GLS:Kac-Moody} and \cite[section 12]{GLS:qcluster} that they are also regarded as specific exchange relations of quantum clusters when we consider the quantum cluster algebra structure given in Theorem \ref{t:qcluster}.
\begin{remark}
The explicit powers ${a/2}, {b/2}$ of $v$ in \eqref{eq:T-sys'} can be calculated (see \cite[Proposition 5.5]{GLS:qcluster}). However we do not need them in this paper. 
\end{remark}
Let $\bm{i}=(i_1,\dots i_{\ell(w_0)})\in I(w_0)$. For $j=1,\dots, \ell (w_0)$, write 
\[
\widetilde{F_{-1}^{{\rm {up}}}}(\beta^{\bm{i}}_k, \bm{i}):=\widetilde{D}_{w_{\leq k}\varpi_{i_k}, w_{\leq k-1}\varpi_{i_k}},
\]
here $\beta^{\bm{i}}_k:=w_{\leq k-1}(\alpha_{i_k})$. Note that $\wt \widetilde{F_{-1}^{{\rm {up}}}}(\beta^{\bm{i}}_k, \bm{i})=-\beta^{\bm{i}}_k$. This is an analogue of the dual function of a root vector corresponding to $-\beta^{\bm{i}}_k$. For $\bm{c}=(c_{\beta})_{\beta\in \Delta_+}\in \mathbb{Z}_{\geq 0}^{\Delta_+}$, set
\begin{align}
\widetilde{F_{-1}^{\mathrm{up}}}\left(\bm{c},\bm{i}\right)=v^{-\sum_{s<t}c_sc_t(\beta^{\bm{i}}_s, \beta^{\bm{i}}_t)/2}\widetilde{F_{-1}^{\mathrm{up}}}(\beta^{\bm{i}}_{\ell},\bm{i})^{c_{\ell}}\cdots \widetilde{F_{-1}^{\mathrm{up}}}(\beta^{\bm{i}}_2,\bm{i})^{c_{2}}\widetilde{F_{-1}^{\mathrm{up}}}(\beta^{\bm{i}}_1,\bm{i})^{c_{1}}, 
\label{eq:normPBW}
\end{align}
here we write $c_{s}:=c_{\beta^{\bm{i}}_s}$ for $s=1,\dots, \ell$. 
\begin{remark}
The element $\widetilde{F_{-1}^{\mathrm{up}}}\left(\bm{c},\bm{i}\right)$ coincides with $\ast (F^{\mathrm{up}}\left(\bm{c},\bm{i}\right))$ in \cite{KiO:BZtwist} up to powers of $v^{\pm 1/2}$. See also \cite[Proposition 5.3.5]{KQ}. The subscript $-1$ is attached just for this conventional reason, and it originally comes from the one in $T_{i, -1}'$ in \cite[Chapter 37]{Lus:intro}. 
\end{remark}
In fact, for $\bm{i}'\in [\bm{i}]$, we have 
\[
\widetilde{F_{-1}^{\mathrm{up}}}\left(\bm{c},\bm{i}\right)=\widetilde{F_{-1}^{\mathrm{up}}}\left(\bm{c},\bm{i}'\right)
\]
for all $\bm{c}\in \mathbb{Z}_{\geq 0}^{\Delta_+}$. Therefore, we will write $\widetilde{F_{-1}^{\mathrm{up}}}\left(\bm{c},\bm{i}\right)$ as $\widetilde{F_{-1}^{\mathrm{up}}}\left(\bm{c},[\bm{i}]\right)$. 

\begin{proposition}[{see, for instance, \cite[Chapter 40, 41]{Lus:intro}}]\label{p:PBWbasis}
The set $\{\widetilde{F_{-1}^{\mathrm{up}}}\left(\bm{c},[\bm{i}]\right)\mid \bm{c}\in \mathbb{Z}_{\geq 0}^{\Delta_+}\}$ forms a $\mathbb{Z}[v^{\pm 1/2}]$-basis of $\mathcal{A}_v[N_-]$. 
\end{proposition}
The basis $\{\widetilde{F_{-1}^{\mathrm{up}}}\left(\bm{c},[\bm{i}]\right)\mid \bm{c}\in \mathbb{Z}_{\geq 0}^{\Delta_+}\}$ is called \emph{a normalized dual Poincar\'e-Birkhoff-Witt type basis} (henceforth, \emph{a normalized dual PBW-type basis}). In particular, $\mathcal{A}_v[N_-]$ is generated by $\{\widetilde{F_{-1}^{{\rm {up}}}}(\beta, [\bm{i}])\}_{\beta\in \Delta_+}=\{\widetilde{D}^{\bm{i}}(k, k^-)\mid k=1,\dots, \ell\}$. 

There exists a $\mathbb{Z}$-algebra anti-involution $\sigma'\colon \mathcal{A}_v[N_-]\to \mathcal{A}_v[N_-]$ given by 
\begin{align*}
v^{\pm 1/2}&\mapsto v^{\mp 1/2}
&
\widetilde{F_{-1}^{{\rm {up}}}}(\beta, [\bm{i}])&\mapsto \widetilde{F_{-1}^{{\rm {up}}}}(\beta, [\bm{i}]) 
\end{align*}
for $\beta\in \Delta_+$. We call $\sigma'$ \emph{the twisted dual bar involution}. This is compatible with the bar involution in the theory of quantum cluster algebras, that is, 
\begin{align}
\mathrm{CL}\circ \overline{(\cdot)}=\sigma'\circ \mathrm{CL}.\label{eq:barcompati}
\end{align}

Lusztig \cite{Lus:can1,Lus:quiper,Lus:intro}
and Kashiwara \cite{Kas:Qana} have constructed a specific $\mathbb{Z}[v^{\pm 1/2}]$-basis $\mathbf{B}^{\mathrm{up}}$ of $\mathcal{A}_v[N_-]$, which is called \emph{the dual canonical basis}. Moreover, we can consider the normlized dual canonical bases $\widetilde{\mathbf{B}}^{\mathrm{up}}$. Following the usual convention, we write $\widetilde{\mathbf{B}}^{\mathrm{up}}=\{\widetilde{\Gup}(b)\mid b\in\mathscr{B}(\infty)\}$. In this paper, $\mathscr{B}(\infty)$ is just an index set of the dual canonical basis. See, for instance, \cite{Kas:crys-survey} for the additional remarkable structure, called the Kashiwara crystal structure, on $\mathscr{B}(\infty)$ which will not be used in this paper. The normalized dual canonical basis $\widetilde{\mathbf{B}}^{\mathrm{up}}$ is characterized as follows :  
\begin{theorem}[{\cite[Theorem 4.29]{Kimura:qunip}}]
\label{t:dualcanonical} Let $\bm{i}\in I\left(w_0\right)$.
Then each element $\widetilde{\Gup}(b)$ of $\widetilde{\mathbf{B}}^{\mathrm{up}}$ is characterized by the following conditions :
\begin{enumerate}
\item[(NDCB1)] $\sigma'(\widetilde{\Gup}(b))=\widetilde{\Gup}(b)$, and
\item[(NDCB2)] $\widetilde{\Gup}(b)=\widetilde{F_{-1}^{\mathrm{up}}}(\bm{c},[\bm{i}])+\sum_{\bm{c}'<_{[\bm{i}]}\bm{c}}d_{\bm{c},\bm{c}'}^{[\bm{i}]}\widetilde{F_{-1}^{\mathrm{up}}}(\bm{c}',[\bm{i}])$ with $d_{\bm{c},\bm{c}'}^{[\bm{i}]}\in v\mathbb{Z}[v]$ for some $\bm{c}\in\mathbb{Z}_{\geq 0}^{\Delta_+}$.

Here the condition $\bm{c}'=(c'_{\beta})_{\beta\in \Delta_+}<_{[\bm{i}]}\bm{c}=(c_{\beta})_{\beta\in \Delta_+}$ means that, for all $\bm{i}'\in [\bm{i}]$, there exists $k\in \{1,\dots, \ell(w_0)\}$ such that $c'_{\beta^{\bm{i}'}_1}=c_{\beta^{\bm{i}'}_1},\dots, c'_{\beta^{\bm{i}'}_{k-1}}=c_{\beta^{\bm{i}'}_{k-1}}$ and $c'_{\beta^{\bm{i}'}_{k}}<c_{\beta^{\bm{i}'}_{k}}$.
\end{enumerate}
\end{theorem}

\begin{definition}\label{d:PBWparam}
Theorem \ref{t:dualcanonical} says that each $F_{-1}^{{\rm {up}}}(\bm{c},[\bm{i}])$ determines a unique normalized dual canonical basis element $\Gup(b)$. We write the corresponding element of $\mathscr{B}(\infty)$ as $b_{-1}(\bm{c},[\bm{i}])$. This is called \emph{the Lusztig parametrization} of $\mathscr{B}(\infty)$. 
\end{definition}

In fact, the non-zero normalized unipotent quantum minors are elements of $\widetilde{\mathbf{B}}^{\mathrm{up}}$ \cite[Proposition 4.1]{Kas:Dem}. 
\begin{remark}\label{r:unitri-PBW}
The unitriangular property of dual PBW-type bases with respect to $\sigma'$ is also included in Theorem \ref{t:dualcanonical} in the same manner as in Remark \ref{r:unitri-qt}. That is, we have 
\[
\sigma'\left(\widetilde{F_{-1}^{\mathrm{up}}}(\bm{c},[\bm{i}])\right)\in \widetilde{F_{-1}^{\mathrm{up}}}(\bm{c},[\bm{i}])+\sum_{\bm{c}'<_{[\bm{i}]}\bm{c}}\mathbb{Z}[v^{\pm 1}]\widetilde{F_{-1}^{\mathrm{up}}}(\bm{c}',[\bm{i}]).
\] 
\end{remark}

\section{Quantum $T$-system for type $\mathrm{B}$}\label{s:T-sys}
In this section, we prove a $t$-analogue of the specific identities, called \emph{the $T$-system}, of $q$-characters for type $\mathrm{B}_n^{(1)}$. The corresponding equalities for type $\mathrm{A}_n^{(1)}$, $\mathrm{D}_n^{(1)}$ and $\mathrm{E}_n^{(1)}$ are proved in \cite[Proposition 5.6]{HL:qGro}, and in \cite{Nak:KR} by using other quantum tori. There are some technical differences with these cases as, for example, we do not know a priori certain positivity properties of $(q, t)$-characters. First we work in the general setting in section \ref{s:prel}. 

A $\mathcal{U}_q(\mathcal{L}\mathfrak{g})$-module $V$ is said to be \emph{thin} if its all $l$-weight spaces are of dimension $1$. 
\begin{proposition}[{\cite[Proposition 5.14]{H:small}}]\label{p:thin}
Let $V$ be a thin $\mathcal{U}_q(\mathcal{L}\mathfrak{g})$-module. Then every monomial $m$ occurring in $\chi_q(V)$ satisfies
\[
\max_{j\in I, s\in \mathbb{Z}}\{|u_{j, s}(m)|\}\leq 1. 
\] 
\end{proposition}

\begin{notation}\label{n:Vbar}
Consider a $\mathbb{Z}$-linear map $\mathcal{Y}\to \mathcal{Y}_t, f\mapsto \underline{f}$ given by $m\mapsto\underline{m}$ for all monomials $m$ in $\mathcal{Y}$.
Note that this map is not an algebra homomorphism. For a module $V$ in $\mathcal{C}_{\bullet}$, we set 
\[
[\underline{V}]:=\underline{\chi_q(V)}\in \mathcal{Y}_t. 
\]
Note that it clearly satisfies $\mathrm{ev}_{t=1}([\underline{V}])=\chi_q(V)$, but $[\underline{V}]\notin \mathcal{K}_t$ in general. 
\end{notation}

\begin{proposition}\label{p:thin-qt}
Let $V$ be a thin $\mathcal{U}_q(\mathcal{L}\mathfrak{g})$-module.
Then we have $[\underline{V}]\in \mathcal{K}_t$. 
\end{proposition}
\begin{proof}
Let $j\in I$. By our definition of the quantum Grothendieck ring $\mathcal{K}_t$, we only have to show that $[\underline{V}]\in \mathcal{K}_{j, t}$. Consider the decomposition $\chi_q(V)=\sum_{m'}\lambda_{\{j\}}(m')L_{\{j\}}(m')$ as in \cite[Proposition 3.14]{H:small} for $J = \{j\}$. Then $[\underline{V}]=\sum_{m'}\lambda_{\{j\}}(m')\underline{L_{\{j\}}(m')}$. By Proposition \ref{p:thin}, each $L_{\{j\}}(m')$ is constructed from the $q$-characters of thin simple $\mathcal{U}_{q_j}(\mathcal{L}\mathfrak{g}_{j})(\simeq \mathcal{U}_{q_i}(\mathcal{L}\mathfrak{sl}_2))$-modules (see \cite{H:small} for the precise definition). Hence it follows from the well-known explicit formulas of the $(q, t)$-characters of thin simple $\mathcal{U}_{q_i}(\mathcal{L}\mathfrak{sl}_2))$-modules \cite[Lemma 4.13]{H:qt} that $[\underline{V}]=\sum_{m'}\lambda_{\{j\}}(m')\underline{L_{\{j\}}(m')}\in \mathcal{K}_{j, t}$. 
\end{proof}

As mentioned in the introduction, one crucial result for our proof is the following.

\begin{theorem}[{\cite[Theorem 3.10]{H:minim}}]\label{t:KRthin}
Assume that $\mathfrak{g}$ is of type $\mathrm{A}_n$ or $\mathrm{B}_n$. Then every Kirillov-Reshetikhin module $W^{(i)}_{k, r}$ over $\mathcal{U}_{q}(\mathcal{L}\mathfrak{g})$ is thin. 
\end{theorem}

From now until the end of this section, we assume that $\mathfrak{g}$ is of type $\mathrm{B}_n$. 

\begin{proposition}\label{p:F-KR}
For $i\in I$, $r\in \mathbb{Z}$ and $k\in \mathbb{Z}_{\geq 0}$, we have $[\underline{W_{k,r}^{(i)}}]=F_t(m_{k, r}^{(i)}) \in \mathcal{K}_t$. 
\end{proposition}
\begin{proof}
By Proposition \ref{p:thin-qt} and Theorem \ref{t:KRthin}, we have $[\underline{W_{k,r}^{(i)}}]\in \mathcal{K}_t$, which implies that $[\underline{W_{k,r}^{(i)}}]=F_t(m_{k, r}^{(i)})$ because $W_{k,r}^{(i)}$ is affine-minuscule (see Theorem \ref{t:minuscule}, \ref{t:F-elem}). 
\end{proof}

Now let us establish the quantum $T$-systems in $\mathcal{K}_t$. Let us define the product : 
\[
[\underline{S_{k, r}^{(i)}}]= \begin{cases} [\underline{W_{k, r+2}^{(i-1)}}] [\underline{W_{k,r+2}^{(i+1)}}]\text{ if $i \leq n-2$}, 
\\ [\underline{W_{k,r+2}^{(n-2)}]} [\underline{W_{2k,r+1}^{(n)}}]\text{ if $i = n-1$},
\\ [\underline{W_{s,r+1}^{(n-1)}}][\underline{W_{s,r+3}^{(n-1)}}]  \text{ if $i = n$ and $k = 2s$ is even,}
\\ [\underline{W_{s+1,r+1}^{(n-1)}}][\underline{W_{s,r+3}^{(n-1)}}] \text{ if $i = n$ and $k = 2s+1$ is odd.}
\end{cases}
\]
Here we set $[\underline{W_{a, \ell}^{(0)}}]:=1$ for all $a\in \mathbb{Z}_{\geq 0}$ and $\ell\in \mathbb{Z}$. 
\begin{theorem}[{The quantum $T$-system of type $\mathrm{B}$}]\label{t:T-sysaff} 
For $i\in I$, $r\in \mathbb{Z}$ and $k\in \mathbb{Z}_{> 0}$, there exist $\alpha, \beta\in \mathbb{Z}$ such that the following identity holds in $\mathcal{K}_t$ :
\[
[\underline{W_{k,r}^{(i)}}][\underline{W_{k,r+2r_i}^{(i)}}] = t^{\alpha/2} [\underline{W_{k+1,r}^{(i)}}] [\underline{W_{k-1,r+2r_i}^{(i)}}] + t^{\beta/2} [\underline{S_{k, r}^{(i)}}].
\]
\end{theorem}

\begin{remark}
The powers $\alpha,\beta$ can be computed from the $t$-commutation relations in the quantum torus as in \cite[Proposition 5.6]{HL:qGro}. However we do not use these explicit values in this paper. 
\end{remark}
\begin{remark}
The classical counterpart of the equalities in Theorem \ref{t:T-sysaff} is proved in \cite[Theorem 3.4]{H:KR-T}. By definition, this classical counterpart is a system of identities in Grothendieck ring $K(\mathcal{C}_{\bullet})$, and the corresponding exact sequences are also obtained in \cite{H:KR-T}.
\end{remark}
Before proving Theorem \ref{t:T-sysaff}, we recall one lemma in \cite{H:KR-T}. 

\begin{lemma}[{\cite[Proposition 5.3, Lemma 5.6 and its proof]{H:KR-T}}]\label{l:dominant}
Let $i\in I$, $r\in \mathbb{Z}$ and $k\in \mathbb{Z}_{> 0}$. Set $M:=m_{k, r}^{(i)}m_{k,r+2r_i}^{(i)}=m_{k+1,r}^{(i)}m_{k-1,r+2r_i}^{(i)}$. 
\begin{itemize}
\item[(i)] The $q$-character $\chi_q(W_{k,r}^{(i)})$ contains monomials $\{m_{k, r}^{(i)}\prod_{s=0}^{l-1}A_{i, r+(2k-1-2s)r_i}^{-1}\mid l=0,\dots, k\}$ with multiplicity $1$, and $\{M\prod_{s=0}^{l-1}A_{i, r+(2k-1-2s)r_i}^{-1}\mid l=0,\dots, k\}$ are all the dominant monomials occurring in $\chi_q(W_{k,r}^{(i)})\chi_q(W_{k,r+2r_i}^{(i)})$. Moreover, they also occur with multiplicity $1$ in $\chi_q(W_{k,r}^{(i)})\chi_q(W_{k,r+2r_i}^{(i)})$. 
\item[(ii)] The $q$-character $\chi_q(W_{k-1,r+2r_i}^{(i)})$ contains monomials $\{m_{k-1,r+2r_i}^{(i)}\prod_{s=0}^{l-1}A_{i, r+(2k-1-2s)r_i}^{-1}\mid l=0,\dots, k-1\}$ with multiplicity $1$, and $\{M\prod_{s=0}^{l-1}A_{i, r+(2k-1-2s)r_i}^{-1}\mid l=0,\dots, k-1\}$ are all the dominant monomials occurring in $\chi_q(W_{k+1,r}^{(i)})\chi_q(W_{k-1,r+2r_i}^{(i)})$. Moreover they also occur with multiplicity $1$ in $\chi_q(W_{k+1,r}^{(i)})\chi_q(W_{k-1,r+2r_i}^{(i)})$. 
\item[(iii)] The monomial $M\prod_{s=0}^{k-1}A_{i, r+(2k-1-2s)r_i}^{-1}$ is the unique dominant monomial occurring in $\chi_q(S_{k, r}^{(i)})(:=\mathrm{ev}_{t=1}([\underline{S_{k, r}^{(i)}}]
))$. 
\end{itemize}
\end{lemma}
\begin{proof}[{Proof of Theorem \ref{t:T-sysaff}}]
The monomials occurring in $[\underline{W_{k,r}^{(i)}}][\underline{W_{k,r+2r_i}^{(i)}}]$, $[\underline{W_{k+1,r}^{(i)}}] [\underline{W_{k-1,r+2r_i}^{(i)}}]$ and $[\underline{S_{k, r}^{(i)}}]$ are the same as in the corresponding $q$-characters at $t = 1$. 
Hence the dominant monomials occurring in these elements are completely described by Lemma \ref{l:dominant}. Besides, an element in $\mathcal{K}_t$ is characterized by the multiplicities of its dominant monomials (recall the uniqueness of $F_t(m)$ in Theorem \ref{t:F-elem}), hence it suffices to show that the dominant monomials occurring in both-hand sides of the equality in the theorem match. More precisely, by Lemma \ref{l:dominant}, we only have to check that there exist $\alpha, \beta'\in \mathbb{Z}$ such that  
\begin{align}
\left(\sum_{l=0}^{k}\underline{M_{l}}\right) \underline{m_{k,r+2r_i}^{(i)}}=t^{\alpha/2} \underline{m_{k+1,r}^{(i)}}\left(\sum_{l=0}^{k-1}\underline{M'_{l}}\right)+t^{\beta'/2}\underline{m_{k, r}^{(i)}m_{k,r+2r_i}^{(i)}\prod_{s=0}^{k-1}A_{i, r+(2k-1-2s)r_i}^{-1}},\label{eq:dominant}
\end{align}
here 
$$M_{l}:=m_{k, r}^{(i)}\prod_{s=0}^{l-1}A_{i, r+(2k-1-2s)r_i}^{-1}\text{ and }M'_{l}:=m_{k-1,r+2r_i}^{(i)}\prod_{s=0}^{l-1}A_{i, r+(2k-1-2s)r_i}^{-1}.$$ 
We can show this equality directly. 
\end{proof}
\begin{remark}
There is another argument to prove \eqref{eq:dominant}  : the multiplicities of commutative monomials in $(\sum_{l=0}^{k}\underline{M_{l}}) \underline{m_{k,r+2r_i}^{(i)}}$ (resp.~$m_{k+1,r}^{(i)}(\sum_{l=0}^{k-1}\underline{M'_{l}})$) are the same as those for the corresponding $\mathcal{L}sl_2$-case up to overall power of $t^{1/2}$. This follows from the fact that, in the quantum torus $\mathcal{Y}_t$, the $t$-commutation relations between $A_{i,s}^{-1}$, $A_{i,s'}^{-1}$ and between $Y_{i,s}$,$A_{i,s'}^{-1}$ (with the same index $i$) are the same as in the $sl_2$-case ($q$ being replaced by $q_i$). See Proposition \ref{p:A-comm}. Consequently, as the quantum $T$-system has been established in the $\mathcal{L}sl_2$-case \cite[Proposition 5.6]{HL:qGro}, we obtain the desired equality. 
\end{remark}
\section{{The isomorphism $\widetilde{\Phi}$}}\label{s:mainisom}
In this section, we show that the isomorphism $\widetilde{\Phi}^T$ in Theorem \ref{t:torusisom} restricts to an isomorphism between the corresponding quantized coordinate algebra and the truncated quantum Grothendieck ring (Theorem \ref{t:mainisom}). Recall Remark \ref{r:qcluster}.
 This implies that the quantized coordinate algebra of type $\mathrm{A}_{2n-1}$ is isomorphic to the quantum Grothendieck ring of $\mathcal{C}_{\mathcal{Q}^{\flat}}$ (Corollary \ref{c:mainisom}). We also show that the elements of the normalized dual canonical basis are sent to the $(q, t)$-characters of simple modules under this isomorphism. 

Recall the settings and notations in subsection \ref{ss:subcat} and section \ref{s:qtoriisom}, that is, let $\mathcal{Q}$ be a Dynkin quiver of type $\mathrm{A}_{2n-2}$, $\xi\colon I=\{1, 2,\dots, 2n-2\}\to\mathbb{Z}$ be an associated height function, $\flat\in \{>, <\}$,  $I_{\mathrm{A}}:=\{1,\dots, 2n-1\}$, $I_{\mathrm{B}}:=\{1,\dots, n\}$. For $(i, r)\in \overline{I}_{\xi}^{\mathrm{tw}, \flat}$,  $k(i, r)$ is defined as \eqref{eq:kir}. By Theorem \ref{t:minuscule} and Proposition \ref{p:F-KR}, we have 
\begin{align}
F_t(m_{k(i, r), r}^{(i)})^T=[\underline{W_{k(i, r), r}^{(i)}}]^T=\underline{m_{k(i, r), r}^{(i)}}\label{eq:tr-KR}
\end{align}
for $(i, r)\in \overline{I}_{\xi}^{\mathrm{tw}, \flat}$.
\begin{theorem}\label{t:mainisom}
The $\mathbb{Z}$-algebra isomorphism $\widetilde{\Phi}^{T}\colon\mathcal{T}_{\mathcal{Q}^{\flat}} \to \mathcal{Y}_{t, \mathcal{Q}^{\flat}}$  restricts to an isomorphism 
\[
\mathcal{A}_v[N_-^{\mathrm{A}_{2n-1}}]\to \mathcal{K}_{t, \mathcal{Q}^{\flat}}^T.
\]
This isomorphism sends the system of quantum determinantal identities associated with $[\mathcal{Q}^{\flat}]$ in Proposition \ref{p:T-sys} to (the truncation of) the quantum $T$-system given in Theorem \ref{t:T-sysaff}. Moreover, this isomorphism induces a bijection between $\widetilde{\mathbf{B}}^{\mathrm{up}}$ and $\{L_t(m)^T\mid m\in \mathbb{B}^{\xi, \flat}\}$. More precisely, if $\widetilde{\Phi}^T(\widetilde{\Gup}(b_{-1}(\bm{c},[\mathcal{Q}^{\flat}])))=L_t(m)^T$, then 
\[
u_{i, r}(m)=c_{(i, r)}
\]
for all $(i, r)\in \overline{I}_{\xi}^{\mathrm{tw}, \flat}$. Here $\bm{c}\in \mathbb{Z}_{\geq 0}^{\Delta_+}$ is regarded as an element $\bm{c}=(c_{(i, r)})_{(i, r)\in \overline{I}_{\xi}^{\mathrm{tw}, \flat}}$ of $\mathbb{Z}_{\geq 0}^{\overline{I}_{\xi}^{\mathrm{tw}, \flat}}$ via $\overline{\Omega}_{\xi}^{\mathrm{tw}, \flat}$. This monomial $m$ will be denoted by $m_{\mathrm{B}}(\bm{c})$. 
\end{theorem}
\begin{proof}
Recall the notation in Proposition \ref{p:T-sys}, and identify $\Delta_+$ with $\overline{I}_{\xi}^{\mathrm{tw}, \flat}$ via $\overline{\Omega}_{\xi}^{\mathrm{tw}, \flat}$. In this proof, we abbreviate $\widetilde{D}^{[\mathcal{Q}^{\flat}]}((i, r), (j, s))$ to $\widetilde{D}((i, r), (j, s))$. 

To prove $\widetilde{\Phi}^T(\mathcal{A}_v[N_-^{\mathrm{A}_{2n-1}}])=\mathcal{K}_{t, \mathcal{Q}^{\flat}}^T$, we only have to show that 
\begin{align}
\widetilde{\Phi}^T(\widetilde{D}((i, r), (i, r+2sr_i)))=F_t(m_{s, r}^{(i)})^T\label{eq:corresp}
\end{align}
for $s=0, 1,\dots, k(i, r)$, where we set $(i, r+2k(i, r)r_i):=0$, here $0$ stands for $0$ in the sense of Notation \ref{n:indexplus} (then we have $(i, r)^{-}=(i, r+2r_i)$). Indeed, the equality \eqref{eq:corresp} implies that the generators of $\mathcal{A}_v[N_-^{\mathrm{A}_{2n-1}}]$ are sent to the generators of $\mathcal{K}_{t, \mathcal{Q}^{\flat}}^T$ (see Lemma \ref{l:monoidal} and Proposition \ref{p:PBWbasis}). 
Denote the equality \eqref{eq:corresp} by $\mathrm{T}((i, r), (i, r+2sr_i))$, and define $\mathrm{T}(0, 0)$ as the trivial equality $1=1$.
\begin{claim*}
Suppose that 
\begin{itemize}
\item $\mathrm{T}((i, r+2r_i), (i, r+2(s+1)r_i))$, 
\item $\mathrm{T}((i, r), (i, r+2(s+1)r_i))$, 
\item $\mathrm{T}((i, r+2r_i), (i, r+2sr_i))$, and 
\item $\mathrm{T}((i, r)^-(\jmath), (i, r+2sr_i)^-(\jmath))$ for $\jmath\in I_{\mathrm{A}}\setminus\{\res^{[\mathcal{Q}^{\flat}]}(i, r)\}$
\end{itemize}
hold for some $(i, r)\in \overline{I}_{\xi}^{\mathrm{tw}, \flat}$ with $k(i, r)\geq 2$, and some $s\in \{1,2,\dots, k(i, r)-1\}$. Then $\mathrm{T}((i, r), (i, r+2sr_i))$ holds. 
\end{claim*}
\begin{proof}[Proof of Claim]
It is shown in \cite[section 13]{GLS:Kac-Moody} and \cite[section 12]{GLS:qcluster} that the equality \eqref{eq:T-sys'} correspond to a specific exchange relations of quantum clusters in the quantum cluster algebra structure given in Theorem \ref{t:qcluster}. In particular, 
\begin{itemize}
\item[(QC)] $\widetilde{D}((i, r+2r_i), (i, r+2(s+1)r_i))$, $\widetilde{D}((i, r), (i, r+2(s+1)r_i))$, $\widetilde{D}((i, r+2r_i), (i, r+2sr_i))$ and non-trivial $\widetilde{D}((i, r)^-(\jmath), (i, r+2sr_i)^-(\jmath))$ for $\jmath\in I_{\mathrm{A}}\setminus\{\res^{[\mathcal{Q}^{\flat}]}(i, r)\}$ simultaneously belong to a certain quantum cluster (see Appendix \ref{a:qclus} for the definition of quantum clusters). 
\end{itemize}

Let $\mathcal{F}(\mathcal{T}_{\mathcal{Q}^{\flat}})$ and $\mathcal{F}(\mathcal{Y}_{t, \mathcal{Q}^{\flat}})$ be the skew-field of fractions of $\mathcal{T}_{\mathcal{Q}^{\flat}}$ and $\mathcal{Y}_{t, \mathcal{Q}^{\flat}}$, respectively. Note that $\mathcal{T}_{0}$ and $\mathcal{Y}_{t, \mathcal{Q}^{\flat}}$ are Ore domains (cf.~\cite[Appendix A]{BZ:qcluster}). We can extend the $\mathbb{Z}$-algebra anti-involution $\overline{(\cdot)}$ on $\mathcal{T}_{\mathcal{Q}^{\flat}}$ (resp.~$\mathcal{Y}_{t, \mathcal{Q}^{\flat}}$) to the $\mathbb{Q}$-algebra anti-involution on $\mathcal{F}(\mathcal{T}_{\mathcal{Q}^{\flat}})$ (resp.~$\mathcal{F}(\mathcal{Y}_{t, \mathcal{Q}^{\flat}})$), denoted again by $\overline{(\cdot)}$, and extend the $\mathbb{Z}$-algebra isomorphism $\widetilde{\Phi}^T$ to the $\mathbb{Q}$-algebra isomorphim $\widetilde{\Phi}^T\colon\mathcal{F}(\mathcal{T}_{\mathcal{Q}^{\flat}}) \to \mathcal{F}(\mathcal{Y}_{t, \mathcal{Q}^{\flat}})$. (They still satisfy $\widetilde{\Phi}^T\circ \overline{(\cdot)}=\overline{(\cdot)}\circ\widetilde{\Phi}^T$. ) Note that $\mathcal{F}(\mathcal{T}_{\mathcal{Q}^{\flat}})$ can be also regarded as the skew-field of fractions of $\mathcal{A}_v[N_-^{\mathrm{A}_{2n-1}}]$ and the extension of the algebra anti-involution $\sigma'$ to that on $\mathcal{F}(\mathcal{T}_{\mathcal{Q}^{\flat}})$ coincides with $\overline{(\cdot)}$. 

By \eqref{eq:T-sys'} and (QC) (see also \eqref{eq:qmonominv} in Appendix \ref{a:qclus}), the element $\widetilde{D}((i, r), (i, r+2sr_i))$ is characterized as the $\overline{(\cdot)}$-invariant element of the form 
\begin{align*}
&\left(v^{\alpha/2}\widetilde{D}((i, r), (i, r+2(s+1)r_i))\widetilde{D}((i, r+2r_i), (i, r+2sr_i))\right.\\
&\left.+v^{\beta/2}\dprod_{\jmath:|\jmath-\res^{[\mathcal{Q}^{\flat}]}(i, r)|=1}\widetilde{D}((i, r)^-(\jmath), (i, r+2sr_i)^-(\jmath))\right)\\
&\times\widetilde{D}((i, r+2r_i), (i, r+2(s+1)r_i))^{-1}
\end{align*}
for some $\alpha, \beta\in \mathbb{Z}$ in $\mathcal{F}(\mathcal{T}_{\mathcal{Q}^{\flat}})$. (The integers $\alpha$ and $\beta$ are uniquely determined from the $\overline{(\cdot)}$-invariance property.) Therefore, taking our assumption into the account, we can deduce that the element $\widetilde{\Phi}^T(\widetilde{D}((i, r), (i, r+2sr_i)))$ is characterized as the $\overline{(\cdot)}$-invariant element of the form 
\begin{align*}
&\left(t^{\alpha'/2}F_t(m_{s+1, r}^{(i)})^T
F_t(m_{s-1, r+2r_i}^{(i)})^T+t^{\beta'/2}\dprod_{\jmath: |\jmath-\res^{[\mathcal{Q}^{\flat}]}(i, r)|=1}F_t(m_{0}(\jmath))^T\right)(F_t(m_{s, r+2r_i}^{(i)})^T)^{-1}
\end{align*}
for some $\alpha', \beta'\in \mathbb{Z}$ in $\mathcal{F}(\mathcal{Y}_{t, \mathcal{Q}^{\flat}})$, here $m_0(\jmath)$ is the monomial such that $\widetilde{\Phi}^T(\widetilde{D}((i, r)^-(\jmath), (i, r+2sr_i)^-(\jmath)))=F_t(m_{0}(\jmath))^T$.

On the other hand, by Theorem \ref{t:T-sysaff}, $F_t(m_{s, r}^{(i)})^T(=[\underline{W_{s, r}^{(i)}}]^T)$ is the $\overline{(\cdot)}$-invariant element of the form 
\begin{align*}
\left(t^{\alpha''/2}F_t(m_{s+1, r}^{(i)})^TF_t(m_{s-1, r+2r_i}^{(i)})^T+t^{\beta''/2}[\underline{S_{s, r}^{(i)}}]^T\right)(F_t(m_{s, r+2r_i}^{(i)})^T)^{-1} 
\end{align*}
for some $\alpha'', \beta''\in \mathbb{Z}$ in $\mathcal{F}(\mathcal{Y}_{t, \mathcal{Q}^{\flat}})$. Therefore, it remains to show that 
\begin{align}
&\{ m_0(\jmath)\mid \jmath\in I_{\mathrm{A}}, |\jmath-\res^{[\mathcal{Q}^{\flat}]}(i, r)|=1\}\label{eq:comparison}\\ 
&=\{m_{s', r'}^{(i')}\mid F_t(m_{s', r'}^{(i')})\ \text{occurs in the product in the definition of}\ [\underline{S_{s, r}^{(i)}}]\}. \notag 
\end{align}
Denote the left-hand side of \eqref{eq:comparison} by $M_0$, and write the set $\{ ((i, r)^-(\jmath), (i, r+2sr_i)^-(\jmath))\mid \jmath\in I_{\mathrm{A}}, |\jmath-\res^{[\mathcal{Q}^{\flat}]}(i, r)|=1\}$ as $R_0$. 

If $i\leq n-2$, then, by twisted convexity of $\Upsilon_{[\mathcal{Q}^{\flat}]}$, we have
\[
R_0=\{((i-1, r+2), (i-1, r+4s+2)), ((i+1, r+2), (i+1, r+4s+2))\}.
\]
(If $i=1$, we ignore $((0, r+2), (0, r+4s+2))$. Henceforth, we always ignore such undefined terms.) Note that, when $k(i, r)\geq 2$ and $s\in\{1,\dots, k(i,r)-1\}$, we have $(i\pm 1, r+2), (i\pm 1, r+4s-2)\in \overline{I}_{\xi}^{\mathrm{tw}, \flat}$ ($(i\pm 1, r+4s+2)$ may not belong to $\overline{I}_{\xi}^{\mathrm{tw}, \flat}$). Therefore, 
\[
M_0=\{m_{s, r+2}^{(i-1)}, m_{s, r+2}^{(i+1)}\},
\]
which satisfies \eqref{eq:comparison}. 

If $i=n-1$, then similarly we have
\[
R_0=\{((n-2, r+2), (n-2, r+4s+2)), ((n, r+1), (n, r+4s+1))\}.
\]
Note that, when $k(n-1, r)\geq 2$ and $s\in \{1,\dots, k(n-1,r)-1\}$, we have $(n-2, r+2), (n-2, r+4s-2), (n, r+1), (n, r+4s-1)\in \overline{I}_{\xi}^{\mathrm{tw}, \flat}$ ($(n-2, r+4s+2)$ and $(n, r+4s+1)$ may not belong to $\overline{I}_{\xi}^{\mathrm{tw}, \flat}$). Therefore, 
\[
M_0=\{m_{s, r+2}^{(n-2)}, m_{2s, r+1}^{(n)}\},
\]
which satisfies \eqref{eq:comparison}.

If $i=n$ and $s$ is even, then similarly we have 
\[
R_0=\{((n-1, r+1), (n-1, r+2s+1)), ((n-1, r+3), (n-1, r+2s+3))\}.
\]
Note that, when $k(n, r)\geq 3$ and $s\in \{2,\dots, k(n,r)-1\}$ (remark that $s$ is even), we have $(n-1, r+1), (n-1, r+2s-3), (n-1, r+3), (n-1, r+2s-1)\in \overline{I}_{\xi}^{\mathrm{tw}, \flat}$ ($(n-1, r+2s+3)$ and $(n-1, r+2s+1)$ may not belong to $\overline{I}_{\xi}^{\mathrm{tw}, \flat}$). Therefore, 
\[
M_0=\{m_{s/2, r+1}^{(n-1)}, m_{s/2, r+3}^{(n-1)}\},
\]
which satisfies \eqref{eq:comparison}.

If $i=n$ and $s$ is odd, then similarly we have
\[
R_0=\{((n-1, r+1), (n-1, r+2s+3)), ((n-1, r+3), (n-1, r+2s+1))\}.
\]
Note that, when $s=1$, $(n-1, r+1), (n-1, r+2s-1)\in\overline{I}_{\xi}^{\mathrm{tw}, \flat}$ and $\widetilde{D}((n-1, r+3), (n-1, r+2s+1))=1$ (its corresponding monomial is $1$), otherwise,
 we have $(n-1, r+1), (n-1, r+2s-1), (n-1, r+3), (n-1, r+2s-3)\in \overline{I}_{\xi}^{\mathrm{tw}, \flat}$ ($(n-1, r+2s+3)$ and $(n-1, r+2s+1)$ may not belong to $\overline{I}_{\xi}^{\mathrm{tw}, \flat}$). Therefore, 
\[
M_0=\{m_{(s+1)/2, r+1}^{(n-1)}, m_{(s-1)/2, r+3}^{(n-1)}\},
\]
which satisfies \eqref{eq:comparison}, which completes the proof of Claim.
\end{proof}

It follows form Theorem \ref{t:torusisom} and \eqref{eq:tr-KR} that the equality $\mathrm{T}((i, r), 0)$ holds for all $(i, r)\in \overline{I}_{\xi}^{\mathrm{tw}, \flat}$ ($0$ stands for $0$ in the sense of Notation \ref{n:indexplus}). By Claim and the specific mutation sequence in \cite[section 13]{GLS:Kac-Moody} and \cite[section 12]{GLS:qcluster}, we can obtain $\mathrm{T}((i, r), (i, r+2sr_i))$ for all $s=1,\dots, k(i, r)$ and all $(i, r)\in \overline{I}_{\xi}^{\mathrm{tw}, \flat}$ from the equalities $\mathrm{T}((i, r), 0)$, $(i, r)\in \overline{I}_{\xi}^{\mathrm{tw}, \flat}$. This proves $\Phi(\mathcal{A}_v[N_-^{\mathrm{A}_{2n-1}}])=\mathcal{K}_{t, \mathcal{Q}^{\flat}}^{T}$. Moreover, in the proof of Claim, we prove that our isomorphism sends a system of quantum determinantal identities associated with $[\mathcal{Q}^{\flat}]$ in Proposition \ref{p:T-sys} to (the truncation of) the quantum $T$-system given in Theorem \ref{t:T-sysaff}.

Next, we show the bijection between distinguished bases. It follows from (NDCB1) in Theorem \ref{t:dualcanonical} and the equality $\widetilde{\Phi}^T\circ \overline{(\cdot)}=\overline{(\cdot)}\circ\widetilde{\Phi}^T$ that 
\[
\overline{\widetilde{\Phi}^T(\widetilde{\Gup}(b))}=\widetilde{\Phi}^T(\widetilde{\Gup}(b)).
\]
Hence, by Theorem \ref{t:normqt} (Remark \ref{r:tr-qt}) and Theorem \ref{t:dualcanonical}, it remains to show that 
\begin{align}
\widetilde{\Phi}^T(\widetilde{F_{-1}^{\mathrm{up}}}(\bm{c},[\mathcal{Q}^{\flat}]))=E_t(m_{\mathrm{B}}(\bm{c}))^{T}
\label{eq:PBWcorresp}
\end{align}
for every $\bm{c}\in \mathbb{Z}_{\geq 0}^{\overline{I}_{\xi}^{\mathrm{tw}, \flat}}$.

Recall that we have 
\[
\widetilde{F_{-1}^{\mathrm{up}}}(\bm{e}_{(i, r)}, [\mathcal{Q}^{\flat}])=\widetilde{D}((i, r), (i, r)^-)
\]
for all $(i, r)\in \overline{I}_{\xi}^{\mathrm{tw}, \flat}$, here $\{\bm{e}_{(i, r)}\mid (i, r)\in \overline{I}_{\xi}^{\mathrm{tw}, \flat}\}$ is the standard basis of $\bm{c}\in \mathbb{Z}^{\overline{I}_{\xi}^{\mathrm{tw}, \flat}}$. Hence, by \eqref{eq:corresp}, 
\begin{align}
\widetilde{\Phi}^T(\widetilde{F_{-1}^{\mathrm{up}}}(\bm{e}_{(i, r)},[\mathcal{Q}^{\flat}]))=F_t(Y_{i, r})^T=E_t(Y_{i, r})^{T}.\label{eq:rootcprresp}
\end{align}
This proves \eqref{eq:PBWcorresp} for $\bm{c}=\bm{e}_{(i, r)}$. By \eqref{eq:normPBW}, 
\begin{align*}
\widetilde{F_{-1}^{\mathrm{up}}}(\bm{c},[\mathcal{Q}^{\flat}])=v^{\alpha/2}\dprod_{r\in \mathbb{Z}}\left(\prod_{i; (i, r)\in \overline{I}_{\xi}^{\mathrm{tw}, \flat}}\widetilde{F_{-1}^{\mathrm{up}}}(\bm{e}_{(i, r)},[\mathcal{Q}^{\flat}])^{c_{(i, r)}}\right)
\end{align*}
for some $\alpha\in\mathbb{Z}$. Note that this ordering of the product matches the compatible reading of $\Upsilon_{[\mathcal{Q}^{\flat}]}$. On the other hand, for $m\in \mathbb{B}^{\xi, \flat}$, 
\begin{align*}
E_t(m)=t^{\beta/2}\dprod_{r\in \mathbb{Z}}\left(\prod_{i; (i, r)\in \overline{I}_{\xi}^{\mathrm{tw}, \flat}}F_{t}(Y_{i, r})^{u_{i, r}(m)}\right).
\end{align*}
Therefore, by \eqref{eq:rootcprresp}, we have 
\[
\widetilde{\Phi}^T(\widetilde{F_{-1}^{\mathrm{up}}}(\bm{c},[\mathcal{Q}^{\flat}]))=t^{\beta(\bm{c})/2}E_t(m_{\mathrm{B}}(\bm{c}))^T
\]
for some $\beta(\bm{c})\in\mathbb{Z}$. It suffices to show that $\beta(\bm{c})=0$ for every $\bm{c}\in \mathbb{Z}_{\geq 0}^{\overline{I}_{\xi}^{\mathrm{tw}, \flat}}$. By Remark \ref{r:unitri-qt} and \ref{r:unitri-PBW}, we have 
\begin{align*}
&\overline{E_t(m)}\in E_t(m)+\sum_{m'<m}\mathbb{Z}[t^{\pm 1}]E_t(m'), \\
&\sigma'(\widetilde{F_{-1}^{\mathrm{up}}}(\bm{c},[\mathcal{Q}^{\flat}]))\in \widetilde{F_{-1}^{\mathrm{up}}}(\bm{c},[\mathcal{Q}^{\flat}])+\sum_{\bm{c}'<_{[\mathcal{Q}^{\flat}]}\bm{c}}\mathbb{Z}[v^{\pm 1}]\widetilde{F_{-1}^{\mathrm{up}}}(\bm{c}',[\mathcal{Q}^{\flat}]). 
\end{align*}
Therefore, 
\begin{align*}
&\overline{\widetilde{\Phi}^T(\widetilde{F_{-1}^{\mathrm{up}}}(\bm{c},[\mathcal{Q}^{\flat}]))}=\widetilde{\Phi}^T(\sigma'(\widetilde{F_{-1}^{\mathrm{up}}}(\bm{c},[\mathcal{Q}^{\flat}])))\in t^{\beta(\bm{c})/2}E_t(m_{\mathrm{B}}(\bm{c}))^T+\sum_{m'\neq m_{\mathrm{B}}(\bm{c})}\mathbb{Z}[t^{\pm 1/2}]E_t(m')^T\\
&\overline{\widetilde{\Phi}^T(\widetilde{F_{-1}^{\mathrm{up}}}(\bm{c},[\mathcal{Q}^{\flat}]))}=
\overline{t^{\beta(\bm{c})/2}E_t(m_{\mathrm{B}}(\bm{c}))^T}
\in t^{-\beta(\bm{c})/2}E_t(m_{\mathrm{B}}(\bm{c}))^T+\sum_{m'\neq m_{\mathrm{B}}(\bm{c})}\mathbb{Z}[t^{\pm 1/2}]E_t(m')^T. 
\end{align*}
Hence, $\beta(\bm{c})/2=-\beta(\bm{c})/2$, which implies $\beta(\bm{c})=0$. 
\end{proof}

The following is an immediate consequence of Theorem \ref{t:mainisom} obtained from Proposition \ref{p:trunc}.
\begin{corollary}\label{c:mainisom}
There exists a $\mathbb{Z}$-algebra isomorphism $\widetilde{\Phi}\colon\mathcal{A}_v[N_-^{\mathrm{A}_{2n-1}}]\to \mathcal{K}_{t, \mathcal{Q}^{\flat}}$ such that $(\cdot)^{T}\circ \widetilde{\Phi}=\widetilde{\Phi}^T$. This isomorphism induces a bijection between $\widetilde{\mathbf{B}}^{\mathrm{up}}$ and $\widetilde{\mathbf{S}}^{\xi, \flat}:=\{L_t(m)\mid m\in \mathbb{B}^{\xi, \flat}\}$. 
\end{corollary}
\begin{remark}\label{r:qclustr}
Combining Corollary \ref{c:mainisom} with Theorem \ref{t:qcluster}, we can immediately conclude that the quantum Grothendieck ring $\mathcal{K}_{t, \mathcal{Q}^{\flat}}$ has a quantum cluster algebra structure. It follows from the observation in Remark \ref{r:HLquiver} that the initial matrix $\widetilde{B}^{[\mathcal{Q}^{\flat}]}$ of this quantum cluster algebra is a submatrix of the infinite matrix introduced in \cite{HL:JEMS2016} for type $\mathrm{B}_n^{(1)}$. 
\end{remark}
\section{{Corollaries of the isomorphism $\widetilde{\Phi}$}}\label{s:cor}
In this section, we provide applications of the isomorphism $\widetilde{\Phi}$ which is obtained in section \ref{s:mainisom}. In subsection \ref{ss:cor}, we prove several corollaries of the isomorphism $\widetilde{\Phi}$, including the  positivity properties (P1)--(P3) in the introduction and Conjecture \ref{co:qt} for $\mathcal{C}_{\mathcal{Q}^{\flat}}$, in particular the conjectural formula \eqref{eq:fconj} for the multiplicity of simple modules in standard modules. Once Conjecture \ref{co:qt} for $\mathcal{C}_{\mathcal{Q}^{\flat}}$ is proved, we can efficiently use its quantum Grothendieck ring to investigate  $\mathcal{C}_{\mathcal{Q}^{\flat}}$. In subsection \ref{ss:ex}, we provide an example of such calculation. 

\subsection{Corollaries of the isomorphism $\widetilde{\Phi}$}\label{ss:cor}
\begin{corollary}[{(P1) : Positivity of the coefficients of Kazhdan-Lusztig type polynomials}]\label{c:posKL}
Recall the notation in Theorem \ref{t:mainisom} and Corollary \ref{c:mainisom}. We have the following :
\begin{itemize}
\item[(1)] $\widetilde{\Phi}(\widetilde{F_{-1}^{\mathrm{up}}}(\bm{c},[\mathcal{Q}^{\flat}]))=E_t(m_{\mathrm{B}}(\bm{c}))$ for all $\bm{c}\in \mathbb{Z}_{\geq 0}^{\overline{I}_{\xi}^{\mathrm{tw}, \flat}}$.
\item[(2)] For $\bm{c}\in \mathbb{Z}_{\geq 0}^{\overline{I}_{\xi}^{\mathrm{tw}, \flat}}$ and $m\in \mathbb{B}^{\xi, \flat}$, write 
\begin{align*}
\widetilde{F_{-1}^{\mathrm{up}}}(\bm{c},[\mathcal{Q}^{\flat}])=\sum _{\bm{c}'}p_{\bm{c}, \bm{c}'}(v)\widetilde{\Gup}(b_{-1}(\bm{c}',[\mathcal{Q}^{\flat}]))&&&E_t(m)=\sum_{m'}P_{m, m'}(t)L_t(m').
\end{align*}
with $p_{\bm{c}, \bm{c}'}(v)\in \mathbb{Z}[v]$ and $P_{m, m'}(t)\in \mathbb{Z}[t^{-1}]$. Then we have
\[
p_{\bm{c}, \bm{c}'}(t^{-1})=P_{m_{\mathrm{B}}(\bm{c}), m_{\mathrm{B}}(\bm{c}')}(t).
\]
Moreover, $P_{m_{\mathrm{B}}(\bm{c}), m_{\mathrm{B}}(\bm{c}')}(t)\in \mathbb{Z}_{\geq 0}[t^{-1}]$.
\end{itemize}
\end{corollary}
\begin{proof}
The assertion (1) immediately follows from \eqref{eq:PBWcorresp}. The first statement of (2) is a consequence of Theorem \ref{t:mainisom}, Corollary \ref{t:mainisom} and (1). The second statement in (2) follows from the known positivity of $p_{\bm{c}, \bm{c}'}(v)$, which was proved by Kato \cite[Theorem 4.17]{Kato:KLR} and McNamara \cite[Theorem 3.1]{McN:finite}, through the categorification of dual PBW-type bases by using the quiver-Hecke algebras. 
\end{proof}
\begin{remark}
The positivity of $p_{\bm{c}, \bm{c}'}(v)$ for type $\mathrm{A}_N$, $\mathrm{D}_N$ and $\mathrm{E}_N$ was originally proved by Lusztig in his original paper of canonical bases \cite[Corollary 10.7]{Lus:can1} for \emph{adapted} reduced words of $w_0$, through his geometric realization of the elements of the canonical bases and PBW-type bases. However every reduced word in $[\mathcal{Q}^{\flat}]$ is non-adapted (Remark \ref{r:adapted}). In \cite[subsection 5.1]{Oya:qcoord}, the second author provides a method for deducing this positivity for \emph{arbitrary} reduced words of $w_0$ (when $\mathfrak{g}$ is of type $\mathrm{A}_N$, $\mathrm{D}_N$ and $\mathrm{E}_N$) from Lusztig's results. 
\end{remark}
\begin{corollary}[(P2) : Positivity of structure constants]\label{c:posstr}
For $m_1, m_2\in \mathbb{B}^{\xi, \flat}$, write 
\begin{align*}
L_t(m_1)L_t(m_2)
=\sum_{m\in\mathbb{B}^{\xi, \flat}}c_{m_1, m_2}^{m}L_t(m).
\end{align*}
Then we have
\[
c_{m_1, m_2}^{m}\in \mathbb{Z}_{\geq 0}[t^{\pm 1/2}].
\]
\end{corollary}
\begin{proof}
Through $\widetilde{\Phi}$, this assertion follows from the corresponding positivity of structure constants of (dual) canonical bases proved in \cite[Theorem 11.5]{Lus:quiper}.
\end{proof}
\begin{corollary}[(P3) : Positivity of the coefficients of the truncated $(q, t)$-characters]\label{c:poscoeff}
For $m\in \mathbb{B}^{\xi, \flat}$, write 
\begin{align*}
L_t(m)^T=\sum_{m':\text{monomial in }\mathcal{Y}}a[m; m']\underline{m'}
\end{align*}
with $a[m; m']\in \mathbb{Z}[t^{\pm 1/2}]$. Then, 
\[
a[m; m']\in \mathbb{Z}_{\geq 0}[t^{\pm 1/2}].
\]
\end{corollary}
\begin{proof}
By the definition of $\widetilde{\Phi}^T$ in Theorem \ref{t:torusisom}, we have 
\[
\widetilde{\Phi}^T(\widetilde{D}((i, r+2r_i), 0)^{-1}\widetilde{D}((i, r), 0))=t^{\alpha(i, r)/2}\widetilde{Y}_{i, r}
\]
for some $\alpha(i, r)\in \mathbb{Z}$.  Therefore, via $(\widetilde{\Phi}^T)^{-1}$, the truncated $(q, t)$-characters of simple modules deduce the following expression of the elements of the normalized dual canonical basis :
\begin{align}
\widetilde{\Gup}(b)=\sum_{\bm{v}=(v_{(i, r)})\in \mathbb{Z}^{\overline{I}_{\xi}^{\mathrm{tw}, \flat}}}\alpha[b; \bm{v}]\dprod_{(i, r)\in \overline{I}_{\xi}^{\mathrm{tw}, \flat}}\widetilde{D}((i, r), 0)^{v_{(i, r)}}\label{eq:exp}
\end{align}
for some $\alpha[b; \bm{v}]\in \mathbb{Z}[v^{\pm 1/2}]$ by Theorem \ref{t:mainisom}, here we consider an arbitrary total ordering on $\overline{I}_{\xi}^{\mathrm{tw}, \flat}$. 
It suffices to prove that $\alpha[b; \bm{v}]\in \mathbb{Z}_{\geq 0}[v^{\pm 1/2}]$ for every $b$ and $\bm{v}$. 

Multiplying an element $\mathcal{D}:=\overrightarrow{\prod}_{(i, r)\in \overline{I}_{\xi}^{\mathrm{tw}, \flat}}\widetilde{D}((i, r), 0)^{w_{(i, r)}}$ such that $w_{(i, r)}, w_{(i, r)}+v_{(i, r)} \in \mathbb{Z}_{\geq 0}$  for all $(i, r)\in \overline{I}_{\xi}^{\mathrm{tw}, \flat}$ by the both-hand sides of \eqref{eq:exp}, we obtain 
\begin{align}
\mathcal{D}\widetilde{\Gup}(b)=\sum_{\bm{v}=(v_{(i, r)})\in \mathbb{Z}^{\overline{I}_{\xi}^{\mathrm{tw}, \flat}}}\alpha'[b; \bm{v}]\dprod_{(i, r)\in \overline{I}_{\xi}^{\mathrm{tw}, \flat}}\widetilde{D}((i, r), 0)^{w_{(i, r)}+v_{(i, r)}}\label{eq:exp'}
\end{align}
for some $\alpha'[b; \bm{v}]\in \mathbb{Z}[v^{\pm 1/2}]$. It is known that any products of the elements $\widetilde{D}((i, r), 0)$, $(i, r)\in \overline{I}_{\xi}^{\mathrm{tw}, \flat}$ belong to $v^{\mathbb{Z}/2}\widetilde{\mathbf{B}}^{\mathrm{up}}$ \cite[Theorem 6.26]{Kimura:qunip}. Therefore, $\mathcal{D}$ and $\overrightarrow{\prod}_{(i, r)\in \overline{I}_{\xi}^{\mathrm{tw}, \flat}}\widetilde{D}((i, r), 0)^{w_{(i, r)}+v_{(i, r)}}$ belong to $v^{\mathbb{Z}/2}\widetilde{\mathbf{B}}^{\mathrm{up}}$. Hence the equality \eqref{eq:exp'} is the expansion of $\mathcal{D}\widetilde{\Gup}(b_{-1}(\bm{c},\bm{i}_{0}))$ with respect to the normalized dual canonical basis. By the positivity of the structure constants of dual canonical bases proved in \cite[Theorem 11.5]{Lus:quiper}, we obtain $\alpha'[b; \bm{v}]\in \mathbb{Z}_{\geq 0}[v^{\pm 1/2}]$ for every $b$ and $\bm{v}$, which is equivalent to $\alpha[b; \bm{v}]\in \mathbb{Z}_{\geq 0}[v^{\pm 1/2}]$. 
\end{proof}

\begin{corollary}[The $(q, t)$-characters of Kirillov-Reshetikhin modules]\label{c:qtKR}
For a Kirillov-Reshetikhin module $W^{(i)}_{k, r}$ in $\mathcal{C}_{\mathcal{Q}^{\flat}}$, we have
\[
L_t(m^{(i)}_{k, r})=[\underline{W^{(i)}_{k, r}}].  
\]
In particular, $\mathrm{ev}_{t=1}(L_t(m^{(i)}_{k, r}))=[W^{(i)}_{k, r}]$. 
\end{corollary}
\begin{proof}
By \eqref{eq:corresp}, $F_t(m^{(i)}_{k, r})$ corresponds to a certain normalized unipotent quantum minor via $\widetilde{\Phi}$, in particular, an element of $\widetilde{\mathbf{B}}^{\mathrm{up}}$. Therefore, by Corollary \ref{c:mainisom},  $F_t(m^{(i)}_{k, r})=L_t(m)$ for some $m\in \mathbb{B}^{\xi, \flat}$. The maximum monomial in $F_t(m^{(i)}_{k, r})$ is $\underline{m^{(i)}_{k, r}}$, and that in $L_t(m)$ is $\underline{m}$. Hence $m=m^{(i)}_{k, r}$. Therefore, by Proposition \ref{p:F-KR}, 
\[
L_t(m^{(i)}_{k, r})=F_t(m^{(i)}_{k, r})=[\underline{W^{(i)}_{k, r}}].
\] 
\end{proof}

The following remarkable theorem, due to Kashiwara and Oh, is proved by means of the celebrated \emph{generalized quantum affine Schur-Weyl dualities}, which is developed by Kang, Kashiwara, Kim and Oh \cite{KKK:1,KKK:2,KKKO:3,KKKO:4} (see also \cite{KKKO:Simp}) : 

\begin{theorem}[{\cite[Corollary 6.6, Corollary 6.7]{KO}}]\label{t:KO}
A $\mathbb{Z}$-algebra isomorphism 
\[
[\mathscr{F}]\colon \mathrm{ev}_{v=1}(\mathcal{A}_v[N_-^{\mathrm{A}_{2n-1}}])\to K(\mathcal{C}_{\mathcal{Q}^{\flat}})
\]
given by 
\[
\mathrm{ev}_{v=1}(\widetilde{F_{-1}^{\mathrm{up}}}(\bm{e}_{(i, r)},[\mathcal{Q}^{\flat}]))\mapsto [L(Y_{i, r})], 
\]
for $(i, r)\in \overline{I}_{\xi}^{\mathrm{tw}, \flat}$ (see just before \eqref{eq:rootcprresp} for the definition of $\bm{e}_{(i, r)}$) maps the dual canonical basis $\mathrm{ev}_{v=1}(\widetilde{\mathbf{B}}^{\mathrm{up}})$ specialized at $v=1$ to the set of classes of simple modules $\{[L(m)]\mid m\in \mathbb{B}^{\xi, \flat}\}$ .
\end{theorem}
\begin{remark}
Kashiwara and Oh also deal with some monoidal subcategories of $\mathcal{C}_{\bullet}$ for type $\mathrm{C}_n^{(1)}$ in \cite{KO} (in this case, the corresponding quantized coordinate algebra is of type $\mathrm{D}_{n+1}$). Their isomorphisms are obtained from generalized quantum affine Schur-Weyl duality functors which are conjectually equivalences of categories \cite[Conjecture 6.10]{KO}. See also Remark \ref{r:convention}. 
\end{remark}
We already knew that $\mathrm{ev}_{t=1}(F_t(Y_{i, r}))=\chi_q(L(Y_{i, r}))$. Therefore, by Theorem \ref{t:KO}, we have the following : 
\begin{corollary}\label{c:KO}
\begin{itemize}
\item[(1)] Let $\widetilde{\Phi}\mid_{v=t=1}$ be the $\mathbb{Z}$-algebra isomorphism $\mathrm{ev}_{v=1}(\mathcal{A}_v[N_-^{\mathrm{A}_{2n-1}}])\to \mathrm{ev}_{t=1}(\mathcal{K}_{t, \mathcal{Q}^{\flat}})\simeq K(\mathcal{C}_{\mathcal{Q}^{\flat}})$ induced from $\widetilde{\Phi}$. Then we have $\widetilde{\Phi}\mid_{v=t=1}=[\mathscr{F}]$. 
\item[(2)] $\chi_q(L(m))=\mathrm{ev}_{t=1}(L_t(m))$ for all $m\in \mathbb{B}^{\xi, \flat}$. 
\end{itemize}
\end{corollary}
Corollary \ref{c:KO} (2) gives the affirmative answer to Conjecture \ref{co:qt} for $\mathcal{C}_{\mathcal{Q}^{\flat}}$. Therefore, by Remark \ref{r:KLalg}, the multiplicity of the simple module $L(m')$ in the standard module $M(m)$ for $m, m'\in \mathbb{B}^{\xi, \flat}$ can be computed by  the Kazhdan-Lusztig algorithm in the quantum Grothendieck ring $\mathcal{K}_{t, \mathcal{Q}^{\flat}}$. 

\subsection{An example of calculation in quantum Grothendieck rings}\label{ss:ex}
By Corollary \ref{c:KO}, we can use the quantum Grothendieck ring $\mathcal{K}_{t, \mathcal{Q}^{\flat}}$ to obtain certain results on $\mathcal{C}_{\mathcal{Q}^{\flat}}$. Let us give an explicit example of such computations. 
Consider the monomial 
\[
m =Y_{n,r}Y_{n,r + 2p}
\]
where $p\geq 0$ in $\mathcal{Y}$ of type $\mathrm{B}_n^{(1)}$. We calculate the following : 

\begin{proposition}\label{p:length2} The module $L(m)$ is affine-minuscule (see just after Theorem \ref{t:minuscule} for the definition). 
Moreover, the standard module 
\[
M(m) = L(Y_{n, r})\otimes L(Y_{n, r+2p})
\] 
has length at most $2$ :

If $p$ is odd and $p\leq 2n - 1$,  then we have 
\[
[M(m)] = [L(m)] + [L(m')] 
\]
in $K(\mathcal{C}_{\bullet})$ (in particular, $M(m)$ is of length $2$) with 
\[
m' = Y_{n - (p+1)/2, r + p}.
\]
Otherwise $M(m)$ is simple and $L_t(m) = E_t(m)$.
\end{proposition}

\begin{remark}
In particular cases of the above proposition, the socle (or head) of $M(m)$ was computed in \cite[Theorem 8.1]{CP:Dorey}. 
\end{remark}
\begin{proof} First the $q$-character of $L(Y_{n,r})$ is explicitly known (see references in \cite{H:minim}) and is equal to
\begin{equation}\label{exp}\sum_{0\leq k\leq n, n > i_1 > \cdots > i_k\geq 0}
(A_{n,r+1}A_{n-1,r+3}\cdots A_{n-i_1,r + 1 + 2i_1})^{-1} (A_{n,r+5}A_{n-1,r+7}\cdots A_{n - i_2,r + 5 + 2i_2})^{-1}$$
$$\cdots (A_{n,r + 4k -3}\cdots A_{n-i_k,r +4k - 3 + 2i_k})^{-1}.\end{equation}
This can be checked directly as this formula has a unique dominant monomial $Y_{n,r}$ and for each $1\leq j\leq n$ it has a decomposition as in \cite[Proposition 2.15]{H:minim}.

Now consider a dominant monomial $M$ in $\chi_q(L(Y_{n,r}))\chi_q(L(Y_{n,r + 2p}))$. Then it is of the form 
$$M = M_1 Y_{n,r+2p}$$ 
where $M_1$ is a monomial occurring in $\chi_q(L(Y_{n,r}))$ 
(otherwise $M$ would be right-negative in the sense of \cite{FM} and would not be dominant). 
Hence $M_1 = 1$ or  $M_1$ is of the form described in Equation (\ref{exp}) 
with $i_k = 0, i_{k-1} = 1, i_{k-2} = 2,\cdots, i_1 = k-1$ and $r + 2p = r + 4k - 3 + 2i_k + 1$ that is
$p = 2k - 1$. In this case $M_1 = Y_{n, r + 4k - 2}^{-1}Y_{n - k,r +2k -1}$ and $M = Y_{n - (p+1)/2, r + p}$. 

Suppose that $p$ is even or $p > 2n -1$ : $m$ is the unique dominant monomial in $\chi_q(M(m))$. 
Then $M(m)$ is special and so is simple. Hence we obtain the desired result. 

Suppose that $p$ is odd and $p\leq 2n - 1$ :  $\chi_q(M(m))$ has two dominant monomials $m$ and $m'$. A direct computation in the quantum torus gives that the respective multiplicities of $\underline{m}$ and $\underline{m}'$ in $L_t(Y_{n, r})L_t(Y_{n, r+2p})$ are $t^{-\alpha}$ and $t^{-\alpha - 1}$ for a certain $\alpha$, here note that, by Proposition \ref{p:F-KR} (or Corollary \ref{c:qtKR}),   
\begin{align}
L_t(Y_{i, r})=[\underline{L(Y_{i, r})}](=F_t(Y_{i, r})) \label{eq:fundamental}
\end{align}
for any $i\in I$. Hence the respective multiplicities of $\underline{m}$ and $\underline{m}'$ in $E_t(m)$ are $1$ and $t^{- 1}$. By Theorem \ref{t:normqt}, 
\[
E_t(m)= L_t(m)+ t^{-1}P'_{m,m'}(t)L_t(m')
\]
for some $P'_{m,m'}(t)\in \mathbb{Z}[t^{-1}]$. Let $P_1(t)$ be the multiplicity of $\underline{m}'$ in $L_t(m)$. Then, by the above argument, we have 
\[
t^{-1}=P_1(t)+t^{-1}P'_{m,m'}(t). 
\]
Since $L_t(m)$ is invariant by the bar involution, we have $P_1(t^{-1})=P_1(t)$, hence $P_1(t)=0$ and $P'_{m,m'}(t)=1$. Therefore, 
\[
E_t(m)= L_t(m)+ t^{-1}L_t(m'). 
\]
Hence, by Corollary \ref{c:KO} (2), we obtain 
\[
[M(m)]=[L(m)]+[L(m')]
\]
and $L(m)$ is affine-minuscule. 
\end{proof}
\begin{remark}
We can check that this result is ``correct'' : let us prove Proposition \ref{p:length2} in a more ``representation-theoretic'' way without using the $(q,t)$-characters and the theorems in this paper. Suppose that $p$ is odd and $p\leq 2n - 1$ (this case is the only case where we used $(q,t)$-characters in the proof). It follows from  \cite[Theorem 5.1]{H:small}\footnote{There is a typo in condition (iv) of Theorem 5.1 of that paper : $m''\neq M$} that $m'$ does not
occur in $\chi_q(L(m))$. Hence $M(m)$ has two simple special constituents $L(m)$ and $L(m')$ :
\[
[M(m)] = [L(m)] + [L(m')].
\]
\end{remark}

\section{Comparison between type A and B}\label{s:AB}
In this section, we combine our results with the results of \cite{HL:qGro}. In \cite{HL:qGro}, it is shown that the category $\mathcal{C}_{\mathcal{Q}}$ for type $\mathrm{A}_{2n-1}^{(1)}$ (recall subsection \ref{ss:HLsubcat}) is also isomorphic to the quantized coordinate algebra of $\mathrm{A}_{2n-1}$. We recall the precise statement in subsection \ref{ss:HLisom}. Combining this result with Corollary \ref{c:mainisom}, we obtain an isomorphism between quantum Grothendieck rings of distinguished monoidal subcategories of $\mathcal{C}_{\bullet}$ of type $\mathrm{A}_{2n-1}^{(1)}$ and $\mathrm{B}_n^{(1)}$. In subsection \ref{ss:explicit}, we give an explicit correspondence of the $(q, t)$-characters of simple modules under this isomorphism for special quivers (Theorem \ref{t:explicitcorresp}). This is computed from the well-known formulas for the change of the Lusztig parametrizations (Definition \ref{d:PBWparam}) associated with a change of reduced words.

\subsection{Quantum Grothendieck ring isomorphisms between type $\mathrm{A}$ and $\mathrm{B}$}\label{ss:HLisom}
Let $\mathcal{Q}'_{2n-1}$ be a Dynkin quiver of type $\mathrm{A}_{2n-1}$, and $\xi'$ an associated height function. Write $I_{\mathrm{A}}:=\{1, \dots, 2n-1\}$ and denote by $\mathcal{Y}^{\mathrm{A}}$ (resp.~$\mathcal{K}^{\mathrm{A}}_t$, $L^{\mathrm{A}}(m)$, $E_t^{\mathrm{A}}(m)$ and  $L_t^{\mathrm{A}}(m)$) the Laurent polynomial algebra $\mathcal{Y}$ (resp.~the quantum Grothendieck ring $\mathcal{K}_t$, the simple module $L(m)$ in $\mathcal{C_{\bullet}}$, the $(q, t)$-character of $M(m)$ and $L(m)$) for type $\mathrm{A}_{2n-1}^{(1)}$. Recall the labelling $(\widehat{I}_{\mathrm{A}})_{\xi'}$ of the positive roots $\Delta_+$ of type $\mathrm{A}_{2n-1}$ defined just after Proposition \ref{p:ARlabel}. Recall $\mathbb{B}^{\xi'}\subset\mathcal{Y}^{\mathrm{A}}$ in Definition \ref{d:HLsubcat}. 

Set 
\[
\mathcal{K}_{t, \mathcal{Q}'_{2n-1}}(=K_t(\mathcal{C}_{\mathcal{Q}'_{2n-1}})):=\bigoplus_{m\in \mathbb{B}^{\xi'}}\mathbb{Z}[t^{\pm 1/2}]E_t^{\mathrm{A}}(m)=\bigoplus_{m\in \mathbb{B}^{\xi'}}\mathbb{Z}[t^{\pm 1/2}]L_t^{\mathrm{A}}(m)\subset \mathcal{K}^{\mathrm{A}}_t. 
\]
This is a subalgebra of $\mathcal{K}^{\mathrm{A}}_t$ (cf.~subsection \ref{ss:qGroB}), called \emph{the quantum Grothendieck ring of $\mathcal{C}_{\mathcal{Q}'_{2n-1}}$}. The first author and Leclerc proved the following : 
\begin{theorem}[{\cite[Theorem 1.2, Theorem 6.1]{HL:qGro}}]\label{t:isomtypeA}
We identify $\Delta_+$ with its lebelling set $(\widehat{I}_{\mathrm{A}})_{\xi'}$. Then there exists a $\mathbb{Z}$-algebra isomorphism
\[
\widetilde{\Phi}_{\mathrm{A}}\colon\mathcal{A}_{v}[N_-^{\mathrm{A}_{2n-1}}]\to \mathcal{K}_{t, \mathcal{Q}'_{2n-1}}
\]
given by 
\begin{align*}
v^{\pm 1/2}\mapsto t^{\mp 1/2}&&& \widetilde{D}^{[\mathcal{Q}'_{2n-1}]}((i, r), (i, r+2s))\mapsto L_t^{\mathrm{A}}(m_{s, r}^{(i)})
\end{align*}
for $(i, r)\in (\widehat{I}_{\mathrm{A}})_{\xi'}$ (we adopt the convention concerning $0$ similar to \eqref{eq:corresp}). Moreover, this isomorphism induces a bijection between $\widetilde{\mathbf{B}}^{\mathrm{up}}$ and $\widetilde{\mathbf{S}}^{\xi'}:=\{L_t^{\mathrm{A}}(m)\mid m\in \mathbb{B}^{\xi'}\}$. More precisely, if $\widetilde{\Phi}_{\mathrm{A}}(\widetilde{\Gup}(b_{-1}(\bm{c},[\mathcal{Q}'_{2n-1}]))=L_t^{\mathrm{A}}(m)$, then 
\[
c_{(i, r)}=u_{i, r}(m)
\]
for all $(i, r)\in (\widehat{I}_{\mathrm{A}})_{\xi'}$. This monomial $m$ will be denoted by $m_{\mathrm{A}}(\bm{c})$. 
\end{theorem}
Let $\mathcal{Q}_{2n-2}$ be a Dynkin quiver of type $\mathrm{A}_{2n-2}$ and $\xi$ an associated height function. Denote by $\widetilde{\Phi}_{\mathrm{B}}$ (resp.~$L_t^{\mathrm{B}}(m)$) the isomorphism $\widetilde{\Phi}$ in Corollary \ref{c:mainisom} for $\mathcal{C}_{\mathcal{Q}_{2n-2}^{\flat}}$ (resp.~the $(q, t)$-character of $L_t(m)$ for type $\mathrm{B}_{n}^{(1)}$). 

By Corollary \ref{c:mainisom} and Theorem \ref{t:isomtypeA}, we obtain the following : 
\begin{theorem}\label{t:maincorresp}
There exists a $\mathbb{Z}[t^{\pm 1/2}]$-algebra isomorphism $\widetilde{\Phi}_{\mathrm{A}\to \mathrm{B}}\colon \mathcal{K}_{t, \mathcal{Q}'_{2n-1}}\to \mathcal{K}_{t, \mathcal{Q}_{2n-2}^{\flat}}$ given by $\widetilde{\Phi}_{\mathrm{B}}\circ \widetilde{\Phi}_{\mathrm{A}}^{-1}$. Moreover, this restricts to a bijection from $\widetilde{\mathbf{S}}^{\xi'}$ to $\widetilde{\mathbf{S}}^{\xi, \flat}$. 
\end{theorem}
\begin{remark}\label{r:stdPBW}
The quiver $\mathcal{Q}_{2n-2}$ is irrelevant to $\mathcal{Q}'_{2n-1}$. We should also remark that the $(q, t)$-characters of standard modules in $\mathcal{C}_{\mathcal{Q}'_{2n-1}}$ are not necessarily sent to those in $\mathcal{C}_{\mathcal{Q}_{2n-2}^{\flat}}$ by $\widetilde{\Phi}_{\mathrm{A}\to \mathrm{B}}$. From the viewpoint of quantized coordinate algebras, this observation corresponds to the fact that there are several choices of normalized dual PBW-type bases depending on the choice of reduced words of the longest element $w_0$, unlike the dual canonical basis. In other words, we can deal with two irrelevant objects, the $(q, t)$-characters of standard modules in $\mathcal{C}_{\mathcal{Q}'_{2n-1}}$ and $\mathcal{C}_{\mathcal{Q}_{2n-2}^{\flat}}$, in the same framework of normalized dual PBW-type bases by considering the quantized coordinate algebra $\mathcal{A}_{v}[N_-^{\mathrm{A}_{2n-1}}]$.
\end{remark}
The inverse of $\widetilde{\Phi}_{\mathrm{A}\to \mathrm{B}}$ is denoted by $\widetilde{\Phi}_{\mathrm{B}\to \mathrm{A}}$. 

\subsection{The explicit description of $\widetilde{\Phi}_{\mathrm{A}\to \mathrm{B}}$ and $\widetilde{\Phi}_{\mathrm{B}\to \mathrm{A}}$}\label{ss:explicit}
In this subsection, we provide a more explicit description of $\widetilde{\Phi}_{\mathrm{A}\to \mathrm{B}}$ and $\widetilde{\Phi}_{\mathrm{B}\to \mathrm{A}}$ under the following assumptions :
\begin{itemize}
\item  $\mathcal{Q}_{2n-2}(=:\overleftarrow{\mathcal{Q}})$ is a quiver 

\hfill
\begin{xy} 0;<1pt,0pt>:<0pt,-1pt>:: 
(210,0) *+{{2n-2}} ="5",
(160,0) *+{ } ="4",
(80,0) *+{ } ="3",
(40,0) *+{2} ="2",
(0,0) *+{1} ="1",
"5", {\ar"4"},
"4", {\ar@{--}"3"},
"3", {\ar"2"},
"2", {\ar"1"},
\end{xy}
\hfill\hfill

endowed with the associated height function $\xi_0$ such that $\xi_0(1)=0$.

\item $\mathcal{Q}'_{2n-1}(=:\overleftarrow{\mathcal{Q}}')$ is a quiver 

\hfill
\begin{xy} 0;<1pt,0pt>:<0pt,-1pt>:: 
(280,0) *+{{2n-1}} ="6",
(210,0) *+{{2n-2}} ="5",
(160,0) *+{ } ="4",
(80,0) *+{ } ="3",
(40,0) *+{2} ="2",
(0,0) *+{1} ="1",
"6", {\ar"5"},
"5", {\ar"4"},
"4", {\ar@{--}"3"},
"3", {\ar"2"},
"2", {\ar"1"},
\end{xy}
\hfill\hfill

endowed with the associated height function such that $\xi'_0(1)=0$.

\item $\flat=<$. 
\end{itemize}

In this section, we write $J_{\mathrm{B}}:=\widehat{I}_{\xi_0}^{\mathrm{tw}, <}$ and $J_{\mathrm{A}}:=(\widehat{I}_{\mathrm{A}})_{\xi'_0}$ for short. 

Let $\bm{c}_0, \bm{c}_1\in \mathbb{Z}_{\geq 0}^{\Delta_+}$. Suppose that 
\begin{align}
\widetilde{\Phi}_{\mathrm{A}\to \mathrm{B}}(L_t^{\mathrm{A}}(m_{\mathrm{A}}(\bm{c}_1)))=L_t^{\mathrm{B}}(m_{\mathrm{B}}(\bm{c}_0)).\label{eq:ABcorresp}
\end{align}
Then 
\[
\widetilde{\Gup}(b_{-1}(\bm{c}_1,[\overleftarrow{\mathcal{Q}}']))=\widetilde{\Gup}(b_{-1}(\bm{c}_0,[\overleftarrow{\mathcal{Q}}^{<}])).
\]
Hence, to describe the correspondence $\widetilde{\Phi}_{\mathrm{A}\to \mathrm{B}}$ and $\widetilde{\Phi}_{\mathrm{B}\to \mathrm{A}}$,  we only have to calculate the change of the Lusztig parametrizations associated with the change of the corresponding commutation classes. Recall the following  proposition : 
\begin{proposition}[{\cite[Chapter 42]{Lus:intro}}]\label{p:localchange}
Let $\bm{i}=(\imath_1, \imath_2,\dots,\imath_{\ell})$ and $\bm{i}'=(\imath'_1, \imath'_2,\dots,\imath'_{\ell})$ be reduced words of $w_0^{\mathrm{A}_{2n-1}}$. 
Suppose that, for some $k\in \{2, 3,\dots, {\ell}-1\}$ and $\imath, \jmath\in I_{\mathrm{A}}$ with $|\imath-\jmath|=1$, 
\begin{align}
\begin{cases}
(\imath_{k-1}, \imath_k, \imath_{k+1})=(\imath, \jmath, \imath),\\
(\imath'_{k-1}, \imath'_k, \imath'_{k+1})=(\jmath, \imath, \jmath),\\
\imath_s=\imath'_s\ \text{for all}\ s\ \text{with}\ s\neq k, k\pm 1,
\end{cases}\label{eq:3move}
\end{align}
Then $b_{-1}(\bm{c},[\bm{i}])=b_{-1}(\bm{c}',[\bm{i}'])$ if and only if
\[
\begin{cases}
c'_{k-1}=c_k+c_{k+1}-\min\{c_{k-1}, c_{k+1}\},\\
c'_k=\min\{c_{k-1}, c_{k+1}\},\\
c'_{k+1}=c_{k-1}+c_k-\min\{c_{k-1}, c_{k+1}\},\\
c'_s=c_s\ \text{for all}\ s\ \text{with}\ s\neq k, k\pm 1,
\end{cases}
\]
here $c_t:=c_{\beta_t^{\bm{i}}}$ and $c'_t:=c_{\beta_t^{\bm{i}'}}$ for $t=1,\dots, \ell$. 
\end{proposition}
When two commutation classes of reduced words of $w_0^{\mathrm{A}_{2n-1}}$ contain reduced words $\bm{i}$ and $\bm{i}'$, respectively, satisfying \eqref{eq:3move}, we say that these commutation classes are related via \emph{3-move}. If we find explicit 3-moves which connect $[\overleftarrow{\mathcal{Q}}']$ and $[\overleftarrow{\mathcal{Q}}^<]$, we can describe $\widetilde{\Phi}_{\mathrm{A}\to \mathrm{B}}$ and $\widetilde{\Phi}_{\mathrm{B}\to \mathrm{A}}$ explicitly by Proposition \ref{p:localchange}. From now on, we calculate such 3-moves. 

Let 
\[
\widetilde{J}:=I_{\mathrm{A}}\times \left( [0, -2(2n-3)]\cap \frac{1}{2}\mathbb{Z}\right). 
\]

Then $J_{\mathrm{B}}$ is considered as a subset of $\widetilde{J}$. More explicitly, $J_{\mathrm{B}}$ is the following set : 
\begin{align*}
&\{(\imath, -\imath +2-2k)\mid k=0,\dots, 2n-1-\imath\ \text{and}\  \imath=n+1,\dots, 2n-1\}\\
&\cup\{(n, -n+\frac{3}{2}-k)\mid k=0, \dots, 2n-2\}\\
&\cup\{(\imath, -\imath+1-2k)\mid k=0,\dots, 2n-2-\imath\ \text{and}\  \imath=1,\dots, n-1\}.
\end{align*}
For example, the subset $J_{\mathrm{B}}$ for $n=5$ is visualized as follows :

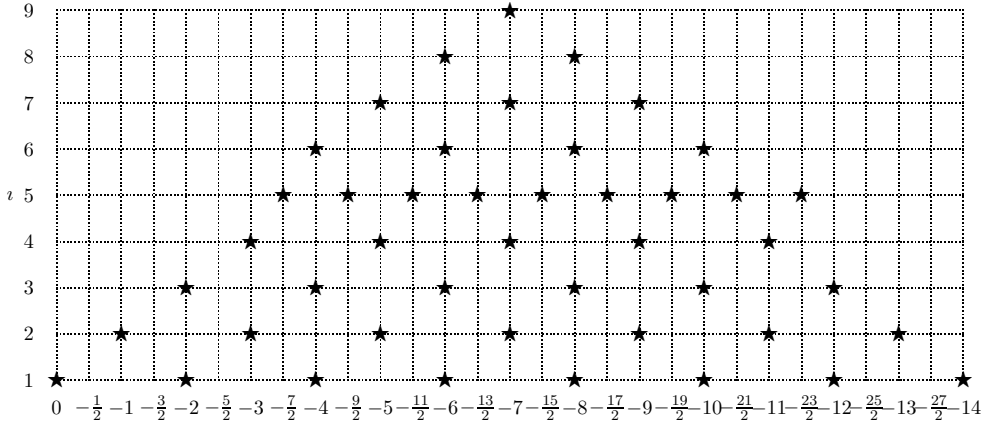
\hfill
\scalebox{0.7}{
\hspace{-50pt}
\begin{xy} 0;<1pt,0pt>:<0pt,-1pt>:: 
(245,0) *+{\bigstar},
(210,25) *+{\bigstar},
(280,25) *+{\bigstar},
(175,50) *+{\bigstar},
(245,50) *+{\bigstar},
(315,50) *+{\bigstar},
(140,75) *+{\bigstar},
(210,75) *+{\bigstar},
(280,75) *+{\bigstar},
(350,75) *+{\bigstar},
(122.5,100) *+{\bigstar},
(157.5,100) *+{\bigstar},
(192.5,100) *+{\bigstar},
(227.5,100) *+{\bigstar},
(262.5,100) *+{\bigstar},
(297.5,100) *+{\bigstar},
(332.5,100) *+{\bigstar},
(367.5,100) *+{\bigstar},
(402.5,100) *+{\bigstar},
(105,125) *+{\bigstar},
(175,125) *+{\bigstar},
(245,125) *+{\bigstar},
(315,125) *+{\bigstar},
(385,125) *+{\bigstar},
(70,150) *+{\bigstar},
(140,150) *+{\bigstar},
(210,150) *+{\bigstar},
(280,150) *+{\bigstar},
(350,150) *+{\bigstar},
(420,150) *+{\bigstar},
(35,175) *+{\bigstar},
(105,175) *+{\bigstar},
(175,175) *+{\bigstar},
(245,175) *+{\bigstar},
(315,175) *+{\bigstar},
(385,175) *+{\bigstar},
(455,175) *+{\bigstar},
(0,200) *+{\bigstar},
(70,200) *+{\bigstar},
(140,200) *+{\bigstar},
(210,200) *+{\bigstar},
(280,200) *+{\bigstar},
(350,200) *+{\bigstar},
(420,200) *+{\bigstar},
(490,200) *+{\bigstar},
(-25,100) *+{\imath},
(-15,200) *+{1},
(-15,175) *+{2},
(-15,150) *+{3},
(-15,125) *+{4},
(-15,100) *+{5},
(-15,75) *+{6},
(-15,50) *+{7},
(-15,25) *+{8},
(-15,0) *+{9},
(0,215) *+{0},
(17.5,215) *+{-\frac{1}{2}},
(35,215) *+{-1},
(52.5,215) *+{-\frac{3}{2}},
(70,215) *+{-2},
(87.5,215) *+{-\frac{5}{2}},
(105,215) *+{-3},
(122.5,215) *+{-\frac{7}{2}},
(140,215) *+{-4},
(157.5,215) *+{-\frac{9}{2}},
(175,215) *+{-5},
(192.5,215) *+{-\frac{11}{2}},
(210,215) *+{-6},
(227.5,215) *+{-\frac{13}{2}},
(245,215) *+{-7},
(262.5,215) *+{-\frac{15}{2}},
(280,215) *+{-8},
(297.5,215) *+{-\frac{17}{2}},
(315,215) *+{-9},
(332.5,215) *+{-\frac{19}{2}},
(350,215) *+{-10},
(367.5,215) *+{-\frac{21}{2}},
(385,215) *+{-11},
(402.5,215) *+{-\frac{23}{2}},
(420,215) *+{-12},
(437.5,215) *+{-\frac{25}{2}},
(455,215) *+{-13},
(472.5,215) *+{-\frac{27}{2}},
(490,215) *+{-14},
{\ar@{.} (0,200); (0,0)},
{\ar@{.} (35,200); (35,0)},
{\ar@{.} (70,200); (70,0)},
{\ar@{.} (105,200); (105,0)},
{\ar@{.} (122.5,200); (122.5,0)},
{\ar@{.} (140,200); (140,0)},
{\ar@{.} (157.5,200); (157.5,0)},
{\ar@{.} (175,200); (175,0)},
{\ar@{.} (192.5,200); (192.5,0)},
{\ar@{.} (210,200); (210,0)},
{\ar@{.} (227.5,200); (227.5,0)},
{\ar@{.} (245,200); (245,0)},
{\ar@{.} (262.5,200); (262.5,0)},
{\ar@{.} (280,200); (280,0)},
{\ar@{.} (297.5,200); (297.5,0)},
{\ar@{.} (315,200); (315,0)},
{\ar@{.} (332.5,200); (332.5,0)},
{\ar@{.} (350,200); (350,0)},
{\ar@{.} (367.5,200); (367.5,0)},
{\ar@{.} (385,200); (385,0)},
{\ar@{.} (402.5,200); (402.5,0)},
{\ar@{.} (420,200); (420,0)},
{\ar@{.} (455,200); (455,0)},
{\ar@{.} (490,200); (490,0)},
{\ar@{.} (17.5,200); (17.5,0)},
{\ar@{.} (52.5,200); (52.5,0)},
{\ar@{.} (87.5,200); (87.5,0)},
{\ar@{.} (437.5,200); (437.5,0)},
{\ar@{.} (472.5,200); (472.5,0)},
{\ar@{.} (0,0); (490,0)},
{\ar@{.} (0,25); (490,25)},
{\ar@{.} (0,50); (490,50)},
{\ar@{.} (0,75); (490,75)},
{\ar@{.} (0,100); (490,100)},
{\ar@{.} (0,125); (490,125)},
{\ar@{.} (0,150); (490,150)},
{\ar@{.} (0,175); (490,175)},
{\ar@{.} (0,200); (490,200)},
\end{xy}
}
\hfill
\hfill

Any subset $J'$ of $\widetilde{J}$ is represented as a function $\delta_{J'}\colon \widetilde{J}\to \{0, 1\}$ given by 
\[
\delta_{J}(\imath, r):=\begin{cases}
1&\text{if}\ (\imath, r)\in J',\\
0&\text{if}\ (\imath, r)\notin J'.
\end{cases}
\]

We call a map $\delta\colon \widetilde{J}\to \{0, 1\}$ \emph{well-arranged} if, for every fixed $r$, the simple reflections $\{s_{\imath}\mid \imath\in I_{\mathrm{A}}, \delta(\imath, r)=1\}$ mutually commute. 

For a well-arranged map $\delta\colon \widetilde{J}\to \{0, 1\}$, set 
\begin{align}
\bm{s}(\delta):=\dprod_{r\in \mathbb{Z}}\left(\prod_{\imath: \delta(\imath, -r)=1} s_{\imath}\right)\in W^{\mathrm{A}_{2n-1}}, \label{eq:def-sdelta}
\end{align}
and let 
\[
[\bm{i}(\delta)]:=[(\imath_1^{(0)},\dots, \imath_{k_0}^{(0)}, \imath_1^{(\frac{1}{2})},\dots, \imath_{k_{\frac{1}{2}}}^{(\frac{1}{2})}, \dots, \imath_1^{(-2(2n-3))},\dots, \imath_{k_{-2(2n-3)}}^{(-2(2n-3))})]
\]
such that 
\[
\{\imath_1^{(s)},\dots, \imath_{k_s}^{(s)}\}=\{\imath\mid \delta(\imath, s)=1\}. 
\]
Note that by the definition of well-arranged property the element $\bm{s}(\delta)$ and the commutation class $[\bm{i}(\delta)]$ is well-defined. We call $\delta$ \emph{reduced} if $[\bm{i}(\delta)]$ is a commutation class of a reduced word of $\bm{s}(\delta)$. 

Let $\delta_{\mathrm{B}}\colon \widetilde{J}\to \{0, 1\}$ be the function associated with $J_{\mathrm{B}}$. This map is well-arranged and reduced.

\begin{lemma}\label{l:braidrel}
Let $\imath_0\in I_{\mathrm{A}}\setminus\{1\}$, $r_0$ an integer with $0<r_0<-2(2n-3)$ and $\delta\colon\widetilde{J}\to \{0, 1\}$ a well-arranged map. Suppose that $\delta$ satisfies 
\begin{align*}
&\delta(\jmath, r_0+\frac{1}{2})=\delta(\jmath, r_0-\frac{1}{2})=\delta_{\imath_0\jmath}\ \text{for all}\ \jmath\in I_{\mathrm{A}},\\
&\delta(\imath_0-1, r_0)=1,\ \delta(\imath_0+1, r_0)=0\ (\delta(2n, r_0):=0).
\end{align*}
Set
\begin{align}
\delta'(\imath, r):=&\delta(\imath, r)-\delta_{(\imath, r), (\imath_0, r_0+\frac{1}{2})}-\delta_{(\imath, r), (\imath_0, r_0-\frac{1}{2})}-\delta_{(\imath, r), (\imath_0-1, r_0)}\\ \label{eq:braid}
&+\delta_{(\imath, r), (\imath_0-1, r_0+\frac{1}{2})}+\delta_{(\imath, r), (\imath_0-1, r_0-\frac{1}{2})}+\delta_{(\imath, r), (\imath_0, r_0)}\notag
\end{align}
for $(\imath, r)\in \widetilde{J}$, here we consider the usual addition and subtraction of integers $0, 1$. Then $\delta'$ is a map from $\widetilde{J}$ to $\{0, 1\}$, and it is well-arranged. Moreover, $s(\delta)=s(\delta')$, and $[\bm{i}(\delta)]$, $[\bm{i}(\delta)]$ are related via 3-move. In particular, $\delta$ is reduced if and only if $\delta'$ is reduced. 
\end{lemma}
\begin{remark}
The difference between the subsets corresponding to $\delta$ and $\delta'$ is visualized as follows : 

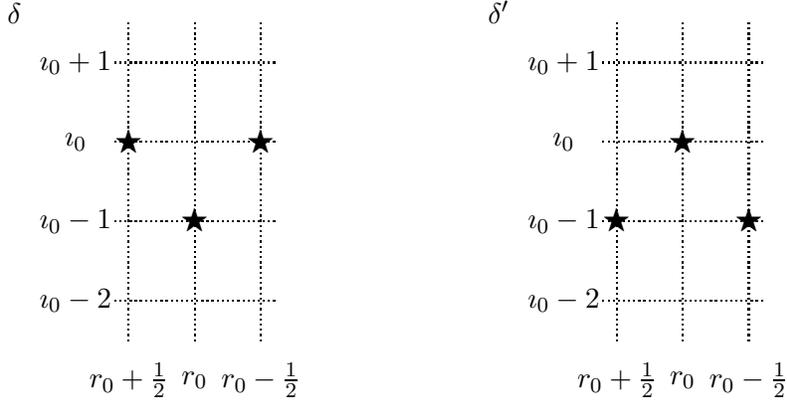
\hfill
$\delta$
\begin{xy}0;<1pt,0pt>:<0pt,-1pt>:: 
(0,15) *+{\imath_0+1},
(0,45) *+{\imath_0},
(0,75) *+{\imath_0-1},
(0,105) *+{\imath_0-2},
(20,135) *+{r_0+\frac{1}{2}},
(45,135) *+{r_0},
(70,135) *+{r_0-\frac{1}{2}},
(20,45) *+{\bigstar},
(70,45) *+{\bigstar},
(45,75) *+{\bigstar},
{\ar@{.} (20,120); (20,0)},
{\ar@{.} (45,120); (45,0)},
{\ar@{.} (70,120); (70,0)},
{\ar@{.} (15,15); (75,15)},
{\ar@{.} (15,45); (75,45)},
{\ar@{.} (15,75); (75,75)},
{\ar@{.} (15,105); (75,105)},
\end{xy}
\hfill
$\delta'$
\begin{xy}0;<1pt,0pt>:<0pt,-1pt>:: 
(0,15) *+{\imath_0+1},
(0,45) *+{\imath_0},
(0,75) *+{\imath_0-1},
(0,105) *+{\imath_0-2},
(20,135) *+{r_0+\frac{1}{2}},
(45,135) *+{r_0},
(70,135) *+{r_0-\frac{1}{2}},
(45,45) *+{\bigstar},
(20,75) *+{\bigstar},
(70,75) *+{\bigstar},
{\ar@{.} (20,120); (20,0)},
{\ar@{.} (45,120); (45,0)},
{\ar@{.} (70,120); (70,0)},
{\ar@{.} (15,15); (75,15)},
{\ar@{.} (15,45); (75,45)},
{\ar@{.} (15,75); (75,75)},
{\ar@{.} (15,105); (75,105)},
\end{xy}
\hfill\hfill

\end{remark}
\begin{proof}
Since $\delta$ is well-arranged, we have 
\[
\delta(\imath_0, r_0)=\delta(\imath_0-2, r_0)=0.
\]
Therefore the assertions follow from the braid relations in $W^{\mathrm{A}_{2n-1}}$. 
\end{proof}
The map $\delta'$ in Lemma \ref{l:braidrel} will be denoted by $\beta(\imath_0, r_0)(\delta)$. Set 
\begin{align*}
J_{\delta}:=\{(\imath, r)\mid \delta(\imath, r)=1\}&&&J_{\beta(\imath_0, r_0)(\delta)}:=\{(i, r)\mid \beta(\imath_0, r_0)(\delta)(\imath, r)=1\}.
\end{align*}
Assume that $\delta$ is reduced and $s(\delta)=w_0^{\mathrm{A}_{2n-1}}$. Then $J_{\delta}$ and $J_{\beta(i_0, r_0)(\delta)}$ are naturally considered as the vertex sets of the combinatorial Auslander-Reiten quivers $\Upsilon_{[\bm{i}(\delta)]}$ and $\Upsilon_{[\bm{i}(\beta(\imath_0, r_0)(\delta))]}$, hence they give labellings of $\Delta_+$.

Let $\bm{c}=(c_{(\imath, r)})_{(\imath, r)\in J_{\delta}}\in \mathbb{Z}_{\geq 0}^{J_{\delta}}$ and $\bm{c}'=(c'_{(\imath, r)})_{(\imath, r)\in J_{\beta(\imath_0, r_0)(\delta)}}\in \mathbb{Z}_{\geq 0}^{J_{\beta(\imath_0, r_0)(\delta)}}$. By Proposition \ref{p:localchange},  $b_{-1}(\bm{c},[\bm{i}(\delta)])=b_{-1}(\bm{c}',[\bm{i}(\beta(\imath_0, r_0)(\delta))])$ if and only if
\[
\begin{cases}
c'_{(\imath_0-1, r_0+\frac{1}{2})}=c_{(\imath_0-1, r_0)}+c_{(\imath_0, r_0-\frac{1}{2})}-\min\{c_{(\imath_0, r_0+\frac{1}{2})}, c_{(\imath_0, r_0-\frac{1}{2})}\},\\
c'_{(\imath_0, r_0)}=\min\{c_{(\imath_0, r_0+\frac{1}{2})}, c_{(\imath_0, r_0-\frac{1}{2})}\},\\
c'_{(\imath_0-1, r_0-\frac{1}{2})}=c_{(\imath_0, r_0+\frac{1}{2})}+c_{(\imath_0-1, r_0)}-\min\{c_{(\imath_0, r_0+\frac{1}{2})}, c_{(\imath_0, r_0-\frac{1}{2})}\},\\
c'_{(\imath, r)}=c_{(\imath, r)}\ \text{for all}\ (\imath, r)\in J_{\delta}\setminus \{(\imath_0, r_0+\frac{1}{2}), (\imath_0-1, r_0), (\imath_0, r_0-\frac{1}{2})\}.
\end{cases}
\] 
Note that 
\[
J_{\delta}\setminus \{(\imath_0, r_0+\frac{1}{2}), (\imath_0-1, r_0), (\imath_0, r_0-\frac{1}{2})\}=J_{\beta(\imath_0, r_0)(\delta)}\setminus \{(\imath_0-1, r_0+\frac{1}{2}), (\imath_0, r_0), (\imath_0-1, r_0-\frac{1}{2})\}.
\] 
In particular, 

\noindent if $(c_{(\imath_0, r_0+\frac{1}{2})}, c_{(\imath_0-1, r_0)}, c_{(\imath_0, r_0-\frac{1}{2})})=\begin{cases}
(1, 0, 0)\\
(0, 1, 0)\\
(0, 0, 1)
\end{cases}$, then $(c'_{(\imath_0-1, r_0+\frac{1}{2})}, c'_{(\imath_0, r_0)}, c'_{(\imath_0-1, r_0-\frac{1}{2})})=\begin{cases}
(0, 0, 1)\\
(1, 0, 1)\\
(1, 0, 0),
\end{cases}$

\noindent if $(c'_{(\imath_0-1, r_0+\frac{1}{2})}, c'_{(\imath_0, r_0)}, c'_{(\imath_0-1, r_0-\frac{1}{2})})=\begin{cases}
(1, 0, 0)\\
(0, 1, 0)\\
(0, 0, 1)
\end{cases}$, then $(c_{(\imath_0, r_0+\frac{1}{2})}, c_{(\imath_0-1, r_0)}, c_{(\imath_0, r_0-\frac{1}{2})})=\begin{cases}
(0, 0, 1)\\
(1, 0, 1)\\
(1, 0, 0).
\end{cases}$

The following proposition provides explicit 3-moves between $[\overleftarrow{\mathcal{Q
}}]$ and $[\overleftarrow{\mathcal{Q
}}']$. See also Example \ref{e:n=5pictures} below.

\begin{proposition}\label{p:transform}
For $\imath=2,3,\dots, n$, set 
\[
\bm{\beta}_{\imath}:=\dprod_{k=0, 1,\dots, \imath-2}\beta(\imath, -2n+\imath-2k). 
\]
Then the following hold : 
\begin{itemize}
\item[(1)] $\bm{\beta}_2\bm{\beta}_3\cdots \bm{\beta}_n(\delta_{\mathrm{B}})$, denoted by $\delta_{\mathrm{A}}$, is well-defined and the corresponding subset $\widehat{J}_{\mathrm{A}}$ of $\widetilde{J}$ is given by 
\begin{align*}
&\{(\imath, -\imath+1-2k)\mid k=0,\dots, n-1-\imath\ \text{and}\ \imath=1,\dots, n-1\},\\
&\cup\{(\imath, -2n+\imath+\frac{3}{2})\mid \imath=1,\dots, n\}\\
&\cup\{(\imath, -2n+\imath-2k)\mid k=0,\dots, \imath -2\ \text{and}\ \imath=2,\dots, n\}\\
&\cup\{(\imath, -2n-\imath +\frac{5}{2})\mid \imath=1,\dots, n-1\}\\
&\cup\{(\imath, -2n-\imath+1-2k)\mid k=0,\dots, n-2-\imath\ \text{and}\ \imath=1,\dots, n-2\}\\
&\cup\{(\imath, -\imath +2-2k)\mid k=0,1,\dots, 2n-1-\imath\ \text{and}\ \imath=n+1,\dots, 2n-1\}.
\end{align*}
\item[(2)] $[\bm{i}(\delta_{\mathrm{A}})]=[\overleftarrow{\mathcal{Q}}']$. 
\end{itemize}
\end{proposition}

\begin{proof}
Let $\delta_{\mathrm{A}}'$ be the function $\widetilde{J}\to\{0, 1\}$ which corresponds to the subset explicitly written in the statement (1). For $\imath=1, 2,\dots, n$, define $\delta_{\imath}\colon \widetilde{J}\to\{0, 1\}$ by
\begin{align*}
\delta_{\imath}(\jmath, r):=\begin{cases}
\delta_{\mathrm{A}}'(\jmath, r)&\text{if}\ \jmath>\imath,\\
\displaystyle \sum_{k=0}^{n-\imath -1}\delta_{r, -\imath+1-2k}+\sum_{k=0}^{2\imath -1-\delta_{\imath, n}}\delta_{r, -2n+\imath+\frac{3}{2}-k}+\sum_{k=0}^{n-\imath -2}\delta_{r, -2n-\imath+1-2k}&\text{if}\ \jmath=\imath,\\
\delta_{\mathrm{B}}(\jmath, r)&\text{if}\ \jmath<\imath.
\end{cases}
\end{align*}
Note that $\delta_{n}=\delta_{\mathrm{B}}$ and $\delta_1=\delta_{\mathrm{A}}'$. To prove (1), it remains to show that $\bm{\beta}_k\delta_{k}$ is well-defined and is equal to $\delta_{k-1}$. By the explicit description, we can check the assumption of Lemma \ref{l:braidrel} directly and obtain the desired assertion straightforwardly. 

Next, by Theorem \ref{t:commequiv}, we only have to show that the combinatorial Auslander-Reiten quiver $\Upsilon_{[\bm{i}(\delta_{\mathrm{A}})]}$ is isomorphic to $\Upsilon_{[\overleftarrow{\mathcal{Q}}']}$ as a quiver by the correspondence preserving residues. Now $\Upsilon_{[\overleftarrow{\mathcal{Q}}']}$ is characterized by the following two properties :
\begin{itemize}
\item[(a)] the number of vertices of residue $\imath$ is equal to $2n-\imath$ for $\imath\in I_{\mathrm{A}}$, 
\item[(b)] convexity in the sense of Lemma \ref{l:dia}.  
\end{itemize}
By the explicit description, the quiver $\Upsilon_{[\bm{i}(\delta_{\mathrm{A}})]}$ satisfies (a). The property (b) corresponds to the property that 
\begin{itemize}
\item[] for any $\imath, \imath'\in I_{\mathrm{A}}$ with $|\imath-\imath'|=1$ and any pair $\{(\imath, r), (\imath, r^-)\}\subset\widehat{J}_{\mathrm{A}}$ with $r^-=\min\{s\mid s>r, (\imath, s)\in \widehat{J}_{\mathrm{A}}\}$, there uniquely exists $(\imath', r')\in \widehat{J}_{\mathrm{A}}$ such that $r<r'<r^-$.
\end{itemize}
We can also prove this property straightforwardly by the explicit description. 
\end{proof}

\begin{example}\label{e:n=5pictures}
The following is an example of the procedure in Proposition \ref{p:transform} for $n=5$. 

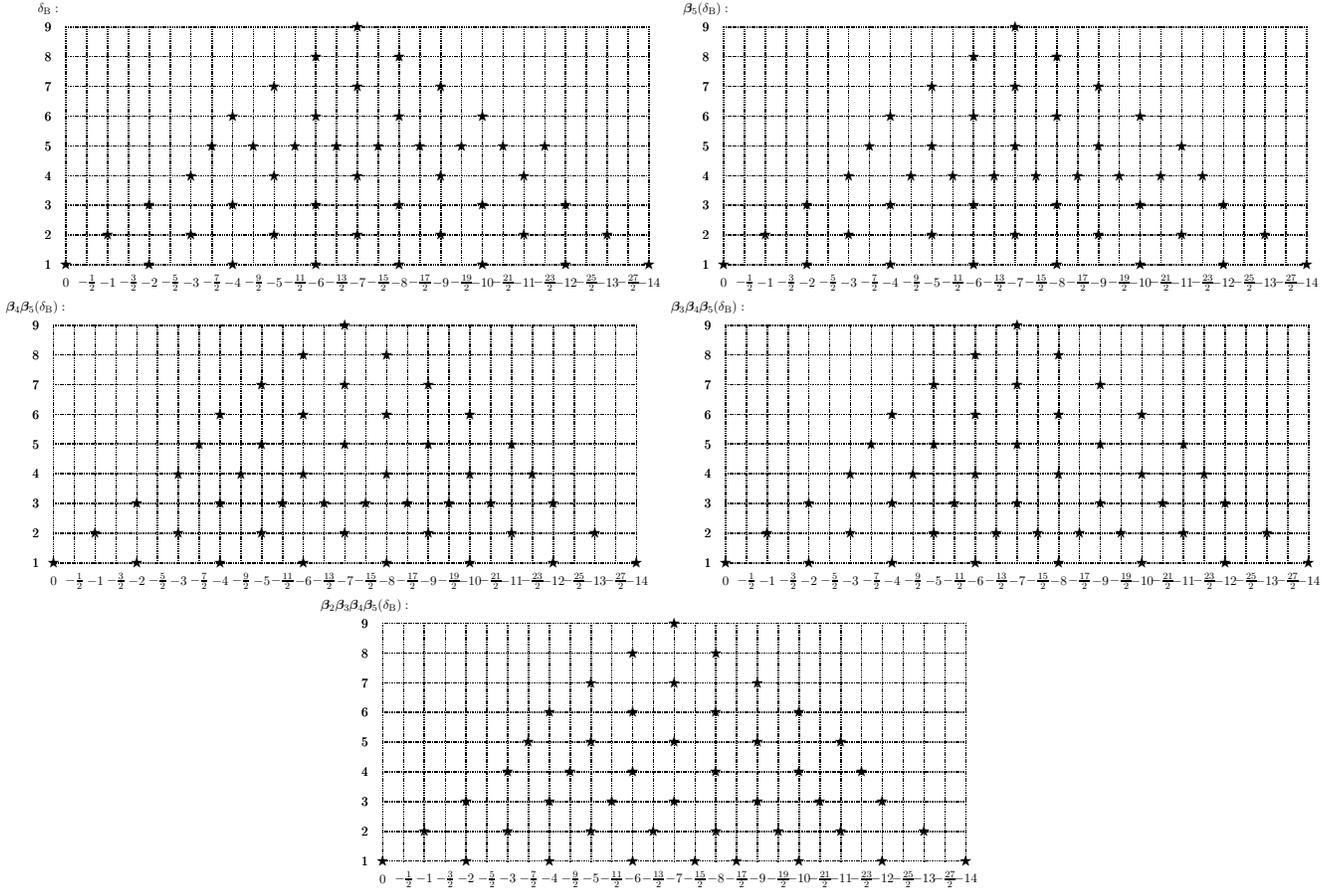
\scalebox{0.45}{
\hspace{-50pt}
\begin{xy} 0;<1pt,0pt>:<0pt,-1pt>:: 
(245,0) *+{\bigstar},
(210,25) *+{\bigstar},
(280,25) *+{\bigstar},
(175,50) *+{\bigstar},
(245,50) *+{\bigstar},
(315,50) *+{\bigstar},
(140,75) *+{\bigstar},
(210,75) *+{\bigstar},
(280,75) *+{\bigstar},
(350,75) *+{\bigstar},
(122.5,100) *+{\bigstar},
(157.5,100) *+{\bigstar},
(192.5,100) *+{\bigstar},
(227.5,100) *+{\bigstar},
(262.5,100) *+{\bigstar},
(297.5,100) *+{\bigstar},
(332.5,100) *+{\bigstar},
(367.5,100) *+{\bigstar},
(402.5,100) *+{\bigstar},
(105,125) *+{\bigstar},
(175,125) *+{\bigstar},
(245,125) *+{\bigstar},
(315,125) *+{\bigstar},
(385,125) *+{\bigstar},
(70,150) *+{\bigstar},
(140,150) *+{\bigstar},
(210,150) *+{\bigstar},
(280,150) *+{\bigstar},
(350,150) *+{\bigstar},
(420,150) *+{\bigstar},
(35,175) *+{\bigstar},
(105,175) *+{\bigstar},
(175,175) *+{\bigstar},
(245,175) *+{\bigstar},
(315,175) *+{\bigstar},
(385,175) *+{\bigstar},
(455,175) *+{\bigstar},
(0,200) *+{\bigstar},
(70,200) *+{\bigstar},
(140,200) *+{\bigstar},
(210,200) *+{\bigstar},
(280,200) *+{\bigstar},
(350,200) *+{\bigstar},
(420,200) *+{\bigstar},
(490,200) *+{\bigstar},
(-15,200) *+{1},
(-15,175) *+{2},
(-15,150) *+{3},
(-15,125) *+{4},
(-15,100) *+{5},
(-15,75) *+{6},
(-15,50) *+{7},
(-15,25) *+{8},
(-15,0) *+{9},
(-15,-15) *+{\delta_{\mathrm{B}} :},
(-15,200) *+{1},
(-15,175) *+{2},
(-15,150) *+{3},
(-15,125) *+{4},
(-15,100) *+{5},
(-15,75) *+{6},
(-15,50) *+{7},
(-15,25) *+{8},
(-15,0) *+{9},
(0,215) *+{0},
(17.5,215) *+{-\frac{1}{2}},
(35,215) *+{-1},
(52.5,215) *+{-\frac{3}{2}},
(70,215) *+{-2},
(87.5,215) *+{-\frac{5}{2}},
(105,215) *+{-3},
(122.5,215) *+{-\frac{7}{2}},
(140,215) *+{-4},
(157.5,215) *+{-\frac{9}{2}},
(175,215) *+{-5},
(192.5,215) *+{-\frac{11}{2}},
(210,215) *+{-6},
(227.5,215) *+{-\frac{13}{2}},
(245,215) *+{-7},
(262.5,215) *+{-\frac{15}{2}},
(280,215) *+{-8},
(297.5,215) *+{-\frac{17}{2}},
(315,215) *+{-9},
(332.5,215) *+{-\frac{19}{2}},
(350,215) *+{-10},
(367.5,215) *+{-\frac{21}{2}},
(385,215) *+{-11},
(402.5,215) *+{-\frac{23}{2}},
(420,215) *+{-12},
(437.5,215) *+{-\frac{25}{2}},
(455,215) *+{-13},
(472.5,215) *+{-\frac{27}{2}},
(490,215) *+{-14},
{\ar@{.} (0,200); (0,0)},
{\ar@{.} (35,200); (35,0)},
{\ar@{.} (70,200); (70,0)},
{\ar@{.} (105,200); (105,0)},
{\ar@{.} (122.5,200); (122.5,0)},
{\ar@{.} (140,200); (140,0)},
{\ar@{.} (157.5,200); (157.5,0)},
{\ar@{.} (175,200); (175,0)},
{\ar@{.} (192.5,200); (192.5,0)},
{\ar@{.} (210,200); (210,0)},
{\ar@{.} (227.5,200); (227.5,0)},
{\ar@{.} (245,200); (245,0)},
{\ar@{.} (262.5,200); (262.5,0)},
{\ar@{.} (280,200); (280,0)},
{\ar@{.} (297.5,200); (297.5,0)},
{\ar@{.} (315,200); (315,0)},
{\ar@{.} (332.5,200); (332.5,0)},
{\ar@{.} (350,200); (350,0)},
{\ar@{.} (367.5,200); (367.5,0)},
{\ar@{.} (385,200); (385,0)},
{\ar@{.} (402.5,200); (402.5,0)},
{\ar@{.} (420,200); (420,0)},
{\ar@{.} (455,200); (455,0)},
{\ar@{.} (490,200); (490,0)},
{\ar@{.} (17.5,200); (17.5,0)},
{\ar@{.} (52.5,200); (52.5,0)},
{\ar@{.} (87.5,200); (87.5,0)},
{\ar@{.} (437.5,200); (437.5,0)},
{\ar@{.} (472.5,200); (472.5,0)},
{\ar@{.} (0,0); (490,0)},
{\ar@{.} (0,25); (490,25)},
{\ar@{.} (0,50); (490,50)},
{\ar@{.} (0,75); (490,75)},
{\ar@{.} (0,100); (490,100)},
{\ar@{.} (0,125); (490,125)},
{\ar@{.} (0,150); (490,150)},
{\ar@{.} (0,175); (490,175)},
{\ar@{.} (0,200); (490,200)},
\end{xy}
}
\scalebox{0.45}{
\begin{xy} 0;<1pt,0pt>:<0pt,-1pt>:: 
(245,0) *+{\bigstar},
(210,25) *+{\bigstar},
(280,25) *+{\bigstar},
(175,50) *+{\bigstar},
(245,50) *+{\bigstar},
(315,50) *+{\bigstar},
(140,75) *+{\bigstar},
(210,75) *+{\bigstar},
(280,75) *+{\bigstar},
(350,75) *+{\bigstar},
(122.5,100) *+{\bigstar},
(175,100) *+{\bigstar},
(245,100) *+{\bigstar},
(315,100) *+{\bigstar},
(385,100) *+{\bigstar},
(105,125) *+{\bigstar},
(157.5,125) *+{\bigstar},
(192.5,125) *+{\bigstar},
(227.5,125) *+{\bigstar},
(262.5,125) *+{\bigstar},
(297.5,125) *+{\bigstar},
(332.5,125) *+{\bigstar},
(367.5,125) *+{\bigstar},
(402.5,125) *+{\bigstar},
(70,150) *+{\bigstar},
(140,150) *+{\bigstar},
(210,150) *+{\bigstar},
(280,150) *+{\bigstar},
(350,150) *+{\bigstar},
(420,150) *+{\bigstar},
(35,175) *+{\bigstar},
(105,175) *+{\bigstar},
(175,175) *+{\bigstar},
(245,175) *+{\bigstar},
(315,175) *+{\bigstar},
(385,175) *+{\bigstar},
(455,175) *+{\bigstar},
(0,200) *+{\bigstar},
(70,200) *+{\bigstar},
(140,200) *+{\bigstar},
(210,200) *+{\bigstar},
(280,200) *+{\bigstar},
(350,200) *+{\bigstar},
(420,200) *+{\bigstar},
(490,200) *+{\bigstar},
(-15,200) *+{1},
(-15,175) *+{2},
(-15,150) *+{3},
(-15,125) *+{4},
(-15,100) *+{5},
(-15,75) *+{6},
(-15,50) *+{7},
(-15,25) *+{8},
(-15,0) *+{9},
(-15,-15) *+{\bm{\beta}_5(\delta_{\mathrm{B}}) :},
(-15,200) *+{1},
(-15,175) *+{2},
(-15,150) *+{3},
(-15,125) *+{4},
(-15,100) *+{5},
(-15,75) *+{6},
(-15,50) *+{7},
(-15,25) *+{8},
(-15,0) *+{9},
(0,215) *+{0},
(17.5,215) *+{-\frac{1}{2}},
(35,215) *+{-1},
(52.5,215) *+{-\frac{3}{2}},
(70,215) *+{-2},
(87.5,215) *+{-\frac{5}{2}},
(105,215) *+{-3},
(122.5,215) *+{-\frac{7}{2}},
(140,215) *+{-4},
(157.5,215) *+{-\frac{9}{2}},
(175,215) *+{-5},
(192.5,215) *+{-\frac{11}{2}},
(210,215) *+{-6},
(227.5,215) *+{-\frac{13}{2}},
(245,215) *+{-7},
(262.5,215) *+{-\frac{15}{2}},
(280,215) *+{-8},
(297.5,215) *+{-\frac{17}{2}},
(315,215) *+{-9},
(332.5,215) *+{-\frac{19}{2}},
(350,215) *+{-10},
(367.5,215) *+{-\frac{21}{2}},
(385,215) *+{-11},
(402.5,215) *+{-\frac{23}{2}},
(420,215) *+{-12},
(437.5,215) *+{-\frac{25}{2}},
(455,215) *+{-13},
(472.5,215) *+{-\frac{27}{2}},
(490,215) *+{-14},
{\ar@{.} (0,200); (0,0)},
{\ar@{.} (35,200); (35,0)},
{\ar@{.} (70,200); (70,0)},
{\ar@{.} (105,200); (105,0)},
{\ar@{.} (122.5,200); (122.5,0)},
{\ar@{.} (140,200); (140,0)},
{\ar@{.} (157.5,200); (157.5,0)},
{\ar@{.} (175,200); (175,0)},
{\ar@{.} (192.5,200); (192.5,0)},
{\ar@{.} (210,200); (210,0)},
{\ar@{.} (227.5,200); (227.5,0)},
{\ar@{.} (245,200); (245,0)},
{\ar@{.} (262.5,200); (262.5,0)},
{\ar@{.} (280,200); (280,0)},
{\ar@{.} (297.5,200); (297.5,0)},
{\ar@{.} (315,200); (315,0)},
{\ar@{.} (332.5,200); (332.5,0)},
{\ar@{.} (350,200); (350,0)},
{\ar@{.} (367.5,200); (367.5,0)},
{\ar@{.} (385,200); (385,0)},
{\ar@{.} (402.5,200); (402.5,0)},
{\ar@{.} (420,200); (420,0)},
{\ar@{.} (455,200); (455,0)},
{\ar@{.} (490,200); (490,0)},
{\ar@{.} (17.5,200); (17.5,0)},
{\ar@{.} (52.5,200); (52.5,0)},
{\ar@{.} (87.5,200); (87.5,0)},
{\ar@{.} (437.5,200); (437.5,0)},
{\ar@{.} (472.5,200); (472.5,0)},
{\ar@{.} (0,0); (490,0)},
{\ar@{.} (0,25); (490,25)},
{\ar@{.} (0,50); (490,50)},
{\ar@{.} (0,75); (490,75)},
{\ar@{.} (0,100); (490,100)},
{\ar@{.} (0,125); (490,125)},
{\ar@{.} (0,150); (490,150)},
{\ar@{.} (0,175); (490,175)},
{\ar@{.} (0,200); (490,200)},
\end{xy}
}
\scalebox{0.45}{
\hspace{-50pt}
\begin{xy} 0;<1pt,0pt>:<0pt,-1pt>:: 
(245,0) *+{\bigstar},
(210,25) *+{\bigstar},
(280,25) *+{\bigstar},
(175,50) *+{\bigstar},
(245,50) *+{\bigstar},
(315,50) *+{\bigstar},
(140,75) *+{\bigstar},
(210,75) *+{\bigstar},
(280,75) *+{\bigstar},
(350,75) *+{\bigstar},
(122.5,100) *+{\bigstar},
(175,100) *+{\bigstar},
(245,100) *+{\bigstar},
(315,100) *+{\bigstar},
(385,100) *+{\bigstar},
(105,125) *+{\bigstar},
(157.5,125) *+{\bigstar},
(192.5,150) *+{\bigstar},
(227.5,150) *+{\bigstar},
(262.5,150) *+{\bigstar},
(297.5,150) *+{\bigstar},
(332.5,150) *+{\bigstar},
(367.5,150) *+{\bigstar},
(402.5,125) *+{\bigstar},
(70,150) *+{\bigstar},
(140,150) *+{\bigstar},
(210,125) *+{\bigstar},
(280,125) *+{\bigstar},
(350,125) *+{\bigstar},
(420,150) *+{\bigstar},
(35,175) *+{\bigstar},
(105,175) *+{\bigstar},
(175,175) *+{\bigstar},
(245,175) *+{\bigstar},
(315,175) *+{\bigstar},
(385,175) *+{\bigstar},
(455,175) *+{\bigstar},
(0,200) *+{\bigstar},
(70,200) *+{\bigstar},
(140,200) *+{\bigstar},
(210,200) *+{\bigstar},
(280,200) *+{\bigstar},
(350,200) *+{\bigstar},
(420,200) *+{\bigstar},
(490,200) *+{\bigstar},
(-15,200) *+{1},
(-15,175) *+{2},
(-15,150) *+{3},
(-15,125) *+{4},
(-15,100) *+{5},
(-15,75) *+{6},
(-15,50) *+{7},
(-15,25) *+{8},
(-15,0) *+{9},
(-15,-15) *+{\bm{\beta}_4\bm{\beta}_5(\delta_{\mathrm{B}}) :},
(-15,200) *+{1},
(-15,175) *+{2},
(-15,150) *+{3},
(-15,125) *+{4},
(-15,100) *+{5},
(-15,75) *+{6},
(-15,50) *+{7},
(-15,25) *+{8},
(-15,0) *+{9},
(0,215) *+{0},
(17.5,215) *+{-\frac{1}{2}},
(35,215) *+{-1},
(52.5,215) *+{-\frac{3}{2}},
(70,215) *+{-2},
(87.5,215) *+{-\frac{5}{2}},
(105,215) *+{-3},
(122.5,215) *+{-\frac{7}{2}},
(140,215) *+{-4},
(157.5,215) *+{-\frac{9}{2}},
(175,215) *+{-5},
(192.5,215) *+{-\frac{11}{2}},
(210,215) *+{-6},
(227.5,215) *+{-\frac{13}{2}},
(245,215) *+{-7},
(262.5,215) *+{-\frac{15}{2}},
(280,215) *+{-8},
(297.5,215) *+{-\frac{17}{2}},
(315,215) *+{-9},
(332.5,215) *+{-\frac{19}{2}},
(350,215) *+{-10},
(367.5,215) *+{-\frac{21}{2}},
(385,215) *+{-11},
(402.5,215) *+{-\frac{23}{2}},
(420,215) *+{-12},
(437.5,215) *+{-\frac{25}{2}},
(455,215) *+{-13},
(472.5,215) *+{-\frac{27}{2}},
(490,215) *+{-14},
{\ar@{.} (0,200); (0,0)},
{\ar@{.} (35,200); (35,0)},
{\ar@{.} (70,200); (70,0)},
{\ar@{.} (105,200); (105,0)},
{\ar@{.} (122.5,200); (122.5,0)},
{\ar@{.} (140,200); (140,0)},
{\ar@{.} (157.5,200); (157.5,0)},
{\ar@{.} (175,200); (175,0)},
{\ar@{.} (192.5,200); (192.5,0)},
{\ar@{.} (210,200); (210,0)},
{\ar@{.} (227.5,200); (227.5,0)},
{\ar@{.} (245,200); (245,0)},
{\ar@{.} (262.5,200); (262.5,0)},
{\ar@{.} (280,200); (280,0)},
{\ar@{.} (297.5,200); (297.5,0)},
{\ar@{.} (315,200); (315,0)},
{\ar@{.} (332.5,200); (332.5,0)},
{\ar@{.} (350,200); (350,0)},
{\ar@{.} (367.5,200); (367.5,0)},
{\ar@{.} (385,200); (385,0)},
{\ar@{.} (402.5,200); (402.5,0)},
{\ar@{.} (420,200); (420,0)},
{\ar@{.} (455,200); (455,0)},
{\ar@{.} (490,200); (490,0)},
{\ar@{.} (17.5,200); (17.5,0)},
{\ar@{.} (52.5,200); (52.5,0)},
{\ar@{.} (87.5,200); (87.5,0)},
{\ar@{.} (437.5,200); (437.5,0)},
{\ar@{.} (472.5,200); (472.5,0)},
{\ar@{.} (0,0); (490,0)},
{\ar@{.} (0,25); (490,25)},
{\ar@{.} (0,50); (490,50)},
{\ar@{.} (0,75); (490,75)},
{\ar@{.} (0,100); (490,100)},
{\ar@{.} (0,125); (490,125)},
{\ar@{.} (0,150); (490,150)},
{\ar@{.} (0,175); (490,175)},
{\ar@{.} (0,200); (490,200)},
\end{xy}
}
\scalebox{0.45}{
\begin{xy} 0;<1pt,0pt>:<0pt,-1pt>:: 
(245,0) *+{\bigstar},
(210,25) *+{\bigstar},
(280,25) *+{\bigstar},
(175,50) *+{\bigstar},
(245,50) *+{\bigstar},
(315,50) *+{\bigstar},
(140,75) *+{\bigstar},
(210,75) *+{\bigstar},
(280,75) *+{\bigstar},
(350,75) *+{\bigstar},
(122.5,100) *+{\bigstar},
(175,100) *+{\bigstar},
(245,100) *+{\bigstar},
(315,100) *+{\bigstar},
(385,100) *+{\bigstar},
(105,125) *+{\bigstar},
(157.5,125) *+{\bigstar},
(210,125) *+{\bigstar},
(280,125) *+{\bigstar},
(350,125) *+{\bigstar},
(402.5,125) *+{\bigstar},
(70,150) *+{\bigstar},
(140,150) *+{\bigstar},
(192.5,150) *+{\bigstar},
(245,150) *+{\bigstar},
(315,150) *+{\bigstar},
(367.5,150) *+{\bigstar},
(420,150) *+{\bigstar},
(35,175) *+{\bigstar},
(105,175) *+{\bigstar},
(175,175) *+{\bigstar},
(227.5,175) *+{\bigstar},
(262.5,175) *+{\bigstar},
(297.5,175) *+{\bigstar},
(332.5,175) *+{\bigstar},
(385,175) *+{\bigstar},
(455,175) *+{\bigstar},
(0,200) *+{\bigstar},
(70,200) *+{\bigstar},
(140,200) *+{\bigstar},
(210,200) *+{\bigstar},
(280,200) *+{\bigstar},
(350,200) *+{\bigstar},
(420,200) *+{\bigstar},
(490,200) *+{\bigstar},
(-15,200) *+{1},
(-15,175) *+{2},
(-15,150) *+{3},
(-15,125) *+{4},
(-15,100) *+{5},
(-15,75) *+{6},
(-15,50) *+{7},
(-15,25) *+{8},
(-15,0) *+{9},
(-15,-15) *+{\bm{\beta}_3\bm{\beta}_4\bm{\beta}_5(\delta_{\mathrm{B}}) :},
(-15,200) *+{1},
(-15,175) *+{2},
(-15,150) *+{3},
(-15,125) *+{4},
(-15,100) *+{5},
(-15,75) *+{6},
(-15,50) *+{7},
(-15,25) *+{8},
(-15,0) *+{9},
(0,215) *+{0},
(17.5,215) *+{-\frac{1}{2}},
(35,215) *+{-1},
(52.5,215) *+{-\frac{3}{2}},
(70,215) *+{-2},
(87.5,215) *+{-\frac{5}{2}},
(105,215) *+{-3},
(122.5,215) *+{-\frac{7}{2}},
(140,215) *+{-4},
(157.5,215) *+{-\frac{9}{2}},
(175,215) *+{-5},
(192.5,215) *+{-\frac{11}{2}},
(210,215) *+{-6},
(227.5,215) *+{-\frac{13}{2}},
(245,215) *+{-7},
(262.5,215) *+{-\frac{15}{2}},
(280,215) *+{-8},
(297.5,215) *+{-\frac{17}{2}},
(315,215) *+{-9},
(332.5,215) *+{-\frac{19}{2}},
(350,215) *+{-10},
(367.5,215) *+{-\frac{21}{2}},
(385,215) *+{-11},
(402.5,215) *+{-\frac{23}{2}},
(420,215) *+{-12},
(437.5,215) *+{-\frac{25}{2}},
(455,215) *+{-13},
(472.5,215) *+{-\frac{27}{2}},
(490,215) *+{-14},
{\ar@{.} (0,200); (0,0)},
{\ar@{.} (35,200); (35,0)},
{\ar@{.} (70,200); (70,0)},
{\ar@{.} (105,200); (105,0)},
{\ar@{.} (122.5,200); (122.5,0)},
{\ar@{.} (140,200); (140,0)},
{\ar@{.} (157.5,200); (157.5,0)},
{\ar@{.} (175,200); (175,0)},
{\ar@{.} (192.5,200); (192.5,0)},
{\ar@{.} (210,200); (210,0)},
{\ar@{.} (227.5,200); (227.5,0)},
{\ar@{.} (245,200); (245,0)},
{\ar@{.} (262.5,200); (262.5,0)},
{\ar@{.} (280,200); (280,0)},
{\ar@{.} (297.5,200); (297.5,0)},
{\ar@{.} (315,200); (315,0)},
{\ar@{.} (332.5,200); (332.5,0)},
{\ar@{.} (350,200); (350,0)},
{\ar@{.} (367.5,200); (367.5,0)},
{\ar@{.} (385,200); (385,0)},
{\ar@{.} (402.5,200); (402.5,0)},
{\ar@{.} (420,200); (420,0)},
{\ar@{.} (455,200); (455,0)},
{\ar@{.} (490,200); (490,0)},
{\ar@{.} (17.5,200); (17.5,0)},
{\ar@{.} (52.5,200); (52.5,0)},
{\ar@{.} (87.5,200); (87.5,0)},
{\ar@{.} (437.5,200); (437.5,0)},
{\ar@{.} (472.5,200); (472.5,0)},
{\ar@{.} (0,0); (490,0)},
{\ar@{.} (0,25); (490,25)},
{\ar@{.} (0,50); (490,50)},
{\ar@{.} (0,75); (490,75)},
{\ar@{.} (0,100); (490,100)},
{\ar@{.} (0,125); (490,125)},
{\ar@{.} (0,150); (490,150)},
{\ar@{.} (0,175); (490,175)},
{\ar@{.} (0,200); (490,200)},
\end{xy}
}

\hfill
\scalebox{0.45}{
\begin{xy} 0;<1pt,0pt>:<0pt,-1pt>:: 
(245,0) *+{\bigstar},
(210,25) *+{\bigstar},
(280,25) *+{\bigstar},
(175,50) *+{\bigstar},
(245,50) *+{\bigstar},
(315,50) *+{\bigstar},
(140,75) *+{\bigstar},
(210,75) *+{\bigstar},
(280,75) *+{\bigstar},
(350,75) *+{\bigstar},
(122.5,100) *+{\bigstar},
(175,100) *+{\bigstar},
(245,100) *+{\bigstar},
(315,100) *+{\bigstar},
(385,100) *+{\bigstar},
(105,125) *+{\bigstar},
(157.5,125) *+{\bigstar},
(210,125) *+{\bigstar},
(280,125) *+{\bigstar},
(350,125) *+{\bigstar},
(402.5,125) *+{\bigstar},
(70,150) *+{\bigstar},
(140,150) *+{\bigstar},
(192.5,150) *+{\bigstar},
(245,150) *+{\bigstar},
(315,150) *+{\bigstar},
(367.5,150) *+{\bigstar},
(420,150) *+{\bigstar},
(35,175) *+{\bigstar},
(105,175) *+{\bigstar},
(175,175) *+{\bigstar},
(227.5,175) *+{\bigstar},
(262.5,200) *+{\bigstar},
(297.5,200) *+{\bigstar},
(332.5,175) *+{\bigstar},
(385,175) *+{\bigstar},
(455,175) *+{\bigstar},
(0,200) *+{\bigstar},
(70,200) *+{\bigstar},
(140,200) *+{\bigstar},
(210,200) *+{\bigstar},
(280,175) *+{\bigstar},
(350,200) *+{\bigstar},
(420,200) *+{\bigstar},
(490,200) *+{\bigstar},
(-15,200) *+{1},
(-15,175) *+{2},
(-15,150) *+{3},
(-15,125) *+{4},
(-15,100) *+{5},
(-15,75) *+{6},
(-15,50) *+{7},
(-15,25) *+{8},
(-15,0) *+{9},
(-15,-15) *+{\bm{\beta}_2\bm{\beta}_3\bm{\beta}_4\bm{\beta}_5(\delta_{\mathrm{B}}) :},
(-15,200) *+{1},
(-15,175) *+{2},
(-15,150) *+{3},
(-15,125) *+{4},
(-15,100) *+{5},
(-15,75) *+{6},
(-15,50) *+{7},
(-15,25) *+{8},
(-15,0) *+{9},
(0,215) *+{0},
(17.5,215) *+{-\frac{1}{2}},
(35,215) *+{-1},
(52.5,215) *+{-\frac{3}{2}},
(70,215) *+{-2},
(87.5,215) *+{-\frac{5}{2}},
(105,215) *+{-3},
(122.5,215) *+{-\frac{7}{2}},
(140,215) *+{-4},
(157.5,215) *+{-\frac{9}{2}},
(175,215) *+{-5},
(192.5,215) *+{-\frac{11}{2}},
(210,215) *+{-6},
(227.5,215) *+{-\frac{13}{2}},
(245,215) *+{-7},
(262.5,215) *+{-\frac{15}{2}},
(280,215) *+{-8},
(297.5,215) *+{-\frac{17}{2}},
(315,215) *+{-9},
(332.5,215) *+{-\frac{19}{2}},
(350,215) *+{-10},
(367.5,215) *+{-\frac{21}{2}},
(385,215) *+{-11},
(402.5,215) *+{-\frac{23}{2}},
(420,215) *+{-12},
(437.5,215) *+{-\frac{25}{2}},
(455,215) *+{-13},
(472.5,215) *+{-\frac{27}{2}},
(490,215) *+{-14},
{\ar@{.} (0,200); (0,0)},
{\ar@{.} (35,200); (35,0)},
{\ar@{.} (70,200); (70,0)},
{\ar@{.} (105,200); (105,0)},
{\ar@{.} (122.5,200); (122.5,0)},
{\ar@{.} (140,200); (140,0)},
{\ar@{.} (157.5,200); (157.5,0)},
{\ar@{.} (175,200); (175,0)},
{\ar@{.} (192.5,200); (192.5,0)},
{\ar@{.} (210,200); (210,0)},
{\ar@{.} (227.5,200); (227.5,0)},
{\ar@{.} (245,200); (245,0)},
{\ar@{.} (262.5,200); (262.5,0)},
{\ar@{.} (280,200); (280,0)},
{\ar@{.} (297.5,200); (297.5,0)},
{\ar@{.} (315,200); (315,0)},
{\ar@{.} (332.5,200); (332.5,0)},
{\ar@{.} (350,200); (350,0)},
{\ar@{.} (367.5,200); (367.5,0)},
{\ar@{.} (385,200); (385,0)},
{\ar@{.} (402.5,200); (402.5,0)},
{\ar@{.} (420,200); (420,0)},
{\ar@{.} (455,200); (455,0)},
{\ar@{.} (490,200); (490,0)},
{\ar@{.} (17.5,200); (17.5,0)},
{\ar@{.} (52.5,200); (52.5,0)},
{\ar@{.} (87.5,200); (87.5,0)},
{\ar@{.} (437.5,200); (437.5,0)},
{\ar@{.} (472.5,200); (472.5,0)},
{\ar@{.} (0,0); (490,0)},
{\ar@{.} (0,25); (490,25)},
{\ar@{.} (0,50); (490,50)},
{\ar@{.} (0,75); (490,75)},
{\ar@{.} (0,100); (490,100)},
{\ar@{.} (0,125); (490,125)},
{\ar@{.} (0,150); (490,150)},
{\ar@{.} (0,175); (490,175)},
{\ar@{.} (0,200); (490,200)},
\end{xy}
}
\hfill
\hfill
 
\end{example}

By the way, we have
\[
J_{\mathrm{A}}=\{(\imath, -\imath+1-2k)\mid k=0, 1,\dots, 2n-\imath-1\ \text{and}\ \imath\in I_{\mathrm{A}}\}. 
\]
By Proposition \ref{p:transform} (2), we have a bijection $\widehat{J}_{\mathrm{A}}\to J_{\mathrm{A}}$ given by
\[
(\imath, r)\mapsto (\imath, -\imath+1-2k)\ \text{with}\ k=\#\{r'\mid r'> r, (\imath, r')\in \widehat{J}_{\mathrm{A}}\}. 
\]
The explicit form of this bijection is given by
\begin{align*}
&(\imath, -\imath+1-2k)\mapsto (\imath, -\imath+1-2k)\ \text{for}\ k=0,\dots, n-1-\imath\ \text{and}\ \imath=1,\dots, n-1,\\
&(\imath, -2n+\imath+\frac{3}{2})\mapsto (\imath, -2n+\imath+1)\ \text{for}\ \imath=1,\dots, n,\\
&(\imath, -2n+\imath-2k)\mapsto (\imath, -2n+\imath-1-2k)\ \text{for}\ k=0,\dots, \imath -2\ \text{and}\ \imath=2,\dots, n,\\
&(\imath, -2n-\imath +\frac{5}{2})\mapsto (\imath, -2n-\imath+1)\ \text{for}\ \imath=1,\dots, n-1,\\
&(\imath, -2n-\imath+1-2k)\mapsto (\imath, -2n-\imath-1-2k)\ \text{for}\ k=0,\dots, n-2-\imath\ \text{and}\ \imath=1,\dots, n-2,\\
&(\imath, -\imath +2-2k)\mapsto (\imath, -\imath+1-2k)\ \text{for}\ k=0,1,\dots, 2n-\imath -1\ \text{and}\ \imath=n+1,\dots, 2n-1.
\end{align*}

By Proposition \ref{p:transform}, we already knew an explicit method for obtaining $\widehat{J}_{\mathrm{A}}$ from $J_{\mathrm{B}}$ by repeated application of $\beta(\imath, r)$, and we also already described the explicit bijection between $\widehat{J}_{\mathrm{A}}$ and $J_{\mathrm{A}}$ as above. Therefore, we can describe the correspondence between $m_{\mathrm{A}}(\bm{c}_1)$ and $m_{\mathrm{B}}(\bm{c}_0)$ in \eqref{eq:ABcorresp} by direct calculation (recall also the observation before Proposition \ref{p:transform}). The detailed calculation is left to the reader :
\begin{theorem}\label{t:explicitcorresp}
The following give the explicit images of the $(q, t)$-characters of fundamental modules under $\widetilde{\Phi}_{\mathrm{A}\to \mathrm{B}}$ and $\widetilde{\Phi}_{\mathrm{B}\to \mathrm{A}}$ :
\begin{align*}
&\widetilde{\Phi}_{\mathrm{A}\to \mathrm{B}}(L_t^{\mathrm{A}}(Y_{\imath, -\imath+1-2k}))=L_t(Y_{\imath, -2\imath +2-4k})\ \text{for}\ k=0,\dots, n-1-\imath\ \text{and}\ \imath=1,\dots,n-1,\\
&\widetilde{\Phi}_{\mathrm{A}\to \mathrm{B}}(L_t^{\mathrm{A}}(Y_{\imath, -2n+\imath+1}))=L_t(Y_{n, -6n+4\imath+3})\ \text{for}\ \imath=1,\dots, n,\\
&\widetilde{\Phi}_{\mathrm{A}\to \mathrm{B}}(L_t^{\mathrm{A}}(Y_{\imath, -2n+\imath-1-2k}))=L_t(Y_{n, -2n+1-4k}Y_{n, -6n+4\imath -1-4k})\ \text{for}\ k=0,\dots, \imath-2\ \text{and}\ \imath=2,\dots,n,\\
&\widetilde{\Phi}_{\mathrm{A}\to \mathrm{B}}(L_t^{\mathrm{A}}(Y_{\imath, -2n-\imath+1}))=L_t(Y_{n, -2n-4\imath +5})\ \text{for}\ \imath=1,\dots,n-1,\\
&\widetilde{\Phi}_{\mathrm{A}\to \mathrm{B}}(L_t^{\mathrm{A}}(Y_{\imath, -2n-\imath-1-2k}))=L_t(Y_{\imath, -4n-2\imath +2-4k})\ \text{for}\ k=0,\dots, n-2-\imath\ \text{and}\ \imath=1,\dots,n-2,\\
&\widetilde{\Phi}_{\mathrm{A}\to \mathrm{B}}(L_t^{\mathrm{A}}(Y_{\imath, -\imath+1-2k}))=L_t(Y_{2n-\imath, -2\imath +4-4k})\ \text{for}\ k=0,\dots,2n-1-\imath\ \text{and}\ \imath=n+1,\dots,2n-1.
\end{align*}
\begin{align*}
&\widetilde{\Phi}_{\mathrm{B}\to \mathrm{A}}(L_t(Y_{i, -2i +2-4k}))=L_t^{\mathrm{A}}(Y_{i, -i+1-2k})\ \text{for}\ k=0,\dots,n-1-i\ \text{and}\ i=1,\dots,n-1,\\
&\widetilde{\Phi}_{\mathrm{B}\to \mathrm{A}}(L_t(Y_{i, -4n+2i +2-4k}))=L_t^{\mathrm{A}}(Y_{i-k , -2n+i +1-k}Y_{1+k, -2n-k})\ \text{for}\ k=0,\dots,i-1\ \text{and}\ i=1,\dots,n-1,\\
&\widetilde{\Phi}_{\mathrm{B}\to \mathrm{A}}(L_t(Y_{i, -4n-2i +2-4k}))=L_t^{\mathrm{A}}(Y_{i, -2n-i-1-2k})\ \text{for}\ k=0,\dots,n-2-i\ \text{and}\ i=1,\dots,n-2,\\
&\widetilde{\Phi}_{\mathrm{B}\to \mathrm{A}}(L_t(Y_{n, -2n+3-4k}))=L_t^{\mathrm{A}}(Y_{n-k, -n+1-k})\ \text{for}\ k=0,\dots,n-1,\\
&\widetilde{\Phi}_{\mathrm{B}\to \mathrm{A}}(L_t(Y_{n, -2n+1-4k}))=L_t^{\mathrm{A}}(Y_{1+k, -2n-k})\ \text{for}\ k=0,\dots,n-2,\\
&\widetilde{\Phi}_{\mathrm{B}\to \mathrm{A}}(L_t(Y_{n-i, -2n-2i +4-4k}))=L_t^{\mathrm{A}}(Y_{n+i, -n-i+1-2k})\ \text{for}\ k=0,\dots, n-i -1\ \text{and}\ i=1,\dots,n-1.
\end{align*}
\end{theorem}
\begin{remark}
The correspondence given by $\widetilde{\Phi}_{\mathrm{A}\to \mathrm{B}}$ in Theorem \ref{t:explicitcorresp}  coincides with Kashiwara-Kim-Oh's correspondence \cite[section 3]{KKO:AB} up to shift of spectral parameters when we restrict it to $\mathcal{C}_{\overleftarrow{\mathcal{Q}}'}$. We should note that their approach is based on the generalized quantum affine Schur-Weyl dualities.
\end{remark}
\begin{example}\label{ex:ABcorresp2}
The following is the explicit correspondence in the case $n=2$ (type $\mathrm{A}_3^{(1)}/\mathrm{B}_2^{(1)}$). 
\begin{align*}
&\widetilde{\Phi}_{\mathrm{A}\to \mathrm{B}}(L_t^{\mathrm{A}}(Y_{1, 0}))=L_t^{\mathrm{B}}(Y_{1, 0}),&&&
&\widetilde{\Phi}_{\mathrm{A}\to \mathrm{B}}(L_t^{\mathrm{A}}(Y_{1, -2}))=L_t^{\mathrm{B}}(Y_{2, -5}),\\
&\widetilde{\Phi}_{\mathrm{A}\to \mathrm{B}}(L_t^{\mathrm{A}}(Y_{1, -4}))=L_t^{\mathrm{B}}(Y_{2, -3}),&&&
&\widetilde{\Phi}_{\mathrm{A}\to \mathrm{B}}(L_t^{\mathrm{A}}(Y_{2, -1}))=L_t^{\mathrm{B}}(Y_{2, -1}),\\
&\widetilde{\Phi}_{\mathrm{A}\to \mathrm{B}}(L_t^{\mathrm{A}}(Y_{2, -3}))=L_t^{\mathrm{B}}(m_{2, -5}^{(2)}),&&&
&\widetilde{\Phi}_{\mathrm{A}\to \mathrm{B}}(L_t^{\mathrm{A}}(Y_{3, -2}))=L_t^{\mathrm{B}}(Y_{1, -2}).
\end{align*}
\begin{align*}
&\widetilde{\Phi}_{\mathrm{B}\to \mathrm{A}}(L_t^{\mathrm{B}}(Y_{1, 0}))=L_t^{\mathrm{A}}(Y_{1, 0}),&&&
&\widetilde{\Phi}_{\mathrm{B}\to \mathrm{A}}(L_t^{\mathrm{B}}(Y_{1, -4}))=L_t^{\mathrm{A}}(m_{2, -4}^{(1)}),\\
&\widetilde{\Phi}_{\mathrm{B}\to \mathrm{A}}(L_t^{\mathrm{B}}(Y_{2, -1}))=L_t^{\mathrm{A}}(Y_{2, -1}),&&&
&\widetilde{\Phi}_{\mathrm{B}\to \mathrm{A}}(L_t^{\mathrm{B}}(Y_{2, -3}))=L_t^{\mathrm{A}}(Y_{1, -4}),\\
&\widetilde{\Phi}_{\mathrm{B}\to \mathrm{A}}(L_t^{\mathrm{B}}(Y_{2, -5}))=L_t^{\mathrm{A}}(Y_{1, -2}),&&&
&\widetilde{\Phi}_{\mathrm{B}\to \mathrm{A}}(L_t^{\mathrm{B}}(Y_{1, -2}))=L_t^{\mathrm{A}}(Y_{3, -2}).
\end{align*}
\end{example}
\begin{example}\label{ex:ABcorresp3}
The following is the explicit correspondence in the case $n=3$ (type $\mathrm{A}_5^{(1)}/\mathrm{B}_3^{(1)}$). Here ``not KR'' means ``not the $(q, t)$-character of a Kirillov-Reshetikhin module''. 
\begin{align*}
&\widetilde{\Phi}_{\mathrm{A}\to \mathrm{B}}(L_t^{\mathrm{A}}(Y_{1, 0}))=L_t^{\mathrm{B}}(Y_{1, 0}),&&&
&\widetilde{\Phi}_{\mathrm{A}\to \mathrm{B}}(L_t^{\mathrm{A}}(Y_{1, -2}))=L_t^{\mathrm{B}}(Y_{1, -4}),\\
&\widetilde{\Phi}_{\mathrm{A}\to \mathrm{B}}(L_t^{\mathrm{A}}(Y_{1, -4}))=L_t^{\mathrm{B}}(Y_{3, -11}),&&&
&\widetilde{\Phi}_{\mathrm{A}\to \mathrm{B}}(L_t^{\mathrm{A}}(Y_{1, -6}))=L_t^{\mathrm{B}}(Y_{3, -5}),\\
&\widetilde{\Phi}_{\mathrm{A}\to \mathrm{B}}(L_t^{\mathrm{A}}(Y_{1, -8}))=L_t^{\mathrm{B}}(Y_{1, -12}),&&&
&\widetilde{\Phi}_{\mathrm{A}\to \mathrm{B}}(L_t^{\mathrm{A}}(Y_{2, -1}))=L_t^{\mathrm{B}}(Y_{2, -2}),\\
&\widetilde{\Phi}_{\mathrm{A}\to \mathrm{B}}(L_t^{\mathrm{A}}(Y_{2, -3}))=L_t^{\mathrm{B}}(Y_{3, -7}), &&&
&\widetilde{\Phi}_{\mathrm{A}\to \mathrm{B}}(L_t^{\mathrm{A}}(Y_{2, -5}))=L_t^{\mathrm{B}}(Y_{3, -5}Y_{3, -11}), (\text{not KR})\\
&\widetilde{\Phi}_{\mathrm{A}\to \mathrm{B}}(L_t^{\mathrm{A}}(Y_{2, -7}))=L_t^{\mathrm{B}}(Y_{3, -9}),&&&
&\widetilde{\Phi}_{\mathrm{A}\to \mathrm{B}}(L_t^{\mathrm{A}}(Y_{3, -2}))=L_t^{\mathrm{B}}(Y_{3, -3}),\\
&\widetilde{\Phi}_{\mathrm{A}\to \mathrm{B}}(L_t^{\mathrm{A}}(Y_{3, -4}))=L_t^{\mathrm{B}}(m_{2, -7}^{(3)}),&&&
&\widetilde{\Phi}_{\mathrm{A}\to \mathrm{B}}(L_t^{\mathrm{A}}(Y_{3, -6}))=L_t^{\mathrm{B}}(m_{2, -11}^{(3)}),\\
&\widetilde{\Phi}_{\mathrm{A}\to \mathrm{B}}(L_t^{\mathrm{A}}(Y_{4, -3}))=L_t^{\mathrm{B}}(Y_{2, -4}),&&&
&\widetilde{\Phi}_{\mathrm{A}\to \mathrm{B}}(L_t^{\mathrm{A}}(Y_{4, -5}))=L_t^{\mathrm{B}}(Y_{2, -8}),\\
&\widetilde{\Phi}_{\mathrm{A}\to \mathrm{B}}(L_t^{\mathrm{A}}(Y_{5, -4}))=L_t^{\mathrm{B}}(Y_{1, -6}).&&&&
\end{align*}
\begin{align*}
&\widetilde{\Phi}_{\mathrm{B}\to \mathrm{A}}(L_t^{\mathrm{B}}(Y_{1, 0}))=L_t^{\mathrm{A}}(Y_{1, 0}),&&&
&\widetilde{\Phi}_{\mathrm{B}\to \mathrm{A}}(L_t^{\mathrm{B}}(Y_{1, -4}))=L_t^{\mathrm{A}}(Y_{1, -2}),\\
&\widetilde{\Phi}_{\mathrm{B}\to \mathrm{A}}(L_t^{\mathrm{B}}(Y_{1, -8}))=L_t^{\mathrm{A}}(m_{2, -6}^{(1)}),&&&
&\widetilde{\Phi}_{\mathrm{B}\to \mathrm{A}}(L_t^{\mathrm{B}}(Y_{1, -12}))=L_t^{\mathrm{A}}(Y_{1, -8}),\\
&\widetilde{\Phi}_{\mathrm{B}\to \mathrm{A}}(L_t^{\mathrm{B}}(Y_{2, -2}))=L_t^{\mathrm{A}}(Y_{2, -1}),&&&
&\widetilde{\Phi}_{\mathrm{B}\to \mathrm{A}}(L_t^{\mathrm{B}}(Y_{2, -6}))=L_t^{\mathrm{A}}(Y_{2, -3}Y_{1, -6}), (\text{not KR})\\
&\widetilde{\Phi}_{\mathrm{B}\to \mathrm{A}}(L_t^{\mathrm{B}}(Y_{2, -10}))=L_t^{\mathrm{A}}(Y_{1, -4}Y_{2, -7}), (\text{not KR})&&&
&\widetilde{\Phi}_{\mathrm{B}\to \mathrm{A}}(L_t^{\mathrm{B}}(Y_{3, -3}))=L_t^{\mathrm{A}}(Y_{3, -2}),\\
&\widetilde{\Phi}_{\mathrm{B}\to \mathrm{A}}(L_t^{\mathrm{B}}(Y_{3, -5}))=L_t^{\mathrm{A}}(Y_{1, -6}),&&&
&\widetilde{\Phi}_{\mathrm{B}\to \mathrm{A}}(L_t^{\mathrm{B}}(Y_{3, -7}))=L_t^{\mathrm{A}}(Y_{2, -3}),\\
&\widetilde{\Phi}_{\mathrm{B}\to \mathrm{A}}(L_t^{\mathrm{B}}(Y_{3, -9}))=L_t^{\mathrm{A}}(Y_{2, -7}),&&&
&\widetilde{\Phi}_{\mathrm{B}\to \mathrm{A}}(L_t^{\mathrm{B}}(Y_{3, -11}))=L_t^{\mathrm{A}}(Y_{1, -4}),\\
&\widetilde{\Phi}_{\mathrm{B}\to \mathrm{A}}(L_t^{\mathrm{B}}(Y_{2, -4}))=L_t^{\mathrm{A}}(Y_{4, -3}),&&&
&\widetilde{\Phi}_{\mathrm{B}\to \mathrm{A}}(L_t^{\mathrm{B}}(Y_{2, -8}))=L_t^{\mathrm{A}}(Y_{4, -5}),\\
&\widetilde{\Phi}_{\mathrm{B}\to \mathrm{A}}(L_t^{\mathrm{B}}(Y_{1, -6}))=L_t^{\mathrm{A}}(Y_{5, -4}).&&&&
\end{align*}
\end{example}

\appendix
\section{The inverse of the quantum Cartan matrix of type $\mathrm{B}_n$}\label{a:inv}
Recall the notation in section \ref{s:qtori}. Assume that the Cartan matrix $C$ is of type $\mathrm{B}_n$. The following lemma provides the explicit formula of $\widetilde{C}(z)_{ji}$ for $i, j\in I$ as an element of $\mathbb{Z}\dbracket{z^{-1}}$ (see also Example \ref{e:invcarB_2}, \ref{e:invcarB-pic} and Remark \ref{r:explicitinvcarB}). The formula of $\widetilde{C}(z)_{ji}$ in different forms can be also found in \cite[Appendix C]{FR:defW-alg} and \cite[Appendix A]{GT}. 
\begin{lemma}\label{l:explicitinvcarB}
For $j\in I$ and $k\in \mathbb{Z}_{\geq 0}$, we have
\begin{align*}
\widetilde{c}_{ji}(-2k-1)=0\ \text{for}\ i\in I\setminus\{n\},&&& \widetilde{c}_{jn}(-2k-2)=0.
\end{align*}
For $j\in I$ and $k=0, 1,\dots, n-1$, we have
\begin{align}
\widetilde{c}_{ji}(-2k)&=
\begin{cases}
1&\text{if}\ (\ast),\\
0&\text{otherwise,}
\end{cases}\ \text{for}\ i\in I\setminus\{n\},&
\widetilde{c}_{jn}(-2k-1)&=
\begin{cases}
1&\text{if}\ (\ast\ast)\\
0&\text{otherwise}
\end{cases}\label{leq:explicitinvcarB}
\end{align}
here 
\begin{align*}
(\ast)&=
\begin{cases}
[i+1\leq k+j \leq 2n-i-1\ \text{and}\ 1-i\leq k-j\leq i-1\\
\text{with}\ (k+j)-(i+1), (k-j)-(1-i)\in 2\mathbb{Z}_{\geq 0}],\\
or\\
[2n-i\leq k+j].
\end{cases}\\
(\ast\ast)&=[k+j-n\in 2\mathbb{Z}_{\geq 0}].
\end{align*}

Moreover, for $j\in I$ and $k=0, 1,\dots, n-1$, we have 
\begin{align*}
\widetilde{c}_{ji}(-2k)&=\widetilde{c}_{ji}(-4n+2+2k)\ \text{for}\ i\in I\setminus\{n\},&
\widetilde{c}_{jn}(-2k-1)&=\widetilde{c}_{jn}(-4n+3+2k).
\end{align*}
\end{lemma}
\begin{remark}
Note that, when we consider the tables as in Example \ref{e:invcarB-pic} (in particular, fix $i\in I\setminus\{n\}$), the first condition of $(\ast)$ corresponds to the rectangle whose corners are $(j, -2k)=(i, -2), (n-1, -2(n-i)), (1, -2i), (n-i, -2(n-1))$. In particular, the points in this rectangle satisfy the condition $1\leq k\leq n-1$ and $1\leq j\leq n-1$. 
\end{remark}
\begin{proof}[{Proof of Lemma \ref{l:explicitinvcarB}}]
By considering the coefficients of $z^k$ with $k\geq 0$ in \eqref{eq:invcar}, we obtain, for $i, j\in I$, 
\begin{align*}
\widetilde{c}_{ji}(-r_j)&=\delta_{ji}& 
\widetilde{c}_{ji}(r)&=0\ \text{for}\ r>-r_j.
\end{align*}
Moreover, by considering the coefficients of $z^r$ with $r\leq -1$ in \eqref{eq:invcar}, we have
\begin{align}
&\widetilde{c}_{ji}(r-2)+\widetilde{c}_{ji}(r+2)-\sum_{|j-k|=1}\widetilde{c}_{ki}(r)=0\ \text{for}\ j\in I\setminus\{n\},\label{eq:rec1}\\
&\widetilde{c}_{ni}(r-1)+\widetilde{c}_{ni}(r+1)-\widetilde{c}_{n-1 i}(r-1)-\widetilde{c}_{n-1 i}(r+1)=0.\label{eq:rec2}
\end{align}
Hence, by induction on $k$, we have $\widetilde{c}_{ji}(-2k-1)=\widetilde{c}_{jn}(-2k-2)=0$ for every $i\in I\setminus\{n\}$, $j\in I$ and $k\in \mathbb{Z}_{\geq 0}$. 

From now, we shall prove \eqref{leq:explicitinvcarB} for $i\in I\setminus\{n\}$. (The proof of \eqref{leq:explicitinvcarB} for $\widetilde{c}_{jn}(-2k-1)$ is also given by the similar arguments, and it is simpler. Hence we omit details of this case.) 
Fix $i\in I\setminus\{n\}$. We already have $\widetilde{c}_{ji}(0)=0$ for $j\in I$ and $\widetilde{c}_{ji}(-2)=\delta_{ji}$ for $j\in I\setminus\{n\}$. Moreover, by \eqref{eq:rec2}, 
\[
\widetilde{c}_{ni}(-2)=-\widetilde{c}_{ni}(0)+\widetilde{c}_{n-1 i}(-2)+\widetilde{c}_{n-1 i}(0)=\delta_{n-1 i}.
\]
On the other hand, if $k=0$, there is no $j$ which satisfies the condition $(\ast)$. Moreover, if $k=1$, the condition $(\ast)$ is satisfied if and only if  
\[
j=\begin{cases}
i&\text{if}\ i\neq n-1, \\
n-1, n&\text{if}\ i= n-1.
\end{cases}
\]
Hence, when $k=0, 1$, \eqref{leq:explicitinvcarB} holds for $i\in I\setminus\{n\}$. Assume that $2\leq k(\leq n-1)$. By \eqref{eq:rec1}, we have  
\begin{align}
\widetilde{c}_{ji}(-2k)=-\widetilde{c}_{ji}(-2(k-2))+\widetilde{c}_{j+1i}(-2(k-1))+\widetilde{c}_{j-1i}(-2(k-1)), \label{eq:rec1'}
\end{align}
for $j\in I\setminus\{n\}$, here $c_{-1i}(-2(k-1)):=0$, and, by \eqref{eq:rec2}, 
\begin{align}
\widetilde{c}_{ni}(-2k)=-\widetilde{c}_{ni}(-2(k-1))+\widetilde{c}_{n-1i}(-2k)+\widetilde{c}_{n-1i}(-2(k-1)). \label{eq:rec2'}
\end{align}

We consider the following cases : 
\begin{itemize}
\item[(a)] $k+j\leq 2n-i-1$ and $1-i>k-j$, 
\item[(b)] $k+j\leq 2n-i-1$ and $1-i\leq k-j$,
\item[(c)] $k+j\geq 2n-i$. 
\end{itemize}
\noindent Case (a) : By \eqref{eq:rec1'} and \eqref{eq:rec2'}, each $\widetilde{c}_{ji}(-2k)$ is determined by the terms which also satisfy (a).  
Hence, under the condition (a), we obtain $\widetilde{c}_{ji}(-2k)=0$ by  induction on $k$. (When $i=n-1$, there are no pairs $(j, k)$ which satisfy (a) and $k\geq 2$.) 

\noindent Case (b) : Note that in this case $j\leq n-1$. Hence $\widetilde{c}_{ji}(-2k)$'s in this region are determined by \eqref{eq:rec1'}, $\widetilde{c}_{ji}(-2k')=0$ for $k'\in\mathbb{Z}_{\leq 0}$ and $\widetilde{c}_{ji}(-2)=\delta_{ji}$. 

Write $\widetilde{c}_i(k+j, k-j):=\widetilde{c}_{ji}(-2k)$. Then the equality \eqref{eq:rec1'} is the equality of the following form :  
\begin{align}
\widetilde{c}_i(X, Y)=-\widetilde{c}_i(X-2, Y-2)+\widetilde{c}_i(X, Y-2)+\widetilde{c}_i(X-2, Y). \label{eq:rec1''}
\end{align}
Therefore, we can obtain the following by inductive calculation : $\widetilde{c}_{ji}(-2k)=1$ if 
\begin{align*}
&i+1\leq k+j (=X) \leq 2n-i-1\ \text{and}\ 1-i\leq k-j (=Y)\leq i-1\\
&\text{with}\ (k+j)-(i+1), (k-j)-(1-i)\in 2\mathbb{Z}_{\geq 0},
\end{align*} 
and otherwise $\widetilde{c}_{ji}(-2k)=0$ under the condition (b). A picture in Example \ref{e:invcarB-pic} might be helpful for consideration.

\noindent Case (c) : 
Consider the case $k+j=2n-i$. If $j=n$, then, by \eqref{eq:rec2} (or \eqref{eq:rec2'}) and the results in the cases (a) and (b), we have 
\begin{align*}
\widetilde{c}_{ni}(-2(n-i))&=-\widetilde{c}_{ni}(-2(n-i-1))+\widetilde{c}_{n-1i}(-2(n-i))+\widetilde{c}_{n-1i}(-2(n-i-1))\\
&=0+1+0=1. 
\end{align*}
The remaining argument is similar to the case (b) (here we also use the result in the case (b)). The details are left the reader. Consequently, we obtain that  $\widetilde{c}_{ji}(-2k)$ is always equal to $1$ in this case.

To prove $\widetilde{c}_{ji}(-2k)=\widetilde{c}_{ji}(-4n+2+2k)$ for $i\in I\setminus\{n\}$, $j\in I$, and $k=0, 1,\dots, n-1$, we only have to prove them for $k=n-2, n-1$ because of the form of the recurrence relations \eqref{eq:rec1} and \eqref{eq:rec2}, and they are shown by direct calculation via these recurrence relations and \eqref{leq:explicitinvcarB}. 

To prove $\widetilde{c}_{jn}(-2k-1)=\widetilde{c}_{ji}(-4n+3+2k)$ for $j\in I$ and $k=0, 1,\dots, n-1$, we only have to prove them for $k=n-2$ as above, and they are also shown by similar direct calculation. 
\end{proof}
\begin{remark}\label{r:explicitinvcarB}
If we continue the calculation by using \eqref{eq:rec1} and \eqref{eq:rec2}, we can also prove that 
\begin{itemize}
\item $\widetilde{c}_{ji}(-4n)=-\delta_{ji}$ for $i\in I\setminus\{n, n-1\}$ and $j\in I$,
\item $\widetilde{c}_{jn-1}(-4n)=-\delta_{jn-1}-\delta_{jn}$ for $j\in I$.
\end{itemize}
Hence the integers $\widetilde{c}_{ji}(-4n+2), \widetilde{c}_{ji}(-4n)$ coincide with $\widetilde{c}_{ji}(0), \widetilde{c}_{ji}(-2)$ for $i\in I\setminus\{n\}$ and $j\in I$ modulo overall sign. 

Similarly, we can show that  
\begin{itemize}
\item $\widetilde{c}_{jn}(-4n+1)=-\delta_{jn}$ and $\widetilde{c}_{jn}(-4n-1)=-\delta_{jn-1}$ for $j\in I$. 
\end{itemize}
Hence the integers $\widetilde{c}_{jn}(-4n+1), \widetilde{c}_{jn}(-4n-1)$ coincide with $\widetilde{c}_{jn}(-1), \widetilde{c}_{jn}(-3)$ for $j\in I$ modulo overall sign. Therefore, by \eqref{eq:rec1} and \eqref{eq:rec2}, we have 
\[
\widetilde{c}_{ji}(-4n+2+r)=-\widetilde{c}_{ji}(r)
\]
for all $i, j\in I$ and $r\in \mathbb{Z}_{\leq 0}$. This periodicity deduces the general formulas for $\widetilde{c}_{ij}(r)$ from Lemma \ref{l:explicitinvcarB}. We should remark that 
\[
-4n+2=-2h^{\vee}, 
\]
here $h^{\vee}$ is the dual Coxeter number for type $\mathrm{B}_n^{(1)}$ (see Remark \ref{r:convention}). See also \cite[Corollary 2.3]{HL:qGro}. 
\end{remark}

\section{Quantum cluster algebras}\label{a:qclus}
In this appendix, we review the definition of the quantum cluster algebra $\mathcal{A}(\Lambda, \widetilde{B})$, following \cite{BZ:qcluster}. Recall the notation in subsection \ref{ss:qclus}. 

Let $(\Lambda, \widetilde{B})$ a compatible pair. For $k\in J_e$, define $E^{(k)}=(e_{ij})_{i,j\in J}$ and $F^{(k)}=(f_{ij})_{i,j\in J_e}$ as follows : 
\begin{align*}
e_{ij}=\begin{cases}
\delta_{ij} & \text{if}\ j\neq k,\\
-1 & \text{if}\ i=j=k,\\
\max(0,-b_{ik}) & \text{if}\ i\neq j=k,
\end{cases} &  & f_{ij}=\begin{cases}
\delta_{ij} & \text{if}\ i\neq k,\\
-1 & \text{if}\ i=j=k,\\
\max(0,b_{kj}) & \text{if}\ i=k\neq j.
\end{cases}
\end{align*}
Then \emph{the mutation of $(\Lambda,\widetilde{B})$ in direction $k$} transforms $(\Lambda,\widetilde{B})$ into the pair $\mu_{k}(\Lambda,\widetilde{B}):=(\mu_{k}(\widetilde{B}),\mu_{k}(\Lambda))$ given by 
\begin{align*}
\mu_{k}(\Lambda)=(E^{(k)})^{T}\Lambda E^{(k)} &  & \mu_{k}(\widetilde{B})=E^{(k)}\widetilde{B}F^{(k)}.
\end{align*}
Then $\mu_{k}(\Lambda,\widetilde{B})$ is again compatible \cite[Proposition 3.4]{BZ:qcluster}. Moreover $\mu_{k}(\mu_{k}(\Lambda,\widetilde{B}))=(\Lambda,\widetilde{B})$ \cite[Proposition 3.6]{BZ:qcluster}.

\emph{A toric frame} in the skew-field $\mathcal{F}(\Lambda)$ of fraction is a map $M\colon \mathbb{Z}^{J}\to \mathcal{F}\setminus \{0\}$ of the form :  
\[
M(\bm{c})=\tau (X^{\eta(\bm{c})}),
\]
for $\bm{c}\in\mathbb{Z}^J$, here $\tau$ is a $\mathbb{Q}(v^{1/2})$-algebra automorphism of $\mathcal{F}$, and $\eta\colon\mathbb{Z}^J\to \mathbb{Z}^J$ is an isomorphism of $\mathbb{Z}$-modules. 

For a toric frame $M$, define a skew-symmetric $\mathbb{Z}$-bilinear form $\Lambda_M\colon \mathbb{Z}^J\times \mathbb{Z}^J\to \mathbb{Z}$ by $(\bm{c}, \bm{c}')\mapsto \Lambda(\eta(\bm{c}), \eta(\bm{c}'))$. Then, 
\begin{itemize}
\item $M(\bm{c})M(\bm{c}')=v^{\Lambda_M(\bm{c}, \bm{c}')}M(\bm{c}')M(\bm{c})$,
\item $M(0)=1$, $M(\bm{c})^{-1}=M(-\bm{c})$.
\end{itemize}
A pair $(M, \widetilde{B}_M)$ is called \emph{a quantum seed} if 
\begin{itemize}
\item $M$ is a toric frame of $\mathcal{F}$,
\item $\widetilde{B}_M=(b_{ij}^M)_{i\in J, j\in J_e}$ is a $J\times J_e$-integer matrix, and
\item $(\Lambda_M, \widetilde{B}_M)$ is compatible. 
\end{itemize}
The positive integers $\bm{d}\in (\mathbb{Z}_{>0})^{J_e}$ corresponding to the compatible pair $(\Lambda_M, \widetilde{B}_M)$ is denoted by $\bm{d}=(d_i^M)_{i\in J_e}$ (recall the beginning of subsection \ref{ss:qclus}). For $k\in J_e$, define $\mu_k(M)\colon \mathbb{Z}^{J}\to \mathcal{F}\setminus \{0\}$ as  
\begin{itemize}
\item $\displaystyle (\mu_k(M))(\bm{c})=\sum_{s=0}^{c_k}\left[\begin{array}{c}c_k\\s\end{array}\right]_{v^{d_k^M/2}}M(E^{(k)}\bm{c}+s\widetilde{\bm{b}}^{M, k})$, here $\widetilde{\bm{b}}^{M, k}:=(b_{ik}^M)_{i\in J}\in \mathbb{Z}^J$, if $c_k\geq 0$,
\item $\displaystyle (\mu_k(M))(\bm{c})=(\mu_k(M))(-\bm{c})^{-1}$ if $c_k\leq 0$.
\end{itemize}
Then $\mu_k(M)$ is also a toric frame of $\mathcal{F}$, and $(\Lambda_{\mu_k(M)}, \mu_k(\widetilde{B}_M))=\mu_k(\Lambda_M, \widetilde{B}_M)$. Therefore $\mu_k(M, \widetilde{B}_M):=(\mu_k(M), \mu_k(\widetilde{B}_M))$ is also a quantum seed \cite[Proposition 4.7]{BZ:qcluster}. Moreover, $\mu_k(\mu_k(M, \widetilde{B}_M))=(M, \widetilde{B}_M)$ \cite[Proposition 4.10]{BZ:qcluster}. We say that the quantum seed $\mu_k(M, \widetilde{B}_M)$ is obtained from $(M, \widetilde{B}_M)$ by \emph{mutation in direction $k$}. If we set  $X_i^M:=M(\bm{e}_i)$ and ${}'X_i^M:=(\mu_k(M))(\bm{e}_i)$ ($i\in J$), then   
\begin{itemize}
\item ${}'X_i^M=X_i^M$ if $i\neq k$, and
\item ${}'X_k^M=M(-\bm{e}_k-\sum_{j:-b_{jk}>0}b_{jk}\bm{e}_j)+M(-\bm{e}_k+\sum_{j:b_{jk}>0}b_{jk}\bm{e}_j)$.
\end{itemize}

Let $\mathbb{T}_e$ be the $\# J_e$-regular tree whose edges are labelled by $J_e$ so that the $\# J_e$-edges emanating from each vertex receive different labels. Write $t\overset{k}{\text{---}} t'$ if $t, t'\in \mathbb{T}_e$ are joined by an edge labelled by $k\in J_e$.

\emph{A quantum cluster pattern} is an assignment of a quantum seed $(M_t, \widetilde{B}_t)$ to every vertex $t\in \mathbb{T}_e$ such that $\mu_k(M_t, \widetilde{B}_t)=(M_{t'}, \widetilde{B}_{t'})$ whenever $t\overset{k}{\text{---}} t'$. Quantum seeds $(M, \widetilde{B}_M)$, $(M', \widetilde{B}_{M'})$ are said to be \emph{mutation equivalent} if there is a quantum cluster pattern such that $(M_t, \widetilde{B}_t)=(M, \widetilde{B}_M)$ and $(M_{t'}, \widetilde{B}_{t'})=(M', \widetilde{B}_{M'})$ for some $t, t'\in \mathbb{T}_e$. 

Let $\mathcal{S}:=\{(M_t, \widetilde{B}_t)\}_{t\in \mathbb{T}_e}$ be a quantum cluster pattern in $\mathcal{F}$. For $t\in \mathbb{T}_e$ and $j\in J$, write $X_{j; t}=M_t(\bm{e}_j)$, and set $\mathcal{X}:=\{X_{j; t}\mid t\in \mathbb{T}_e, j\in J_e\}$. The elements of $\mathcal{X}$ are called \emph{quantum cluster variables}. The set ${\bf X}_f:=\{M_t(\bm{e}_i)\mid i\in J_f\}$ does not depend on $t\in \mathbb{T}_e$. The elements of ${\bf X}_f$ are called \emph{frozen variables}. A set of elements $\{X_{j; t}\mid j\in J\}$ is called \emph{the cluster at $t$} (the frozen variables are included in the clusters). 

\emph{The quantum cluster algebra} $\mathcal{A}(\mathcal{S})$ is the $\mathbb{Z}[v^{\pm 1/2}]$-subalgebra of $\mathcal{F}$ generated by $\mathcal{X}$ and ${\bf X}_f$ (unlike \cite{BZ:qcluster}, we do not add the inverses of the elements of ${\bf X}_f$ in generators). If there exists $t_0\in \mathbb{T}_e$ such that 
$M_{t_0}(\bm{c})=X^{\bm{c}}$ and $\widetilde{B}_{t_0}=\widetilde{B}$, then $\Lambda_{M_{t_0}}=\Lambda$ and we write $\mathcal{A}(\mathcal{S})$ as $\mathcal{A}(\Lambda, \widetilde{B})$. In this case, we have 
\begin{align}
\overline{M_t(\bm{c})}=M_t(\bm{c})\label{eq:qmonominv}
\end{align}
for all $t\in \mathbb{T}_e$ and $\bm{e}\in \mathbb{Z}^J$ \cite[Proposition 6.2]{BZ:qcluster}. 

Remark that $\mathcal{A}(\mathcal{S})$ is isomorphic to $\mathcal{A}(\Lambda_{M_t}, \widetilde{B}_t)$ for any $t\in \mathbb{T}_e$ in general. 

\begin{theorem}[{The quantum Laurent phenomenon \cite[Corollary 5.2]{BZ:qcluster}}]\label{t:laurentpheno}
The quantum cluster algebra $\mathcal{A}(\mathcal{S})$ associated with a cluster pattern $\mathcal{S}=\{(M_t, \widetilde{B}_t)\}_{t\in \mathbb{T}_e}$ is contained in the quantum torus $\mathcal{T}_{M_t}:=\sum_{\bm{c}\in \mathbb{Z}^J}\mathbb{Z}[v^{\pm 1/2}]M_t(\bm{c})$ for any $t\in \mathbb{T}_e$.
\end{theorem}

\section{Unipotent quantum minors}\label{a:QEA}
In this appendix, we review our definitions of the quantized coordinate algebra $\mathcal{A}_v[N_-]$ and unipotent quantum minors. Recall the notation in subsection \ref{ss:QLA}. See also Remark \ref{r:minorconv}. 

\subsection{Quantized coordinate algebras $\mathcal{A}_v[N_-]$}
\begin{definition}
\label{d:QEA} \emph{The quantized enveloping algebra} $\Uv(=\Uv(\mathrm{X}_N))$ of type $\mathrm{X}_N$ over $\mathbb{Q}(v^{1/2})$ is the unital associative $\mathbb{Q}(v^{1/2})$-algebra defined by the generators 
\[
e_{i},f_{i}, t_i^{\pm 1}\;(i\in I),
\]
and the relations (i)--(iv) below : 
\begin{enumerate}
\item[(i)] $t_it_i^{-1}=1=t_i^{-1}t_i$ and  $t_it_j=t_jt_i$ for $i, j\in I$,
\item[(ii)] $t_ie_{j}=v_i^{c_{ij}}e_{j}t_i,\;t_if_{j}=v_i^{-c_{ij}}f_{j}t_i$
for $i, j\in I$,
\item[(iii)] ${\displaystyle \left[e_{i},f_{j}\right]=\delta_{ij}\frac{t_{i}-t_{i}^{-1}}{q_{i}-q_{i}^{-1}}}$ for $i,j\in I$,
\item[(iv)] ${\displaystyle \sum_{k=0}^{1-c_{ij}}(-1)^{k}\left[\begin{array}{c}
1-c_{ij}\\
k
\end{array}\right]_{v_i}x_{i}^{k}x_{j}x_{i}^{1-c_{ij}-k}=0}$ for $i,j\in I$ with $i\neq j$, and $x=e,f$. 
\end{enumerate}
\end{definition}
The $\mathbb{Q}(v^{1/2})$-subalgebra of $\Uv$ generated by $\{f_{i}\}_{i\in I}$ is denoted by $\Uv^-$. This is isomorphic to the algebra $\Uv^-$ in section \ref{s:dualcan} in an obvious way, hence we do not distinguish them. 

For $i\in I$, we define the $\mathbb{Q}(v^{1/2})$-linear maps $e'_{i}$
and $_{i}e'\colon\Uv^{-}\to\Uv^{-}$ by
\[
\begin{array}{lc}
e'_{i}\left(xy\right)=e'_{i}\left(x\right)y+v_{i}^{\langle h_{i}, \wt x\rangle}xe'_{i}\left(y\right), & e'_{i}(f_{j})=\delta_{ij},\\
_{i}e'\left(xy\right)=v_{i}^{\langle h_{i}, \wt y\rangle}{_{i}e'}\left(x\right)y+x\;{_{i}e'}\left(y\right), & _{i}e'(f_{j})=\delta_{ij}
\end{array}
\]
for homogeneous elements $x,y\in\Uv^{-}$. There exists a unique symmetric $\mathbb{Q}(v^{1/2})$-bilinear form $(-,-)_{L}\colon\Uv^{-}\times\Uv^{-}\to\mathbb{Q}(v^{1/2})$
such that
\begin{align*}
(1,1)_{L}=1 & , &  & (f_{i}x,y)_{L}=\frac{1}{1-v_{i}^{2}}(x,e'_{i}(y))_{L}, &  & (xf_{i},y)_{L}=\frac{1}{1-v_{i}^{2}}(x,{_{i}e'}(y))_{L}.
\end{align*}
Then $(-,-)_{L}$ is non-degenerate. See \cite[Chapter 1]{Lus:intro} for more
properties.

Let $\mathcal{U}_{v, \mathbb{Z}}^-$ be the $\mathbb{Z}[v^{\pm 1/2}]$-subalgebra of $\mathcal{U}_{v}^-$ generated by $\{f_i^{(m)}:=f_i^m/[m]_{v_i}[m-1]_{v_i}\cdots [1]_{v_i}\mid i\in I, m\in\mathbb{Z}_{\geq 0}\}$. Then the quantized coordinate algebra $\mathcal{A}_v[N_-]$ is defined as 
\[
\mathcal{A}_v[N_-]:=\{x\in \mathcal{U}_{v}^-\mid (x, \mathcal{U}_{v, \mathbb{Z}}^-)_L\subset \mathbb{Z}[v^{\pm 1/2}]\}. 
\]

\subsection{Unipotent quantum minors}\label{ss:unipminor}
Let $V$ be a $\Uv$-module. For $\mu \in P$, we set
\[
V_{\mu}:= \{u \in V \mid t_i.u=v_i^{\langle h_i, \mu \rangle}u\ \text{\ for\ all\ } i \in I\}.
\]
This is called \emph{the weight space of $V$ of weight $\mu$}. A $\Uv$-module $V=\bigoplus_{\mu\in P}V_{\mu}$ with weight space decomposition is said to be \emph{integrable} if $e_i$ and $f_i$ act locally nilpotently on $V$ for all $i\in I$. 

For $\lambda\in P_{+}$, denote by $V(\lambda)$ the (finite-dimensional) irreducible highest weight $\Uv$-module generated by a highest weight vector $u_{\lambda}$ of weight $\lambda$. The module $V(\lambda)$ is integrable. There exists a unique $\mathbb{Q}(v^{1/2})$-bilinear form $(-,-)_{\lambda}^{\varphi}\colon V(\lambda)\times V(\lambda)\to\mathbb{Q}(v^{1/2})$ such that
\begin{align*}
 \left(u_{\lambda},u_{\lambda}\right)_{\lambda}^{\varphi} & =1 & (x.u_{1},u_{2})_{\lambda}^{\varphi} & =(u_{1},\varphi(x).u_{2})_{\lambda}^{\varphi}
\end{align*}
for $u_{1},u_{2}\in V(\lambda)$ and $x\in\Uv$, where $\varphi\colon\Uv\to\Uv$ is the $\mathbb{Q}(v^{1/2})$-algebra anti-involution defined by 
\begin{align*}
\varphi\left(e_{i}\right) & =f_{i}, & \varphi(f_{i}) & =e_{i}, & \varphi(t_i) & =t_i
\end{align*}
for $i\in I$. The form $(-,-)_{\lambda}^{\varphi}$ is non-degenerate and symmetric. 

For $w\in W$, define the element $u_{w\lambda}\in V(\lambda)$ by
\begin{align*}
u_{w\lambda}=f_{i_{1}}^{(\langle h_{i_{1}}, s_{i_{2}}\cdots s_{i_{\ell}}\lambda\rangle)}\cdots f_{i_{\ell-1}}^{(\langle h_{i_{\ell-1}}, s_{i_{\ell}}\lambda\rangle)}f_{i_{\ell}}^{(\langle h_{i_{\ell}}, \lambda\rangle)}.u_{\lambda}
\end{align*}
for $(i_{1},\dots,i_{\ell})\in I(w)$. It is known that this element does not depend on the choice of $(i_{1},\dots,i_{\ell})\in I(w)$ and $w\in W$ (depends only on $w\lambda$). See, for example, \cite[Proposition 39.3.7]{Lus:intro}. Then $\left(u_{w\lambda},u_{w\lambda}\right)_{\lambda}^{\varphi}=1$.

For $\lambda\in P_{+}$ and $w,w'\in W$, define an element $D_{w\lambda, w'\lambda}\in\Uv^-$ by the
following property :
\[
(D_{w\lambda,w'\lambda},x)_{L}=(u_{w\lambda},x.u_{w'\lambda})_{\lambda}^{\varphi}
\]
for $x\in\Uv^{-}$. By the nondegeneracy of the bilinear form $(-,-)_L$, this element is uniquely determined. An element of this form is called \emph{a unipotent quantum minor}. Moreover, set
\[
\widetilde{D}_{w\lambda,w'\lambda}:=v^{-(w\lambda-w'\lambda, w\lambda-w'\lambda)/4+(w\lambda-w'\lambda, \rho)/2}D_{w\lambda,w'\lambda}. 
\]
This element is called \emph{a normalized unipotent quantum minor}.

\bibliographystyle{jplain} 
\bibliography{ref}
\def\cprime{$'$} \def\cprime{$'$} \def\cprime{$'$} \def\cprime{$'$}

\end{document}